\documentclass[a4paper,12pt,fleqn]{smfart}

\usepackage{ucs}
\usepackage[latin1]{inputenc}

\usepackage[francais,english]{babel} \AutoSpaceBeforeFDP
\usepackage[T1]{fontenc}
\usepackage{amstext,amssymb,amsthm}
\usepackage{smfthm}
\usepackage{graphicx}
\usepackage{color}
\usepackage{epsfig}
\usepackage[all]{xy}
\usepackage{enumerate}
\usepackage{version}
\usepackage{wasysym}
\usepackage{marvosym}

\usepackage{geometry}
\definecolor{darkred}{rgb}{0.9,0.,.2}
\definecolor{darkblue}{rgb}{0.,0.,.6}
\definecolor{darkgreen}{rgb}{0.,.6,0.1}
\usepackage[colorlinks=true,urlcolor=darkblue,linkcolor=darkred,citecolor=darkgreen]{hyperref}
\usepackage{subfig}

\newcommand{\N}{\mathbb{N}}
\newcommand{\Z}{\mathbb{Z}}

\newcommand{\R}{\mathbb{R}}
\newcommand{\C}{\mathcal{C}} 

\newcommand{\A}{\mathbb{A}}

\renewcommand{\ss}{\mathrm{SL_{n+1}(\mathbb{R})}}

\newcommand{\SO}{\mathrm{SO_{n,1}(\mathbb{R})}}

\newcommand*{\so}[1]{\mathrm{SO}_{#1,1}(\mathbb{R})}
\newcommand*{\s}[1]{\mathrm{SL}_{#1}(\mathbb{R})}

\newcommand{\LG}{\Lambda_{\Gamma}}
\newcommand{\G}{\Gamma}
\newcommand{\g}{\gamma}
\renewcommand{\L}{\Lambda}

\renewcommand{\C}{\mathcal{C}}
\newcommand{\U}{\mathcal{U}}
\newcommand{\F}{\mathcal{F}}

\renewcommand{\H}{\mathcal{H}}

\newcommand{\Nn}{\mathcal{N}}
\newcommand{\V}{\mathcal{V}}

\newcommand{\E}{\mathcal{E}}

\renewcommand{\AA}{\mathcal{A}}

\renewcommand{\O}{\Omega}
\newcommand{\Og}{\mathcal{O}_{\Gamma}}
\newcommand{\dO}{\partial \Omega}
\renewcommand{\d}{d_{\Omega}}

\renewcommand{\S}{\mathbb{S}^n}
\renewcommand{\P}{\mathcal{P}}
\newcommand{\D}{\mathcal{D}}
\renewcommand{\A}{\mathbb{A}}

\newcommand{\PP}{\mathbb{P}}

\newcommand{\Quo}{\Omega/\!\raisebox{-.90ex}{\ensuremath{\Gamma}}}
\newcommand*{\Quotient}[2]{\ensuremath{#1/\!\raisebox{-.90ex}{\ensuremath{#2}}}}

\newcommand{\Qug}{\Og /\!\raisebox{-.65ex}{\ensuremath{\Gamma}}}

\newcommand{\Aut}{\textrm{Aut}}

\newcommand{\Vol}{\textrm{Vol}}

\newcommand{\Cone}{\textrm{C\^one}}
\newcommand{\Fix}{\textrm{Fix}}
\newcommand{\Stab}{\textrm{Stab}}

\newcommand{\Isom}{\textrm{Isom}}

\newcommand{\PGL}{\mathrm{PGL}}
\newcommand{\SL}{\mathrm{SL}}
\newcommand{\Hc}{\mathcal{H}}
\newcommand{\Cc}{\mathcal{C}}
\newcommand{\Pc}{\mathcal{P}}
\newcommand{\Dc}{\mathcal{D}}
\newcommand{\Hb}{\mathbb{H}}
\newcommand{\Ec}{\mathcal{E}}
\newcommand{\Ab}{\mathbb{A}}

\renewcommand{\leq}{\leqslant}
\renewcommand{\geq}{\geqslant}

\newcommand{\triple}{\mathrm{v\hspace{-.9mm}v\hspace{-.9mm}v}}
\newcommand{\double}{\mathrm{v\hspace{-.9mm}v}}
\newcommand{\simple}{\mathrm{v}}
\usepackage{mathtools}
\usepackage{nicefrac}

\makeatletter
\def\namedlabel#1#2{\begingroup
	#2%
	\def\@currentlabel{#2}%
	\phantomsection\label{#1}\endgroup
}
\makeatother
\usepackage{tikz}
\usepackage{tikz-cd}
\usepackage{fp}
\usetikzlibrary{calc}
\usetikzlibrary{shapes,arrows}
\usetikzlibrary{fit,positioning}
\usetikzlibrary{patterns,decorations.pathreplacing}
\makeatletter
\def\calcLength(#1,#2)#3{%
	\pgfpointdiff{\pgfpointanchor{#1}{center}}%
	{\pgfpointanchor{#2}{center}}%
	\pgf@xa=\pgf@x%
	\pgf@ya=\pgf@y%
	\FPeval\@temp@a{\pgfmath@tonumber{\pgf@xa}}%
	\FPeval\@temp@b{\pgfmath@tonumber{\pgf@ya}}%
	\FPeval\@temp@sum{(\@temp@a*\@temp@a+\@temp@b*\@temp@b)}%
	\FProot{\FPMathLen}{\@temp@sum}{2}%
	\FPround\FPMathLen\FPMathLen5\relax
	\global\expandafter\edef\csname #3\endcsname{\FPMathLen}
}
\makeatother

\theoremstyle{plain}
\newtheorem{qu}{Question}
\newtheorem*{fait}{Fait}
\newtheorem{theorem}{Theorem}[subsection]
\newtheorem{lemma}[theorem]{Lemma}
\newtheorem{fact}[theorem]{Fact}
\newtheorem{corollary}[theorem]{Corollary}
\newtheorem{proposition_english}[theorem]{Proposition}

\theoremstyle{definition}

\newtheorem{definition_english}[theorem]{Definition}

\theoremstyle{remark}
\newtheorem*{nota}{Notations}
\newtheorem{remark}[theorem]{Remark}

\title{Finitude g\'eom\'etrique en g\'eom\'etrie de Hilbert}

\author{
	Pierre-Louis Blayac
}
\address{Université de Strasbourg, IRMA, Strasbourg, France}
\email{blayac@unistra.fr}

\author{
	Mickaël Crampon
}
\author{
	Ludovic Marquis
}
\address{Université de Rennes, CNRS, IRMAR - UMR 6625, F-35000 Rennes, France}
\email{ludovic.marquis@univ-rennes.fr}

 \dedicatory{
Avec un appendice \'ecrit avec 
	Constantin Vernicos
\\
And an Erratum/addendum \ref{errata_part} written in english by 
	Pierre-Louis Blayac
 and
 	Ludovic Marquis
concatenated here as an independent paper
 }

\date{} 

\address{\newline{}}

\begin{document}
\frontmatter





\begin{abstract}
On \'etudie la notion de finitude g\'eom\'etrique pour certaines g\'eom\'etries de Hilbert d\'efinies par un ouvert strictement convexe \`a bord de classe $\C^1$.\\
La d\'efinition dans le cadre des espaces Gromov-hyperboliques fait intervenir l'action du groupe discret consid\'er\'e sur le bord de l'espace. On montre, en construisant explicitement un contre-exemple, que cette d\'efinition doit \^etre renforc\'ee pour obtenir des d\'efinitions \'equivalentes en termes de la g\'eom\'etrie de l'orbifold quotient, similaires \`a celles obtenues par Brian Bowditch \cite{MR1218098} en g\'eom\'etrie hyperbolique.\\

\end{abstract}

\begin{altabstract}
We study the notion of geometrical finiteness for those Hilbert geometries defined by strictly convex sets with $\C^1$ boundary.\\
In Gromov-hyperbolic spaces, geometrical finiteness is defined by a property of the group action on the boundary of the space. We show by constructing an explicit counter-example that this definition has to be strenghtened in order to get equivalent characterizations in terms of the geometry of the quotient orbifold, similar to those obtained by Brian Bowditch \cite{MR1218098} in hyperbolic geometry.
\end{altabstract}

\maketitle

\scriptsize
\tableofcontents
\normalsize
\frontmatter

\section{Introduction}

La notion de finitude g\'eom\'etrique a suscit\'e l'int\'er\^et de nombreux g\'eom\`etres dans l'\'etude des groupes klein\'eens de dimension $3$. On peut notamment citer Leon Greenberg \cite{MR0200446}, Lars Ahlfors \cite{MR0194970}, Albert Marden \cite{MR0217287,MR0349992}, Alan Beardon and Bernard Maskit \cite{MR0333164}, William Thurston \cite{CoursdeThurston}. Six d\'efinitions \'equivalentes avaient alors \'et\'e introduites, parmi lesquelles subsistent seulement cinq en dimension sup\'erieure. Il revient \`a Brian Bowditch, dans une \'etude tr\`es d\'etaill\'ee \cite{MR1218098}, d'avoir effectu\'e cette extension \`a la dimension quelconque. Dans ce travail, Bowditch discute \'egalement de façon tr\`es compl\`ete le probl\`eme essentiel de l'existence d'un domaine fondamental ayant un nombre fini de faces: il s'agit de la sixi\`eme d\'efinition de finitude g\'eom\'etrique, qui n'est plus \'equivalente aux autres en dimension sup\'erieure ou \'egale à 4.\\

Dans \cite{MR1317633}, Bowditch \'etend ces consid\'erations \`a la courbure n\'egative pinc\'ee. Dans ce texte, nous nous proposons d'\'etudier, comme promis dans \cite{Crampon:2011fk}, ce qu'il se passe dans le cadre des g\'eom\'etries de Hilbert. \\

Les g\'eom\'etries de Hilbert peuvent \^etre vues comme des g\'en\'eralisations de la g\'eom\'etrie hyperbolique, dont la d\'efinition se base sur le mod\`ele de Beltrami-Klein: il s'agit d'un espace m\'etrique $(\O,\d)$, o\`u $\O$ est un ouvert proprement convexe de l'espace projectif $\PP^n$ et $\d$ la distance d\'efinie par

$$
\begin{array}{ccc}
d_{\O}(x,y) = \frac{1}{2}\ln \big([p:x:y:q]\big) & \textrm{et} &
d_{\O}(x,x)=0,\ x,y\in\O,\ x\not=y.
\end{array}
$$
Ici, $p$ et $q$ sont les points d'intersection de la droite $(xy)$ et du bord $\partial \O$ de $\O$ tels que $x$ soit entre $p$ et $y$, et $y$ soit entre $x$ et $q$ (voir figure \ref{dis}). Par proprement convexe, nous voulons dire que l'adh\'erence de $\O$ dans $\PP^n$ \'evite au moins un hyperplan projectif; de fa\c con \'equivalente, il existe une carte affine dans laquelle $\O$ appara\^it comme un ouvert convexe relativement compact.\\

\begin{center}
\begin{figure}[h!]
  \centering
\includegraphics[width=5cm]{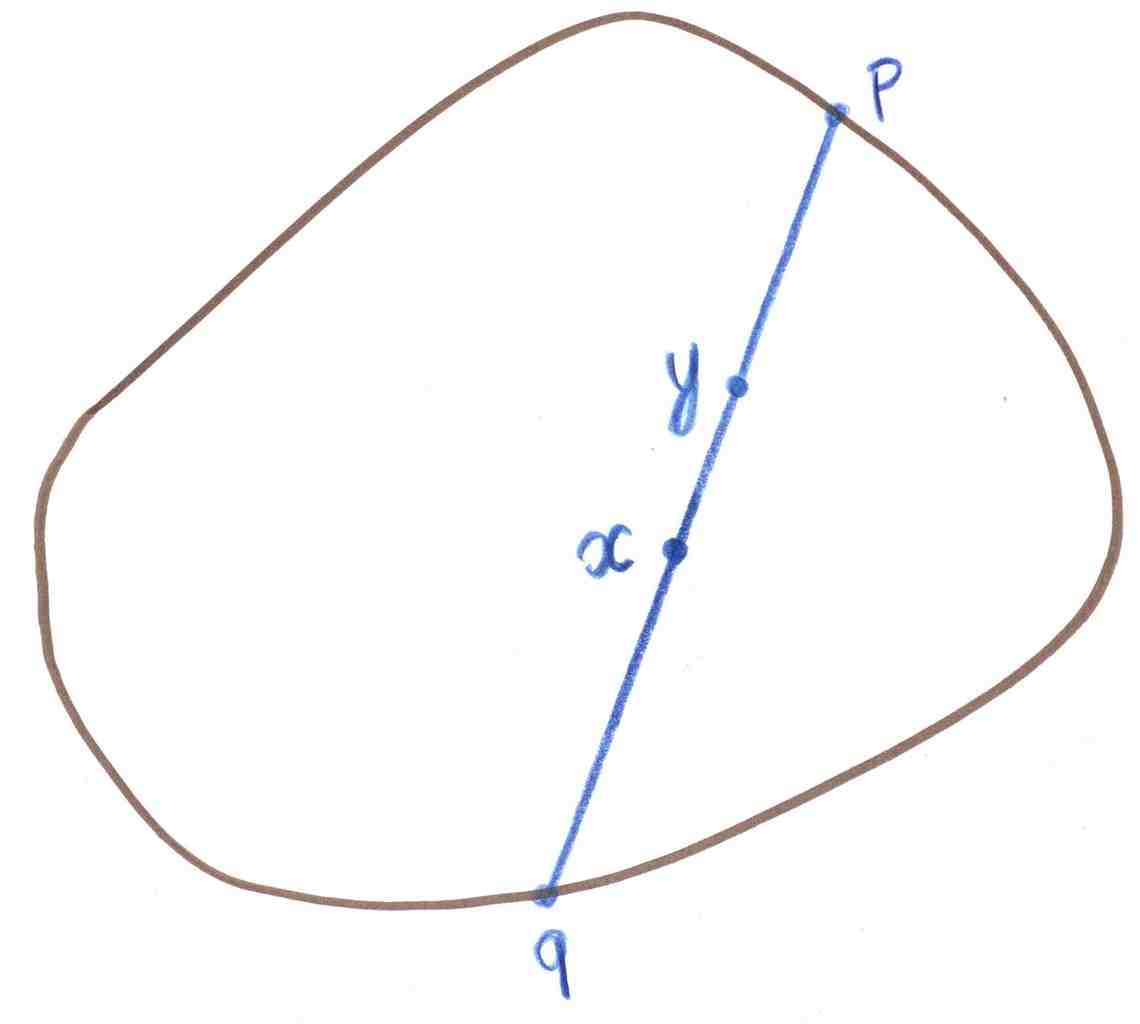}
\caption{La distance de Hilbert}
\label{dis}
\end{figure}
\end{center}

Ces g\'eom\'etries furent introduites par Hilbert comme exemples d'espaces dans lesquelles les droites sont des g\'eod\'esiques. Ce qui nous int\'eresse ici, c'est l'\'etude des quotients d'une g\'eom\'etrie de Hilbert donn\'ee. Il faut noter tout de suite que le groupe des transformations projectives pr\'eservant le convexe $\O$, que nous noterons $\Aut(\O)$, est un sous-groupe du groupe d'isom\'etries $\Isom(\O,\d)$. On ne sait pas en g\'en\'eral dire si ces deux groupes co\"incident. Par contre, il n'est pas difficile de voir que c'est le cas d\`es que $\O$ est strictement convexe: c'est une cons\'equence du fait que les droites sont alors les seule g\'eod\'esiques. Sinon, les seuls cas compris sont ceux des poly\`edres: si $\O$ est un poly\`edre, $\Aut(\O)=\Isom(\O,\d)$ si et seulement si $\O$ n'est pas un simplexe; lorsque $\O$ est un simplexe, $\Aut(\O)$ est d'indice $2$ dans $\Isom(\O,\d)$.\\

Les g\'eom\'etries de Hilbert ont connus un regain d'int\'er\^et dans les, disons, deux derni\`eres d\'ecennies. Pour ce qui va nous concerner ici, il convient de citer dans les r\^oles principaux William Goldman et Yves Benoist. L'article \cite{MR1053346} de Goldman de 1990 est consacr\'e aux surfaces compactes projectives convexes, autrement dit aux quotients compacts d'une g\'eom\'etrie de Hilbert plane. Yves Benoist s'est lui int\'eress\'e \`a la situation bien plus g\'en\'erale d'un sous-groupe discret de $\ss$ pr\'eservant un ouvert proprement convexe de $\PP^n$ \cite{MR1767272}; il a ensuite clarifi\'e, dans sa s\'erie d'articles sur les convexes \emph{divisibles}\footnote{Un ouvert proprement convexe est dit \emph{divisible} lorsqu'il existe un sous-groupe discret de $\Aut(\O)$ tel que $\Quo$ soit compact. On dit alors que le groupe $\G$ divise le convexe $\O$.} \cite{MR2094116,MR2010735,MR2195260,MR2218481}, le cas des quotients compacts d'une g\'eom\'etrie $(\O,\d)$ par un sous-groupe discret de $\
Aut(\O)$. Dans les deux cas, notons que les auteurs restent tributaires de travaux des ann\'ees 60, notamment ceux de Benz\'ecri, Kac, Koszul et Vinberg \cite{MR0124005,MR0239529,MR0208470}.\\

Parmi les convexes divisibles, l'ellipso\"ide, qui d\'efinit une g\'eom\'etrie hyperbolique, est un cas bien \`a part. En fait, un th\'eor\`eme d'\'Edith Soci\'e-M\'ethou affirme que, d\`es que le bord du convexe $\O$ est de classe $\C^2$ \`a hessien d\'efini positif, le groupe d'isom\'etries de $(\O,\d)$ est compact, sauf si, bien s\^ur, c'est un ellipso\"ide \cite{MR1981171}. Un des accomplissements des auteurs pr\'ec\'edents est bien d'avoir montr\'e qu'il existe malgr\'e tout de nombreux autres convexes divisibles. Le premier exemple avait \'et\'e donn\'e par Kac et Vinberg dans les ann\'ees 60 \cite{MR0208470}. En dimension $2$, le r\'esultat de Goldman est quantitatif: l'espace des structures projectives convexes non \'equivalentes sur une surface de genre $g\geqslant 2$ est hom\'eomorphe \`a $\R^{16g-16}$, alors que l'espace des structures hyperboliques non \'equivalentes est lui hom\'eomorphe \`a $\R^{6g-6}$. En dimension plus grande, on ne dispose que de th\'eor\`emes d'existence: d'une part, il est 
possible dans certains cas, par des techniques de pliage, de d\'eformer contin\^ument une structure hyperbolique en une structure projective convexe; d'autre part, il existe des exemples de quotients exotiques \cite{MR2295544,MR2350468}, c'est-\`a-dire de vari\'et\'es compactes projectives strictement convexes, qui n'admettent pas de structure hyperbolique. L'\'etude quantitative de la dimension $2$ et la construction d'exemples par pliage de structures hyperboliques ont \'et\'e g\'en\'eralis\'ees au cas du volume fini par le second auteur \cite{MR2740643,Marquis:2010fk}.\\

Jusque-l\`a, sans le savoir, nous n'avons parl\'e que de situations dans lesquelles l'ouvert convexe est strictement convexe. Rappelons le r\'esultat suivant:
\begin{theo}[Benoist \cite{MR2094116}]\label{benoistmain}
Soit $\O$ un convexe divisible, divis\'e par $\G \leqslant \Aut(\O)$. Les propositions suivantes sont \'equivalentes.
\begin{enumerate}[(i)]
 \item L'ouvert $\O$ est strictement convexe;
 \item Le bord $\dO$ est de classe $\C^1$;
 \item L'espace m\'etrique $(\O,\d)$ est Gromov-hyperbolique;
 \item Le groupe $\G$ est Gromov-hyperbolique.
\end{enumerate}
\end{theo}
Ce th\'eor\`eme a \'et\'e \'etendu dans la pr\'epublication \cite{Cooper:2011fk} au cas des convexes \emph{quasi-divisibles}, c'est-\`a-dire ayant un quotient de volume fini. Cette claire dichotomie ne peut plus exister pour des quotients plus g\'en\'eraux et nous allons voir pourquoi.\\

Dans \cite{MR1767272}, Benoist explique que, sous des hypoth\`eses minimales, si un sous-groupe discret $\G$ de $\ss$ pr\'eserve un ouvert proprement convexe de $\PP^n$, alors il pr\'eserve un convexe minimal $\O_{min}$ et un maximal $\O_{max}$: tout ouvert proprement convexe pr\'eserv\'e par $\G$ est coinc\'e entre $\O_{min}$ et $\O_{max}$. Le convexe $\O_{min}$ n'est rien d'autre que l'int\'erieur de l'enveloppe convexe $C(\LG)$ de l'ensemble limite $\LG$ de $\G$, d\'efini comme l'adh\'erence des points attractifs des \'el\'ements proximaux de $\G$. Le convexe $\O_{max}$ se d\'eduit par dualit\'e du convexe minimal de l'action duale de $\G$.\\
Prenons l'exemple simple d'un sous-groupe discret $\G$ d'isom\'etries du plan hyperbolique $\E=\mathbb{H}^2$. Dans ce cas, l'ensemble limite peut \^etre d\'efini dynamiquement comme l'ensemble $\LG = \overline{\G.o}\smallsetminus \G.o \subset \partial\E$, o\`u $o$ est un point quelconque de $\E$. Lorsque l'action du groupe $\G$ est cocompacte ou de volume fini, l'ensemble limite est pr\'ecis\'ement $\partial\E$ tout entier. Cette propri\'et\'e va en fait rester vraie pour une g\'eom\'etrie de Hilbert d\'efinie par un ouvert strictement convexe $\O$: on aura ainsi $\O_{min}=\O_{max}=\O$, autrement dit que $\G$ ne pr\'eserve pas d'autre ouvert proprement convexe que $\O$. C'est la propri\'et\'e essentielle, que l'on va perdre pour un quotient g\'en\'eral, qui permet d'obtenir le th\'eor\`eme \ref{benoistmain}.\\

\begin{center}
\begin{figure}[h!]\label{trompet}
  \centering
\includegraphics[width=7cm]{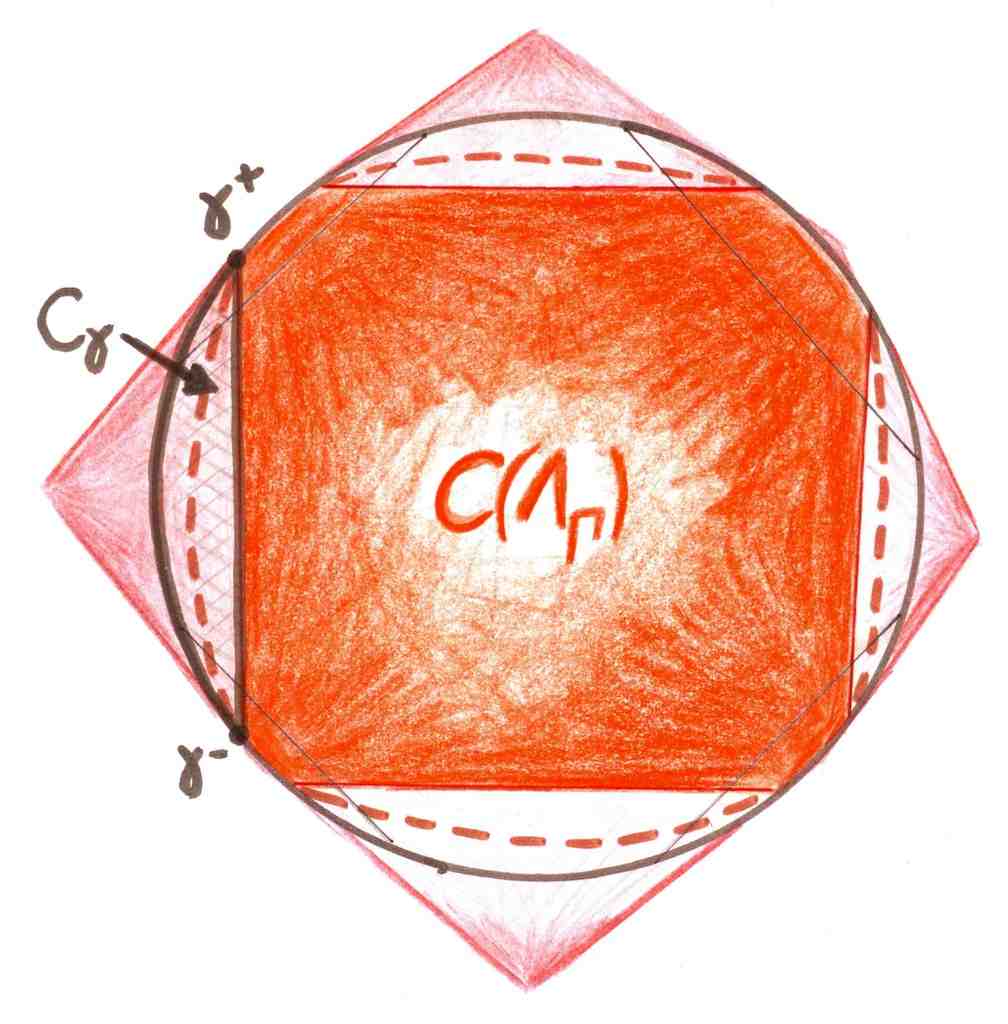}\ \includegraphics[width=7cm]{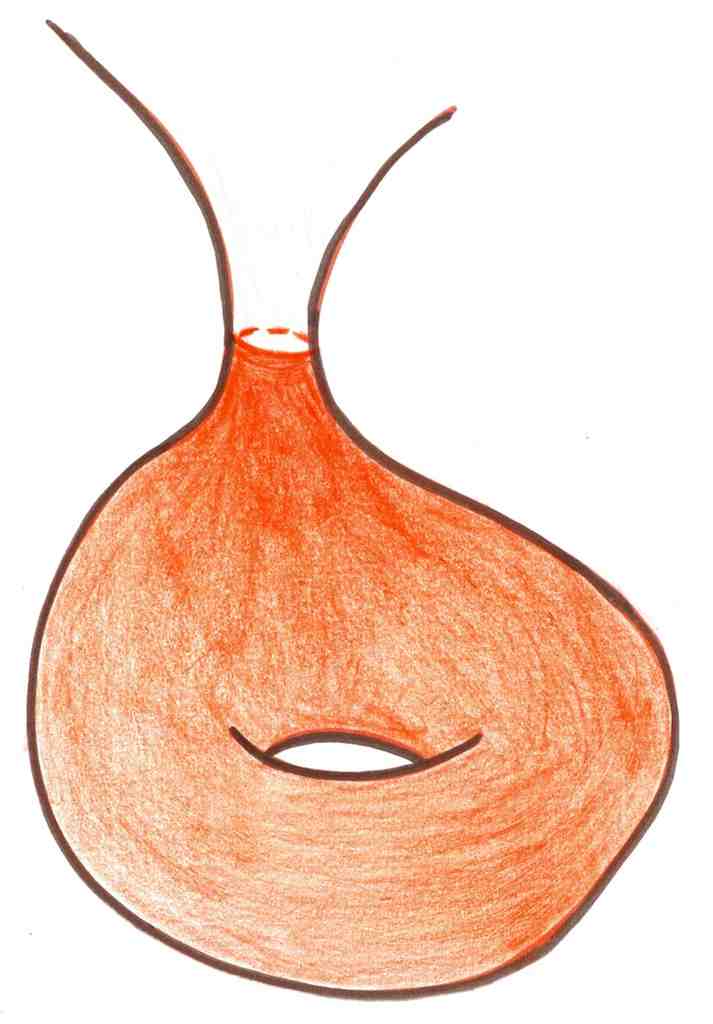}
\caption{Convexe-cocompact}

\end{figure}
\end{center}

En effet, consid\'erons maintenant un sous-groupe convexe cocompact $\G$, dont, disons, le quotient $\E/\G$ est une surface de genre $1$ avec une pointe (qui ressemble \`a une trompette). Le groupe fondamental de la pointe, isomorphe \`a $\Z$, est repr\'esent\'e par un \'el\'ement hyperbolique $\g\in\G$ qui correspond \`a la g\'eod\'esique ferm\'ee \`a la base de la trompette. L'ensemble $C(\LG)$ est un ouvert convexe qui n'est ni strictement convexe ni \`a bord $\C^1$. L'ensemble $\partial\E\smallsetminus C(\LG)$ est exactement l'orbite sous $\G$ de l'ouvert convexe $C_{\g}$ d\'elimit\'e par l'axe $(\g^+\g^-)$ de $\g$ et l'arc de cercle dans $\partial\E$ reliant $\g^+$ \`a $\g^-$. Il n'est pas difficile de voir qu'on peut modifier la partie circulaire du bord de $C_{\g}$ par une courbe $\g$-invariante de telle fa\c con que le convexe que cette courbe d\'efinit avec l'axe de $\g$ aient les propri\'et\'es que l'on veut: strictement convexe mais pas \`a bord $\C^1$, \`a bord $\C^1$ mais pas strictement convexe. 
En copiant cette ``pi\`ece'' via $\G$, on peut ainsi voir que le groupe $\G$ agit sur des ouverts convexes aux caract\'eristiques bien diff\'erentes, et qu'ainsi on ne peut esp\'erer un r\'esultat du type du th\'eor\`eme \ref{benoistmain}. Ce qu'il est raisonnable de se demander toutefois, ce serait:

\begin{qu}\label{quest}
Soit $\G$ un sous-groupe discret (irr\'eductible) de $\ss$. A t-on \'equivalence entre les points suivants?
\begin{enumerate}[(i)]
 \item $\G$ pr\'eserve un ouvert strictement convexe;
 \item $\G$ pr\'eserve un ouvert convexe \`a bord $\C^1$;
 \item $\G$ pr\'eserve un ouvert strictement convexe \`a bord $\C^1$.
\end{enumerate}
\end{qu}

C'est charm\'e par cette id\'ee que nous allons dans cet article ne consid\'erer que des g\'eom\'etries de Hilbert d\'efinies par un ouvert strictement convexe \`a bord $\C^1$. Il est plus agr\'eable de travailler avec une telle g\'eom\'etrie, plus proche de la g\'eom\'etrie hyperbolique, et une r\'eponse affirmative \`a la question pr\'ec\'edente permettrait, pour des probl\`emes ne d\'ependant que du groupe $\G$, de se ramener \`a un tel convexe. De plus, ce sont des hypoth\`eses essentielles pour les r\'esultats g\'eom\'etriques de cet article, ainsi que pour l'article \cite{Mickael:2012fk}  \`a venir, dans lequel nous \'etudierons la dynamique du flot g\'eod\'esique sur certains quotients g\'eom\'etriquement finis.\\

M\^eme si on r\'epondait affirmativement \`a la question \ref{quest}, cela laisserait de c\^ot\'e les cas o\`u les deux propri\'et\'es, stricte convexit\'e et bord $\C^1$, tombent en d\'efaut. En fait, lorsque le convexe n'est ni strictement convexe ni \`a bord $\C^1$, le sentiment tr\`es na\"if est que la g\'eom\'etrie de Hilbert qu'il d\'efinit a plus \`a voir avec la g\'eom\'etrie riemanienne de courbure n\'egative \emph{ou nulle} qu'avec la g\'eom\'etrie hyperbolique. Or, en g\'eom\'etrie riemanienne de courbure n\'egative ou nulle, les situations peuvent \^etre tr\`es diverses et c'est chose peu ais\'ee voire vaine que de se mettre d'accord sur une notion de finitude g\'eom\'etrique.\\

\subsection{Pr\'esentation des r\'esultats}

Il nous a fallu un certain temps avant de trouver la bonne d\'efinition de la notion de finitude g\'eom\'etrique. Parmi les d\'efinitions \'equivalentes de Bowditch, celle qui semblait la plus simple et directe \`a adapter \'etait la suivante: l'action d'un sous-groupe discret $\G$ sur $\O$ est g\'eom\'etriquement finie si tout point de l'ensemble limite est soit conique soit parabolique born\'e. C'est d'ailleurs la d\'efinition que l'on retrouve dans les travaux post\'erieurs de Bowditch et Yaman qui concernent des espaces plus g\'en\'eraux que les vari\'et\'es riemanniennes de courbure n\'egative.\\
Malheureusement, nous n'arrivions par \`a montrer que les autres d\'efinitions de Bowditch, qui font intervenir plus directement la g\'eom\'etrie du quotient, \'etaient \'equivalentes \`a la pr\'ec\'edente. Nous n'y parvenions qu'en faisant une hypoth\`ese suppl\'ementaire sur les points paraboliques, que l'on devait supposer \emph{uniform\'ement born\'es} (d\'efinition \ref{def_uni_born}).
%

Le r\'esultat est alors le suivant (on se reportera au texte pour les d\'efinitions, parties \ref{par_geo_fini} et \ref{section_decomposition} essentiellement; elles sont similaires \`a celles qu'on trouve en g\'eom\'etrie hyperbolique).

\begin{theo}[Th\'eorème \ref{theo_geo_finie}]\label{main1}
Soient $\O$ un ouvert proprement convexe, strictement convexe et \`a bord $\C^1$, $\G$ un sous-groupe discret de $\Aut(\O)$, et $M=\Quo$ le quotient correspondant. Les propositions suivantes sont \'equivalentes:
\begin{enumerate}
\item[(GF)] Tout point de $\LG$ est soit un point limite conique soit un point parabolique \underline{uniform\'ement} born\'e;
\item[(TF)] Le quotient $\Quotient{\Og}{\G}$ est une orbifold à bord qui est l'union d'un compact et d'un nombre fini de projections de r\'egions paraboliques standards disjointes;
\item[(PEC)] La partie \'epaisse du c\oe ur convexe de $M$ est compacte;
\item[(PNC)] La partie non cuspidale du c\oe ur convexe de $M$ est compacte;
\item[(VF)] Le $1$-voisinage du c\oe ur convexe de $\Quo$ est de volume fini et le groupe $\G$ est de type fini.
\end{enumerate}
En particulier, un tel quotient est \emph{sage} (c'est-à-dire est l'int\'erieur d'une orbifold compacte à bord) et par suite le groupe $\G$ est de pr\'esentation finie.
\end{theo}

Nous avons longtemps pens\'e que cet \'ecart \'etait d\^u \`a une d\'efaillance de notre part, et qu'on devrait pouvoir enlever l'hypoth\`ese d'uniformit\'e. En fait, il s'av\`ere que non:

\begin{prop}[Proposition \ref{contrex_geo_finie}]\label{contreex}
Il existe un ouvert proprement convexe $\O \subset \PP^4$, strictement convexe \`a bord $\C^1$, qui admet une action d'un sous-groupe $\G$ de $\Aut(\O)$ telle que tout point de $\LG$ est soit conique soit parabolique born\'e, mais pas uniform\'ement born\'e.
\end{prop}

C'est ce qui nous a amen\'e \`a introduire les d\'efinitions suivantes de finitude g\'eom\'etrique, qui respectent les terminologies introduites jusque-l\`a (voir partie \ref{par_geo_fini}):

\begin{defi}
Soient $\O$ un ouvert proprement convexe, strictement convexe et \`a bord $\C^1$ et $\G$ un sous-groupe discret de $\Aut(\O)$.
\begin{itemize}
 \item L'action de $\G$ est dite g\'eom\'etriquement finie \underline{sur $\dO$} si tout point de l'ensemble limite est soit conique soit parabolique born\'e.
 \item L'action de $\G$ est dite g\'eom\'etriquement finie \underline{sur $\O$} si tout point de l'ensemble limite est soit conique soit parabolique \emph{uniform\'ement} born\'e. On dira aussi dans ce cas que le quotient $\Quo$ est g\'eom\'etriquement fini.
\end{itemize}
\end{defi}

Parmi les actions g\'eom\'etriquement finies, on distingue celles qui sont de covolume fini, et qui avaient d\'ej\`a \'et\'e \'etudi\'ees, notamment de fa\c con compl\`ete en dimension $2$, par le second auteur:

\begin{coro}[Corollaire \ref{theo_vol_fini}]
Soient $\O$ un ouvert proprement convexe de $\PP^n$, strictement convexe et \`a bord $\C^1$ et $\G$ un sous-groupe discret de $\Aut(\O)$. Les propositions suivantes sont \'equivalentes:
\begin{itemize}
 \item l'action de $\G$ sur $\O$ est de covolume fini;
 \item l'action de $\G$ sur $\O$ est g\'eom\'etriquement finie et $\LG=\dO$;
 \item l'action de $\G$ sur $\dO$ est g\'eom\'etriquement finie et $\LG=\dO$.
\end{itemize}
\end{coro}

D'autres r\'esultats apparaissent au fil du texte. Une grande partie de notre travail a consist\'e \`a comprendre les bouts d'un quotient g\'eom\'etriquement fini, autrement dit les sous-groupes paraboliques qui apparaissent; c'est le

\begin{theo}[Corollaire \ref{para_cas_gene}]\label{para}
Soient $\O$ un ouvert proprement convexe de $\PP^n$, strictement convexe et \`a bord $\C^1$ et $\G$ un sous-groupe discret de $\Aut(\O)$. Si $p$ est un point parabolique uniform\'ement born\'e de $\LG$ de stabilisateur $\P=Stab_{\G}(p)$, alors le groupe $\P$ est conjugu\'e dans $\ss$ \`a un sous-groupe parabolique de $\SO$. En particulier, le groupe $\P$ est virtuellement isomorphe \`a $\Z^{d}$, o\`u $1 \leqslant d \leqslant n-1$ est sa dimension cohomologique virtuelle.
\end{theo}

Ce r\'esultat permet d'adapter une d\'emonstration de Benoist dans \cite{MR1767272} pour obtenir le th\'eorème suivant, qu'on trouve dans \cite{MR1767272} dans le cas o\`u l'action du groupe est cocompacte:

\begin{theo}[Th\'eorème \ref{zariski-dense}]\label{zaza}
Soit $\G$ un sous-groupe discret \emph{irr\'eductible} de $\Aut(\O)$. Si  $\G$ contient un sous-groupe parabolique uniform\'ement born\'e de dimension cohomologique $n-1$ ou $n-2$, alors l'adh\'erence de Zariski de $\G$ est soit $\ss$ soit conjugu\'ee à $\SO$.
\end{theo}

Ce th\'eor\`eme tombe en d\'efaut d\`es que l'ensemble limite ne contient pas de points paraboliques, ou que ceux-ci ne sont pas uniform\'ement born\'es. Ces contre-exemples sont directement reli\'es \`a celui que l'on construit dans la proposition \ref{contreex}.\\

On se rappelle que dans le th\'eor\`eme \ref{benoistmain} de Benoist, l'existence d'un quotient compact pour un ouvert strictement strictement convexe implique que la g\'eom\'etrie de Hilbert qu'il d\'efinit \'etait Gromov-hyperbolique, tout comme le groupe cocompact impliqu\'e. Voici le pendant de ce r\'esultat dans notre cas:

\begin{theo}[Theorème \ref{eqgromovhyp}]\label{gro}
Soient $\O$ un ouvert proprement convexe de $\PP^n$, strictement convexe et \`a bord $\C^1$ et $\G$ un sous-groupe discret de $\Aut(\O)$. Si l'action de $\G$ sur $\O$ est g\'eom\'etriquement finie, alors l'espace m\'etrique $(C(\LG),\d)$ est Gromov-hyperbolique et le groupe $\G$ est relativement hyperbolique par rapport \`a ses sous-groupes paraboliques maximaux.
\end{theo}

Remarquons bien s\^ur que l'espace m\'etrique  $(C(\LG),d_{C(\LG)})$ n'est pas en g\'en\'eral Gromov-hyperbolique. En fait, ce sera le cas seulement lorsque $\LG=\dO$:

\begin{coro}[Corollaire \ref{ghyp_vol_fini}]
Soit $\O$ un ouvert proprement convexe de $\PP^n$, strictement convexe et \`a bord $\C^1$. Si $\O$ admet une action de covolume fini, alors l'espace m\'etrique $(\O,\d)$ est Gromov-hyperbolique.
\end{coro}

En ce qui concerne la recherche d'une action g\'eom\'etriquement finie sur $\dO$ qui ne le serait pas sur $\O$, il convient de noter tout de suite les restrictions suivantes, qui donnent des informations sur le type d'espaces et de groupes que l'on obtient dans de tels exemples:

\begin{theo}[Proposition \ref{bilan_geo}]
Soient $\O$ un ouvert proprement convexe de $\PP^n$, strictement convexe et \`a bord $\C^1$ et $\G$ un sous-groupe discret de $\Aut(\O)$. Les propositions suivantes sont \'equivalentes:
\begin{enumerate}[(i)]
 \item l'action de $\G$ sur $\O$ est g\'eom\'etriquement finie;
 \item l'action de $\G$ sur $\dO$ est g\'eom\'etriquement finie et les sous-groupes paraboliques de $\G$ sont conjugu\'es à des sous-groupes paraboliques de $\SO$;
 \item l'action de $\G$ sur $\dO$ est g\'eom\'etriquement finie et l'espace m\'etrique $(C(\LG),\d)$ est Gromov-hyperbolique.
\end{enumerate}
En particulier, si $n=2$ ou $3$, l'action de $\G$ est g\'eom\'etriquement finie sur $\O$ si et seulement si elle l'est sur $\dO$.
\end{theo}

Disons quelques mots sur l'exemple que nous donnons pour affirmer la proposition \ref{contreex}. Il s'agit de consid\'erer la repr\'esentation sph\'erique de $\s2$ dans $\R^5$, dont l'ensemble limite dans $\PP^4$ est la courbe Veronese. Nous prouvons que cette action de $\s2$ sur $\PP^4$ pr\'eserve une famille $\{\O_r\}_{r\in[0,+\infty]}$ d'ouverts proprement convexes qui, hormis $\O_0$ et $\O_{\infty}$, sont tous strictement convexes \`a bord $\C^1$. Tous nos exemples proviennent alors des images par cette repr\'esentation de r\'eseaux de $\s2$.

\subsection{Autres travaux sur le sujet}

Personne encore ne s'\'etait encore int\'eress\'e \`a la notion de finitude g\'eom\'etrique en g\'eom\'etrie de Hilbert mais d'autres travaux ont \'et\'e propos\'es r\'ecemment sur les quotients non compacts de g\'eom\'etries de Hilbert. Nous faisons principalement r\'ef\'erence \`a l'article de Suhyoung Choi \cite{Choi:2010fk} et celui de Daryl Cooper, Darren Long et Stephan Tillmann \cite{Cooper:2011fk}.\\

Choi a une approche diff\'erente qui consiste \`a partir de la vari\'et\'e ou de l'orbifold et de chercher ce qu'implique l'existence d'une structure projective (strictement) convexe. Il s'int\'eresse \'egalement \`a l'espace des modules de telles structures et il n'est pas clair que les r\'esultats obtenus puissent s'appliquer directement dans les cas consid\'er\'es ici; il semble que cela reste d\'ependant de la question \ref{quest}.\\

Dans \cite{Cooper:2011fk}, les auteurs font des hypoth\`eses moins restrictives sur l'ouvert convexe consid\'er\'e. C'est ainsi, par exemple, qu'en s'affranchissant de l'hypoth\`ese de r\'egularit\'e $\C^1$, ils peuvent donner la version correspondante du th\'eor\`eme \ref{benoistmain} de Benoist dans le cas d'une action de covolume fini. Cela aurait ici compliqu\'e et d\'evi\'e le propos de prendre en compte des cas plus g\'en\'eraux, et nous l'avons gliss\'e, encore une fois, dans la question \ref{quest}. Notons que notre travail pr\'esente quelques points communs, ce qui n'est pas \'etonnant, avec \cite{Cooper:2011fk}. En particulier, le lemme \ref{uni}  appara\^it aussi dans \cite{Cooper:2011fk}, encore une fois avec l'hypoth\`ese  de r\'egularit\'e en moins. Des \'el\'ements de la partie 6 y sont aussi pr\'esents.\\

\subsection{Plan de l'article}

Terminons cette introduction en expliquant o\`u l'on trouvera quoi.
Apr\`es des rappels de g\'eom\'etrie de Hilbert, nous classifions et d\'ecrivons dans la section 3 les automorphismes d'une g\'eom\'etrie de Hilbert d\'efinie par un ouvert strictement convexe (et \`a bord $\C^1$). C'est la classification classique, selon la distance de translation, entre isom\'etrie hyperbolique, parabolique et elliptique, qu'on trouve en g\'eom\'etrie hyperbolique ou de courbure n\'egative.\\
La quatri\`eme partie s'int\'eresse au bord $\dO$, aux points de l'ensemble limite $\LG$, et \`a l'action du groupe $\G$ sur son ensemble limite et son domaine de discontinuit\'e $\dO\smallsetminus\LG$. L'ouvert convexe $\O$ est \`a partir d'ici suppos\'e strictement convexe \`a bord $\C^1$ mais le groupe $\G$ est un sous-groupe discret quelconque de $\Aut(\O)$.\\
On trouve les d\'efinitions de finitude g\'eom\'etrique dans la partie 5, dans laquelle on justifie notre terminologie en faisant r\'ef\'erence \`a celle qui a \'et\'e employ\'ee par d'autres auteurs.\\
Dans la section 6, nous rappelons le lemme de Margulis, que nous avons prouv\'e dans \cite{Crampon:2011fk}, et en tirons les cons\'equences sur la g\'eom\'etrie d'un quotient $\Quo$ d'une g\'eom\'etrie de Hilbert.\\
La section 7 \'etudie plus en d\'etail les groupes paraboliques. C'est avec la section suivante le c\oe ur de ce travail. On y prouve les th\'eor\`emes \ref{para} et \ref{zaza}, ainsi que d'autres r\'esultats concernant l'action des groupes paraboliques qui nous seront utiles ensuite.\\
La section 8 est consacr\'ee \`a la d\'emonstration du th\'eor\`eme \ref{main1}. On y trouve aussi les descriptions des actions convexes-cocompactes et de covolume fini.\\
Nous d\'emontrons le th\'eor\`eme \ref{gro}, qui concerne les propri\'et\'es d'hyperbolicit\'e m\'etrique, dans la section 9. La derni\`ere section permet elle de faire la distinction entre les deux notions de finitude g\'eom\'etrique que nous avons introduites, sur $\O$ et $\dO$. Nous montrons qu'elles sont en fait \'equivalentes en dimension 2 et 3, puis construisons un exemple, en dimension 4, d'une action g\'eom\'etriquement finie sur $\dO$ mais pas sur $\O$.\\
En annexe, nous d\'emontrons un petit r\'esultat concernant le volume des pics, que nous esp\'erons pouvoir g\'en\'eraliser dans un prochain travail; il s'agit d'un travail en commun avec Constantin Vernicos.

\subsection*{Remerciements}

Profitons-en donc pour remercier Constantin pour son aide et son int\'er\^et. Nous tenons \'egalement \`a remercier Yves Benoist, Serge Cantat, Fran\c coise Dal'bo et Patrick Foulon dont les discussions et les connaissances ne sont pas pour rien dans ce travail.\\
Le premier auteur est financ\'e par le programme FONDECYT N$^\circ$ 3120071 de la CONICYT (Chile).


\mainmatter

\section{G\'eom\'etrie de Hilbert}\label{geo_hilbert}

\par{
Cette partie constitue une introduction très rapide à la g\'eom\'etrie de
Hilbert. Pour une introduction plus complète, on pourra lire \cite{MR2270228,MR1238518} ou les livres \cite{MR0075623,MR0054980}.
}

\subsection{Distance et volume}
\par{
Une \emph{carte affine} $A$ de $\PP^n$ est le compl\'ementaire d'un hyperplan projectif. Une carte affine possède une structure naturelle d'espace affine. Un ouvert $\O$ de $\PP^n$ diff\'erent de $\PP^n$ est \emph{convexe} lorsqu'il est inclus dans une carte affine et qu'il est convexe dans cette carte. Un ouvert convexe $\O$ de $\PP^n$ est dit \emph{proprement convexe} lorsqu'il existe une carte affine contenant son adh\'erence $\overline{\O}$. Autrement dit, un ouvert convexe est proprement convexe lorsqu'il ne contient pas de droite affine. Un ouvert proprement convexe $\O$ de $\PP^n$ est dit \emph{strictement convexe} lorsque son bord $\partial \O$ ne contient pas de segment non trivial. 
}
\\
\label{base}

Hilbert a introduit sur un ouvert proprement convexe $\O$ de $\PP^n$ la distance qui porte aujourd'hui son nom. Pour $x \neq y \in \O$, on note $p,q$ les points d'intersection de la droite $(xy)$ et du bord $\partial \O$ de $\O$, de telle fa\c con que $x$ soit entre $p$ et $y$, et $y$ entre $x$ et $q$ (voir figure \ref{dist}). On pose

$$
\begin{array}{ccc}
d_{\O}(x,y) = \displaystyle\frac{1}{2}\ln \big([p:x:y:q]\big) = \displaystyle\frac{1}{2}\ln \bigg(\frac{|py|\cdot |
qx|}{|px| \cdot |qy|} \bigg) & \textrm{et} &
d_{\O}(x,x)=0,
\end{array}
$$
o\`u
\begin{enumerate}
\item la quantit\'e $[p:x:y:q]$ d\'esigne le birapport des points $p,x,y,q$;
\item $| \cdot |$ est une norme euclidienne quelconque sur une carte affine $A$ qui contient l'adh\'erence $\overline{\O}$ de $\O$.
\end{enumerate}

Le birapport \'etant une notion projective, il est clair que $d_{\O}$ ne d\'epend ni du choix de $A$, ni du choix de la norme euclidienne sur $A$.

\begin{center}
\begin{figure}[h!]
  \centering
\includegraphics[width=6cm]{hilbertdistance.jpg}\ \includegraphics[width=7cm]{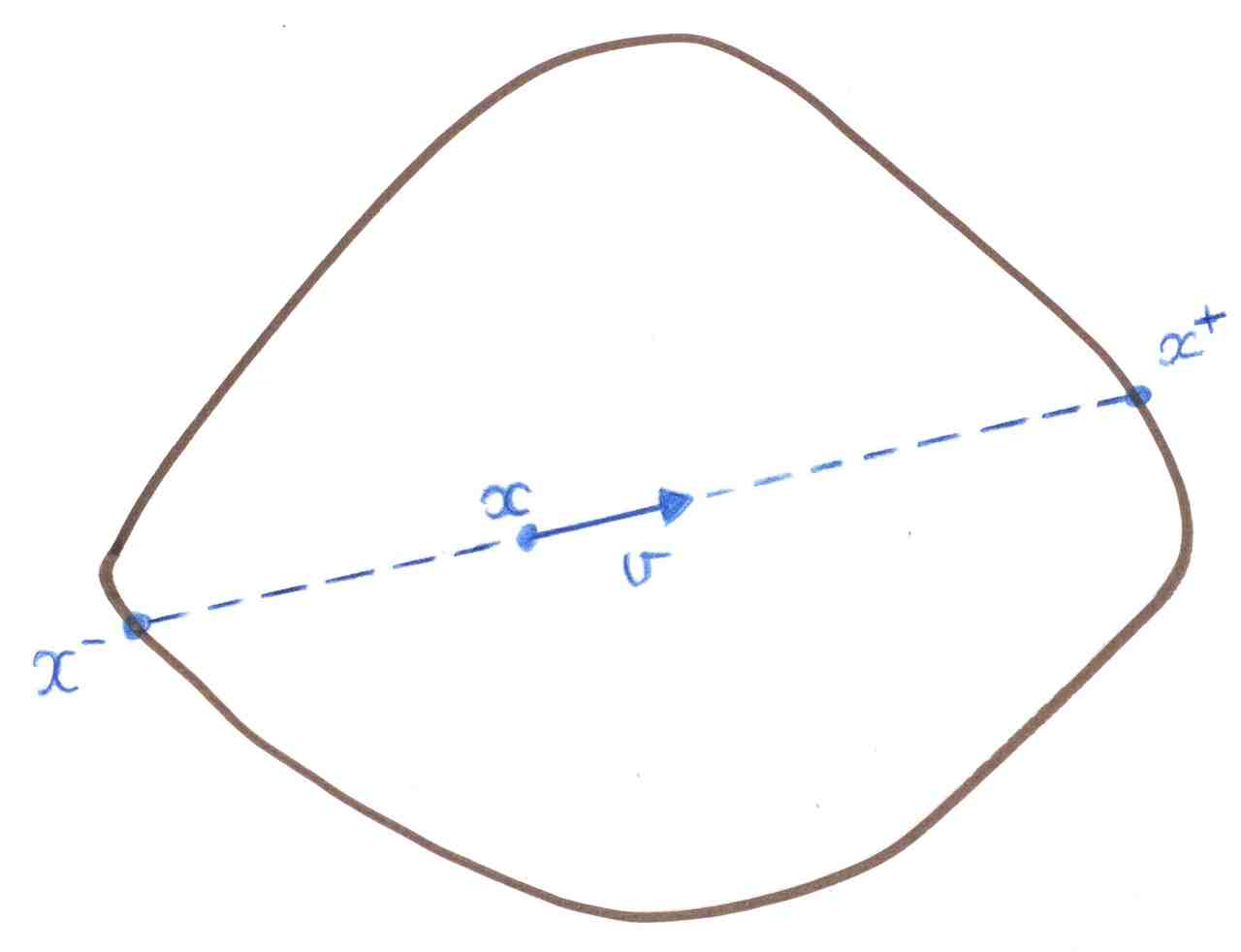}
\caption{La distance de Hilbert et la norme de Finsler}
\label{dist}
\end{figure}
\end{center}

\begin{fait}
Soit $\O$ un ouvert proprement convexe de $\PP^n$.
\begin{enumerate}
\item $d_{\O}$ est une distance sur $\O$;
\item $(\O,d_{\O})$ est un espace m\'etrique complet;
\item La topologie induite par $d_{\O}$ coïncide avec celle induite par $\PP^n$;
\item Le groupe $\Aut(\O)$ des transformations projectives de $\ss$ qui pr\'eservent $\O$ est un sous-groupe ferm\'e de $\ss$ qui agit par isom\'etries sur $(\O,d_{\O})$. Il agit donc proprement sur $\O$.
\end{enumerate}
\end{fait}

La distance de Hilbert $d_{\O}$ est induite par une structure finsl\'erienne sur l'ouvert $\O$. On choisit une carte affine $A$ et une m\'etrique euclidienne $|\cdot|$ sur $A$ pour lesquelles $\O$ appara\^it comme un ouvert convexe born\'e. On identifie le fibr\'e tangent $T \O$ de $\O$ à $\O \times A$. Soient $x \in \O$ et $v \in A$, on note $x^+=x^+(x,v)$ (resp. $x^-$) le point d'intersection de la demi-droite d\'efinie par $x$ et $v$ (resp $-v$) avec $\partial \O$ (voir figure \ref{dist}). On pose
$$F(x,v) = \frac{|v|}{2}\Big(\frac{1}{|xx^-|} + \frac{1}{| xx^+|} \Big),$$
quantit\'e ind\'ependante du choix de $A$ et de $|\cdot|$, puisqu'on ne consid\`ere que des rapports de longueurs.


\begin{fait}
Soient $\O$ un ouvert proprement convexe de $\PP^n$ et $A$ une carte affine qui contient $\overline{\O}$. La distance induite par la m\'etrique finsl\'erienne $F$ est la distance $d_{\O}$. Autrement dit on a les formules suivantes:
\begin{itemize}
\item $\displaystyle{F(x,v) = \left. \frac{d}{dt}\right| _{t=0} d_{\O}(x,x+tv)}$, pour $v \in A$;
\item $d_{\O}(x,y) = \inf \displaystyle\int_0^1 F(\dot\sigma(t))\ dt$, où l'infimum est pris sur les chemins $\sigma$ de classe $\C^1$ tel que $\sigma(0)=x$ et $\sigma(1)=y$.
\end{itemize}
\end{fait}

\par{
Il y a plusieurs mani\`eres naturelles. de construire un volume pour une g\'eom\'etrie de Finsler, la d\'efinition riemannienne acceptant plusieurs g\'en\'eralisations. Nous travaillerons avec le volume de Busemann, not\'e $\textrm{Vol}_{\O}$.\\
Pour le construire, on se donne une carte affine $A$ et une m\'etrique euclidienne $|\cdot|$ sur $A$ pour lesquelles $\O$ appara\^it comme un ouvert convexe born\'e. On note $B(T_x\O)= \{ v \in T_x \O \, | \, F(x,v) < 1\}$ la boule de rayon $1$ de l'espace tangent à $\O$ en $x$, $\textrm{Vol}$ la mesure de Lebesgue sur $A$ associ\'ee \`a $|\cdot|$ et $v_n=\textrm{Vol}(\{ v \in A \, | \, |v| < 1 \})$ le volume de la boule unit\'e euclidienne en dimension $n$.

\par{
Pour tout bor\'elien $\mathcal{A} \subset \O \subset A$, on pose:
$$\textrm{Vol}_{\O} (\mathcal{A})= \int_{\mathcal{A}} \frac{v_n}{\textrm{Vol}(B(T_x\O))}\ d\Vol(x)$$
L\`a encore, la mesure $\textrm{Vol}_{\O}$ est ind\'ependante du choix de $A$ et de $|\cdot|$. En particulier, elle est pr\'eserv\'ee par le groupe $\Aut(\O)$.
}

La proposition suivante permet de comparer deux g\'eom\'etries de Hilbert entre elles.

\begin{prop}\label{compa}
Soient $\O_1$ et $\O_2$ deux ouverts proprement convexes de $\PP^n$ tels que $\O_1 \subset \O_2$.
\begin{itemize}
\item Les m\'etriques finsl\'eriennes $F_1$ et $F_2$ de $\O_1$ et $\O_2$ v\'erifient: $F_2(w) \leqslant F_1(w)$, $w \in T \O_1 \subset T \O_2$, l'\'egalit\'e ayant lieu si et seulement si $x^+_{\O_1}(w)=x^+_{\O_2}(w)$ et $x^-_{\O_1}(w)=x^-_{\O_2}(w)$.
\item Pour tous $x,y \in \O_1$, on a $d_{\O_2}(x,y) \leqslant d_{\O_1}(x,y)$.
\item Les boules m\'etriques v\'erifient, pour tout $x \in \O_1$ et $r>0$, $B_{\O_1}(x,r) \subset B_{\O_2}(x,r)$, avec \'egalit\'e si et seulement si $\O_1=\O_2$. De m\^eme, $B(T_x\O_1) \subset B(T_x\O_2)$.
\item Pour tout bor\'elien $\mathcal{A}$ de $\O_1$, on a $\textrm{Vol}_{\O_2}(\mathcal{A}) \leqslant \textrm{Vol}_{\O_1}(\mathcal{A})$.
\end{itemize}
\end{prop}

\subsection{Fonctions de Busemann et horosphères}\label{busemannhoro}

Nous supposons dans ce paragraphe que l'ouvert proprement convexe $\O$ de $\PP^n$ est strictement convexe et \`a bord $\C^1$. Dans ce cadre, il est possible de d\'efinir les fonctions de Busemann et les horosphères de la m\^eme mani\`ere qu'en g\'eom\'etrie hyperbolique, et nous ne donnerons pas de d\'etails.\\

Pour $\xi\in\dO$ et $x\in \O$, notons $c_{x,\xi}: [0,+\infty)\longrightarrow\O$ la g\'eod\'esique issue de $x$ et d'extr\'emit\'e $\xi$, soit $c_{x,\xi}(0)=x$ et $c_{x,\xi}(+\infty)=\xi$. La \emph{fonction de Busemann} bas\'ee en $\xi\in\dO$ $b_{\xi}(.,.):\O\times\O\longrightarrow\R$ est d\'efinie par:
$$b_{\xi}(x,y) = \lim_{t\to+\infty} \d(y,c_{x,\xi}(t)) - t = \lim_{z\to\xi} \d(y,z) - \d(x,z),\ x,y\in\O.$$
L'existence de ces limites est due aux hypoth\`eses de r\'egularit\'e faites sur $\O$. Les fonctions de Busemann sont de classe $\C^1$.\\

\emph{L'horosph\`ere} bas\'ee en $\xi\in\dO$ et passant par $x\in\O$ est l'ensemble
$$\H_{\xi}(x) = \{y\in\O \ | \ b_{\xi}(x,y) = 0\}.$$
\emph{L'horoboule} bas\'ee en $\xi\in\dO$ et passant par $x\in\O$ est l'ensemble
$$H_{\xi}(x) = \{y\in\O \ | \ b_{\xi}(x,y) < 0\}.$$
L'horoboule bas\'ee en $\xi\in\dO$ et passant par $x\in\O$ est un ouvert strictement convexe de $\O$, dont le bord est l'horosph\`ere correspondante, qui est elle une sous-vari\'et\'e de classe $\C^1$ de $\O$.\\
Dans une carte affine $A$ dans laquelle $\O$ appara\^it comme un ouvert convexe relativement compact, on peut, en identifiant $T\O$ avec $\O \times A$, construire g\'eom\'etriquement l'espace tangent \`a $\H_{\xi}(x)$ en $x$: c'est le sous-espace affine contenant $x$ et l'intersection $T_{\xi}\dO\cap T_{\eta}\dO$ des espaces tangents \`a $\dO$ en $\xi$ et $\eta=(x\xi)\cap\dO\smallsetminus\{\xi\}$.\\
On peut voir que que l'horoboule et l'horosph\`ere bas\'ees en $\xi\in\dO$ et passant par $x\in\O$ sont les limites des boules et des sphères m\'etriques centr\'ees au point $z\in\O$ et passant par $x$ lorsque $z$ tend vers $\xi$.

\subsection{Dualit\'e}\label{def_dualite}

\`A l'ouvert proprement convexe $\O$ de $\PP^n$ est associ\'e l'ouvert proprement convexe dual $\O^*$: on considère un des deux cônes $C\subset \R^{n+1}$ au-dessus de $\O$, et son dual
$$C^* = \{ f\in (\R^{n+1})^*,\ \forall x\in C,\ f(x)>0\}.$$
Le convexe $\O^*$ est par d\'efinition la trace de $C^*$ dans $\PP((\R^{n+1})^*)$.\\

Le bord de $\dO^*$ est facile à comprendre, car il s'identifie à l'ensemble des hyperplans tangent à $\O$. En effet, un hyperplan tangent $T_x$ à $\dO$ en $x$ est la trace d'un hyperplan $H_x$ de $\R^{n+1}$. L'ensemble des formes lin\'eaires dont le noyau est $H_x$ forme une droite de $(\R^{n+1})^*$, dont la trace $x^*$ dans $\PP((\R^{n+1})^*)$ est dans $\dO^*$. Il n'est pas dur de voir qu'on obtient ainsi tout le bord $\dO^*$.\\
Cette remarque permet de voir que le dual d'un ouvert strictement convexe a un bord de classe $C^1$, et inversement. En particulier, lorsque $\O$ est strictement convexe et que son bord est de classe $\C^1$, ce qui est le cas que nous \'etudierons, on obtient une involution continue $x\longmapsto x^*$ entre les bords de $\O$ et $\O^*$.\\

\'Etant donn\'e un sous-groupe discret $\G$ de $\Aut(\O)$, on en d\'eduit aussi une action de $\G$ sur le convexe dual $\O^*$: pour $f\in C^*$ et $\g\in\G$,
$$(\g\cdot f)(x) = f(\g^{-1}x),\ x\in C.$$
Le sous-groupe discret de $\Aut(\O^*)$ ainsi obtenu sera not\'e $\G^*$. Bien entendu, on a $(\O^*)^*=\O$ et $(\G^*)^*=\G$.\\

\subsection{Le th\'eorème de Benz\'ecri}

\par{
On d\'efinit l'espace $X^{\bullet}$ des convexes marqu\'es comme l'ensemble suivant:
$$X^{\bullet} = \{  (\O,x) \,\mid\, \O \textrm{ est un ouvert proprement convexe de } \PP^n \textrm{ et } x \in \O \}$$
muni de la topologie de Hausdorff h\'erit\'ee par la distance canonique sur $\PP^n$.}

Le th\'eor\`eme suivant a dej\`a prouv\'e maintes fois son utilit\'e.

\begin{theo}[Jean-Paul Benz\'ecri \cite{MR0124005}]\label{theo_ben}
$\,$
L'action de $\ss$ sur $X^{\bullet}$ est propre et cocompacte.
\end{theo}

On pourra trouver une preuve de ce th\'eorème aussi dans les notes de cours de William Goldman \cite{NoteGoldman}.


\section{Classification des automorphismes}


\subsection{Le th\'eorème de classification}

\begin{defi}
Soient $\O$ un ouvert proprement convexe et $\g\in\Aut(\O)$. On appelle \emph{distance de translation de $\g$} la quantit\'e $\tau(\g) = \displaystyle{\inf_{x \in \O} d_{\O}(x, \g \cdot x)}$.
\end{defi}

\begin{defi}
Soient $\O$ un ouvert proprement convexe et $\g\in\Aut(\O)$. On dira que $\g$ est:
\begin{enumerate}
\item \emph{hyperbolique} lorsque $\tau(\g) > 0$ et cet infimum est atteint;

\item \emph{quasi-hyperbolique}  lorsque $\tau(\g) > 0$ et cet infimum n'est pas atteint;

\item \emph{elliptique}  lorsque $\tau(\g) = 0$  et cet infimum est atteint, autrement dit $\g$ fixe un point de $\O$;

 \item \emph{parabolique}  lorsque $\tau(\g) = 0$ et cet infimum n'est pas atteint.
\end{enumerate}
\end{defi}

\begin{theo}\label{classi_dym_1}
Soient $\O$ un ouvert strictement convexe et à bord $\C^1$ de $\PP^n$ et $\g \in \Aut(\O)$. On est dans l'un des trois cas exclusifs suivants:
\begin{enumerate}
\item L'automorphisme $\g$ est elliptique.
\item L'automorphisme $\g$ est hyperbolique. Il a exactement deux points fixes $p^+,p^- \in \partial \O$, l'un r\'epulsif et l'autre attractif: la suite $(\g^n)_{n \in \N}$ converge uniform\'ement sur les compacts de $\overline{\O}\smallsetminus \{ p^-\}$ vers $p^+$, et la suite $(\g^{-n})_{n \in \N}$ converge uniform\'ement sur les compacts de $\overline{\O}\smallsetminus \{ p^+\}$ vers $p^-$.
\item L'automorphisme $\g$ est parabolique. Il a exactement un point fixe $p \in \partial \O$ et pr\'eserve toute horosphère bas\'ee en $p$. De plus, la famille $(\g^n)_{n \in \Z}$ converge uniform\'ement sur les compacts de $\overline{\O}\smallsetminus \{ p\}$ vers $p$. Mais $p$ n'est pas un point attractif au sens de la remarque ci-dessous.
\end{enumerate}
En particulier, l'automorphisme $\g$ n'est pas quasi-hyperbolique.
\end{theo}

\begin{center}
\begin{figure}[h!]
  \centering
\includegraphics[width=7cm]{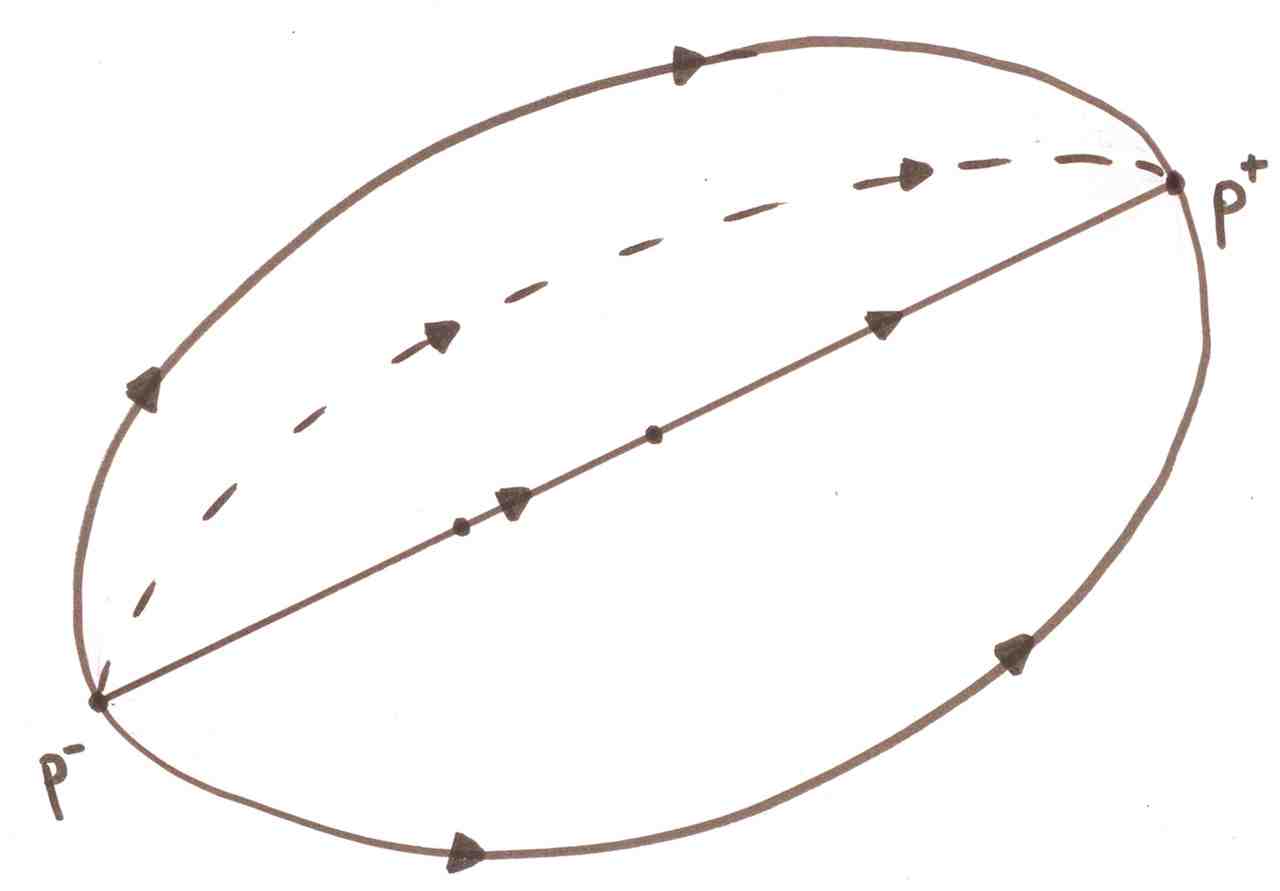}\ \ \includegraphics[width=7cm]{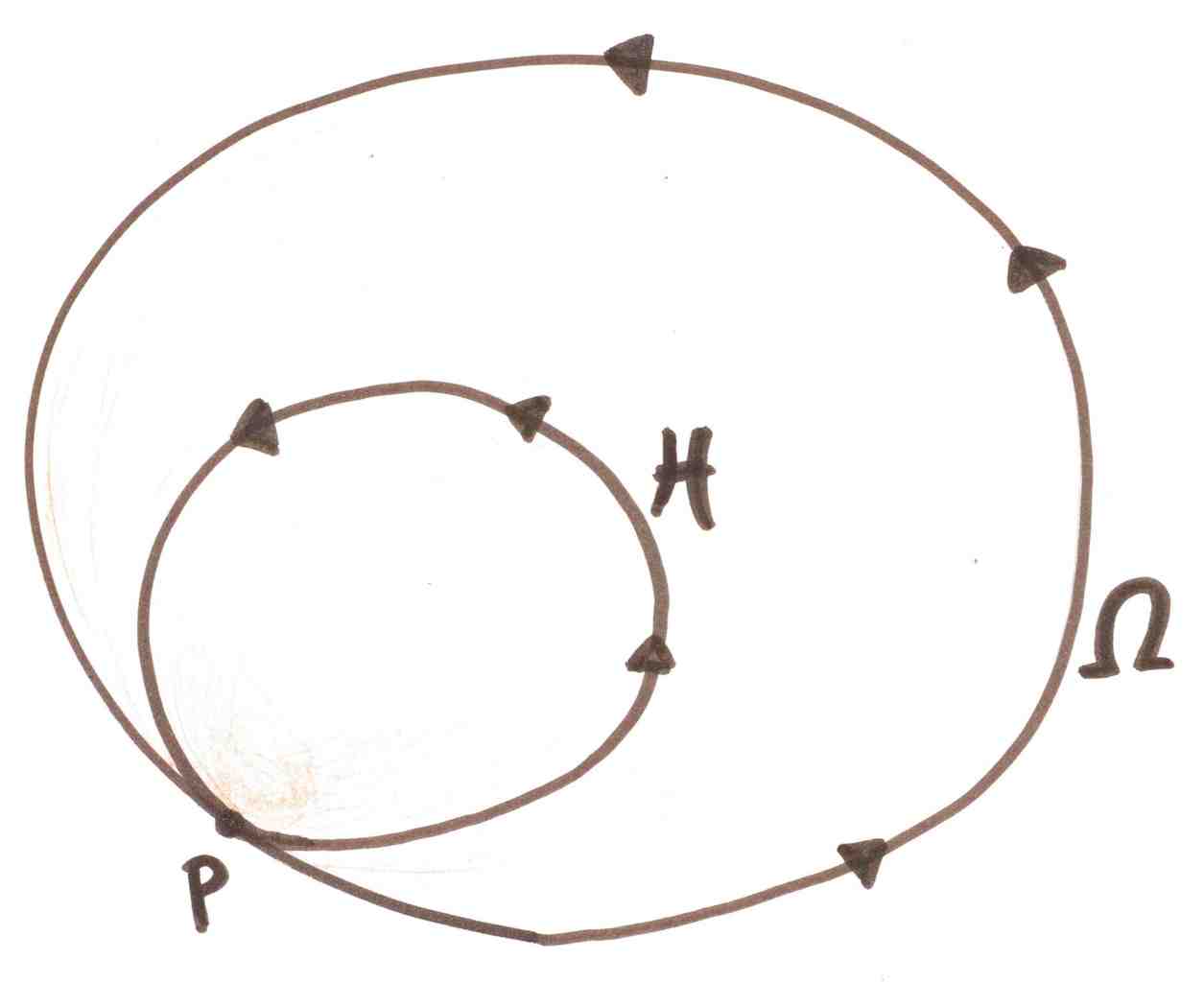}
\caption{Isom\'etries hyperbolique et parabolique}

\end{figure}
\end{center}

\begin{rema}
Un point $x$ est dit \emph{attractif} pour un hom\'eomorphisme $\g$ lorsqu'il existe un voisinage $\U$ de $x$ tel que $(\g^n(\U))_{n\in \N}$ converge vers le singleton $\{ x \}$ en d\'ecroissant. Un point est \emph{r\'epulsif} pour un hom\'eomorphisme $\g$ s'il est attractif pour $\g^{-1}$.
\end{rema}

\subsection{Petites dimensions}
\paragraph*{Dimension 1}

Le lemme suivant est un exercice laiss\'e au lecteur.

\begin{lemm}\label{lem_auto_1}
Tout ouvert proprement convexe de $\PP^1$ est projectivement \'equivalent à $\O_0 = \R_+^*$. De plus, $\Aut(\O_0) = \R_+^*$ via l'action de $\R_+^*$ sur lui-même par homoth\'etie.
\end{lemm}

\begin{rema}
L'action par homoth\'etie $\g$ de rapport $\lambda$ sur $\R_+^*$ est une action par translation de force $\ln(\lambda)$, c'est-à-dire $\forall x \in \R_+^*$, $d_{\R_+^*}(x, \g \cdot x)= \ln(\lambda)$.
\end{rema}

\paragraph*{Dimension 2}

On pourra trouver dans \cite{MR1293655,Marquis:2009kq} une classification complète des automorphismes des ouverts proprement convexes de $\PP^2$. On ne donne ici que le lemme n\'ecessaire pour le th\'eorème \ref{classi_dym_1}.

\begin{lemm}[Proposition 2.9 de \cite{Marquis:2009kq}]\label{lem_auto_2}		
Soit $\O$ un ouvert proprement convexe de $\PP^2$. S'il existe un automorphisme $\g \in \Aut(\O)$  qui possède trois points fixes distincts sur $\partial \O$ alors le bord $\partial \O$ de l'ouvert $\O$ contient deux segments distincts et non triviaux. 
\end{lemm}


\subsection{Quelques lemmes}

Sur l'adh\'erence $\overline{\O}$ de tout ouvert proprement convexe $\O$, on peut introduire la relation d'\'equivalence suivante:

\begin{center}\vspace*{.5cm}
\begin{minipage}[center]{10cm}
\begin{itemize}
\item[] $x \sim_{\O} y$
\item[$\Leftrightarrow$]  le  segment $[x,y]$ peut se prolonger strictement\\ à ses deux extr\'emit\'es et rester dans $\overline{\O}$
\item[$\Leftrightarrow$]  les points $x$ et $y$ sont dans la même facette de $\overline{\O}$.
\end{itemize}
\end{minipage}
\vspace*{.5cm}
\end{center}

\par{
On appelle ainsi \emph{facette} les classes de cette relation d'\'equivalence. Le \emph{support} d'une facette est l'espace projectif qu'elle engendre. On remarquera que les facettes de $\overline{\O}$ sont des ouverts proprement convexes de leur support. Lorsqu'une facette est un singleton $\{p\}$, le point $p$ est dit \emph{extr\'emal}.
}
 
\begin{lemm}\label{force_strict}
Soit $\O$ un ouvert proprement convexe. Soient $(x_n)_{n \in \N}$ et $(y_n)_{n \in \N}$ deux suites de points de $\O$ telles que:
\begin{enumerate}
\item la suite $(x_n)_{n \in \N}$ converge vers un point $x_{\infty} \in \partial \O$;
\item la suite $(d_{\O}(x_n, y_n))_{n \in \N}$ est major\'ee;
\item le point $x_{\infty}$ est un point extr\'emal de $\O$.
\end{enumerate}
Alors, la suite $(y_n)_{n \in \N}$ converge vers le point $x_{\infty} \in \partial \O$.
\end{lemm}

\begin{proof}
C'est une cons\'equence de la proposition suivante \ref{non_stric_conv1} qui est d\'emontr\'e dans \cite{Marquis:2010fk}.
\end{proof}

\begin{prop}\label{non_stric_conv1}
Soient $\O$ un ouvert proprement convexe de $\PP^n$ et $x_{\infty}$ un point de $\dO$. On note $S$ la facette de $\overline{\O}$ contenant $x_{\infty}$ et $E$ son support.

Pour toute suite de points $(x_n)$ de $\O$ et tout r\'eel $R>0$, si la suite $(x_n)$ tend vers $x_{\infty}$ alors la suite $(B^{\O}_{x_n}(R))$ converge vers la boule $B^{S}_{x_{\infty}}(R)$ pour la distance de Hausdorff induite par la distance canonique $d_{can}$ de $\PP^n$ (voir figure \ref{bouboule}).	
\end{prop}

\begin{center}
\begin{figure}[h!]
  \centering
\includegraphics[width=6cm]{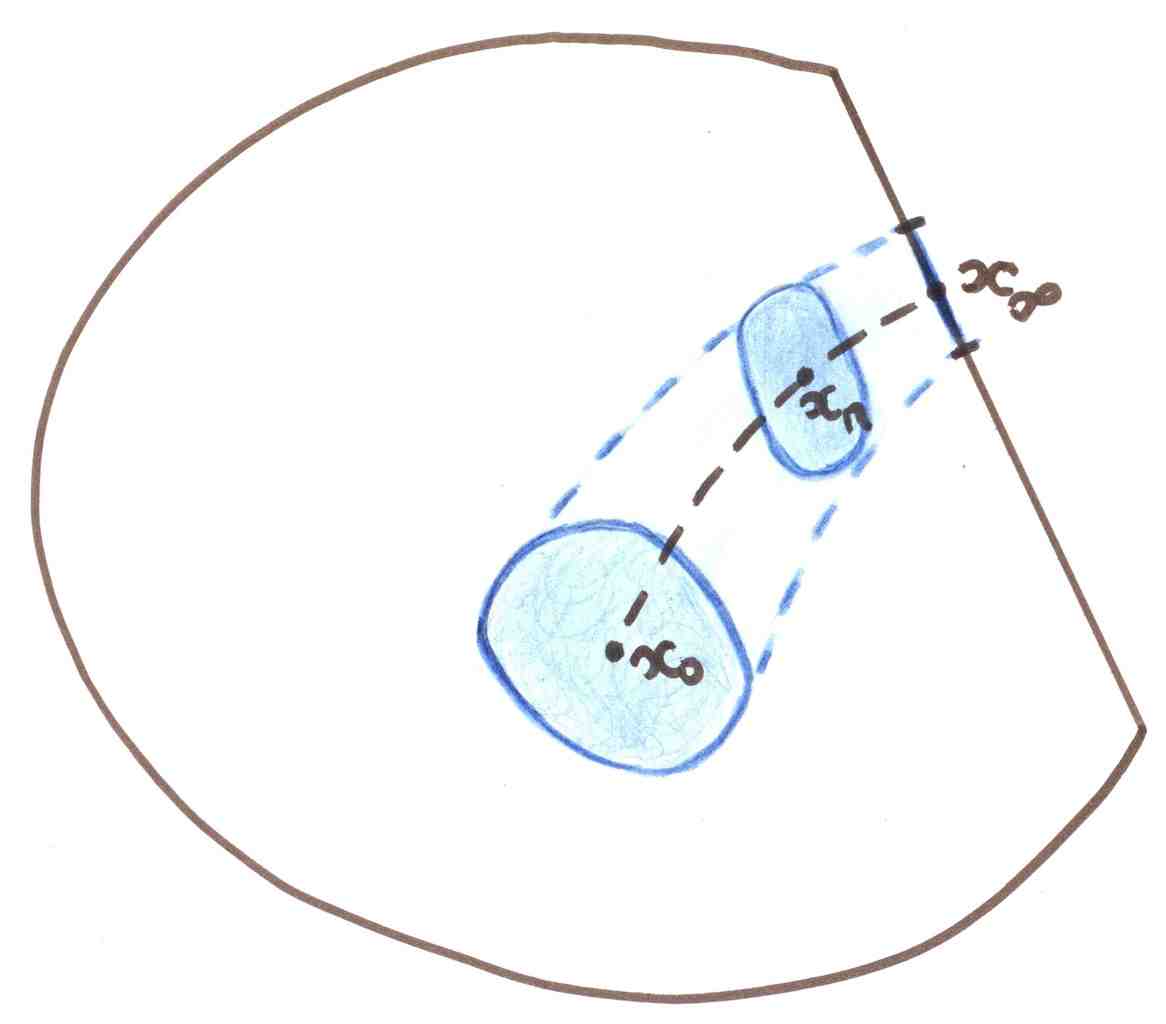}
\caption{D\'eg\'en\'erescence des boules\label{bouboule}}
\end{figure}
\end{center}

\begin{lemm}\label{conv_uni}
Soient $\O$ un ouvert proprement convexe et $\g \in \Aut(\O)$. S'il existe un point $x \in \O$ tel que la suite $(\g^n \cdot x)_{n \in \N}$ converge vers un point extr\'emal $p \in \partial \O$, alors la suite $(\g^n)_{n \in \N}$ converge uniform\'ement sur les compacts de $\O$ vers $p$.
\end{lemm}

\begin{proof}
Commençons par montrer que la suite $(\g^n)_{n \in \N}$ converge simplement vers $p$ sur $\O$. Soit $y\in \O$. Il suffit d'appliquer le lemme \ref{force_strict} pr\'ec\'edent aux suites $x_n = \g^n \cdot x$ et $y_n =  \g^n \cdot y$. La suite $(d_{\O}(x_n, y_n))_{n \in \N}$ est bien major\'ee puisque, $\g$ \'etant une isom\'etrie, elle est constante \'egale à $d_{\O}(x,y)$.\\
On obtient la convergence uniforme sur les compacts pour la même raison. En effet, comme les $(\g^n)_{n \in \N}$ sont des isom\'etries, elles forment en particulier une famille \'equicontinue d'applications.
\end{proof}

\begin{lemm}\label{accu_fixe}
Soit $\O$ un ouvert proprement convexe. Tout point d'accumulation dans $\partial \O$ de la suite $(\g^n \cdot x)_{n \in \N}$ qui est un point extr\'emal de $\O$ est un point fixe de $\g$.
\end{lemm}

\begin{proof}
Soit $p$ un point d'accumulation de la suite $(\g^n \cdot x)_{n \in \N}$ qui est dans $\partial \O$. Il existe une extractrice $(n_i)_{i \in \N}$ tel que $\lim_{ i \to \infty} \g^{n_i} \cdot x= p$. Le lemme \ref{force_strict} montre que la suite $(\g^{1+n_i} \cdot x)$ converge vers $p$ car $p$ est extr\'emal. L'application $\g$ est continue sur $\PP^n  \supset \overline{\O}$, il vient que $\g(p) = p$.
\end{proof}

\subsection{D\'emonstration du th\'eorème de classification \ref{classi_dym_1}}

Nous aurons besoin de la proposition suivante:

\begin{prop}[Lemme 3.2 de \cite{MR2195260}]\label{brouwer}
Si un \'el\'ement $\g \in \ss$ pr\'eserve un ouvert proprement convexe alors le rayon spectral $\rho(\g)$ (c'est-\`a-dire le module de la plus grande valeur propre de $\g$) est une valeur propre dont la droite propre appartient à $\overline{\O}$. En particulier, tout automorphisme d'un ouvert proprement convexe possède un point fixe dans $\overline{\O}$.
\end{prop}

\begin{proof}[D\'emonstration du th\'eorème \ref{classi_dym_1}]
\par{
D'après la proposition \ref{brouwer}, l'hom\'eomorphisme $\g:\overline{\O} \rightarrow \overline{\O}$ possède un point fixe dans $\overline{\O}$. S'il existe un point $x \in \O$ fix\'e par $\g$, alors $\g$ est elliptique et il n'y a rien à montrer. On peut donc supposer que tout point fixe de $\g$ est dans $\partial \O$. Nous allons à pr\'esent distinguer 3 cas.
} 
 \begin{enumerate}
\item Il existe au moins trois points distincts $x,y,z \in \partial \O$ fix\'es par $\g$.  
\item  Il existe exactement deux points distincts $x,y \in \partial \O$ fix\'es par $\g$.
\item L'automorphisme $\g$ fixe un et un seul point de $\partial \O$.
\end{enumerate}
\par{
Commençons par montrer que le premier cas est exclu. Les points $x,y,z$ ne sont pas align\'es car le convexe $\O$ est strictement convexe. Le plan projectif $P$ engendr\'e par les points $x,y,z$ est pr\'eserv\'e par $\g$, tout comme l'ouvert proprement convexe $P \cap \O$ de $P$. Comme $P$ est un espace projectif de dimension 2, le lemme \ref{lem_auto_2} montre que le bord du convexe $P \cap \O$ contient un segment non trivial. Par cons\'equent, $\O$ n'est pas strictement convexe, ce qui contredit l'hypoth\`ese.
} 
\\
\par{
Si on est dans le second cas alors le segment ouvert $s=]x,y[$ de $\overline{\O}$ est pr\'eserv\'e par $\g$ et inclus dans $\O$ puisque $\O$ est strictement convexe. Le lemme \ref{lem_auto_1} montre que l'\'el\'ement $\g$ agit comme une translation sur $s$ et que l'un des points $x,y$ est attractif pour l'action de $\g$ sur $s$ et l'autre est r\'epulsif. On note $p^+$ l'attractif et $p^-$ le r\'epulsif. Le lemme \ref{conv_uni} montre que $(\g^n)_{n \in \N}$ converge uniform\'ement sur les compacts de $\O$ vers $p^+$, et la suite $(\g^{-n})_{n \in \N}$ converge uniform\'ement sur les compacts de $\O$ vers $p^-$. Montrons la convergence sur les compacts de $\overline{\O}\smallsetminus \{ p^-\}$. On se donne un compact $K$ de $\overline{\O}\smallsetminus \{ p^-\}$ et on choisit un hyperplan $H$ de $\O$ qui s\'epare $p^-$ de $K$. Les convexes $\g^n(H)\cap \O$ convergent vers $p^+$ et donc $K$ aussi. On procède de la même façon avec $p^-$.\\
Il nous reste à montrer qu'un tel \'el\'ement est hyperbolique. Pour cela, on va montrer que pour tout $b \in \O \smallsetminus s$, et pour tout point $a \in s$, $d_{\O}(b, \g  b) > d_{\O}(a, \g  a)$. Sur le segment ouvert $s$, l'\'el\'ement $\g$ agit comme une translation, la quantit\'e $d_{\O}(a, \g  a)$ ne d\'epend donc pas du point $a \in s$. On note $H^+$ (resp. $H^-$) l'hyperplan tangent à $\partial \O$ en $p^+$ (resp. $p^-$). Soit $H_b$ l'hyperplan passant par $H^+ \cap H^-$ et le point $b$. On note $a$ l'unique point de l'intersection $H_b \cap s$. La distance de Hilbert est d\'efinie \`a l'aide de birapports et par cons\'equent on a: $d_{\O}(a, \g(a)) = \frac{1}{2}\ln([H^-:H_b:\g(H_b):H^+])$. De plus, comme l'ouvert $\O$ est strictement convexe, on a $[H^-:H_b:\g(H_b):H^+] < [q_b : b : \g(b) : p_b ]$, où $q_b,p_b$ sont les points d'intersections de la droite $(b \,\g(b))$ avec $\partial \O$, tels que $\g(b)$ soit entre $b$ et $p_b$ (voir figure \ref{classi}). L'infimum de la distance de translation de $\g$ est donc 
atteint par tout point de $s$ et seulement par les points de $s$. En particulier, l'automorphisme $\g$ n'est pas quasi-hyperbolique.
}

\begin{center}
\begin{figure}[h!]
  \centering
\includegraphics[width=9cm]{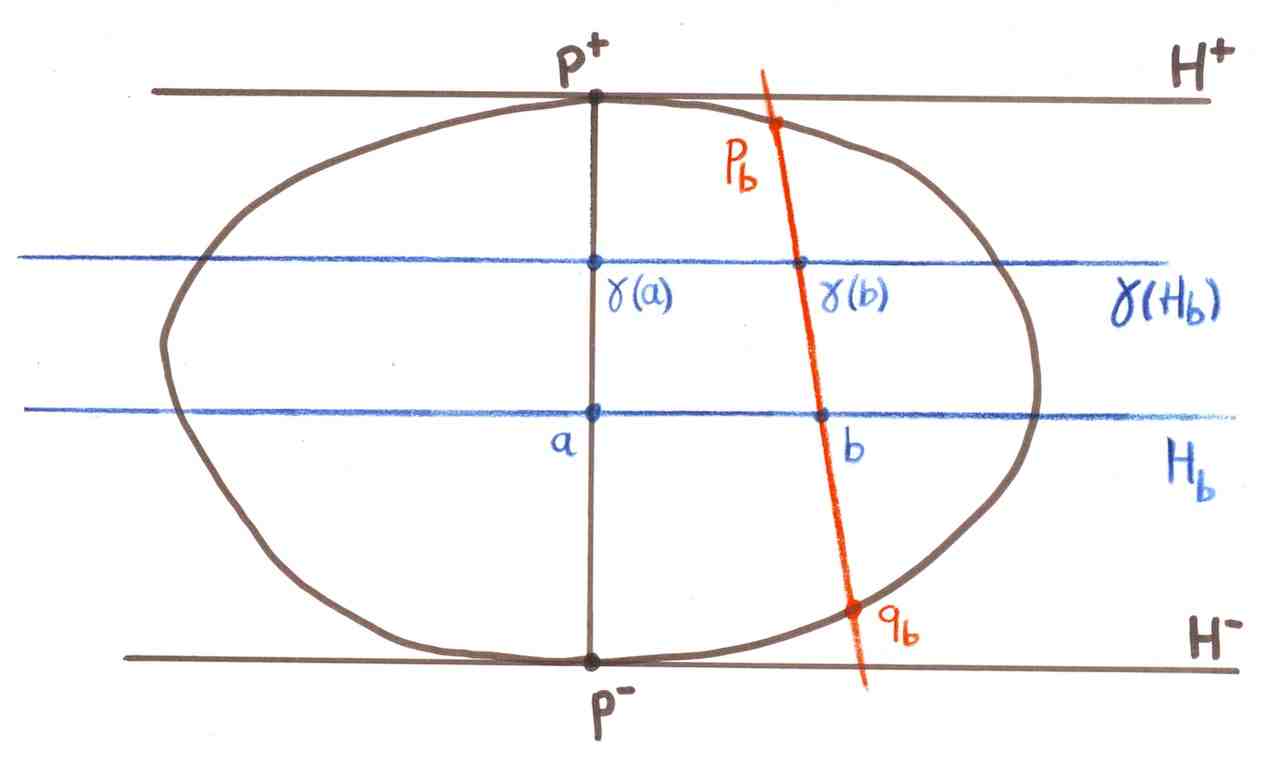}
\caption{L'automorphisme $\g$ est hyperbolique
\label{classi}
}
\end{figure}
\end{center}

\par{
Enfin, si l'automorphisme $\g$ fixe un et un seul point $p$ de $\partial \O$, le lemme \ref{accu_fixe} montre que pour tout point $x \in \O$, le seul point d'accumulation de la bi-suite $(\g^n  x)_{n \in \Z}$ est l'unique point fixe de $\g$. Par cons\'equent d'après le lemme \ref{conv_uni}, la bi-suite $(\g^n  x)_{n \in \Z}$ converge vers $p$ uniform\'ement sur les compacts de $\O$. Un raisonnement analogue au pr\'ec\'edent montre que la convergence a lieu sur les compacts de $\overline{\O}\smallsetminus \{ p\}$.\\
Montrons maintenant que $\tau(\g) = 0$. Pour cela, on se donne un point $x \in \O$ et une suite $(x_n)_{n \in \N}$ de points de la demi-droite $[xp[$ qui converge vers $p$. La suite de points $(\g  x_n)_{n \in \N}$ est donc sur la demi-droite $[\g x\ p[$ et converge vers $p$. La suite des droites $(x_n\ \g x_n)$ converge vers la droite intersection du plan projectif engendr\'e par $p,x,\g x$ et de l'hyperplan tangent à $\O$ en $p$. Comme le bord du convexe est de classe $\C^1$, on en conclut que $d_{\O}(x_n,\g  x_n)$ tend vers $0$.

\begin{center}
\begin{figure}[h!]
  \centering
\includegraphics[width=7cm]{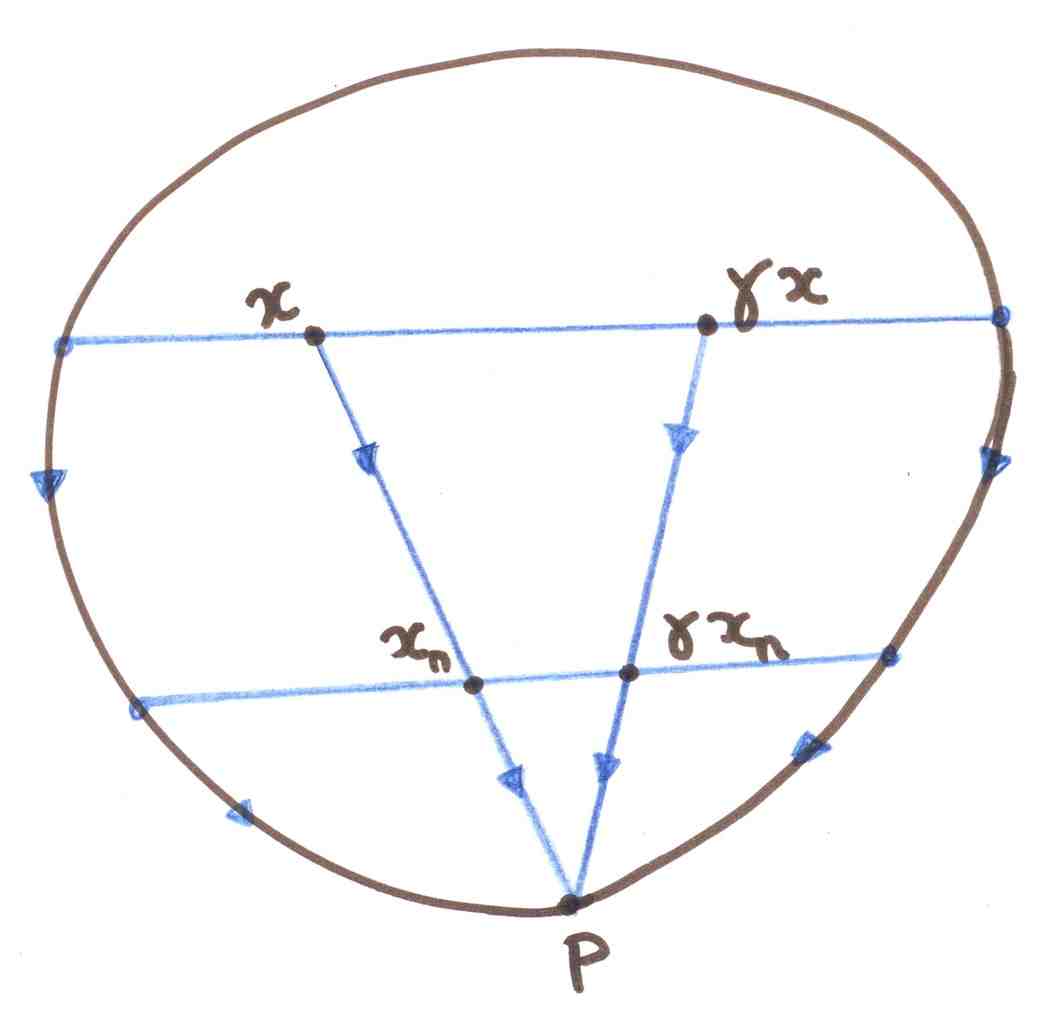}
\caption{La distance de translation est nulle}
\end{figure}
\end{center}

Il reste \`a montrer que $\g$ pr\'eserve toute horosphère bas\'ee en $p$. Voyons d'abord que les fonctions de Busemann bas\'ees en $p$ sont invariantes par $\g$: pour tous $o,x\in\O$,
$$\begin{array}{rl}
b_p(\g o,\g x) = \displaystyle\lim_{z\to p} \d(\g o,z) - \d(\g x,z) & = \displaystyle\lim_{z \to p} \d(\g o,\g z) - \d(\g x,\g z)\\
& =  \displaystyle\lim_{z\to p} \d(o,z) - \d(x,z)\\
&= b_p(o,z),
  \end{array}
$$
puisque, si $z$ tend vers $p$, $\g z$ \'egalement. Ainsi, pour tout $x\in\O$,
$$\H_p(\g x)=\{y\in\O,\ b_p(\g x,y) = 0\} = \{y\in\O,\ b_p(x,\g^{-1}y) = 0\} = \g H_p(x);$$
autrement dit, $\g$ pr\'eserve l'ensemble des horosph\`eres bas\'ees en $p$. Maintenant, pour tous $x,y \in \O$, on a
$$b_p(x,\g x) = b_p(x,y) + b_p(y,\g y) + b_p(\g y,\g x) =  b_p(y,\g y):= a \in\R.$$
Or, $|b_p(x,gx)|\leqslant \d(x,gx)$, ce qui implique que pour tout $x\in\O$, $\d(x,\g x) \geqslant |a|$. De $\tau(g)=0$, on d\'eduit que $a=0$, c'est-\`a-dire que $\g x\in \H_p(x)$.
} 
\end{proof}

En fait, la classification du th\'eor\`eme \ref{classi_dym_1} reste valable lorsque l'ouvert est seulement suppos\'e strictement convexe. Pour montrer que la distance de translation d'un automorphisme parabolique $\g$ est nulle, on utilise alors le lemme suivant, d\^u \`a McMullen, et le fait que le rayon spectral de $\g$ est n\'ecessairement $1$ (sinon, $\g$ aurait plus d'un point fixe).\\
Pour des r\'esultats plus g\'en\'eraux, on pourra consulter \cite{Cooper:2011fk}.

\begin{lemm}[Curtis McMullen, Th\'eorème 2.1 de \cite{MR1953192}]
Soient $\O$ un ouvert proprement convexe de $\PP^n$ et $\g\in\Aut(\O)$. On a
$$\frac{1}{2}\ln\bigg(\max\Big(\rho(\g), \rho(\g^{-1}),\rho(\g)\rho(\g^{-1})\Big)\bigg) \leqslant \tau(\g)\leqslant \ln\bigg(\max\Big(\rho(\g), \rho(\g^{-1})\Big)\bigg)$$
En particulier, si $\rho(\g)= \rho(\g^{-1})$ alors $\tau(\g) = \ln(\rho(\g))$; et si $\rho(\g)=1$ alors $\tau(\g)=0$.
\end{lemm}

\vspace*{.5cm}
{\bf \large Dans tout ce qui suit, sauf mention explicite, $\O$  d\'esignera un ouvert proprement convexe, strictement convexe et à bord $\C^1$.}
\vspace*{.5cm}

\subsection{Sous-groupes nilpotents discrets de $\Aut(\O)$}
\paragraph*{Points fixes et discr\'etude}

\begin{prop}\label{point_fixe_dyn}
Soient $\g $ et $\delta$ deux \'el\'ements non elliptiques de $\Aut(\O)$ qui engendrent un sous-groupe discret de $\Aut(\O)$. Supposons que $\g$ et $\delta$ fixent un même point $x \in \partial \O$. 
\begin{enumerate}
\item Si $\g$ est parabolique, alors $\delta$ est parabolique.
\item Si $\g$ est hyperbolique, alors $\delta$ est hyperbolique et il existe $k,l \in \Z$ tel que $\g^k =\delta^l$.
\end{enumerate}
\end{prop}

\begin{proof}
Supposons pour commencer $\g$ hyperbolique. On peut supposer que le point fixe attractif de $\g$ est $x$, et on appelle $y$ son point r\'epulsif. On veut montrer que $\delta$ est hyperbolique et fixe le point $y$.\\
Si l'\'el\'ement $\delta$ ne fixe pas $y$, l'\'el\'ement  $\g' = \delta \g \delta^{-1}$ est hyperbolique, fixe le point $x$ et le point $\delta(y) \neq y$. Il pr\'eserve donc le segment $[x,\delta(y)]$. Or, si $z\in ]x,y[$, la famille de points (ultimement) distincts $\g'^{-n}\g^n \cdot z$ s'accumule dans $\O$, ce qui contredit la discr\'etude de l'action de $\G$ sur $\O$.\\
Ainsi, si $\g$ est hyperbolique et si $\delta$ fixe $x$ alors $\delta$ fixe aussi le point $y$. Par suite, $\delta$ est hyperbolique grâce au th\'eorème \ref{classi_dym_1}. Le groupe engendr\'e par $\g$ et $\delta$ agit proprement sur le segment $]x,y[ \subset \O$. Or, le groupe $\Aut(]x,y[)$ est isomorphe à $\R$; il existe donc des entiers $k,l \in \Z$ tel que $\g^k =\delta^l$.\\

Enfin, si $\g$ est parabolique et si $\delta$ fixe $x$ alors on vient de voir que $\delta$ ne peut être hyperbolique. Il est donc parabolique via le th\'eorème \ref{classi_dym_1}.

\end{proof}

\paragraph*{Points fixes et groupes libres}

Un simple argument de ping-pong donne la

\begin{prop}\label{ping-pong}
Soient $\g$ et $\delta$ deux \'el\'ements non elliptiques de $\Aut(\O)$ dont les points fixes sont deux \`a deux disjoints. Supposons que $\Fix(\g)$ et $\Fix(\delta)$ sont deux ensembles disjoints. Le groupe engendr\'e par les \'el\'ements $\g$ et $\delta$ est un sous-groupe discret de $\Aut(\O)$ qui contient un groupe libre à deux g\'en\'erateurs.
\end{prop}

\paragraph*{Les sous-groupes nilpotents discrets de $\Aut(\O)$}

\begin{coro}
Soit $\G$ un sous-groupe discret, nilpotent, infini, et sans torsion de $\Aut(\O)$. Alors
\begin{enumerate}
\item soit tous les \'el\'ements de $\G\smallsetminus\{Id\}$ sont hyperboliques et $\G$ est isomorphe à $\Z$;
\item soit tous les \'el\'ements de $\G\smallsetminus\{Id\}$ sont paraboliques.
\end{enumerate}
\end{coro}

\begin{proof}
La proposition \ref{ping-pong} montre que tous les \'el\'ements de $\G$ doivent avoir un point fixe commun, sinon le groupe $\G$ contiendrait un groupe libre non ab\'elien et ne serait donc pas nilpotent. La proposition \ref{point_fixe_dyn} montre qu'alors les \'el\'ements de $\G$ (diff\'erents de l'identit\'e) sont tous hyperboliques ou bien tous paraboliques. De plus, s'ils sont tous hyperboliques, le deuxième point de la proposition \ref{point_fixe_dyn} montre que $\G$ est isomorphe à $\Z$.
\end{proof}

On dira par la suite qu'un sous-groupe discret de $\Aut(\O)$ est
\begin{itemize}
 \item \emph{elliptique} si tous ses \'el\'ements sont elliptiques et fixent le même point;
 \item \emph{parabolique} s'il contient un sous-groupe d'indice fini dont tous les \'el\'ements sont paraboliques et fixent le même point;
 \item \emph{hyperbolique} s'il contient un sous-groupe d'indice fini engendr\'e par un \'el\'ement hyperbolique.
\end{itemize}
Le corollaire pr\'ec\'edent montre qu'un sous-groupe discret de $\Aut(\O)$, qui est virtuellement nilpotent et infini, est soit parabolique, soit hyperbolique.

On remarquera qu'un sous-groupe parabolique contient n\'ecessairement uniquement des \'el\'ements paraboliques alors qu'un sous-groupe hyperbolique peut contenir des \'el\'ements elliptiques d'ordre 2 qui \'echange les deux points fixes des \'el\'ements hyperboliques du groupe en question.


\section{Les notions classiques vues dans le monde projectif}

Le but de cette partie est de rappeler les d\'efinitions d'ensemble limite, de domaine de discontinuit\'e, d'action de convergence et de domaine fondamental; cela nous permettra de montrer dans le cadre des g\'eom\'etries de Hilbert des propositions bien connues de g\'eom\'etrie hyperbolique.

\subsection{Ensemble limite et domaine de discontinuit\'e}

Comme en g\'eom\'etrie hyperbolique, on peut d\'efinir l'ensemble limite et le domaine de discontinuit\'e d'un sous-groupe discret de $\Aut(\O)$ de la façon suivante.

\begin{defi}\label{def_ens_lim}
Soit $\G$ un sous-groupe discret de $\Aut(\O)$ et $x\in\O$. L'ensemble limite $\LG$ de $\G$ est le sous-ensemble de $\dO$ suivant:
$$\LG = \overline{\G\cdot x} \smallsetminus \G\cdot x,$$
o\`u $x$ est un point quelconque de $\O$.
\emph{Le domaine de discontinuit\'e} $\Og$ de $\G$ est le compl\'ementaire de l'ensemble limite de $\G$ dans $\O$.
\end{defi}

\begin{center}
\begin{figure}[h!]
  \centering
\includegraphics[width=7cm]{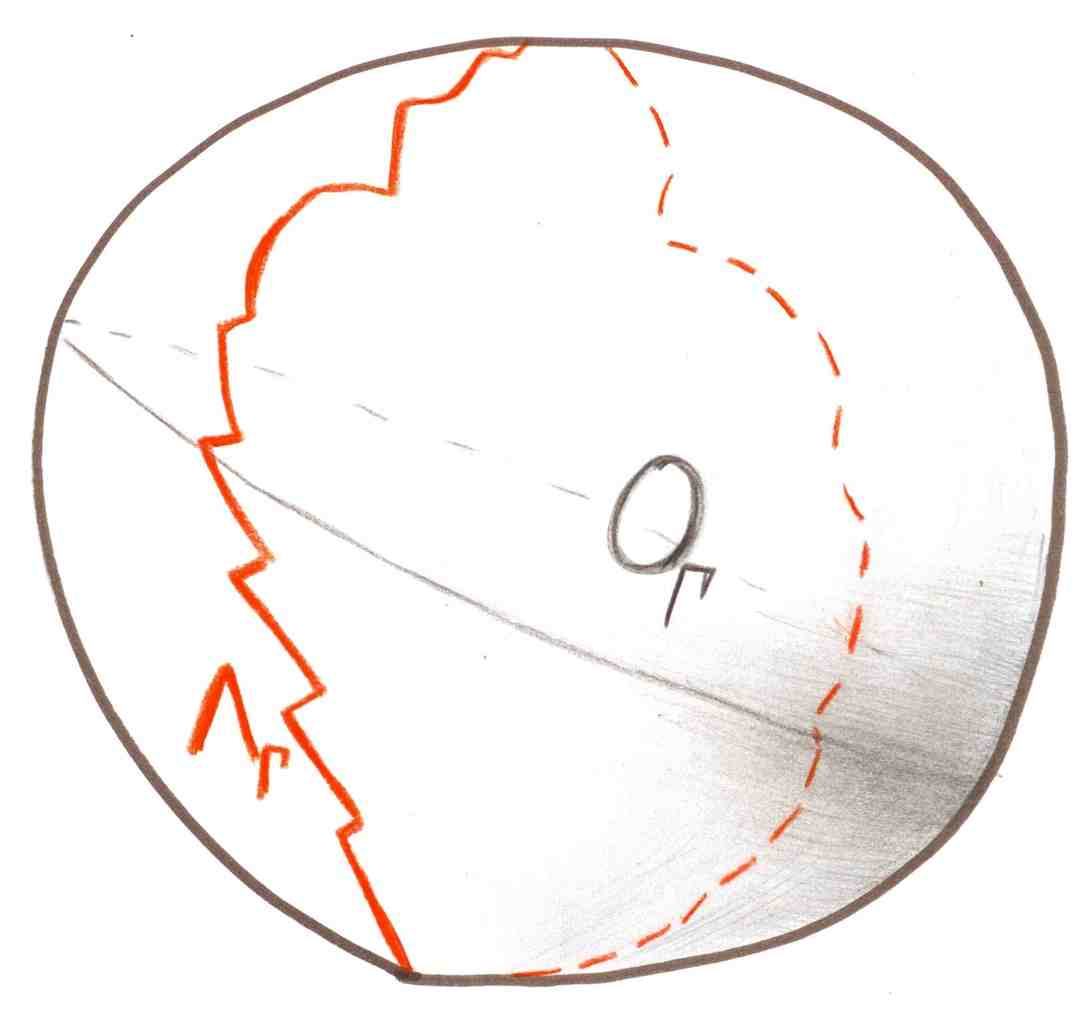}
\caption{Ensemble limite et domaine de discontinuit\'e}
\end{figure}
\end{center}

L'ensemble limite $\LG$, s'il n'est pas infini, est vide ou consiste en 1 ou 2 points, auxquels cas $\G$ est respectivement elliptique, parabolique ou hyperbolique. On dit que $\G$ est \emph{non \'el\'ementaire} si $\LG$ est infini. Dans ce dernier cas, l'ensemble limite $\LG$ est le plus petit ferm\'e $\G$-invariant non vide de $\dO$. Ainsi, $\LG$ est l'adh\'erence des points fixes des \'el\'ements hyperboliques de $\G$. Le lemme suivant d\'ecrit grossièrement l'ensemble limite.

\begin{lemm}\label{bord_dom_fond}
Soit $\G$ un sous-groupe discret non \'el\'ementaire de $\Aut(\O)$. L'ensemble limite $\Lambda_{\G}$ est un compact parfait. De plus, si $\Lambda_{\G} \neq \partial \O$ alors $\Lambda_{\G}$ est d'int\'erieur vide.
\end{lemm}

\begin{proof}
On commence par montrer par l'absurde que $\Lambda_{\G}$ est un compact parfait. Puisque $\LG$ est l'adh\'erence des points fixes des \'el\'ements hyperboliques de $\G$, s'il existe un point isol\'e $x \in \Lambda_{\G}$ alors le point $x$ est fix\'e par un \'el\'ement hyperbolique $\g$. On peut supposer que $x$ est point fixe attractif de $\g$. Comme $\G$ n'est pas \'el\'ementaire, il existe un point $y \in \Lambda_{\G}$ qui n'est pas fix\'e par $\g$. La suite $(\g^n \cdot y)_{n \in \N}$ converge donc vers le point $x$ lorsque $n$ tend vers l'infini et tous les points de cette suite sont diff\'erents de $x$. Ce qui contredit le fait que le point $x$ est isol\'e.

Montrons à pr\'esent le deuxième point. Le ferm\'e $\overline{\Og\cap \dO}$ est un ferm\'e $\G$-invariant donc il contient $\LG$, qui est le plus petit ferm\'e $\G$-invariant. Par cons\'equent, $\overline{\Og\cap \dO} = \overline{\dO}$, autrement dit $\Og\cap\dO$ est dense. Autrement dit encore, $\Lambda_{\G}$ est d'int\'erieur vide.
\end{proof}

\begin{rema} 
Il est possible de d\'efinir l'ensemble limite d'un sous-groupe de $\ss$ agissant sur $\PP^n$ dans des cas plus g\'en\'eraux. On pourra se r\'ef\'erer aux travaux de Benoist \cite{MR1767272},  Yves Guivarc'h \cite{MR1074315} ou Guivarc'h et Jean-Pierre Conze \cite{MR1803464}. En fait, il suffit que le groupe soit irr\'eductible et proximal.
\end{rema}

\begin{defi}
Soit $\Gamma$ un sous-groupe de $\ss$. On dira que $\G$ est \emph{irr\'eductible} lorsque les seuls sous-espaces vectoriels de $\R^{n+1}$ invariants par $\G$ sont $\{0 \}$ et $\R^{n+1}$. On dira que $\G$ est \emph{fortement irr\'eductible} si tous ses sous-groupes d'indice fini sont irr\'eductibles, autrement dit si $\G$ ne pr\'eserve pas une union finie de sous-espaces vectoriels non triviaux.
\end{defi}

\par{
Lorsque $\G$ est un sous-groupe discret de $\Aut(\O)$, $\G$ est irr\'eductible si et seulement si l'int\'erieur $C(\LG)$ de l'enveloppe convexe de son ensemble limite $\LG$ est non vide. Dans ce cas, $C(\LG)$ est le plus petit ouvert convexe de $\PP^n$ pr\'eserv\'e par $\G$. En fait, il n'est pas difficile de voir qu'alors $\G$ est fortement irr\'eductible. En effet, si $G$ est un sous-groupe d'indice fini de $\G$ alors pour tout \'el\'ement hyperbolique $h$ de $\G$, il existe un entier $n \geqslant 1$ tel que $h^n\in G$, et donc  $\Lambda_{G}=\LG$.
}

\subsection{Action de $\G$ sur son domaine de discontinuit\'e}

Le but de cette partie est de montrer le lemme suivant.

\begin{lemm}\label{dom_disc}
Soit $\G$ un sous-groupe discret de $\Aut(\O)$. Le groupe $\G$ agit proprement discontin\^ument sur $\Og$.
\end{lemm}

\paragraph*{Compactification du groupe des transformations projectives de $\PP^n$}

Le groupe $\textrm{PGL}_{n+1}(\R)$ est un ouvert dense de l'espace projectif $\PP(\textrm{End}(\R^{n+1}))$, où $\textrm{End}(\R^{n+1})$ d\'esigne l'espace vectoriel des endomorphismes de $\R^{n+1}$. Ce dernier nous fournit donc une compactification de $\textrm{PGL}_{n+1}(\R)$ en tant qu'espace topologique. On rappelle qu'un \'el\'ement $\g$ de $\PP(\textrm{End}(\R^{n+1}))$ d\'efinit une application de  $\PP^{n} \smallsetminus N(\g)$ vers $\PP^{n} $, où $N(\g)$ est le projectivis\'e du noyau de n'importe quel relev\'e de $\g$ à  $\textrm{End}(\R^{n+1})$.

De plus, la proposition suivante permet de d\'ecrire cette compactification.

\begin{prop}\label{prop_gold_compactification}
Soient $(\g_n)_{n \in \N}$ une suite d'\'el\'ements du groupe $\textrm{PGL}_{n+1}(\R)$ et $\g_{\infty}$ un \'el\'ement de $\PP(\textrm{End}(\R^{n+1}))$. La suite $(\g_n)_{n \in \N}$ converge vers $\g_{\infty}$ dans $\PP(\textrm{End}(\R^{n+1}))$ si et seulement si la suite $(\g_n)_{n \in \N}$ converge vers $\g_{\infty}$ sur tout compact de  $\PP^{n} \smallsetminus N(\g_{\infty})$.
\end{prop}

Cette proposition et des d\'etails sur la compactification du groupe $\textrm{PGL}_{n+1}(\R)$ sont  donn\'es dans \cite{MR0124005} et aussi dans \cite{NoteGoldman}.

\paragraph*{Action de convergence}

\begin{defi}\label{def_action_conv}
Soit $\G$ un groupe agissant par hom\'eomorphisme sur un compact parfait $X$. L'action de $\G$ sur $X$ est \emph{une action de convergence} si, pour toute suite $(\g_n)_{n \in \N}$ de $\G$, il existe une sous-suite $(\g_{n_i})_{i \in \N}$ de $\G$ et deux points $a,b \in X$ tels que $(\g_{n_i})_{i \in \N} \in \G^{\N}$ converge uniform\'ement vers $b$ sur $X \smallsetminus \{ a\}$.
\end{defi}

\begin{prop}\label{prop_pre_action_conv}
Soient $\O$ un ouvert proprement convexe de $\PP^n$ et  $(\g_n)_{n \in \N}$ une suite d'un sous-groupe $\G$ de $\Aut(\O)$. On suppose que la  suite $(\g_n)_{n \in \N}$ converge vers $\g_{\infty}$ dans $\PP(\textrm{End}(\R^{n}))$ et que l'application $\g_{\infty}$ est singulière.\\
Alors les sous-espaces $\textrm{Im}(\g_{\infty})$ et $N(\g_{\infty})$ rencontrent $\overline{\O}$ mais ne rencontrent pas $\O$.\\
En particulier, si le convexe $\O$ est strictement convexe à bord $\C^1$ alors $\textrm{Im}(\g_{\infty})$ est r\'eduite à un point $z$ qui est inclus dans $\partial \O$ et $N(\g_{\infty})$ est un hyperplan dont l'intersection avec $\overline{\O}$ est r\'eduite à un point $x \in \partial \O$. De plus, le point $z$ est dans l'ensemble limite de $\G$.
\end{prop}

\begin{proof}
L'action du groupe $\Aut(\O)$ sur $\O$ est propre. Par cons\'equent, pour tout point $x\in \O$, tout point d'accumulation de la suite $(\g_n(x))_{n \in \N}$ est sur le bord $\partial \O$ de $\O$. Mieux, si un point $x_0 \in \O$ est tel que la suite $(\g_n(x_0))_{n \in \N}$ converge vers un point $y_{x_0} \in \partial \O$, la proposition \ref{non_stric_conv1} montre qu'il existe une facette $S$ de $\overline{\O}$ incluse dans $\partial \O$ contenant $y_{x_0}$ telle que, pour tout $x\in \O$, la suite $(\g_n(x))_{n \in \N}$ sous-converge vers un point $y_x \in S$.

Remarquons ensuite que, par construction de la compactification, l'ensemble $N(\g_{\infty})$ n'est pas vide et n'est pas $\PP^n$ tout entier. Il existe donc un point $x_0 \in \O$ tel que $x_0 \notin N(\g_{\infty})$. Le paragraphe pr\'ec\'edent montre qu'alors aucun point de $\O$ n'est dans $N(\g_{\infty})$ et qu'il existe une facette $S$ de $\overline{\O}$ incluse dans $\partial \O$ et telle que $\g_{\infty}(\O) \subset S$. Comme $\O$ est un ouvert de $\PP^n$, on a $\textrm{Im}(\g_{\infty}) \subset E$, où $E$ est le support de $S$. Ce qui montre le r\'esultat pour $\textrm{Im}(\g_{\infty})$.

Un raisonnement par dualit\'e permet de montrer le second point. Le noyau de $\g^*= ^t\hspace{-.1cm}\g^{-1}$ n'est rien d'autre que le dual de l'image de $\g$. On obtient ainsi le r\'esultat pour $N(\g_{\infty})$ en utilisant le convexe dual $\O^*$ de $\O$ d\'efini au paragraphe \ref{def_dualite}.

Les am\'eliorations dans le cas strictement convexe à bord $\C^1$ sont \'evidentes.
\end{proof}

\begin{theo}\label{action_conv}
Soient $\O$ un ouvert strictement convexe à bord $\C^1$ de $\PP^n$ et $\G$ un sous-groupe discret et irr\'eductible de $\Aut(\O)$. Les actions de $\G$ sur les compacts $\partial \O$ et $\overline{\O}$ sont des actions de convergence.
\end{theo}

\begin{proof}
La proposition \ref{prop_pre_action_conv} montre que tout point d'accumulation d'une suite $(\g_n)_{n \in \N}$ d'automorphismes de $\O$ qui n'est pas stationnaire est de la forme 

$$
\begin{array}{cccc}
b_a : & \overline{\O} & \rightarrow & \overline{\O}\\
         &       x              & \mapsto       &  \left\{ \begin{array}{ccc}
																			b  & $si$ & x\neq a\\
																			a  & $si$ & x=a      \\
         
         \end{array}
         \right.
\end{array}
$$
où le point $b$ est l'unique point de l'image du point d'accumulation $\g_{\infty}$ choisi de la suite $(\g_n)_{n \in \N}$, et le point $a$ est l'intersection du noyau $\g_{\infty}$ avec $\overline{\O}$.
La proposition \ref{prop_gold_compactification} montre que la suite $(\g_n)_{n \in \N}$ sous-converge uniform\'ement sur les compacts de $\overline{\O} \smallsetminus \{ a\}$ vers $b_a$. C'est ce qu'il fallait montrer dans les deux cas.
\end{proof}

\begin{proof}[D\'emonstration du lemme  \ref{dom_disc}]
Supposons que l'action de $\G$ sur $\Og = \overline{\O} \smallsetminus \Lambda_{\G}$ ne soit pas proprement discontinue. Il existe donc un compact $K$ et une suite d'automorphismes $(\g_n)_{n \in \N}$  tel que $\g_n(K) \cap K \neq \varnothing$ pour tout $n \in \N$.

L'action de $\G$ sur $\overline{\O}$ est de convergence (th\'eorème \ref{action_conv}), il existe donc deux points $a$ et $b$ tels que la suite $(\g_n)_{n \in \N}$ sous-converge vers $b$ uniform\'ement sur les compacts de $\overline{\O} \smallsetminus \{ a \}$. De plus, le point $b$ est un point de $\Lambda_{\G}$. Par cons\'equent, il existe un voisinage $U$ de $b$ dans $\overline{\O}$ tel que $U \cap K = \varnothing$.

D'un autre côt\'e, si $n$ est assez grand, on a $\g_n(K) \subset U$, ce qui contredit le fait que $\g_n(K) \cap K \neq \varnothing$ pour tout $n \in \N$.
\end{proof}

\subsection{Domaines fondamentaux}

Le th\'eorème de Dirichlet possède un analogue dans le monde projectif convexe. Rappelons qu'un \emph{domaine fondamental} pour l'action d'un groupe discret $\G$ sur un espace topologique $X$ est un ferm\'e d'int\'erieur non vide $D$ de $X$ tel que $\G\cdot D=X$ et $\g\cdot\mathring{D}\cap\g'\cdot\mathring{D}=\varnothing$ si et seulement si $\g$ et $\g'$ sont deux \'el\'ements distincts de $\G$. Un domaine fondamental est dit \emph{localement fini} si tout compact de $X$ ne rencontre qu'un nombre fini de translat\'es de $D$ par $\G$.

\begin{theo}[Jaejeong Lee, \cite{MR2712298}]\label{lee}
Soient $\O$ un ouvert proprement convexe et $\Gamma$ un sous-groupe discret de $\Aut(\O)$. Il existe un domaine fondamental convexe et localement fini pour l'action de $\G$ sur $\O$.
\end{theo}

On pourra trouver une courte d\'emonstration de ce th\'eorème dans \cite{Marquis:2009kq}.


\section{Action g\'eom\'etriquement finie sur $\LG$ et sur $\O$}\label{par_geo_fini}

\subsection{Action affine des sous-groupes paraboliques}

\begin{defi}
Soit $\O$ un ouvert convexe de $\PP^n$ (a priori non proprement convexe). Un sous-espace affine $\F$ inclus dans $\O$ est dit \emph{maximal} lorsqu'il n'existe pas de sous-espace affine de $\PP^n$ contenant strictement $\F$ et inclus dans $\O$.
\end{defi}

On note $\pi$ la projection naturelle $\R^{n+1}\smallsetminus \{0\} \rightarrow \PP^n$. Tout sous-espace affine $F$ de $\PP^n$ est la projection via $\pi$ d'un sous-espace affine $\tilde{F}$ de $\R^{n+1}$ qui ne contient pas l'origine de $\R^{n+1}$. Deux sous-espaces affines $\tilde{F}$ et $\tilde{F}'$ ont la m\^eme trace $\pi(\tilde{F})=\pi(\tilde{F}')$ si et seulement s'ils engendrent le m\^eme sous-espace vectoriel de $\R^{n+1}$.

On dira que deux sous-espaces affines de $\PP^n$ ont la \emph{même direction} lorsque les sous-espaces affines de $\R^{n+1}$ correspondant ont la même direction, c'est-à-dire la même partie lin\'eaire. La direction commune est un sous-espace vectoriel de $\R^{n+1}$, qui correspond à un sous-espace projectif de $\PP^n$. Par exemple, la direction d'une carte affine est pr\'ecis\'ement son hyperplan \`a l'infini.

\begin{rema}\label{decomp_conv}
Soit $\O$ un ouvert convexe de $\PP^n$. Deux sous-espaces affines maximaux inclus dans $\O$ ont la même direction $F_{max}$, qui est un sous-espace projectif de $\PP^n$. La projection de $\O$ dans l'espace projectif $\PP \bigg( \Quotient{\R^{n+1}}{\tilde{F}_{max}} \bigg)$ est un ouvert proprement convexe, où on a not\'e $\tilde{F}_{max}$ est un relev\'e de $F_{max}$ à $\R^{n+1}$.
\end{rema}

Si $p$ est un point et $A$ une partie de $\PP^n$, on note $\D_p(A)$ l'ensemble des droites concourantes en $p$ et rencontrant $A$. 

\begin{center}
\begin{figure}[h!]
  \centering
\includegraphics[width=10cm]{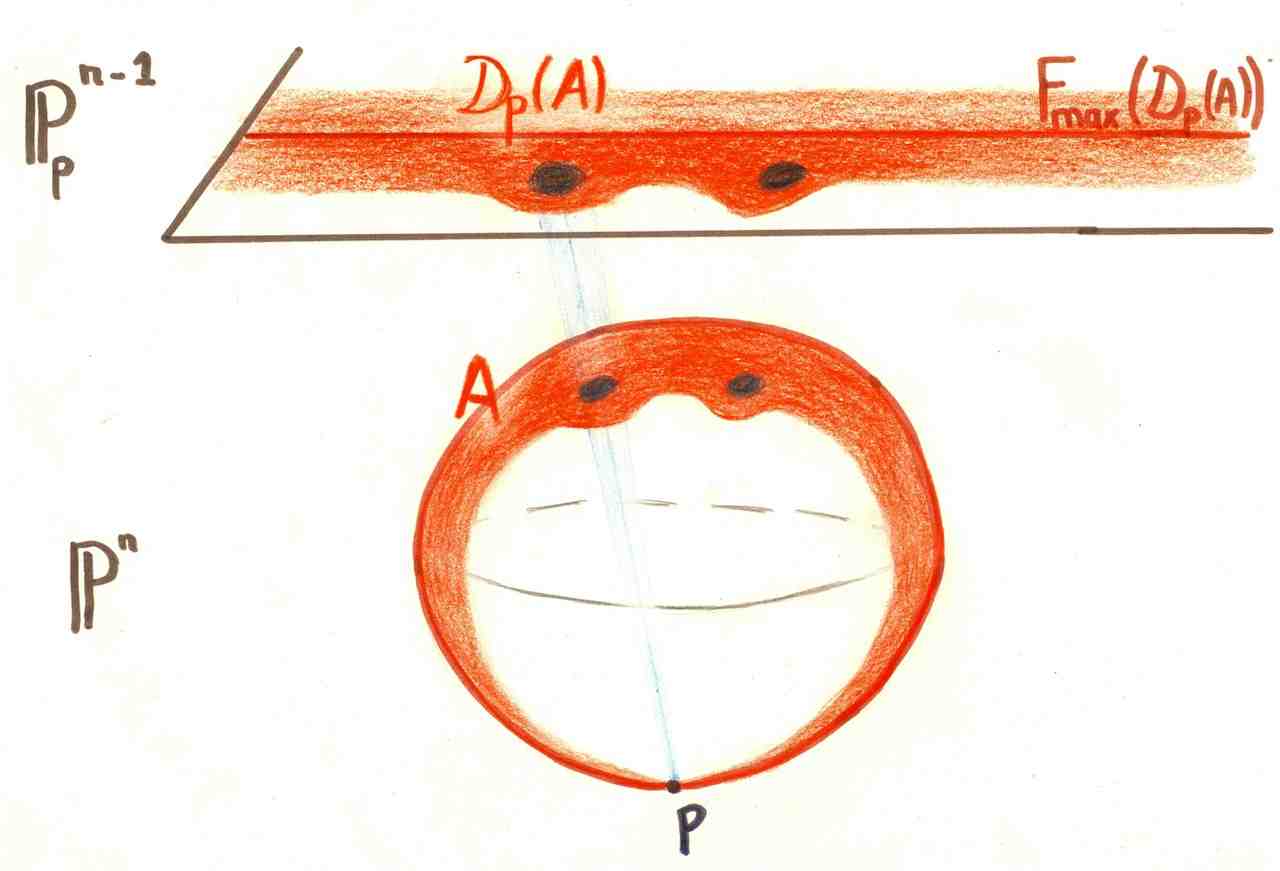}
\caption{}
\end{figure}
\end{center}

La proposition suivante est imm\'ediate.

\begin{prop}\label{tangent}
Soient $\O$ un ouvert \underline{proprement convexe} de $\PP^n$ et $p \in \partial \O$. L'ensemble $\D_p=\D_p(\O)$ des droites concourantes en $p$ et rencontrant $\O$ est un ouvert convexe de l'espace projectif $\PP^{n-1}_p= \PP\left(\Quotient{\R^{n+1}}{p}\right)$ des droites concourantes en $p$.\\
Un point $p \in \partial \O$ est un point de classe $\C^1$ de $\partial \O$ si et seulement si le convexe $\D_p(\O)$ est une carte affine $\A^{n-1}_p$ de $\PP^{n-1}_p$.
\end{prop}

\begin{rema}
Si le point $p \in \partial \O$ n'est pas un point de classe $\C^1$ de $\partial \O$ alors les espaces affines maximaux inclus dans le convexe $\D_p(\O)$ ont la même direction (remarque \ref{decomp_conv}). Cette direction commune est pr\'ecis\'ement l'ensemble des directions dans lesquelles $\partial \O$ est de classe $\C^1$ en $p$.
\end{rema}
On rappelle que, sauf mention explicite, l'ouvert $\O$ est un ouvert proprement convexe strictement convexe à bord $\C^1$.

\begin{lemm}\label{lem_proj_aff}
Soit $\G$ un sous-groupe discret et sans torsion de $\Aut(\O)$, qui fixe un point $p$ de $\partial \O$. Il existe une repr\'esentation fidèle de $\G$ dans le groupe affine $\textrm{Aff}(\R^{n-1})$ des transformations affines de $\R^{n-1}$.
\end{lemm}

\begin{proof}
Si le groupe $\G$ pr\'eserve le point $p$ alors il pr\'eserve l'ensemble des droites passant par $p$. Or, l'ensemble des droites passant par $p$ est un espace projectif $\PP^{n-1}_p$, trace de l'espace vectoriel quotient $\Quotient{\R^{n+1}}{p}$. Le groupe $\G$ agit projectivement sur cet espace projectif $\PP^{n-1}_p$. De plus, comme $p$ est un point $\C^1$ de $\dO$, le groupe $\G$ pr\'eserve l'hyperplan tangent $T_p\partial \O$ à $\partial \O$ en $p$; il agit donc par transformation affine sur l'espace affine $\A^{n-1}_p$ des droites passant par $p$ qui ne sont pas incluses dans $T_p\partial \O$, qui n'est rien d'autre que $\mathcal{D}_p(\O)$.\\
Cette repr\'esentation est fidèle: un \'el\'ement qui fixe toutes les droites issus de $p$, fixerait tous les points de $\partial \O$.
\end{proof}

\begin{nota}
Si $p$ est un point de $\dO$, on notera à pr\'esent $\A^{n-1}_p$ l'espace affine $\D_p(\O)$ des droites passant par $p$ qui ne sont pas contenues dans l'hyperplan tangent $T_p\partial \O$ à $\partial \O$ en $p$.\\
Si $C$ est une partie convexe de $\O$, on d\'esignera par $\overline{\mathcal{D}_p(C)}$ l'adh\'erence, \emph{dans $\A^{n-1}_p$}, de l'ensemble $\mathcal{D}_p(C)$ des droites concourantes en $p$ rencontrant $C$ . Remarquons que si $A$ est une partie de $\overline{O}$, alors $\overline{\mathcal{D}_p(C(A))}$ n'est rien d'autre que l'enveloppe convexe de $\mathcal{D}_p(A\setminus\{p\})$ dans $\A_p^{n-1}$.
\end{nota}

\begin{center}
\begin{figure}[h!]
  \centering
\includegraphics[width=10cm]{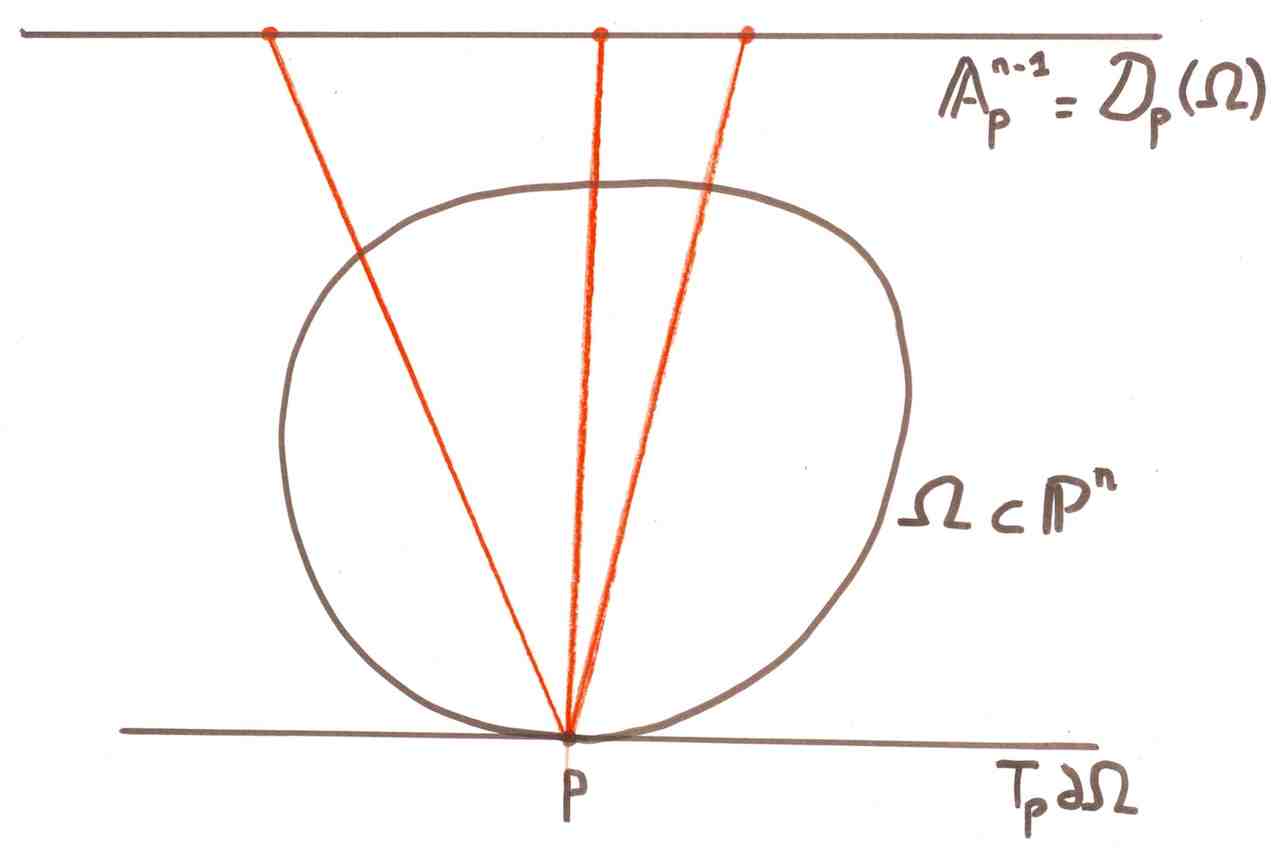}
\caption{}
\end{figure}
\end{center}

\subsection{Finitude g\'eom\'etrique}

Nous allons d\'efinir deux notions de finitude g\'eom\'etrique via la nature des points de l'ensemble limite $\LG$. Pour cela, on s'inspire des d\'efinitions donn\'ees en g\'eom\'etrie hyperbolique ou plus g\'en\'eralement pour les espaces m\'etriques hyperboliques, pour lesquels on dispose des m\^emes objets que dans le cas pr\'esent.

\subsubsection{Points paraboliques born\'es}

La d\'efinition suivante fait l'unanimit\'e pour l'action d'un groupe discret par isom\'etries sur un espace Gromov-hyperbolique. Nous l'adoptons ici.

\begin{defi}
Soit $\G$ un sous-groupe discret de $\Aut(\O)$. Un point $x \in \LG$ est \emph{un point parabolique born\'e}  si l'action du groupe $\Stab_{\G}(x)$ sur $\Lambda_{\G}\smallsetminus \{ x\}$ est cocompacte.\\
Le \emph{rang} d'un point parabolique born\'e $x \in \LG$ est la dimension cohomologique virtuelle du groupe $\Stab_{\G}(x)$. Le point parabolique $x \in \LG$ est dit \emph{de rang maximal} si son rang vaut $\dim \O - 1$, autrement dit si $\Stab_{\G}(x)$ agit de façon cocompacte sur $\dO \smallsetminus \{x\}$. 
\end{defi}

\begin{rema}
La dimension cohomologique d'un groupe discret $\G$ sans torsion est un entier $n_{\G}$ tel que, pour toute action libre et propre de $\G$ sur une vari\'et\'e contractible de dimension $n$, on a $n \geqslant n_{\G}$, avec \'egalit\'e si et seulement si l'action est cocompacte. Si le groupe $\G$ est virtuellement sans torsion, alors on peut montrer que tous ses sous-groupes d'indice fini sans torsion ont la même dimension cohomologique et on appelle ce nombre la dimension cohomologique virtuelle de $\G$. On pourra consulter \cite{MR0422504}.
\end{rema}

Remarquons que si $x$ est un point parabolique born\'e alors $\Stab_{\G}(x)$ est parabolique, c'est-à-dire qu\`a indice fini pr\`es, il est compos\'e uniquement d'\'el\'ements paraboliques qui fixent le même point.

\subsubsection{Points limites coniques}

En g\'eom\'etrie hyperbolique, on trouve la d\'efinition suivante, qui convient \`a notre cadre:

\begin{defi}\label{def_lim_con}
Soit $\G$ un sous-groupe discret de $\Aut(\O)$. On dit qu'un point $x \in \LG$ est un \emph{point limite conique} lorsqu'il existe une suite d'\'el\'ements $(\g_n)_{n \in \N}$ de $\G$, un point $x_0 \in \O$, une demi-droite $[x_1,x[$, et un r\'eel $C > 0$ tel que:
\begin{enumerate}
\item $\g_n \cdot x_0 \underset{ n \to \infty}{\to} x$
\item $d_{\O}(\g_n \cdot x_0, [x_1,x[) \leqslant C$
\end{enumerate}
\end{defi}

\begin{rema}\label{rem_lim_coni}
Un point $x \in \LG$ est un \emph{point limite conique} si et seulement si la projection d'une (et donc de toute) demi-droite terminant en $x$ sur $\Quo$ retourne une infinit\'e de fois dans un compact de $\Quo$.
\end{rema}

Cette d\'efinition de point conique ne convient pas \`a un espace m\'etrique $X$ Gromov-hyperbolique, et on en trouve une autre dans ce contexte: un point $x \in \partial X$ est un point limite conique pour l'action d'un groupe $\G$ sur $X$ lorsqu'il existe deux points distincts $a,b \in \partial X$ et une suite d'\'el\'ements $(\g_n)_{n \in \N}$ de $\G$ tel que $\g_n \cdot x \underset{n \to \infty}{\to} a$ et $\g_n \cdot y \underset{n \to \infty}{\to} b$ pour tout $y \neq x$.\\
Bien s\^ur, cette derni\`ere d\'efinition est \'equivalente \`a la pr\'ec\'edente lorsqu'on l'applique \`a la g\'eom\'etrie hyperbolique. L'avantage de cette dernière d\'efinition est sa nature purement topologique et non g\'eom\'etrique. Cela reste vrai dans notre cas, et cela nous permettra de montrer la proposition \ref{dual_geo_fini}:

\begin{lemm}\label{lem_con_fort=faible}
Soit $\G$ un sous-groupe discret de $\Aut(\O)$. Un point $x \in \LG$ est un point limite conique si et seulement s'il existe deux points $a$ et $b$ distincts de $\dO$ et une suite d'\'el\'ements $(\g_n)_{n \in \N}$ de $\G$ tels que
\begin{itemize}                                                                                                                                                                                        \item $\g_n x$ tend vers $a$;
\item pour tout $y\in\dO\smallsetminus\{x\}$, $\g_n y$ tend vers $b$.                                                                                                                                                                                         \end{itemize}
\end{lemm}

\begin{proof}
Commençons par montrer que cette condition est suffisante. S'il existe deux points distincts $a,b \in \partial \O$ et une suite $(\delta_n)_{n \in \N}$ d'\'el\'ements de $\G$ tel que $\delta_n \cdot x \underset{n \to \infty}{\to} a$ et $\delta_n \cdot y \underset{n \to \infty}{\to} b$ pour tout $y \neq x$. On pose $\g_n = \delta_n^{-1}$ et on se donne $x_0 \in \O$.

La suite $(\g_n \cdot x_0)_{n \in \N} \to x$ car sinon la suite de termes $\delta_n(\g_n \cdot x_0) = x_0$ sous-convergerait vers $b$. Il faut à pr\'esent montrer que la quantit\'e suivante: $d_{\O}(\g_n \cdot x_0, [x_0 , x[ )$ est major\'ee ind\'ependamment de $n$. Mais, les automorphismes $\g_n$ sont des isom\'etries, on a donc  $d_{\O}(\g_n \cdot x_0, [x_0 , x[ ) = d_{\O}(x_0, \delta_n([x_0 , x[) ) \rightarrow  d_{\O}(x_0, ]b,a[ ) < \infty$ car $\delta_n \cdot x_0 \underset{n \to \infty}{\to} b$. La dernière in\'egalit\'e est stricte car $\O$ est strictement convexe.

Montrons à pr\'esent que cette condition est n\'ecessaire.

Il existe un point $x_0 \in \O$ et une suite $(\g_n)_{n \in \N}$ d'\'el\'ements de $\G$ tel que $\g_n \cdot x_0 \underset{n \to \infty}{\to} x$ et $d_{\O}(\g_n \cdot x_0, [x_0 , x[ )$ est major\'ee par une constante $C>0$ ind\'ependamment de $n$. On pose $\delta_n = \g_n^{-1}$, on note $D$ la droite passant par $x_0$ et $x$, enfin on note $q$ le point d'intersection de $D$ avec $\partial \O$ qui n'est pas $x$. Les droites $\delta_n(D)$ forment une famille de droites qui rencontre la boule ferm\'ee de centre $x_0$ et de rayon $C$. On peut donc supposer quitte à extraire que ces droites convergent vers une droite $(a b)$, où les points $a,b \in \partial \O$ et $a \neq b$. On en d\'eduit que $\delta_n \cdot x \underset{n \to \infty}{\to} a$ et  $\delta_n \cdot q \underset{n \to \infty}{\to} b$.

Il vient que pour tout point $y \in [x_0,x[$, on a  $\delta_n \cdot y \underset{n \to \infty}{\to} b$. Il n'est pas difficile d'en d\'eduire alors que pour tout $y \in \overline{\O}$, si $y\neq x$ alors $\delta_n \cdot y \underset{n \to \infty}{\to} b$ car le point $b$ est extr\'emal.
\end{proof}

\subsection{Action g\'eom\'etriquement finie sur $\O$ et $\dO$}

On trouve la d\'efinition suivante, que ce soit en g\'eom\'etrie hyperbolique ou pour un espace Gromov-hyperbolique:

\begin{defi}\label{geofini_ghyp}
Soient $X$  un espace Gromov-hyperbolique et $\G$ un sous-groupe discret d'isom\'etries de $X$. L'action de $\G$ sur $X$ est dite \emph{g\'eom\'etriquement finie} lorsque tout point de l'ensemble limite $\LG$ est un point limite conique ou un point parabolique born\'e.
\end{defi}

En d\'epit des ressemblances, il s'av\`ere que cette d\'efinition ne va pas convenir dans notre cadre. Bien s\^ur, elle convient lorsque la g\'eom\'etrie de Hilbert est Gromov-hyperbolique mais nos hypoth\`eses sur le convexe sont bien plus faibles.\
En g\'eom\'etrie hyperbolique, la finitude g\'eom\'etrique admet des d\'efinitions \'equivalentes de nature plus g\'eom\'etriques, qui justifient l'appelation g\'eom\'etriquement fini. Ces derni\`eres font sens dans notre contexte mais ne sont plus \'equivalentes \`a la pr\'ec\'edente, sinon \`a une version plus forte, qui demande plus aux points paraboliques born\'es. C'est ce que nous introduisons maintenant.

\begin{defi}\label{def_uni_born}
Soit $\G$ un sous-groupe discret de $\Aut(\O)$. Un point $x \in \LG$ est \emph{un point parabolique uniform\'ement born\'e} si l'action du groupe $\Stab_{\G}(x)$ sur $\mathcal{D}_p(\overline{C(\LG\smallsetminus \{x\})})$ est cocompacte.
\end{defi}

\begin{rema}
La notion de point parabolique uniform\'ement born\'e n'a aucun int\'erêt en g\'eom\'etrie hyperbolique, autrement dit dans le cas où $\O$ est un ellipsoïde: en effet, tout point parabolique born\'e est automatiquement uniform\'ement born\'e.\\
Pour voir cela, plaçons-nous dans le modèle du demi-espace de Poincar\'e et supposons que le point $\infty$ est un point parabolique born\'e pour un groupe discret $\G$ d'isom\'etries de l'espace hyperbolique $\mathbb{H}^n$. Le groupe $\Stab_{\G}(\infty)$ agit donc cocompactement sur $\LG \smallsetminus \{ \infty \}$. Le point important est que le groupe $\Stab_{\G}(\infty)$ agit par isom\'etrie euclidienne sur l'espace euclidien $\partial \mathbb{H}^n \smallsetminus \{ \infty \}$. Il existe donc un sous-espace $F$ de celui-ci pr\'eserv\'e par $\Stab_{\G}(\infty)$ sur lequel $\Stab_{\G}(\infty)$ agit cocompactement; de plus, tout sous-espace $F'$ pr\'eserv\'e par $\Stab_{\G}(\infty)$ sur lequel $\Stab_{\G}(\infty)$ agit cocompactement est parallèle à $F$. Ainsi, l'ensemble $\LG \smallsetminus \{ \infty \}$ est inclus dans un voisinage tubulaire de rayon fini de $F$. Comme ce voisinage est convexe, on obtient que le point $\infty$ est un point parabolique uniform\'ement born\'e de $\G$.
\end{rema}

On peut donner alors la d\'efinition suivante:

\begin{defi}\label{geofini_hilbert}
Soit $\G$ un sous-groupe discret de $\Aut(\O)$. L'action de $\G$ sur $\dO$ (resp. $\O$) est dite \emph{g\'eom\'etriquement finie} lorsque tout point de l'ensemble limite est un point limite conique ou un point parabolique born\'e (resp. uniform\'ement born\'e). On dira que le quotient $M=\Quo$ est \emph{g\'eom\'etriquement fini} lorsque l'action de \underline{$\G$ sur $\O$} est g\'eom\'etriquement finie.
\end{defi}

Ceci introduit \underline{deux} notions diff\'erentes a priori: la finitude g\'eom\'etrique de l'action de $\G$ \underline{sur $\dO$} et la finitude g\'eom\'etrique de l'action de $\G$ \underline{sur $\O$}. On verra que ces deux notions sont effectivement diff\'erentes et que cela n'a rien d'\'evident. On essaiera aussi de dire quand elles coïncident. C'est l'objet de la derni\`ere partie.\\

La d\'efinition ``traditionnelle'' de finitude g\'eom\'etrique est donc celle dont on pr\'ecise ici qu'elle est \underline{sur $\dO$}. Comme on le verra dans la partie \ref{mainsection}, celle qui porte sur $\O$ admet des d\'efinitions \'equivalentes concernant la g\'eom\'etrie du quotient $\Quo$. Lorsque l'action de $\G$ est g\'eom\'etriquement finie sur $\dO$ mais pas sur $\O$, le quotient $\Quo$ ne jouit par cons\'equent d'aucune de ces propri\'et\'es g\'eom\'etriques, et on ne saurait qualifier sa g\'eom\'etrie de finie. Nous esp\'erons ainsi justifier notre terminologie.

\subsection{Dualit\'e}

Si $\g\in\Aut(\O)$ est hyperbolique, les seuls hyperplans projectifs tangents à $\dO$ pr\'eserv\'es par $\g$ sont les hyperplans $T_{x_{\g}^+}\dO$ et $T_{x_{\g}^-}\dO$ tangents à $\dO$ en ses deux points fixes. L'\'el\'ement correspondant $\g^*\in\G^*$ est donc aussi hyperbolique, ses points fixes sont $(x_{\g}^+)^* = T_{x_{\g}^+}\dO$ et $(x_{\g}^-)^* = T_{x_{\g}^-}\dO$. De même, on voit que si $\g\in\Aut(\O)$ est un \'el\'ement parabolique fixant $p\in\dO$ alors son dual $\g^*\in\G^*$ est parabolique de point fixe $p^*$. Cela implique en particulier qu'\'etant donn\'e un sous-groupe discret $\G \subset \Aut(\O)$, l'application duale $x\longmapsto x^*$ de $\dO$ dans $\dO^*$ envoie $\LG$ sur $\Lambda_{\G^*}$.

\begin{prop}\label{dual_geo_fini}
Soit $\G$ un sous-groupe discret de $\Aut(\O)$. L'action de $\G$ sur $\dO$ est g\'eom\'etriquement finie si et seulement si l'action de $\G^*$ sur $\dO^*$ est g\'eom\'etriquement finie.
\end{prop}

\begin{proof}
Bien sûr, il suffit de prouver une seule implication. Supposons donc que l'action de $\G$ sur $\dO$ est g\'eom\'etriquement finie. Il suffit de montrer que l'application $x\longmapsto x^*$ de $\dO$ dans $\dO^*$ envoie un point limite conique pour $\G$ sur un point limite conique pour $\G^*$ et un point parabolique born\'e pour $\G$ sur un point parabolique born\'e pour $\G^*$.\\

Soit donc $x\in\LG$ un point limite conique. Il existe donc, d'après le lemme \ref{lem_con_fort=faible}, deux points $a$ et $b$ distincts de $\dO$ et une suite d'\'el\'ements $(\g_n)_{n \in \N}$ de $\G$ tels que $\g_n x$ tend vers $a$ et pour tout $y\in\dO\smallsetminus\{x\}$, $\g_n y$ tend vers $b$. Le convexe $\O$ \'etant suppos\'e strictement convexe à bord $C^1$, cela implique la convergence de $\g_n x^*$ vers $a^*$ et de $\g_n y^*$ vers $b^*$ pour tout $y\not = x$, puisque ces points s'identifient aux plans tangents $T_{\g_n x}\dO$, $T_a\dO$, $T_{\g_n y}\dO$ et $T_b\dO$. Le point $x^*$ est donc un point limite conique.\\

Soit maintenant $x\in\LG$ un point parabolique born\'e. Le groupe $\Stab_{\G^*}(x^*)$ n'est rien d'autre que le groupe $(\Stab_{\G}(x))^*$. Or, $\Stab_{\G}(x)$ agit cocompactement sur $\LG\smallsetminus\{x\}$, donc sur $\{T_y\dO,\ y\in\LG\smallsetminus\{x\}\}$ qui s'identifie à $\Lambda_{\Gamma^*}\smallsetminus\{x^*\}$. Cela montre que $\Stab_{\G^*}(x^*)$ agit cocompactement sur $\Lambda_{\Gamma^*}\smallsetminus\{x^*\}$.
\end{proof}

\begin{rema}
Le corollaire \ref{fin_dual} montrera que l'action de $\G$ sur $\O$ est g\'eom\'etriquement finie si et seulement si l'action de $\G^*$ sur $\O^*$ l'est.
\end{rema}


\section{D\'ecomposition du quotient}\label{section_decomposition}
\subsection{Lemme de Zassenhaus-Kazhdan-Margulis}

Les auteurs ont montr\'e dans \cite{Crampon:2011fk} le lemme suivant qui est le premier pas vers la description des actions g\'eom\'etriquement finies.

\begin{lemm}\label{lem_mar}
En toute dimension $n$, il existe une constante $\varepsilon_n > 0$ tel que: pour tout ouvert proprement convexe $\O$ de $\PP^n$, et tout point $x \in \O$, tout groupe discret engendr\'e par des automorphismes $\g_1,...,\g_p \in \Aut(\O)$ qui v\'erifient $d_{\O}(x,\g_i \cdot x) \leqslant \varepsilon_n$ est virtuellement nilpotent.
\end{lemm}

Une telle constante $\varepsilon_n$ sera appel\'e \emph{constante de Margulis}.

\subsection{D\'ecomposition du quotient}

$\,$\\
$\,$\\
{\bf
Dans toute la suite, on se fixe un r\'eel $\varepsilon>0$ qui est une constante de Margulis pour les ouverts proprement convexes de $\PP^n$. Tous les r\'esultats qui suivent sont ind\'ependants de ce choix.\\
}

On va introduire ici les d\'efinitions et notations que nous utiliserons par la suite. Soit $\G$ un sous-groupe discret de $\Aut(\O)$. Pour tout sous-groupe $G$ de $\G$, on note
\begin{itemize}
 \item pour $x\in\O$, $G_{\varepsilon}(x)$ le groupe engendr\'e par les \'el\'ements $\g \in G$ tels que $d_{\O}(x,\g \cdot x) < \varepsilon$ ;
 \item $\O_{\varepsilon}(G) = \{  x \in \O \,\mid \, G_{\varepsilon}(x) \textrm{ est infini} \}$;
 \item $\O^c_{\varepsilon}(G) = \{  x \in \O \,\mid \, G_{\varepsilon}(x) \textrm{ est infini et parabolique}\}$;
 \item $M_{\varepsilon}(G)=\Quotient{\O_{\varepsilon}(G)}{\G}$ et $M^c_{\varepsilon}(G)=\Quotient{\O^c_{\varepsilon}(G)}{\G}$ les projections de ces diff\'erents ensembles sur $M=\Quo$.

\end{itemize}

Dans le cas où $G$ est le groupe $\G$ tout entier, on abrègera ces notations en $\O_{\varepsilon}$, $\O^c_{\varepsilon}$, $M_{\varepsilon}$ et $M^c_{\varepsilon}$.\\
La partie $M_{\varepsilon}$ est la \emph{partie fine} de $M$. Dans le cas où $M$ est une vari\'et\'e, autrement dit quand $\G$ est sans torsion, c'est l'ouvert des points de $M$ dont le rayon d'injectivit\'e est strictement inf\'erieur à $\varepsilon$.\\
Le compl\'ementaire de $\O_{\varepsilon}$ dans $\O$ sera not\'e $\O^{\varepsilon}$ et sa projection sur $M$, $M^{\varepsilon}$. L'ensemble $M^{\varepsilon}$ est la \emph{partie \'epaisse} de $M$, compl\'ementaire de la partie fine dans $M$. Lorsque $M$ est une vari\'et\'e, c'est l'ensemble des points de $M$ dont le rayon d'injectivit\'e est sup\'erieur ou \'egal à $\varepsilon$.\\
L'ensemble $M^c_{\varepsilon}$ est la \emph{partie cuspidale} de $M$. Son compl\'ementaire dans $\O$ sera not\'e $\O^{nc}_{\varepsilon}$; sa projection $M^{nc}_{\varepsilon}$, compl\'ementaire de $M^c_{\varepsilon}$, est la \emph{partie non cuspidale} de $M$. Enfin, on appellera les composantes connexes de la partie cuspidale de $M^c_{\varepsilon}$ les \emph{cusps} de $M$.\\
Enfin, on d\'esignera par $C(\LG)$ l'enveloppe convexe de $\LG$ dans $\O$. Le \emph{c\oe ur convexe} de $M$ est le quotient est l'adh\'erence du $\Quotient{C(\LG)}{\G}$ dans $\Quo$, on le note $C(M)$.\\
On remarquera que $C(\LG)$ est un ouvert convexe de $\O$ et que $C(M)$ est un ferm\'e de $\Quo$.
Le lemme suivant donne une première description de ces diff\'erentes parties.

\begin{center}
\begin{figure}[h!]
  \centering
\includegraphics[width=10cm]{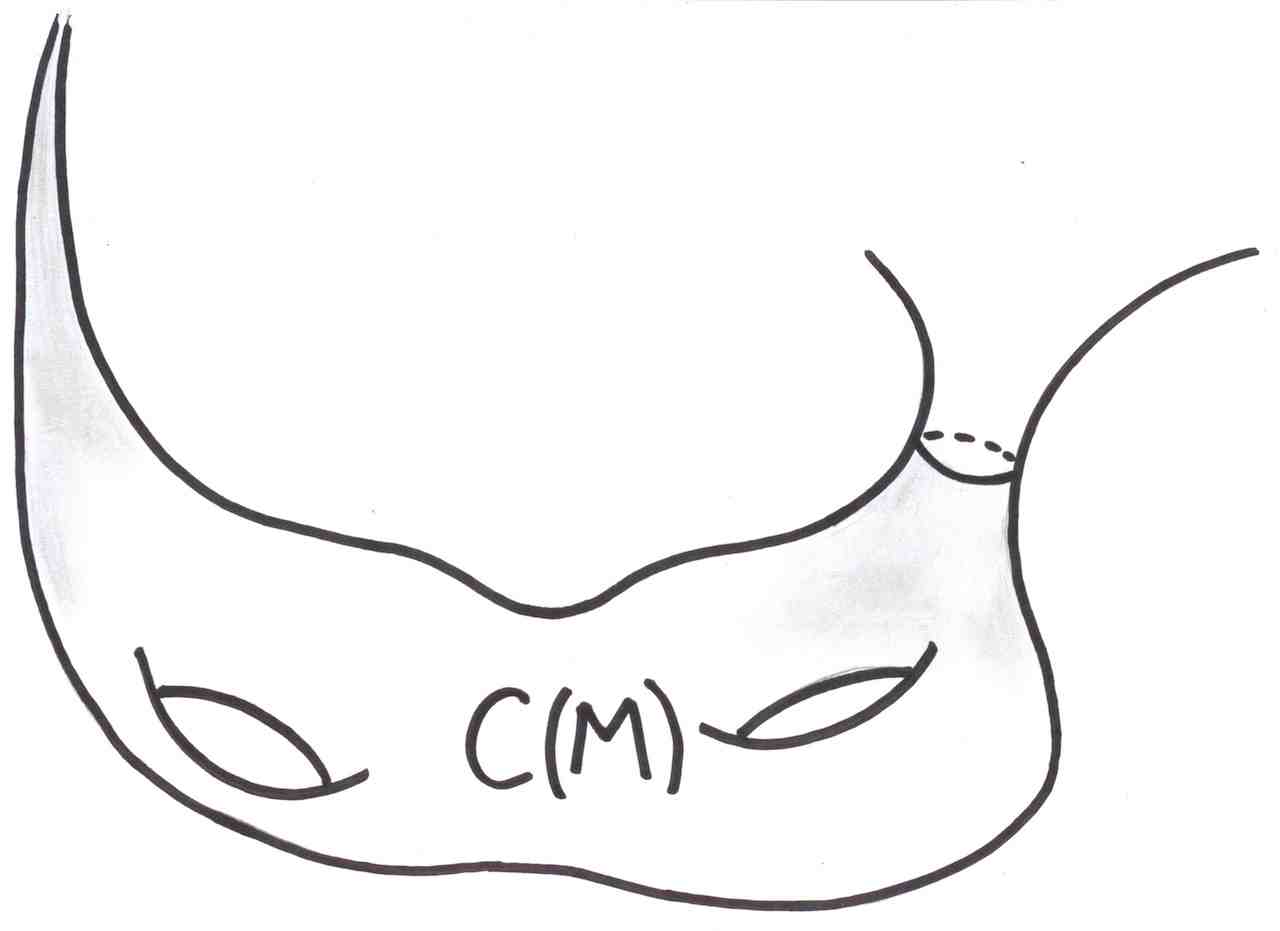}
\caption{Le c\oe ur convexe}
\end{figure}
\end{center}

\begin{center}
\begin{figure}[h!]
  \centering
\includegraphics[width=10cm]{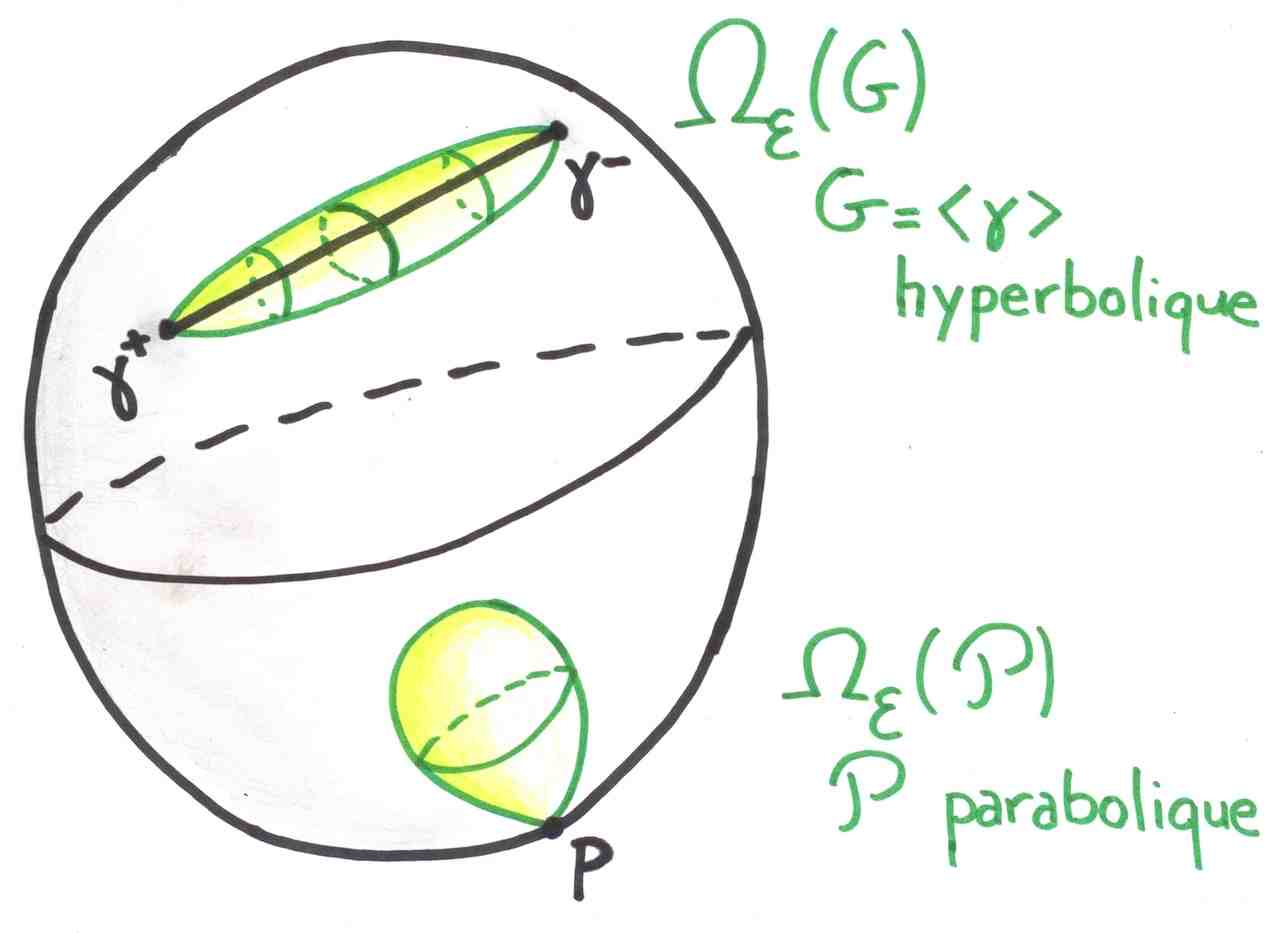} 
\caption{Parties fine et \'epaisse}
\end{figure}
\end{center}

\begin{lemm}\label{decomposition_partie_fine}
Soit $\G$ un sous-groupe discret de $\Aut(\O)$.
\begin{enumerate}
 \item La partie fine de $M$ est la r\'eunion disjointe des parties $M_{\varepsilon}(G)$ où $G$ parcourt les sous-groupes virtuellement nilpotents maximaux de $\G$, c'est-à-dire les sous-groupes hyperboliques et paraboliques maximaux de $\G$.
 \item Les parties $M_{\varepsilon}(G)$, où $G$ parcourt les sous-groupes virtuellement nilpotents maximaux de $\G$, sont connexes, et d'adh\'erences disjointes.
 \item Lorsque $G$ est un sous-groupe hyperbolique de $\G$, la partie $M_{\varepsilon}(G)$ est relativement compacte dans $M$.
 \item\label{partie_fine_cusp} Lorsque $G$ est un sous-groupe parabolique de $\G$ fixant $p\in\dO$, la partie $\O_{\varepsilon}(G)$ est \'etoil\'ee dans $\O$ en $p$, et $p$ est le seul point de $\dO$ adh\'erent à $\O_{\varepsilon}(G)$.
 \item La partie cuspidale est la r\'eunion disjointe des parties $M_{\varepsilon}(G)$, où $G$ parcourt les sous-groupes paraboliques maximaux de $\G$. \item La partie fine de la partie non cuspidale, c'est-à-dire $M^{nc}_{\varepsilon}  = M_{\varepsilon} \smallsetminus M^{c}_{\varepsilon}$, est la r\'eunion disjointe des parties $M_{\varepsilon}(G)$, où $G$ parcourt les sous-groupes hyperboliques maximaux de $\G$.
\end{enumerate}
\end{lemm}

\begin{proof}
\begin{enumerate}

 \item Par d\'efinition, $M_{\varepsilon}(G) \subset M_{\varepsilon}$ pour tout sous-groupe $G$ de $\G$. Maintenant, si $x\in M_{\varepsilon}$, il existe un \'el\'ement non elliptique $\g\in\G$ tel que $d_{\O}(x,\g x)<\varepsilon$. Le groupe engendr\'e par $\g$ est nilpotent et infini, et donc $x\in M_{\varepsilon}(\langle \g \rangle)$. De plus, les parties $\overline{M_{\varepsilon}(G)}$ sont disjointes. En effet, s'il y avait un point $x$ qui \'etait à la fois dans $\overline{M_{\varepsilon}(G)}$ et dans $\overline{M_{\varepsilon}(G')}$, le groupe discret engendr\'e par $G$ et $G'$ serait nilpotent par le lemme de Margulis, contredisant le fait que $G$ et $G'$ sont maximaux.\\
 \item Soit $G$ un groupe virtuellement nilpotent maximal, que l'on peut supposer sans torsion. On va montrer que $M_{\varepsilon}(G)$ est ouvert et ferm\'e dans $M_{\varepsilon}$. L'ouverture de $M_{\varepsilon}(G)$ d\'ecoule de la d\'efinition. Pour la fermeture, consid\'erons une suite $(x_n)$ de points dans $M_{\varepsilon}(G)$ qui converge vers $x$ dans $M_{\varepsilon}$. Il existe ainsi un \'el\'ement non elliptique $\g\in\G$ tel que $d_{\O}(x,\g x)<\varepsilon$. Par continuit\'e, on a aussi $d_{\O}(x_n,\g x_n)<\varepsilon$ lorsque $n$ est assez grand, et ainsi le groupe engendr\'e par $G$ et $\g$ est nilpotent, d'après le lemme de Margulis. Comme $G$ est maximal, on a forc\'ement $\g\in G$ et donc $x\in M_{\varepsilon}(G)$, autrement dit $M_{\varepsilon}(G)$ est ferm\'e.\\
\item Soit $G$ le groupe nilpotent hyperbolique engendr\'e par l'\'el\'ement $\g$. Tout domaine fondamental convexe et ferm\'e $C$ pour l'action de $G$ sur $\O$ intersecte l'axe $a_{\g}$ de $\g$ en une partie compacte. Il est alors clair que $d_{\O}(x,\g x)\geqslant \varepsilon$ dès que $x$ est un point de $C$ dont la distance à l'axe de $\g$ est sup\'erieure à une certaine constante. Autrement dit, $\O_{\varepsilon}(G)\cap C$ est un voisinage relativement compact de $a_{\g}\cap C$, et donc $M_{\varepsilon}(G)$ est relativement compact dans $M$.\\
\item Soit $G$ un sous-groupe parabolique de $\G$ qui fixe le point $p\in\dO$. Prenons $x \in \dO\smallsetminus\{p\}$ et param\'etrons la g\'eod\'esique $(xp)$ par $r:\R \longrightarrow \R$, de telle façon que $r(-\infty)=x,\ r(+\infty)=p,\ r(t)\in(xp), t\in\R$. La convexit\'e de $\O$ montre que la fonction $f : t \longmapsto d_{\O}(r(t), \g r(t))$ est d\'ecroissante. La stricte convexit\'e entraîne que $f$ tend vers $+\infty$ en $-\infty$, et vers $0$ en $+\infty$. C'est exactement ce qu'on voulait montrer.\\
\item Cela d\'ecoule directement de la d\'efinition et du premier point.\\
\item La partie non cuspidale de $M$ est par d\'efinition r\'eunion de la partie \'epaisse et des parties fines non cuspidales. Ces dernières sont exactement les parties $M_{\varepsilon}(G)$, où $G$ parcourt les sous-groupes hyperboliques de $\G$; les points pr\'ec\'edents montrent que ces parties sont connexes, d'adh\'erences compactes et disjointes.
\end{enumerate}
\end{proof}

\begin{center}
\begin{figure}[h!]
  \centering
\includegraphics[width=10cm]{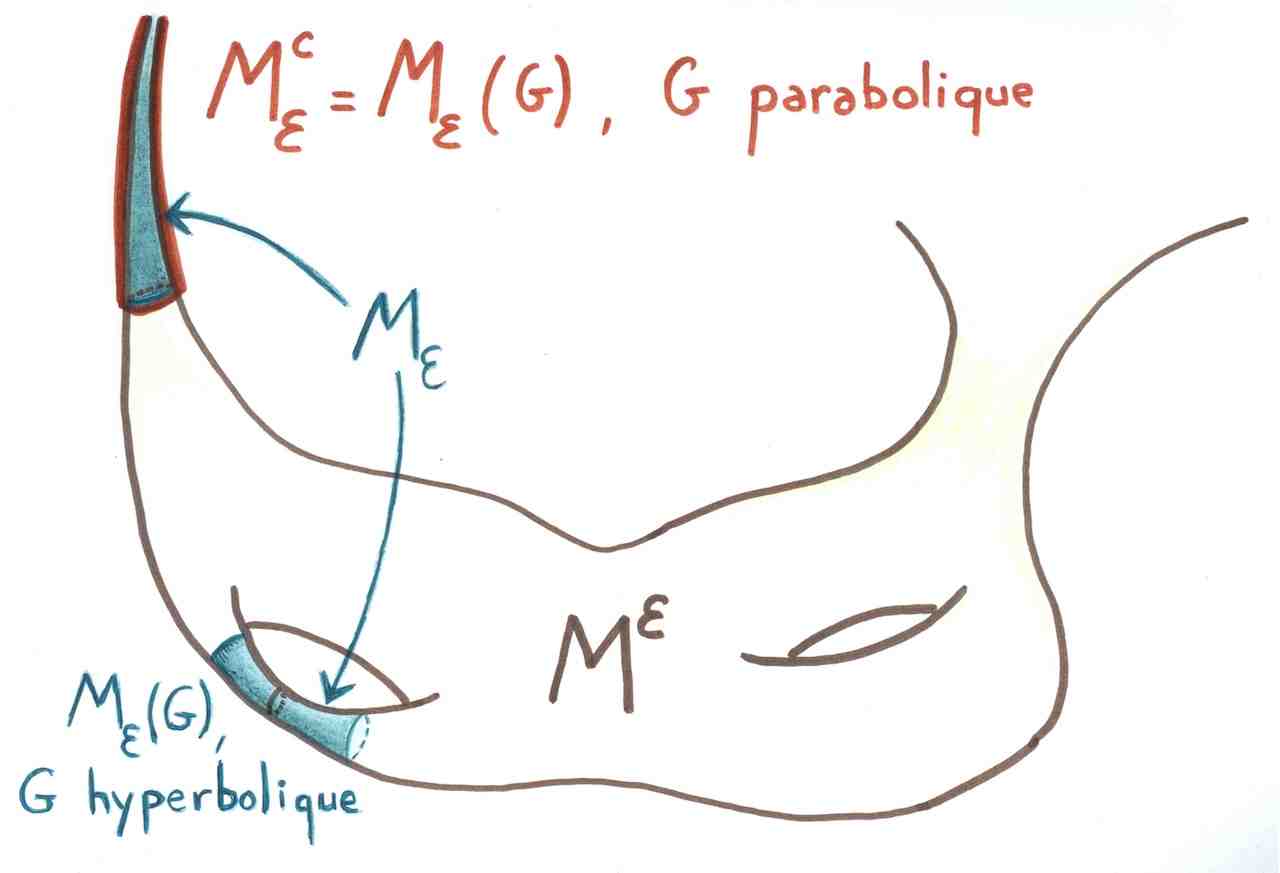}
\caption{D\'ecomposition du quotient}
\end{figure}
\end{center}


\section{Sur les sous-groupes paraboliques}

\subsection{Quelques r\'esultats pr\'eliminaires sur les groupes alg\'ebriques}

Nous allons avoir besoin de plusieurs r\'esultats et d\'efinitions sur les groupes alg\'ebriques lin\'eaires r\'eels; on pourra consulter le livre \cite{MR0396773}.

Soit $G$ un sous-groupe de $\ss$ Zariski-ferm\'e. Un \'el\'ement $g \in G$ est dit \emph{semi-simple} (resp. \emph{unipotent}) lorsque $g$ est diagonalisable sur $\mathbb{C}$ (resp. $(g-1)^{n+1} = 0$). On note $\mathcal{S}(G)$ (resp. $\U(G)$) l'ensemble des \'el\'ements semi-simples (resp. unipotents) de $G$.

L'ensemble $\U(G)$ est un ferm\'e de Zariski de $G$; par contre, l'ensemble $\mathcal{S}(G)$ ne l'est pas en g\'en\'eral.

\begin{prop}[Proposition 19.2 de \cite{MR0396773}]\label{decompo}

Soit $G$ un groupe alg\'ebrique r\'esoluble et connexe. Le groupe $G$ est nilpotent si et seulement si $\mathcal{S}(G)$ est un sous-groupe de $G$. Dans ce cas, l'ensemble $\mathcal{S}(G)$ est un ferm\'e de $G$ pour la topologie de Zariski, le groupe $\mathcal{S}(G)$ est ab\'elien et le groupe $G$ se d\'ecompose en le produit direct $G= \mathcal{S}(G) \times \U(G)$.
\end{prop}

\begin{prop}[Lemme 4.9 de \cite{CoursDeBenoist}]\label{adh_zar_mod}
Soit $\G$ un sous-groupe de $\ss$. Si toutes les valeurs propres de tous les \'el\'ements de $\G$ sont de module 1 alors toutes les valeurs propres de tous les \'el\'ements de l'adh\'erence de Zariski de $\G$ sont aussi de module 1.
\end{prop}

\begin{rema}
Il faut bien faire attention au fait que, dans l'\'enonc\'e pr\'ec\'edent, le corps de base est $\R$. Cette proposition est fausse sur un corps quelconque. Sur le corps des complexes, le groupe compact $\textrm{SU}_n$ est Zariski-dense dans le $\mathbb{C}$-groupe $\textrm{SL}_{n}(\mathbb{C})$; sur les corps $p$-adiques, le groupe compact $\textrm{SL}_n(\Z_p)$ est Zariski dense dans le $\mathbb{Q}_p$-groupe $\textrm{SL}_{n}(\mathbb{Q}_p)$. Pourtant, les valeurs propres des \'el\'ements de ces deux groupes sont toutes de modules 1. Le ph\'enomène exceptionnel qui explique cette proposition sur $\R$ est que le sous-groupe compact maximal $\textrm{SO}_n(\R)$ de $\textrm{SL}_n(\R)$ est Zariski-ferm\'e.
\end{rema}

\begin{theo}[Kostant-Rosenlicht (Th\'eorème 2 de \cite{MR0130878} ou appendice de \cite{MR0296077})]\label{kos_ros}
Soit $U$ un groupe alg\'ebrique unipotent agissant sur un espace affine. Toute orbite de $U$ est Zariski ferm\'e.
\end{theo}

\begin{theo}[Th\'eorème de Mal'cev (Th\'eorème 2.1 de \cite{MR0507234})]\label{rag}
Soit $U$ un sous-groupe Zariski ferm\'e de $\ss$. Si $U$ est unipotent, alors tout sous-groupe discret et Zariski-dense $\G$ de $U$ est un r\'eseau cocompact de $U$.
\end{theo}


\begin{lemm}\label{lemm_gr_alg}
Soit $\P$ un sous-groupe parabolique de $\Aut(\O)$ fixant un point $p$. On note $\Nn$ l'adh\'erence de Zariski de $\P$ et $\U$ le sous-groupe de $\Nn$ constitu\'e des \'el\'ements unipotents de $\Nn$.\\
Le quotient $\Quotient{\Nn}{\U}$ est compact, le groupe $\P$ est un r\'eseau cocompact de $\Nn$, l'action de $\Nn$ sur $\A^{n-1}_p$ est propre et l'action de $\U$ sur $\A^{n-1}_p$ est libre. En particulier, si l'action de $\P$ sur $\dO \smallsetminus \{p\}$ est cocompacte alors l'action de $\U$ sur $\A^{n-1}_p$ est simplement transitive.
\end{lemm}

\begin{proof}
Le groupe $\P$ est virtuellement nilpotent; par cons\'equent, quitte à passer à un sous-groupe d'indice fini, on peut supposer que $\P$ est nilpotent et Zariski-connexe. L'adh\'erence de Zariski $\Nn$ de $\P$ est alors un sous-groupe nilpotent Zariski-ferm\'e de $\ss$. On note $\U$ l'ensemble des \'el\'ements unipotents de $\Nn$ et on note $K$ l'ensemble des \'el\'ements semi-simples de $\Nn$. La proposition \ref{decompo} montre que $_U$ et un groupe et que $\Nn$ est le produit direct de $\U$ et $K$, le groupe $K$ est ab\'elien.

La proposition \ref{adh_zar_mod} montre que toutes les valeurs propres des \'el\'ements de $K$ sont de module 1. Or, les \'el\'ements du groupe ab\'elien $K$ sont tous semi-simples par cons\'equent $K$ est compact.

Montrons à pr\'esent que le groupe discret $\P$ est un r\'eseau du groupe de Lie $\Nn$. Le groupe d\'eriv\'e $[\P,\P]$ de $\P$ est Zariski-dense dans le groupe unipotent $[\Nn,\Nn]=[\U,\U]$. Le th\'eorème \ref{rag} montre que le groupe $[\P,\P]$ est un r\'eseau cocompact de $[\Nn,\Nn]$. Consid\'erons les projections $\pi_1:\Nn \rightarrow \Quotient{\Nn}{[\Nn,\Nn]} = \Quotient{\U}{[\U,\U]} \times K$ et $\pi_2: \Quotient{\Nn}{[\Nn,\Nn]} \rightarrow \Quotient{\U}{[\U,\U]}$. Le quotient $\Quotient{\U}{[\U,\U]}$ est un groupe de Lie ab\'elien unipotent par cons\'equent, il est isomorphe à un espace vectoriel r\'eel. Le groupe $\pi_2 \circ\pi_1(\P)$ est Zariski-dense dans l'espace vectoriel $\Quotient{\U}{[\U,\U]}$, par suite $\pi_2 \circ\pi_1(\P)$ est un sous-groupe cocompact de $\Quotient{\U}{[\U,\U]}$. Il vient que $\pi_1(\P)$ est un sous-groupe cocompact de $\Quotient{\Nn}{[\Nn,\Nn]}$. Donc, $\P$ est un r\'eseau cocompact de $\Nn$.

Ensuite, consid\'erons l'action de $\P$ sur l'espace affine $\A^{n-1}_p$ des droites de $\PP^n$ passant par $p$ mais qui ne sont pas contenu dans l'hyperplan tangent à $\dO$ en $p$. Le groupe $\Nn$ agit aussi sur $\A^{n-1}_p$. L'action de $\P$ sur $\A^{n-1}_p$  est propre et $\P$ est un sous-groupe cocompact de $\Nn$ par suite $\Nn$ agit proprement sur $\A^{n-1}_p$.

Comme l'action de $\Nn$ sur $\A^{n-1}_p$ est propre le stabilisateur de tout point de $\A^{n-1}_p$ est compact. Mais le groupe $\U$ est unipotent et tout \'el\'ement d'un groupe compact est semi-simple. L'action de $\U$ sur $\A^{n-1}_p$  est donc libre.

Enfin, si l'action de $\P$ sur $\A^{n-1}_p$ est cocompact comme l'orbite de tout point de $\A^{n-1}_p$ sous l'action de $\U$ est Zariski-ferm\'e par le th\'eorème \ref{kos_ros}, l'action de $\Nn$ sur $\A^{n-1}_p$ est transitive et l'action de $\U$ sur $\A^{n-1}_p$ est simplement transitive.
\end{proof}

%
%
%

\subsection{Description des sous-groupes paraboliques uniform\'ement born\'es}

Dans cette partie, nous d\'ecrivons les sous-groupes paraboliques des sous-groupes discrets de $\Aut(\O)$ dont l'action est g\'eom\'etriquement finie sur $\O$. Ceux-ci sont en fait conjugu\'es dans $\ss$ à des sous-groupes paraboliques de $\SO$ et donc en particulier virtuellement ab\'eliens.\\

\subsubsection*{Un petit laïus sur les unipotents qui pr\'eservent un convexe}\label{laius}

\begin{defi}
Soit $\g\in \ss$ un \'el\'ement unipotent. On appelle \emph{degr\'e} de $\g$ le plus petit entier $k$ tel que $(\g-1)^k=0$.
\end{defi}

Soit $\g\in \ss$ un \'el\'ement unipotent qui pr\'eserve un ouvert proprement convexe quelconque. Benoist a remarqu\'e dans \cite{MR2218481} (lemme 2.3) que le degr\'e de $\g$ \'etait n\'ecessairement impair. L'argument est très cours, r\'ep\'etons-le pour faciliter la lecture. On regarde l'action de $\g$ sur la sphère projective $\S$, c'est-à-dire le revêtement à  deux feuillets de $\PP^n$. Un calcul explicite de $\g^n$ dans une base donnant une matrice de Jordan montre que, si $k$ est pair alors dans $\S$, on a $\underset{n\to +\infty}{\lim} \g\cdot x = - \underset{n\to -\infty}{\lim} \g\cdot x$ pour tout $x \in \S$ en dehors d'un hyperplan. Par cons\'equent, si $k$ est pair, $\g$ ne peut pr\'eserver d'ouvert proprement convexe.\\

De plus, si l'ouvert $\O$ est strictement convexe, alors il existe un unique bloc de Jordan de $\g$ de degr\'e maximal $k$ et tous les autres blocs de Jordan de $\g$ sont de degr\'e strictement inf\'erieur à $k$. C'est une cons\'equence du th\'eor\`eme \ref{classi_dym_1}. En effet, l'\'el\'element unipotent $\g$ est n\'ecessairement un \'el\'ement parabolique de $\Aut(\O)$; il possède donc un unique point fixe attractif, ce qui impose l'unicit\'e du bloc de degr\'e maximal.

On obtient ainsi que l'unique point fixe $p$ de $\g$ sur $\dO$ est l'image de $(\g-1)^{k-1}$. En effet cet espace est une droite de $\R^{n+1}$: c'est la droite engendr\'ee par le premier vecteur du bloc de Jordan de degr\'e $k$ de $\g$. En fait, il existe un hyperplan $H$ de $\PP^n$ tel que si $x \notin H$ alors $\g^n \cdot x \to p$ lorsque $n\to \pm \infty$.

On obtient aussi l'existence d'une droite attractive. L'image $D$ de $(\g-1)^{k-2}$ est un plan de $\R^{n+1}$, donc une droite de $\PP^n$: c'est le plan engendr\'e par les deux premiers vecteurs du bloc de Jordan de degr\'e $k$ de $\g$. Si $D'$ est une droite de $\PP^n$ et $D' \not\subset H$ alors $\g^n \cdot D' \to D$. On appellera cette droite la \emph{droite attractive} de $\g$. Cette dernière assertion est simplement une cons\'equence du calcul des $\g^i$ et des $(\g-1)^i$ dans une base donnant une matrice de Jordan.

On peut r\'esumer l'essentiel de ce paragraphe dans la proposition suivante:

\begin{prop}
Soit $\g\in \Aut(\O)$ un \'el\'ement unipotent. Le degr\'e $k$ de $\g$ est impair et le bloc de Jordan de degr\'e maximal est unique.
\end{prop}

\begin{defi}
Une courbe $\mathbb{S}^1 \rightarrow \PP^n$ est dite \emph{convexe} lorsqu'elle est incluse dans le bord d'un ouvert proprement convexe.
\end{defi}

\begin{lemm}\label{vero}
Soit $\g\in\Aut(\O)$ un \'el\'ement unipotent. On note $p$ le point de $\dO$ fix\'e par $\g$, $H$ l'hyperplan tangent à $\O$ en $p$ et $\U=\{g^t\}$ le groupe à un paramètre engendr\'e par $\g$. Si $x \notin H$, l'application 
$$
\begin{array}{rcl}
   \PP^1 & \longrightarrow & \PP^n\\
   t \in \R & \longmapsto & \g^t \cdot x\\
   \infty & \longmapsto & p
  \end{array}
$$
d\'efinit une courbe $\C_{x}$ alg\'ebrique, lisse et convexe. De plus, la tangente à $\C_x$ en $p$ est la droite attractive de $\g$.
\end{lemm}

\begin{proof}
Si $\g$ possède un unique bloc de Jordan non trivial, alors, dans un système de coordonn\'ees convenable, $\C_x$ est d\'efinie par $[t:s] \rightarrow [t^{k-1}:t^{k-2}s: ... : s^{k-1}:1:...:1]$ où $k$ est le degr\'e de $\g$; autrement dit, $\C_x$ est la courbe Veronese de degr\'e $k-1$.\\
Il suffit alors d'appliquer cette remarque à chaque bloc de Jordan de $\g$.
\end{proof}

\begin{prop}\label{findieu}
Soit $\g$ (resp. $g$) une matrice unipotente poss\'edant un unique bloc de Jordan de degr\'e maximal impair $k\geqslant 5$ (resp. de degr\'e $3$). On suppose que $\g$ et $g$ ont le même point attractif $p$, la même droite attractive et que $\ker(\g-1)^2=\ker(g-1)^2$. Alors l'\'el\'ement $[\g,g]$ est unipotent de degr\'e $2$. En particulier $[\g,g]$ ne pr\'eserve pas d'ouvert proprement convexe.
\end{prop}

\begin{proof}
C'est un simple calcul. On calcule le bloc principal de $[\g,g]$, pour cela on d\'efinit les matrices suivantes:

$$
\begin{array}{cccccc}
 
 J_k = 
 
 &
 \begin{pmatrix}
1 & 1 & 0 & \cdots  & 0        \\
0 & 1 & 1 &    \ddots    & \vdots \\ 
\vdots & \ddots & \ddots & 1 & 0                  \\
\vdots &  & \ddots & 1 & 1                            \\ 
0 & \cdots & \cdots & 0 & 1                          \\ 
       \end{pmatrix}
      &
       J'_a=
       &
 \begin{pmatrix}
1 & a & \frac{a^2}{2}                  \\
0 & 1 & a &      \\ 
0 & 0 & 1 &                 \\
       \end{pmatrix}
&
M^a_{2,l}=
&
       \begin{pmatrix}
-\frac{a^2}{2} &   \frac{a^2}{2}& \cdots   & -\frac{a^2}{2}  & \frac{a^2}{2}               \\
-a & a & \cdots & -a & a     \\ 
       \end{pmatrix}

       \end{array} 
$$

Ainsi, $J_k$ est le bloc de Jordan canonique de degr\'e $k$, c'est une matrice de taille $k \times k$, $M^a_{2,l}$ est une matrice de taille $2\times l$, où $l$ est un nombre pair et $a \in \R$.

Par hypothèse, les matrices de $\g$ et $g$ ont, dans une base convenable, la forme suivante:
$$
\begin{array}{cccc}

\g =  & 
\begin{pmatrix} J_k &  0 \\ 0  & U \end{pmatrix}
&
\textrm{ et }
&
g = 
\left(
\begin{array}{ccc}
J'_a & 0 &0\\
0      & I_{k-3} &0\\
0&0 & I_{n+1-k}
\end{array}
\right),
\end{array}
$$

\noindent où $U$ est une matrice triangulaire sup\'erieure avec uniquement des 1 sur la diagonale et dont les blocs de Jordan sont de degr\'e strictement inf\'erieur à $k$ et $a \neq 0$. Ainsi, on a,

$$
\begin{array}{cc}

[\g,g] = 

&

\left(
\begin{array}{cccc}
I_{2,2} & 0 & 
\begin{array}{cc} 
M^a_{2,k-3} & 0
\end{array}
\\
0 & 1 & 
\begin{array}{rl}
0 & 0
\end{array}

             \\
0 & 0 & I_{n+1-k}
\end{array}
\right).
\end{array}
$$

Par cons\'equent, $[\g,g]$ est une matrice unipotente de degr\'e $2$.
\end{proof}

Terminons cette partie sur un lemme cl\'e:

\begin{lemm}\label{uni}
Soit $\U$ un sous-groupe unipotent de $\Aut(\O)$ fixant un point $p \in \partial \O$. Si l'action de $\U$ sur $\partial \O \smallsetminus \{ p\}$ est transitive, alors $\O$ est un ellipsoïde.
\end{lemm}

\begin{proof}
Cette proposition se d\'emontre par r\'ecurrence. En dimension $n=2$, l'unique groupe unipotent qui pr\'eserve un convexe est le groupe suivant:

$$
\U = \left\{
\begin{pmatrix}
1 & a & \frac{a^2}{2}\\
0 & 1 & a\\
0 & 0 & 1\\
\end{pmatrix}
, a \in \R
\right\}
$$
Si on note $e_1,e_2,e_3$ les vecteurs de la base canonique de $\R^3$, alors l'orbite sous $\U$ d'un point qui n'est pas sur la droite projective $<e_1,e_2>$ est une ellipse priv\'ee du point $<e_1>$.\\

Supposons maintenant que la propri\'et\'e soit d\'emontr\'ee pour un ouvert convexe de $\PP^{n-1}$ et prenons $\O\subset\PP^n$. On va montrer que le bord de $\O$ est de classe $\C^2$ \`a hessien d\'efini positif. Le th\'eor\`eme \ref{socie} d'\'Edith Soci\'e-M\'ethou permettra de conclure que $\O$ est un ellipso\"ide.\\ 
On note $H$ l'hyperplan tangent à $\O$ en $p$. Comme le groupe $\U$ est unipotent, il pr\'eserve un sous-espace $F$ de dimension $n-2$ inclus dans $H$. L'ensemble des hyperplans contenant $F$ est l'espace projectif $\PP(\Quotient{\R^{n+1}}{\tilde F}) = \PP^1$, o\`u $\tilde F$ d\'esigne le relev\'e de $F$ \`a $\R^{n+1}$. L'action de $\U$ sur $\PP(\Quotient{\R^{n+1}}{\tilde F})$ pr\'eserve l'hyperplan $H$ donc $\U$ agit par transformations affines sur $\PP(\Quotient{\R^{n+1}}{\tilde F}) \smallsetminus {H} = \A^1$. Ces transformations affines \'etant unipotentes, $\U$ agit en fait par translations sur $\PP(\Quotient{\R^{n+1}}{\tilde F}) \smallsetminus {H}$. On obtient donc un morphisme $\varphi : \U \rightarrow \R$.\\
On note $(H_t)_{t \in \PP^1}$ le param\'etrage de la famille des hyperplans de $\PP^n$ contenant $F$ par $\PP^1$, obtenu en posant $H_{\infty} =H$. Ainsi, le noyau $\V$ de $\varphi$ pr\'eserve chacun des hyperplans $H_t$.\\
Par cons\'equent, si $t \neq \infty$, le groupe $\V$ pr\'eserve les ouverts proprement convexe $\O_t = \O \cap H_t$ qui sont strictement convexes à bord $\C^1$. L'action de $\V$ sur $\partial \O_t \smallsetminus \{ p\}$ \'etant clairement transitive, l'hypothèse de r\'ecurrence montre donc que les $\O_t$ sont des ellipsoïdes.\\
Soit $\g \in \U \smallsetminus \V$. Si la droite attractive de $\g$ est incluse dans $F$, alors il existe un \'el\'ement $g \in \V$ tel que $g$ et $\g$ ont le même point fixe $p$ et la même droite attractive, par cons\'equent, le lemme \ref{findieu} montre que l'\'el\'ement $[\g,g]$ ne pr\'eserve pas d'ouvert proprement convexe ce qui est absurde.

Par cons\'equent, la droite attractive de $\g$ n'est pas incluse dans $F$, et le convexe $\O$ est de classe $\C^2$ à hessien d\'efini positif, ainsi le th\'eorème \ref{socie} conclut. En effet, le th\'eor\`eme \ref{kos_ros} appliqu\'e \`a l'action de $\U$ sur l'espace affine $\PP^n\smallsetminus H$ montre que l'ensemble $\dO \smallsetminus \{ p\}$ est Zariski-ferm\'e; il est lisse car le groupe alg\'ebrique $\U$ agit transitivement sur ce dernier. L'ensemble $\dO$ est la compl\'etion alg\'ebrique de $\dO  \smallsetminus \{ p\}$ dans $\PP^n$, c'est une sous-vari\'et\'e de classe $C^2$:  le point $p$ est de classe $\C^2$ puisque dans la direction de $F$ c'est un ellipsoïde, et dans la direction donn\'ee par la droite attractive de $\g$, c'est une courbe alg\'ebrique convexe lisse (lemme \ref{vero}). De la m\^eme fa\c con, le bord du convexe dual $\O^*$ est aussi de classe $\C^2$ et donc $\dO$ est à hessien d\'efini positif. C'est ce qu'il fallait montrer.
\end{proof}

\begin{theo}[Soci\'e-M\'ethou \cite{MR1981171}]\label{socie}
Un ouvert proprement convexe de $\PP^n$ dont le bord est de classe $\C^2$ à hessien d\'efini positif et le groupe d'automorphisme est non compact est un ellipsoïde.
\end{theo}

On peut à pr\'esent se lancer dans l'\'etude des sous-groupes paraboliques uniform\'ement born\'es. Commençons par traiter le cas des

\subsubsection*{Sous-groupes paraboliques de rang maximal}

Le lemme pr\'ec\'edent va permettre d'obtenir le th\'eor\`eme suivant.

\begin{theo}\label{lemmededieu}
Soit $\P$ un sous-groupe parabolique discret de $\Aut(\O)$ fixant $p$. Si le groupe $\P$ est de rang maximal, alors il pr\'eserve des ellipsoïdes $\E^{int}$ et $\E^{ext}$ tels que
\begin{itemize}
 \item $\partial\E^{int}\cap \partial\E^{ext}=\partial\E^{int}\cap\dO = \partial\E^{ext}\cap\dO=\{p\}$;
 \item $\E^{int}\subset \O\subset \E^{ext}$;
 \item $\E^{int}$ est une horoboule de l'espace hyperbolique $(\E^{ext}, d_{\E^{ext}})$.
\end{itemize}
En particulier, le groupe $\P$ est conjugu\'e dans $\ss$ à un sous-groupe parabolique de $\SO$.
\end{theo}

\begin{center}
\begin{figure}[h!]
  \centering
\includegraphics[width=9cm]{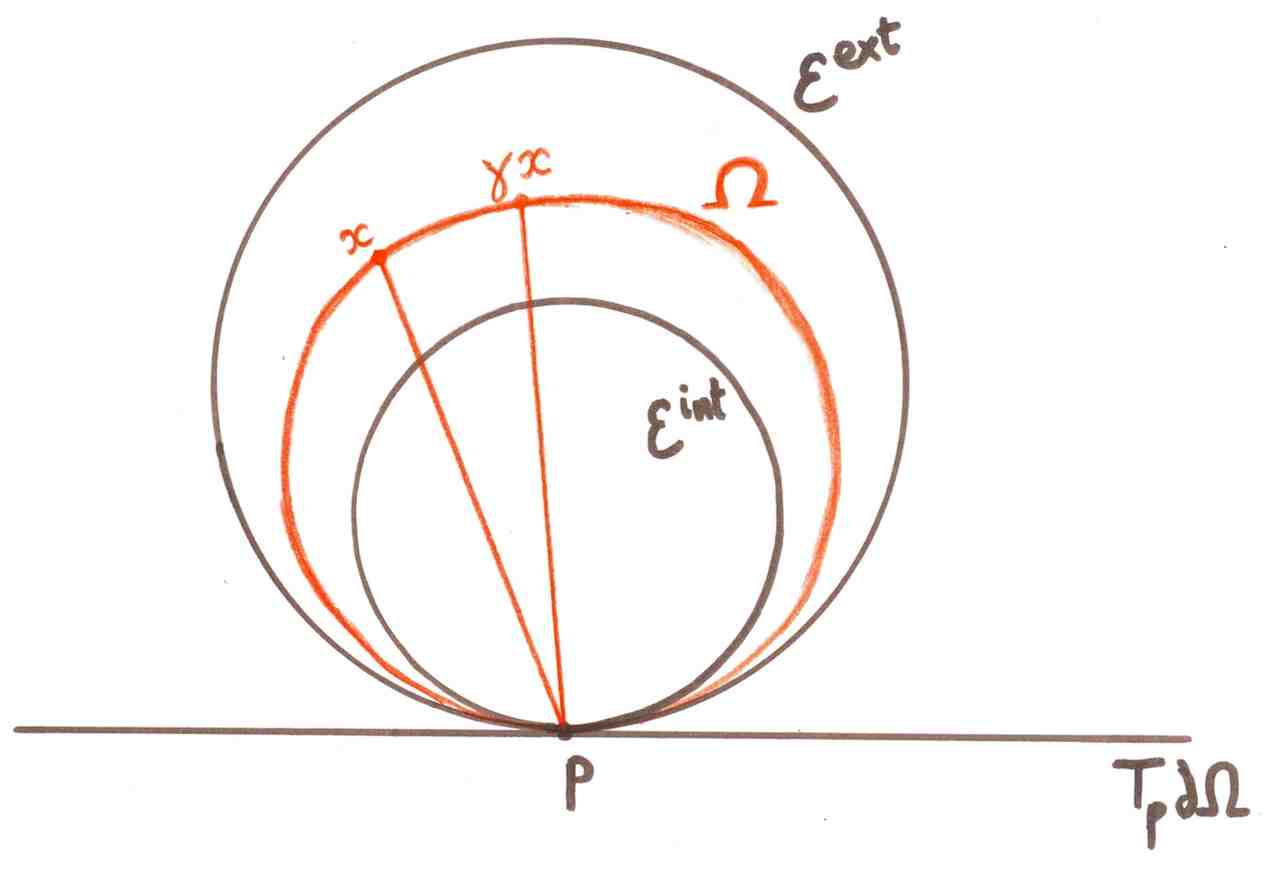}
\caption{$\O$ coinc\'e !}
\end{figure}
\end{center}

\begin{proof}
Soient $\Nn$ l'adh\'erence de Zariski de $\P$ dans $\ss$ et $\U=\U(\Nn)$ le sous-groupe des \'el\'ements unipotents de $\Nn$. Le lemme \ref{lemm_gr_alg} montre que le groupe $\P$ est un r\'eseau cocompact de $\Nn$ et que l'action de $\U$ sur $\A_p^{n-1}$ est simplement transitive.\\
Soient $H$ l'hyperplan tangent à $\dO$ en $p$ et $x\in\PP^n\smallsetminus H$. D'après le th\'eorème \ref{kos_ros} appliqu\'e \`a l'action de $\U$ sur l'espace affine $\PP^n\smallsetminus H$, l'orbite $\U \cdot x$ est une sous-vari\'et\'e alg\'ebrique lisse $C_x$ de $\PP^n\smallsetminus H$; $C_x$ est hom\'eomorphe à $\R^{n-1}$ puisque $\U \cdot (px)=\A_p^{n-1}$ est hom\'eomorphe à $\R^{n-1}$. 

On peut consid\'erer l'enveloppe convexe $\E_x$ de $C_x$ dans $\PP^n$ qui est un ouvert proprement convexe de $\PP^n$: en effet, on a $\lim \g \cdot x =p$ quand $\g$ tend vers l'infini dans $\P$ et donc \'egalement quand $\g$ tend vers l'infini dans $\Nn$. Le groupe $\U$ agit simplement transitivement sur $\partial \E_x \smallsetminus \{ p\}$ car $\partial \E_x \smallsetminus \{ p\}$ se projette  bijectivement sur $\A_p^{n-1}$. Par cons\'equent, $C_x =\partial \E_x \smallsetminus \{ p\}$.

Le bord $\partial \E_x$ de $\E_x$, qui est la compl\'etion alg\'ebrique de $C_x$ dans $\PP^n$ est un ferm\'e de Zariski de $\PP^n$. La vari\'et\'e alg\'ebrique $\partial \E_x$ est partout lisse sauf peut-être en $p$. Comme $\U$ agit transitivement sur l'espace affine $\A_p^{n-1}$ des droites passant par $p$ qui ne sont pas incluses dans $H$, le bord $\partial \E_x$ admet un unique plan tangent en $p$ : l'hyperplan $H$. Comme $\E_x$ est convexe, on en d\'eduit que son bord $\partial\E_x$ est de classe $\C^1$ au point $p$. 

Par cons\'equent, $\E_x$ est un ouvert proprement convexe à bord $\C^1$. Le même raisonnement montre que le dual $\E_x^*$ de $\E_x$ est un ouvert proprement convexe à bord $\C^1$ et donc que $\E_x$ est un ouvert proprement convexe strictement convexe à bord $\C^1$. Le lemme \ref{uni} montre alors que $\E_x$ est un ellipsoïde.\\

Comme l'action de $\P$ sur $\dO\smallsetminus \{p\}$ est cocompacte, on peut trouver $x$ et $y$ tels que $\E_x\subset \O\subset \E_y$. On pose alors $\E^{int}=\E_x$ et $\E^{ext}=\E_y$.
\end{proof}

\begin{rema}\label{pleinplein}
En faisant varier le point $x$ le long d'une droite passant par $p$ et coupant $\O$, on voit que le groupe $\P$ pr\'eserve une famille \`a un param\`etre d'ellipso\"ides tangents \`a $\O$ en $p$.
\end{rema}

Notons tout de suite une cons\'equence de ce r\'esultat.

\begin{coro}\label{loincusp}
Soit $\P$ un sous-groupe parabolique de rang maximal de $\Aut(\O)$ fixant le point $p$ de $\dO$. Le quotient $\Quotient{H}{\P}$ de toute horoboule $H$ bas\'ee en $p$ par $\P$ est de volume fini.
\end{coro}
\begin{proof}
Bien entendu, il suffit de montrer le r\'esultat pour une seule horoboule. Prenons $\E^{int}$ comme dans le th\'eorème \ref{lemmededieu}, et appelons $\Vol^{int}$ le volume hyperbolique qu'il d\'efinit; on a $\Vol^{int}\geqslant \Vol_{\O}$ sur les bor\'eliens de $\E^{int}$ (proposition \ref{compa}). Comme $\P$ agit cocompactement sur $\dO\smallsetminus\{p\}$, on peut choisir une petite horoboule $H$ de $\O$ incluse dans $\E^{int}$ dont le bord ne rencontre celui de $\E^{int}$ qu'en $p$. Cette horoboule $H$ est contenue dans une horoboule $H^{int}$ de $\E^{int}$, de telle façon que $\Quotient{H}{\P}\subset\Quotient{H^{int}}{\P}$ et on a
$$\Vol_{\O}(\Quotient{H}{\P}) \leqslant \Vol^{int}(\Quotient{H^{int}}{\P}).$$
Or, le convexe $\E^{int}$ est un ellipsoïde, la g\'eom\'etrie de Hilbert qui lui est associ\'ee est la g\'eom\'etrie hyperbolique. On sait donc que $\Vol^{int}(\Quotient{H^{int}}{\P})$ est fini.
\end{proof}

\subsubsection*{Cas g\'en\'eral}

Le lemme suivant permet de ramener le cas g\'en\'eral au cas où le rang du sous-groupe parabolique est maximal.

\begin{lemm}\label{cas_general}
Soient $\G$ un sous-groupe discret de $\Aut(\O)$ et $p\in\LG$ un point parabolique uniform\'ement born\'e. Le groupe $\P = \Stab_{\G}(p)$ pr\'eserve un sous-espace projectif $\PP^{d+1}_p$ qui contient $p$ et intersecte $\O$, avec $d$ le rang de $\P$.\\
En particulier, le groupe $\P$ est un sous-groupe parabolique de rang maximal de $\Aut(\O_p)$, o\`u $\O_p$ d\'esigne l'ouvert proprement convexe $\PP^{d+1}_p\cap\O$.
\end{lemm}

\begin{proof}
Voyons l'ensemble $\LG\smallsetminus\{p\}$ comme un sous-ensemble de $\A_p^{n-1}$, et notons $\D=\overline{\mathcal{D}_p(C(\LG))}$: c'est, dans $\A_p^{n-1}$, l'adh\'erence de l'enveloppe convexe de $\LG\smallsetminus\{p\}$. Soit $K$ l'ensemble des sous-espaces affines maximaux inclus dans l'adh\'erence de $\D$. Les \'el\'ements de $K$ ont tous la même direction $D$. L'ensemble $K$ s'identifie à un ferm\'e convexe dans l'espace affine $\Quotient{\A_{p}^{n-1}}{D}$, qui, par d\'efinition, ne contient pas de droite. Montrons qu'il ne contient pas non plus de demi-droite.\\
Pour cela, compactifions l'espace $\A^{n-1}_p$ en $\overline{\A^{n-1}_p}$ en lui ajoutant l'ensemble des demi-droites passant par un point $o \in\A^{n-1}_p$ fix\'e, qui n'est rien d'autre qu'une sph\`ere. Si $x$ est un point de $\A^{n-1}_p$ et $\g$ un \'el\'ement d'ordre infini de $\P$ alors la limite dans $\overline{\A^{n-1}_p}$ de la suite $\g^n \cdot x$ v\'erifie que $\underset{n\to +\infty}{\lim} \g^n\cdot x = - \underset{n\to -\infty} {\lim}\g^n\cdot x $ car le degr\'e de tout \'el\'ement de $\P$ est impair. Ainsi, si $x$ est un point de $\D$, on voit que l'espace des demi-droites incluses dans $K$ est stable par la sym\'etrie centrale de centre $x$; autrement dit, si une demi-droite est dans incluse dans $K$, la droite enti\`ere l'est \'egalement, ce qui est impossible.\\
Par cons\'equent, le ferm\'e $K$ est proprement convexe. L'action de $\P$ sur $K=\Quotient{\A_{p}^{n-1}}{D}$ possède donc un point fixe, le centre de gravit\'e de $K$. Autrement dit, $\P$ pr\'eserve un sous-espaces affine maximal $F$ de $\F$, dont la dimension est n\'ecessairement \'egale à la dimension cohomologique $d$ de $\P$. Il ne reste plus qu'\`a faire machine arri\`ere: $F$ est un sous-espace affine de  $A_p^{n-1} = \Quo{\PP^n}{p} \smallsetminus T_p\dO$, qui engendre le sous-espace projectif $\tilde F$ de $\Quo{\PP^n}{p}$, lui aussi $\P$-invariant; l'espace $\PP_p^{d+1}$ est le relev\'e \`a $\PP^n$ de $\tilde F$.
\end{proof}

Notons $\Cone(p,C(\LG))=\{y\in\PP^n\ |\ y\in (px), x\in C(\LG)\}$. On en d\'eduit le corollaire suivant.

\begin{coro}\label{para_cas_gene}
Soient $\G$ un sous-groupe discret de $\Aut(\O)$ et $p\in\LG$ un point parabolique uniform\'ement born\'e. Alors le groupe $\P = \Stab_{\G}(p)$ est virtuellement isomorphe à $\Z^d$ et pr\'eserve des ellipsoïdes $\E^{int}$ et $\E^{ext}$ tels que
\begin{itemize}
 \item $\partial\E^{int}\cap \partial\E^{ext}=\partial\E^{int}\cap\dO = \partial\E^{ext}\cap\dO=\{p\}$;
 \item $\E^{int} \cap \Cone(p,C(\LG)) \subset \O \cap \Cone(p,C(\LG)) \subset \E^{ext} \cap \Cone(p,C(\LG))$;
 \item $\E^{int}$ est une horoboule de de l'espace hyperbolique $(\E^{ext},d_{\E^{ext}})$.
\end{itemize}
\end{coro}

\begin{proof}
Le lemme pr\'ec\'edent nous fournit un ouvert convexe $\O_p \subset \PP_p^{d+1}$ dont le groupe $\P$ est un sous-groupe parabolique de rang maximal. Prenons deux ellipsoïdes $\P$-invariants $\E_p^{int}$ et $\E_p^{ext}$ de $\PP^{d+1}_p$ comme dans le th\'eorème \ref{lemmededieu}.\\
Il existe donc des ellipsoïdes $\P$-invariants $\E^{int}$ et $\E^{ext}$ de $\PP^n$ tels que $\O_p\cap \E^{int} =\E_p^{int}$ et $\O_p\cap \E^{ext}=\E_p^{ext}$. L'action de $\P$ sur l'adh\'erence, dans $\A_p^{n-1}$ de $\mathcal{D}_p(C(\LG))$ \'etant cocompacte, on peut, quitte à prendre $\E_p^{int}$ et $\E_p^{ext}$ plus petits ou plus grands (c'est possible car, d'apr\`es la remarque \ref{pleinplein} on en a en fait une famille à un paramètre), faire en sorte que $\E^{int}$ et $\E^{ext}$ v\'erifient les conditions de l'\'enonc\'e.
\end{proof}

\begin{center}
\begin{figure}[h!]
  \centering
\includegraphics[width=9cm]{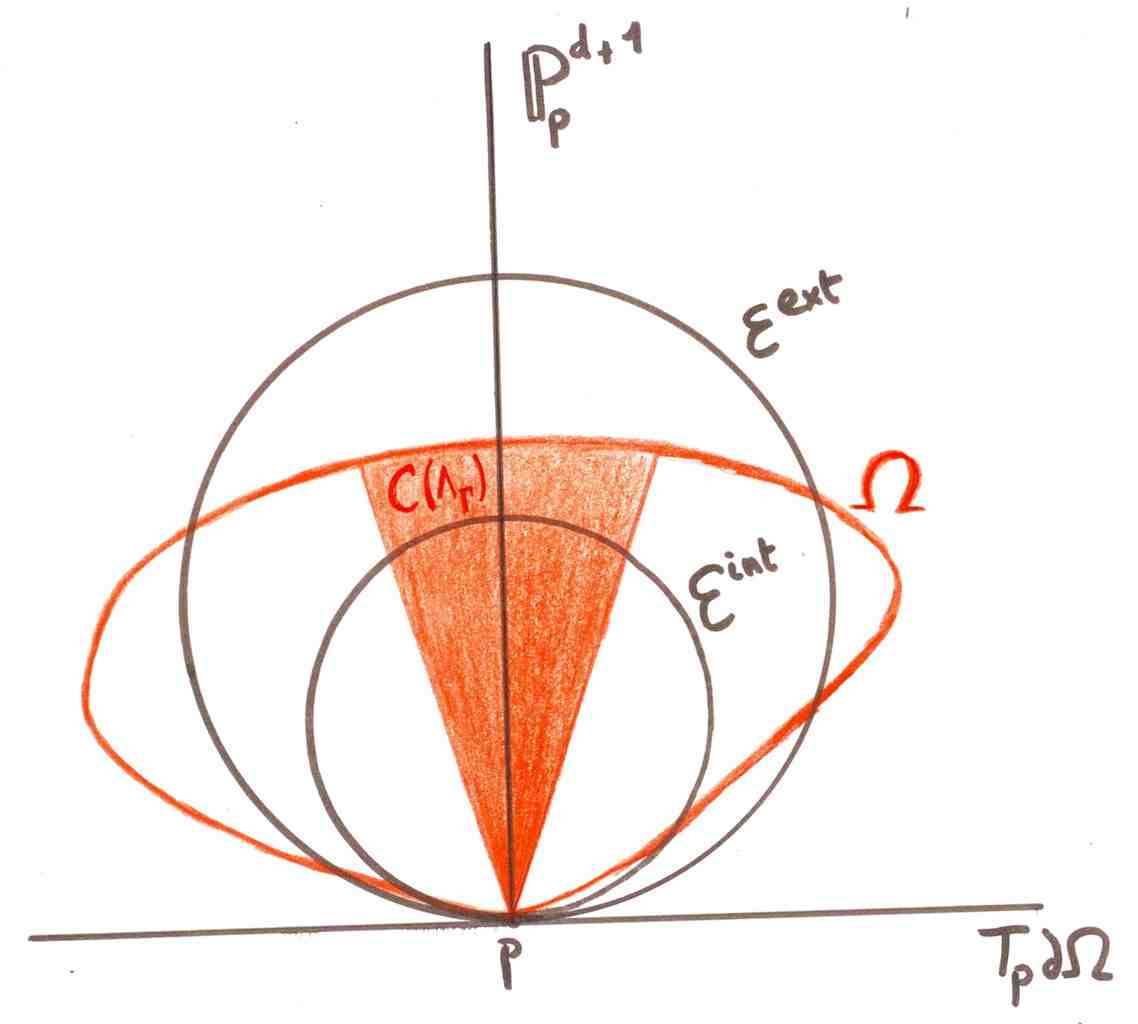}
\caption{Corollaire \ref{para_cas_gene}}
\end{figure}
\end{center}

\subsection{Constructions des r\'egions paraboliques standards}

Rappelons pourquoi il est agr\'eable que nos groupes paraboliques soient conjugu\'es \`a des groupes paraboliques de $\SO$. Ils apparaissent ainsi comme sous-groupes paraboliques d'isom\'etries de l'espace hyperbolique, mais surtout leur action sur $\A_p^{n-1}$ pr\'eserve une m\'etrique euclidienne. Le th\'eor\`eme de Bieberbach permet alors de les d\'ecrire:

\begin{theo}[Bieberbach, Th\'eor\`eme 5.4.4 de \cite{MR2249478}]\label{bieberbach}
Soit $\P$ un sous-groupe discret d'isom\'etries de l'espace euclidien $\mathbb{E}^n$. Il existe une d\'ecomposition $\P=T \times R$ du groupe $\P$ et un sous-espace $E$ de dimension $d$ tels que
\begin{itemize}
 \item le groupe $R$ est fini et agit trivialement sur $E$;
 \item le groupe $T$ est isomorphe \`a $\Z^d$ et agit cocompactement par translations sur $E$.
\end{itemize}
\end{theo}

On va maintenant d\'ecrire l'action d'un sous-groupe parabolique ``uniform\'ement born\'e'' sur $\O$.

\begin{defi}
Soit $\P$ un sous-groupe parabolique de $\Aut(\O)$ fixant le point $p\in\dO$. Une \emph{bande parabolique standard bas\'ee en $p$} est la projection sur $\dO$ d'une partie $\P$-invariante ferm\'ee, convexe et d'int\'erieur non vide de $\A^{n-1}_p$, sur laquelle l'action de $\P$ est cocompacte.
\end{defi}

Remarquons que, bien que les bandes standards soient d\'efinies comme des parties de $\dO \smallsetminus \{ p \}$, elles proviennent de convexes de $\A^{n-1}_p$. On passera souvent d'un point de vue à l'autre, en essayant de rester le plus clair possible.

\begin{prop}\label{soborne}
Soient $\G$ un sous-groupe discret de $\Aut(\O)$ et $\P$ un sous-groupe parabolique de $\G$ fixant le point $p\in\dO$. Les faits suivants sont \'equivalents:
\begin{enumerate}[(i)]
 \item le point parabolique $p$ est uniform\'ement born\'e;
 \item le groupe $\P$ est conjugu\'e à un sous-groupe parabolique de $\SO$;
 \item il existe une bande parabolique standard pour $\P$.
\end{enumerate}
\end{prop}
\begin{proof}
$(i)\Leftrightarrow (ii)$ L'implication $(i)\Rightarrow (ii)$ \'etait l'objet de la partie pr\'ec\'edente. R\'eciproquement, si $\P$ est conjugu\'e à un sous-groupe parabolique de $\SO$, alors il pr\'eserve une m\'etrique euclidienne sur $\A_p^{n-1}$. D'apr\`es le th\'eor\`eme \ref{bieberbach}, il existe un sous-espace $\A^d_p$ de dimension $d\geqslant 1$ sur lequel $\P$ agit par translations et cocompactement. L'ensemble $\LG\smallsetminus\{p\}$ (vu dans $\A_p^{n-1}$) est inclus dans un voisinage de $\A_p^d$ de taille $r$ finie. Ce voisinage est un ensemble convexe et il contient donc aussi l'enveloppe convexe de  $\LG\smallsetminus\{p\}$ dans $\A_p^{n-1}$, sur laquelle le groupe $\P$ agit encore cocompactement. Autrement dit, le point $p$ est un point parabolique uniform\'ement born\'e.\\
$(i)\Rightarrow (iii)$ Si $p$ est un point parabolique uniform\'ement born\'e, l'action de $\P$ sur l'adh\'erence de $C(\LG\smallsetminus\{p\})$ dans $\A_p^{n-1}$ est cocompacte; l'adh\'erence de $C(\LG\smallsetminus\{p\})$ dans $\A_p^{n-1}$ est donc une bande parabolique standard.\\
$(iii)\Rightarrow (ii)$ Supposons qu'il existe une bande standard $B$ pour $\P$. En proc\'edant comme dans la preuve du lemme \ref{cas_general}, on voit que l'ensemble $K$ des espaces affines maximaux inclus dans $B$, qui ont tous la m\^eme direction $D$, est compact. Ainsi, $\P$ stabilise un sous-espace affine $\A_p^d$ sur lequel il agit cocompactement. On en d\'eduit, d'apr\`es le th\'eor\`eme \ref{lemmededieu}, que $\P$ est conjugu\'e \`a un sous-groupe parabolique de $\SO$.
\end{proof}

\begin{defi}
Soit $\P$ un sous-groupe parabolique uniform\'ement born\'e de $\G$ fixant un point $p$. Si $\P$ est de rang maximal alors une \emph{r\'egion parabolique standard bas\'ee en $p$} est une horoboule de centre $p$. Si $\P$ n'est pas de rang maximal alors  une \emph{r\'egion parabolique standard bas\'ee en $p$} est l'enveloppe convexe du compl\'ementaire d'une bande standard d'int\'erieur non vide de $\P$ dans $\overline{\O}$.
\end{defi}

Dans le cas o\`u $\O$ est un ellipso\"ide, on retrouve les r\'egions paraboliques standards consid\'er\'ees par Bowditch \cite{MR1218098}.

\begin{prop}\label{prop_reg}
Soient $\G$ un sous-groupe discret de $\Aut(\O)$ et $p$ un point parabolique uniform\'ement born\'e de $\LG$, de stabilisateur $\P$ dans $\G$.\\
Toute r\'egion parabolique standard $R$ est une partie convexe et $\P$-invariante de $\overline{\O}$, la vari\'et\'e à bord $\Quotient{ (\overline{\O} \smallsetminus (R \cup \{ p \} ) ) }{\P}$ est compacte et l'ensemble $R \cap \O$ est un ouvert.\\
En particulier, si $D_{\P}$ est un domaine fondamental convexe localement fini pour l'action de $\P$ sur $\O$, alors l'adh\'erence de $D_{\P} \smallsetminus R$ dans $\overline{\O}$ ne contient pas $p$.
\end{prop}

\begin{proof}
Si le groupe $\P$ est de rang maximal, c'est \'evident.\\
Soit donc $B$ la bande parabolique standard d\'efinissant la r\'egion parabolique standard $R$. L'ensemble $\mathcal{D}(D_{P})$ est un domaine fondamental pour l'action de $\P$ sur $\A^{n-1}_p$ et l'intersection de $\mathcal{D}(D_{P})$ avec $B$ est un domaine fondamental pour l'action de $\P$ sur $B$, qui est compact. Il vient alors que l'adh\'erence de $D_{\P} \smallsetminus R$ dans $\overline{\O}$ ne contient pas $p$. Ce qui montre que la vari\'et\'e à bord $\Quotient{ (\overline{\O} \smallsetminus (R \cup \{ p \} ) ) }{\P}$ est compacte. Les autres points sont triviaux.
\end{proof}

\begin{rema}\label{chiantissime}
Donnons-nous un sous-groupe discret $\G$ de $\Aut(\O)$ et un point parabolique $p\in\LG$ uniform\'ement born\'e. On peut construire une r\'egion parabolique standard pour le stabilisateur $\P$ de $p$ de la façon suivante.\\
Pour un point $x$ de $C(\LG)$, on consid\`ere le plan tangent $T_x\H_p(x)$; il s\'epare $\O$ en deux ouverts convexes et on appelle $\O(x,p)$ celle qui contient $p$. On obtient une r\'egion parabolique standard en choisissant une horoboule $H_p$ bas\'ee en $p$, de bord l'horosph\`ere $\H_p$, et en posant
$$R_{H_p} = \bigcap_{x\in C(\LG)\cap \H_p} \O(x,p).$$
Plus g\'en\'eralement, on peut consid\'erer les intersections
$$\bigcap_{x\in A\cap \H_p} \O(x,p),$$
pour toute partie $A$ convexe et $\P$-invariante. En particulier, on pourrait prendre pour $A$ un ouvert convexe $\O_p$ pr\'eserv\'e par $\P$.
\end{rema}

\begin{center}
\begin{figure}[h!]
  \centering
\includegraphics[width=5.5cm]{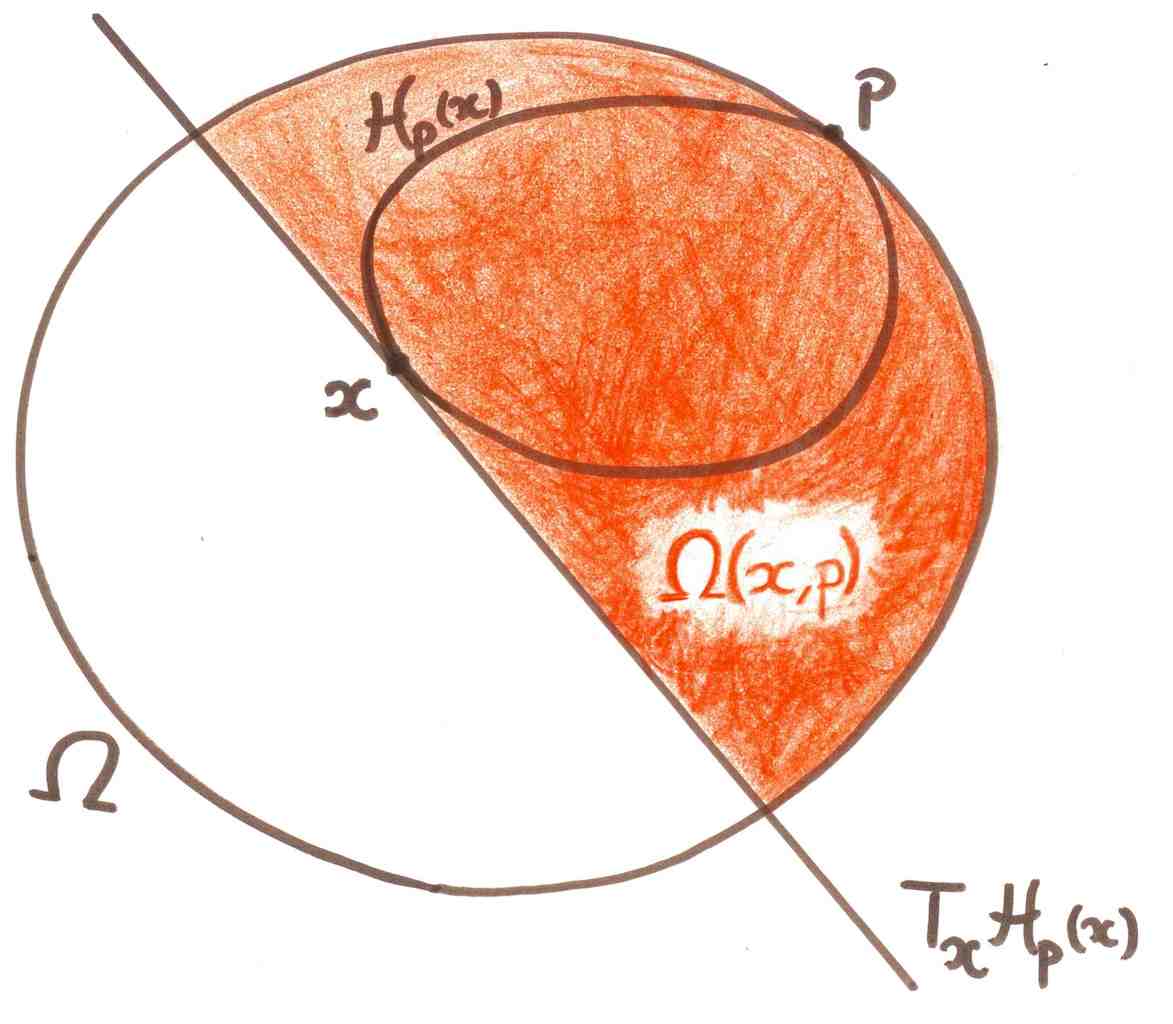}\ \includegraphics[width=5cm]{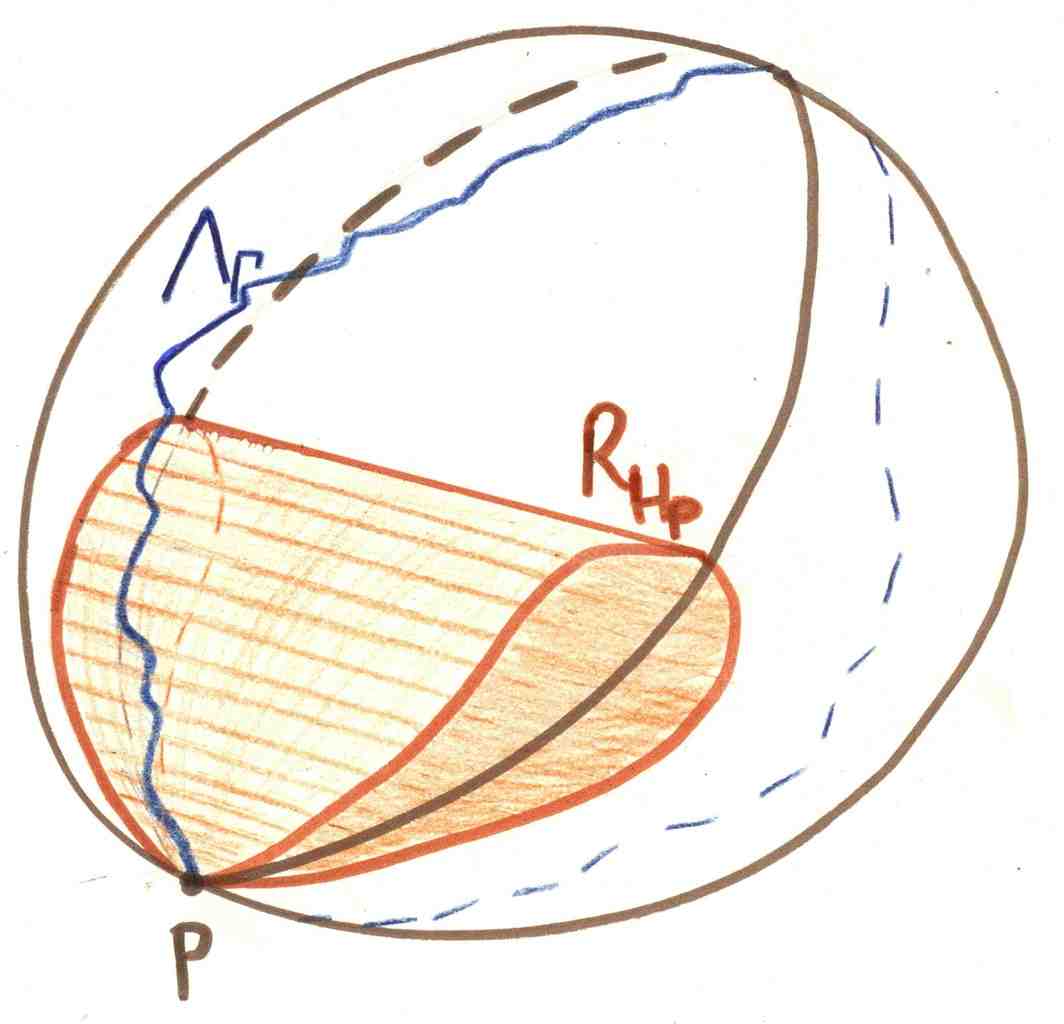}
\caption{Remarque \ref{chiantissime}
\label{reg_standard}
}
\end{figure}
\end{center}


De cette derni\`ere remarque, on d\'eduit le

\begin{coro}\label{horo_stan}
Soient $\G$ un sous-groupe discret de $\Aut(\O)$ et $p$ un point parabolique uniform\'ement born\'e de $\dO$ de stabilisateur $\P$ dans $\G$. Il existe une horoboule $H$ bas\'ee en $p$ telle que
$$R_H \cap C(\LG) = H \cap C(\LG) \subset \O_{\varepsilon}(\P).$$
\end{coro}

\begin{proof}
Fixons $o$ dans $\O$, notons $H_t$ l'horoboule
$$H_t=\{x\in\O,\ b_p(o,x)\leqslant t\},$$
et $\H_t$ l'horosph\`ere au bord de $H_t$. Pour tout $t\in\R$, l'action de $P$ sur $C(\LG)\cap\H_t$ est cocompacte et on a
$$\lim_{t\to+\infty} \sup_{x\in C(\LG)\cap H_t}  \d(x,\g x) = 0.$$ On peut donc choisir $t_0$ assez grand pour que $H_{t_0}\cap C(\LG)$ soit inclus dans $\O_{\varepsilon}(\P)$.\\
De cette fa\c con, la r\'egion parabolique standard construite comme dans la remarque pr\'ec\'edente via
$$R_{H_{t_0}} = \bigcap_{x\in \H_{t_0}\cap C(\LG)} \O(x,p)$$
v\'erifie $R_{H_{t_0}} \cap C(\LG) = H_{t_0}\cap C(\LG)$.
\end{proof}

\begin{nota}
Soit $\G$ un sous-groupe discret de $\Aut(\O)$. On notera $\Pi$ (resp. $\Pi_{ub}$) l'ensemble des points paraboliques (resp. paraboliques uniform\'ement born\'es) de $\G$ et pour tout point $p \in \Pi$ on notera $\P_p = \Stab_p(\G)$.  
\end{nota}

\begin{lemm}\label{region_stan_disj}
Soit $\G$ un sous-groupe discret de $\Aut(\O)$. \`A tout point $p\in \Pi_{ub}$, on peut associer une r\'egion parabolique standard $R_p$ de telle fa\c con que la famille $(R_p)_{p \in \Pi_{ub}}$ soit strictement invariante, c'est-à-dire:
\begin{itemize}
\item $\forall \g \in \G$ et $\forall p \in \Pi_{ub}$, on a $R_{\g p} = \g R_p$
\item $\forall p,q \in \Pi_{ub}$ distincts, $R_p \cap R_q = \varnothing$.
\end{itemize}
\end{lemm}

\begin{proof}
Choisissons une famille de points paraboliques $(p_i)_{i\in I}$ uniform\'ement born\'es telle, que pour tout $p\in \Pi_{ub}$, il existe un unique $i\in I$  et un \'el\'ement $\g\in\G$ tels que $\g p_i=p$. Les stabilisateurs $(\P_{p_i})_{i\in I}$ forment une famille de repr\'esentants des classes de conjugaisons de sous-groupes paraboliques maximaux uniform\'ement born\'es de $\G$.\\

Pour chaque $i\in I$, on fixe une horoboule $H_{p_i}$ bas\'ee en $p_i$ comme dans le corollaire \ref{horo_stan}: on a
$$H_{p_i}\cap C(\LG) \subset \O_{\varepsilon}(\P_{p_i}).$$
Notons $\H_{p_i}$ l'horosph\`ere au bord de $H_{p_i}$. La r\'egion parabolique standard donn\'ee par
$$R_{p_i} = \bigcap_{x\in \H_{p_i}\cap C(\LG)} \O(x,p)$$
v\'erifie $R_{p_i}\cap C(\LG) = H_{p_i}$.\\

\`A chaque point $p=\g p_i$ de $\Pi_{ub}$, on associe l'horoboule $H_p=\g H_{p_i}$ et la r\'egion parabolique standard $R_p=\g R_{p_i}$. La famille $(R_p)_{p\in\Pi_{ub}}$ ainsi construite v\'erifie alors imm\'ediatement le premier point du lemme. Voyons qu'elle v\'erifie aussi le second.\\
Pour cela, prenons deux points distincts $p,q\in\Pi_{ub}$. Les ensembles $\O_{\varepsilon}(\P_p)$ et $\O_{\varepsilon}(\P_q)$ sont disjoints d'apr\`es le lemme \ref{decomposition_partie_fine} et donc les horoboules $H_p$ et $H_q$ \'egalement. La droite $(pq)$ coupe $H_p$ en $P$ et $H_q$ en $Q$. L'intersection des plans tangents \`a $H_p$ et $H_q$ en $P$ et $Q$ v\'erifie (voir section \ref{busemannhoro})
$$T_PH_p\cap T_QH_q = T_p\dO \cap T_q \dO.$$
Ainsi, les ensembles $\O(P,p)$ et $\O(Q,q)$ sont disjoints et par suite $R_p$ et $R_q$ aussi.
\end{proof}

\begin{center}
\begin{figure}[h!]
  \centering
\includegraphics[width=7cm]{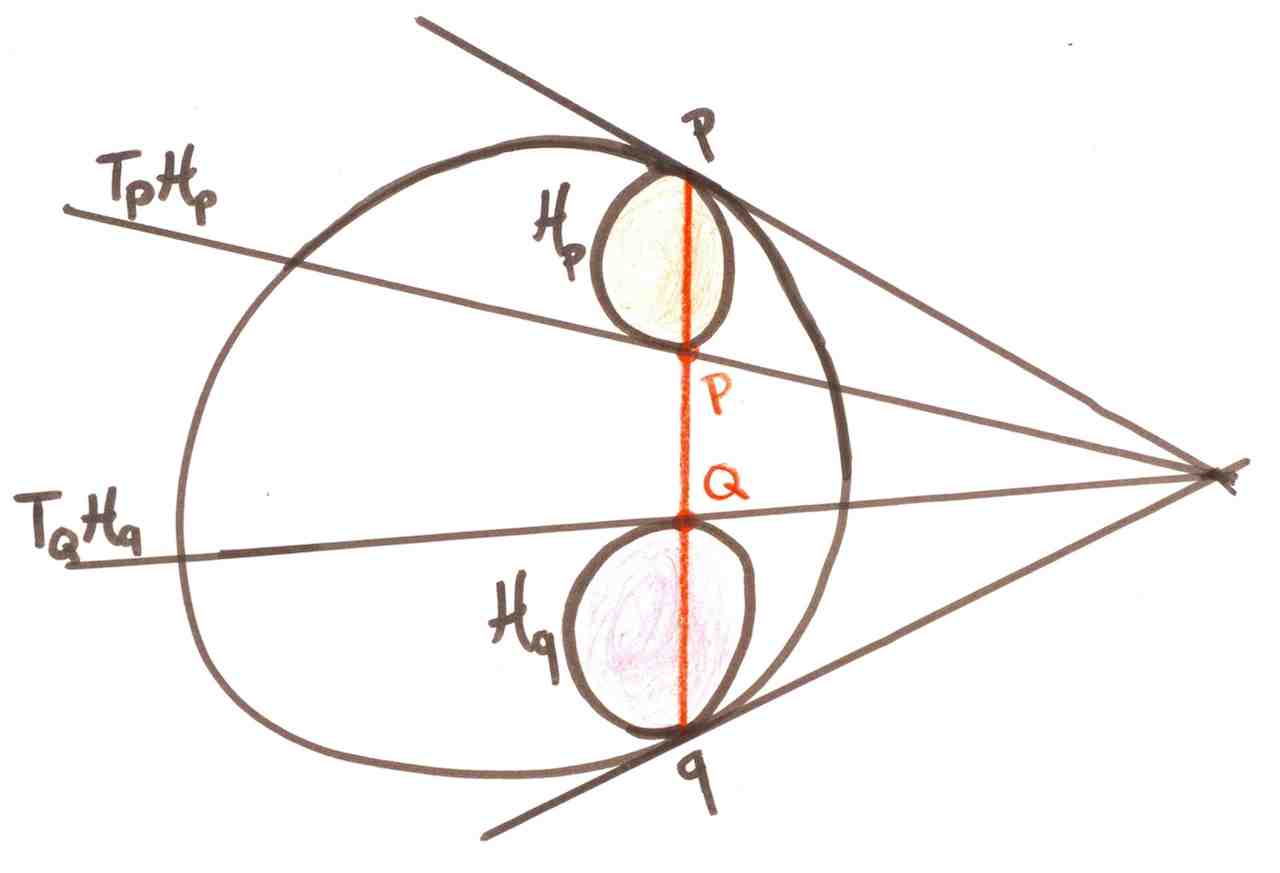}
\caption{Les r\'egions $(R_p)_{p\in\Pi_{ub}}$ sont disjointes}
\end{figure}
\end{center}

\subsection{Adh\'erence de Zariski de $\G$}

Dans \cite{MR1767272}, Yves Benoist a montr\'e le

\begin{theo}[Benoist]\label{benoist_zar}
Soit $\O$ un ouvert proprement convexe strictement convexe de $\PP^n$. Si $\G$ est un sous-groupe de $\Aut(\O)$ agissant de façon cocompacte sur $\O$, alors l'adh\'erence de Zariski de $\G$ est soit $\ss$ soit conjugu\'ee à $\SO$.
\end{theo}

Nous allons, en utilisant les mêmes techniques, montrer un r\'esultat similaire, valable pour les actions g\'eom\'etriquement finies sur $\O$ qui ne sont pas convexes-cocompactes. Dans ce dernier cas, le r\'esultat est faux comme nous le verrons dans la partie \ref{cex}.\\
Signalons en passant que dans \cite{MR2010735}, Benoist a montr\'e le th\'eorème \ref{benoist_zar} en se passant de l'hypothèse de stricte convexit\'e; nous renvoyons à son texte pour un \'enonc\'e pr\'ecis.\\
Notre r\'esultat est le suivant:

\begin{theo}\label{zariski-dense}
Soit $\G$ un sous-groupe discret et irr\'eductible de $\Aut(\O)$. Si  $\LG$ contient un point parabolique uniform\'ement born\'e, alors l'adh\'erence de Zariski de $\G$ est soit $\ss$ soit conjugu\'ee à $\SO$.
\end{theo}







Nous utiliserons le r\'esultat \ref{lemm_benoist} ci-dessous, d\^u \`a Benoist. Pour cela, il nous faut d'abord d\'efinir quelques objets.\\

Soit $G$ un groupe de Lie r\'eel semi-simple connexe.\\

Consid\'erons une repr\'esentation irr\'eductible $\rho$ de $G$ dans $\ss$. On dit qu'elle est \emph{proximale} si tout sous-groupe nilpotent $N$ maximal de $G$ stabilise exactement une droite de $\R^{n+1}$, c'est-\`a-dire un point de $\PP^n$; cette droite est la \emph{droite de plus haut poids} associ\'ee \`a $N$. De fa\c con \'equivalente, la repr\'esentation $\rho$ est proximale s'il existe un \'el\'ement $g\in G$ dont l'image $\rho(g)$ est un \'el\'ement proximal, c'est-\`a-dire que sa valeur propre de module maximal est de multiplicit\'e $1$.\\

Supposons donc que la repr\'esentation $\rho$ est proximale. Pour chaque \'el\'ement proximal $g$, on note $x_g^+$ le point de $\PP^n$ correspondant \`a sa valeur propre de module maximal. Les repr\'esentations proximales ont la propri\'et\'e remarquable qu'il existe un plus petit ferm\'e invariant; on l'appelle \emph{l'ensemble limite de $G$ dans $\PP^n$}, qu'on note $\Lambda_G$.\\

Comme tout groupe semi-simple connexe, $G$ admet une d\'ecomposition d'Iwasawa $G=KAN$, o\`u $K$ est un sous-groupe compact maximal, $A$ un tore maximal, et $N$ un sous-groupe nilpotent maximal. Si $x\in \PP^n$ est la droite de plus haut poids de $N$, comme $A$ normalise $N$ et que $x$ est l'unique point fixe de $N$, on a que $A$ fixe aussi $x$. L'orbite de $x$ sous $G$ est donc celle de $x$ sous le groupe compact $K$, et \`a ce titre, c'est une orbite ferm\'ee; elle est donc \'egale à l'ensemble limite $\Lambda_G$.\\
L'ensemble limite $\Lambda_G$ est donc l'orbite de la droite de plus haut poids $x$ sous $G$. Comme $x$ est fix\'e par $A$, il existe un \'el\'ement proximal $g\in A$ tel que $x_g^+ =x$. Cela permet de voir que 
$\Lambda_G = \{x_g^+,\ g \in G\ \text{proximal}\}.$

\begin{lemm}[Th\'eorème 1.5 et d\'emonstration du th\'eorème 3.6 de \cite{MR1767272}]\label{lemm_benoist}
Soient $\O$ un ouvert proprement convexe et $\G$ un sous-groupe irr\'eductible de $\Aut(\O)$. La composante neutre $G$ de l'adh\'erence de Zariski de $\G$ est un groupe de Lie semi-simple et la repr\'esentation $\rho:G\rightarrow \ss$ est irr\'eductible et proximale.\\
De plus, si l'ensemble limite $\Lambda_{G}$ de $G$ s'identifie au bord d'un ouvert proprement convexe $\O'$, c'est-\`a-dire $\L_G=\partial\O'$, alors $G$ est conjugu\'e à  $\SO$; si $\Lambda_{G}= \PP^n$ alors $G=\ss$.
\end{lemm}

\begin{proof}[O\`u trouver la preuve dans \cite{MR1767272}]
Tout d'abord, Benoist montre que l'adh\'erence de Zariski d'un groupe irr\'eductible $\G$ qui pr\'eserve un ouvert proprement convexe est un groupe de Lie semi-simple; voir la proposition 3.1 et la remarque qui suit le corollaire 3.2.\\
Ensuite, le th\'eorème 1.5 montre que la repr\'esentation de $G$ dans $\ss$ ainsi obtenue est proximale.\\
Enfin, la d\'emonstration du th\'eorème 3.6 de \cite{MR1767272} se divise en deux cas, $\Lambda_{G}= \dO'$ ou $\Lambda_{G}= \PP^n$, et conclut comme indiqu\'e dans l'\'enonc\'e du lemme.
\end{proof}

Dans la d\'emonstration qui suit, on appellera \emph{ellisphère} de dimension $k$ le bord d'un ellipso\"ide de dimension $k+1$.

\begin{proof}[D\'emonstration du th\'eorème \ref{zariski-dense}]
Soit $G$ la composante connexe de l'adh\'erence de Zariski de $\G$. Le lemme \ref{lemm_benoist} montre que $G$ est un groupe de Lie semi-simple et la repr\'esentation $\rho:G\rightarrow \ss$ est irr\'eductible et proximale. Si $G=KAN$ est une d\'ecomposition d'Iwasawa de $G$, alors l'ensemble limite de $G$ est $\Lambda_{G} = K\cdot x$, o\`u $x$ d\'esigne la droite de plus haut poids de $N$. $\L_G$ est ainsi une sous-vari\'et\'e alg\'ebrique compacte connexe de $\PP^n$.\\
Fixons un point parabolique uniform\'ement born\'e $p$ de $\LG$. Notons $\P_p$ le stabilisateur dans $\G$ de $p$ et $\U_p$ le sous-groupe de l'adh\'erence de Zariski de $\P_p$ form\'e par les \'el\'ements unipotents. Le lemme \ref{para_cas_gene} montre que $\U_p$ est un groupe ab\'elien isomorphe à $\R^{k}$. D'apr\`es ce même lemme, il existe un sous-espace $F_p$ de dimension $k$ de l'hyperplan tangent $T_p\dO$ tel que tout sous-espace $H'$ de dimension $k+1$ de $\PP^n$ contenant $F_p$ et intersectant $\O$ est pr\'eserv\'e par $\U_p$; de plus, si $z$ est un point hors de $T_p\dO$, alors l'ensemble $\U_p \cdot z \cup \{ p \}$ est une ellisphère de dimension $k$. Si $z$ est dans $\LG$ ou plus g\'en\'eralement dans $\Lambda_G$, cette ellisph\`ere est incluse dans $\Lambda_G$.
\\

Commen\c cons par le cas simple o\`u le groupe $\P_p$ est de rang maximal. L'ensemble limite $\L_G$ contient alors une ellisphère de dimension $n-1$. Ainsi, soit $\L_G$ est pr\'ecis\'ement cette ellisph\`ere, soit $\L_G$ est de dimension $n$, autrement dit, $\L_G=\PP^n$. Le lemme \ref{lemm_benoist} permet de conclure comme annonc\'e.\\

Traitons maintenant le cas g\'en\'eral en supposant que le groupe parabolique $\P_p$ est de rang $1$. Dans ce cas, le groupe $\U_p$ est un groupe ab\'elien isomorphe à $\R$. Soit $z$ un point de $\L_G\smallsetminus T_p\dO$ et $H_z$ le plan projectif engendr\'e par $z$ et $F_p$, qui est stable sous $\U_p$. L'ensemble limite $\L_G$ contient l'ellipse $\U_p \cdot z \cup \{ p \}$. Par cons\'equent, la sous-vari\'et\'e alg\'ebrique $\L_G \cap H_z$ de $H_z$ est soit une ellipse soit $H_z$ tout entier, et cette conclusion ne d\'epend pas de $z$. Comme $\G$ est irr\'eductible, le cas $\LG \cap H_z=H_z$ implique que $\L_G = \PP^n$ et donc que $G=\ss$ par le lemme \ref{lemm_benoist}.\\
Supposons donc que $\L_G \cap H_z$ est une ellipse. Comme le sous-groupe compact maximal $K$ de $G$ agit transitivement sur $\Lambda_G$, ceci est en fait valable pour tous points $p$ de $\L_G$ et $z$ de $\LG$: il existe une droite $F_p$ de $T_p\Lambda_G$ telle que pour tout sous-espace $H$ de dimension $2$ contenant $F_p$ et non inclus dans $T_p\Lambda_G$, l'intersection $H \cap \Lambda_G$ est une ellipse.\\

On va montrer que $\Lambda_G$ est une ellisphère de dimension $n-1$, en utilisant une r\'ecurrence, dont l'initialisation vient juste d'\^etre faite.\\

Prenons $k\geqslant 1$. Supposons que pour tout point $p$ de $\L_G$, il existe un sous-espace $F^k_p$ de dimension $k$ de $T_p\Lambda_G$ tel que pour tout sous-espace $H$ de dimension $k+1$ contenant $F^k_p$ et non inclus dans $T_p\Lambda_G$ alors $H \cap \Lambda_G$ est une ellisphère de dimension $k$.\\
Voyons que cette propri\'et\'e est encore vraie au rang $k+1$. Par irr\'eductibilit\'e de $\G$, on peut trouver des points $p, q \in \LG$ tels que l'espace engendr\'e par la droite $F_p$ et le sous-espace $F^k_q$ soit de dimension $k+2$. En effet, on aurait sinon que l'intersection $F$ de tous les sous-espaces $F^k_q$, pour $q\in\LG$, est non vide. $F$ serait donc un sous-espace projectif pr\'eserv\'e par $\G$, et donc $\G$ ne serait pas irr\'eductible.\\
Notons $E$ l'espace de dimension $k+2$ engendr\'e par $F_p$ et $F^k_q$. On obtient ainsi deux feuilletages en ellisph\`eres de $E\cap \Lambda_G$ qui n'ont aucune feuille en commun. Cela montre que $E\cap \Lambda_G$ est une ellisphère de dimension $k+1$. L'espace tangent en $p$ \`a cette ellisph\`ere est l'espace $F_{k+1}^p$ que l'on cherchait. On a le r\'esultat pour tout point de $\L_G$ en utilisant l'action de $K$.\\
Le cas $k=n-1$ permet de conclure que $\Lambda_G$ est une ellisphère de dimension $n-1$.
\end{proof}

\section{D\'efinitions \'equivalentes de la finitude g\'eom\'etrique}\label{mainsection}

Le but de cette partie est de montrer notre th\'eorème principal, qui donne des d\'efinitions \'equivalentes de la notion de finitude g\'eom\'etrique \emph{sur $\O$}. En fait, celles-ci sont pr\'ecis\'ement celles que Brian Bowditch \cite{MR1317633} a donn\'ees, en courbure n\'egative pinc\'ee, pour la finitude g\'eom\'etrique telle que d\'efinie en \ref{geofini_ghyp}.\\

Pour être plus pr\'ecis, et plus juste, la première d\'efinition d'action g\'eom\'etriquement finie est due à Lars Alhfors dans \cite{MR0194970} dans le contexte de g\'eom\'etrie hyperbolique de dimension 3. Ahlfors demandait à cette action d'avoir un domaine fondamental qui soit un polyèdre avec un nombre fini de côt\'es. Le temps (sous l'action de Brian Bowditch) a montr\'e que cette d\'efinition n'\'etait pas la bonne en dimension sup\'erieure ou \'egale à 4. Une seconde d\'efinition, (GF) dans ce texte, a \'et\'e propos\'ee par Alan Beardon et Bernard Maskit \cite{MR0333164} pour la dimension 3. William Thurston propose 3 autres d\'efinitions dans ses notes (\cite{MR1435975} chapitre 8), toujours en dimension 3; ce sont les d\'efinitions (PEC), (PNC), (VF) de ce texte. La situation devient vraiment claire lorsque Bowditch \cite{MR1218098,MR1317633} montre qu'en g\'eom\'etrie hyperbolique ou en courbure n\'egative pinc\'ee, toutes ces d\'efinitions sont \'equivalentes et ce quelque soit la dimension.

\begin{theo}\label{theo_geo_finie}
Soient $\G$ un sous-groupe discret de $\Aut(\O)$, et $M=\Quo$ l'orbifold quotient correspondante. Les propositions suivantes sont \'equivalentes.
\begin{enumerate}
\item[(GF)] L'action de $\G$ sur $\O$ est g\'eom\'etriquement finie sur $\O$ (i.e les points de $\LG$ sont des points limites coniques ou des points paraboliques \underline{uniform\'ement} born\'es).
\item[(TF)] Le quotient $\Quotient{\Og}{\G}$ est une orbifold à bord qui est l'union d'un compact et d'un nombre fini de projections de r\'egions paraboliques standards disjointes.
\item[(PEC)] La partie \'epaisse du c\oe ur convexe de $M$, c'est-à-dire $M^{\varepsilon}\cap C(M)$, est compacte.
\item[(PNC)] La partie non cuspidale du c\oe ur convexe de $M$, c'est-à-dire $M^{nc}_{\varepsilon}\cap C(M)$, est compacte.
\item[(VF)] Le $1$-voisinage du c\oe ur convexe de $\Quo$ est de volume fini et le groupe $\G$ est de type fini.
\end{enumerate}
En particulier, le quotient $M=\Quo$ est \emph{sage}, c'est-\`a-dire l'int\'erieur d'une orbifold compacte à bord, et par suite le groupe $\G$ est de pr\'esentation finie.
\end{theo}

\subsection{Finitude topologique}

\begin{lemm}\label{lem_pbord}
Soit $\G$ un sous-groupe discret de $\Aut(\O)$. Soit $D$ un domaine fondamental convexe et localement fini pour l'action de $\G$ sur $\O$. Aucun point de $\partial D \cap \partial \O$ n'est un point limite conique.
\end{lemm}

\begin{proof}
Soient $p$ un point de $\partial D \cap \partial \O$ et $x$ un point de $D$. La demi-droite $[xp[ \subset D$ d\'efinit une demi-g\'eod\'esique de $\Quo$ qui sort de tout compact; par cons\'equent, le point $p$ n'est pas un point limite conique.
\end{proof}

\begin{proof}[D\'emonstration de (GF)$\Rightarrow$(TF)]
Le lemme \ref{dom_disc} montre que le groupe $\G$ agit proprement discontin\^ument sur $\Og$. Le lemme \ref{region_stan_disj} montre que pour tout point point parabolique $p$, il existe une r\'egion parabolique standard $R_p$ bas\'ee en $p$ puisque l'action de $\G$ est g\'eom\'etriquement finie sur $\O$. De plus, le même lemme \ref{region_stan_disj} montre que l'on peut choisir ces r\'egions de telle sorte que la famille $(R_p)_{p \in \Pi}$ soit strictement invariante, puisque l'action est g\'eom\'etriquement finie sur $\O$ ($\Pi$ d\'esigne l'ensemble des points paraboliques).

On considère la partie $K$ de $\Qug$ obtenue en retirant les r\'egions paraboliques standards $R_p$ bas\'ees aux points paraboliques $p$. Il nous reste à montrer que $K$ est compact et que l'ensemble $\Pi$ des points paraboliques est fini modulo $\G$. D'apr\`es le lemme \ref{decomposition_partie_fine}, les composantes connexes du bord de $K$ sont en bijection avec les classes de points paraboliques modulo $\G$. Ainsi, si $K$ est compact, alors l'ensemble $\Quotient{\Pi}{\G}$ est fini. Il suffit donc de montrer la compacit\'e de $K$ pour conclure.

On considère un domaine fondamental convexe et localement fini $D$ pour l'action de $\G$ sur $\O$. On doit montrer que tout point d'accumulation $z$ dans $\overline{\O}$ de $D \smallsetminus \bigcup_{p \in \Pi} R_p$ est un point de $\Og$. Comme l'action de $\G$ sur $\O$ est g\'eom\'etriquement finie sur $\O$, on a $\LG \cap \overline{D} \subset \Pi$ d'après le lemme \ref{lem_pbord}. Le point $z$ est donc soit dans $\Og$ soit un point de $\Pi$. La proposition \ref{prop_reg} montre qu'aucune suite de points de  $D \smallsetminus \bigcup_{p \in \Pi} R_p$ ne peut converger vers un point parabolique.
\end{proof}

\subsection{Parties \'epaisse et non cuspidale}

Donnons maintenant une

\begin{proof}[Preuve de (TF)$\Rightarrow$(PNC)$\Rightarrow$(PEC)]
Supposons que $\G$ v\'erifie (TF). Il existe alors un compact $K$ de $\Og$ et une famille $\G$-\'equivariante $(R_{p_i})_{1\leqslant i\leqslant k}$ de r\'egions paraboliques standards disjointes, bas\'ees en des points paraboliques $p_i$, tels que 
$$\Og = (\G\cdot K) \bigsqcup \sqcup_{i=1}^k \G\cdot R_{p_i}.$$
Le c\oe ur convexe de $M$ est le quotient $\Quotient{\overline{C(\LG)}^{\O}}{\G}$, o\`u $\overline{C(\LG)}^{\O}$  d\'esigne l'adh\'erence de $C(\LG)$ dans $\O$. Or, on a
$$\overline{C(\LG)}^{\O} = \G\cdot (K\cap\overline{C(\LG)}^{\O})  \bigsqcup \sqcup_{i=1}^k \G\cdot (R_{p_i}\cap \overline{C(\LG)}^{\O});$$
autrement dit, $C(M)$ est l'union d'un compact et des projections des $R_{p_i}\cap \overline{C(\LG)}^{\O}$.\\
Le corollaire \ref{soborne} montre que tous les points paraboliques de $\LG$ sont uniform\'ement born\'es. Par cons\'equent, il existe pour chaque $p_i$ une horoboule $H_{p_i}$ bas\'ee en $p_i$ telle que 
$$H_{p_i}\cap C(\LG) \subset R_{p_i}\cap C(\LG).$$ Le lemme \ref{horo_stan} montre qu'on peut choisir $H_{p_i}$ de telle fa\c con que
$$H_{p_i}\cap C(\LG) \subset \O_{\varepsilon}(\Stab_{\G}(p_i)).$$
L'ensemble $K_{p_i} = \overline{C(\LG) \cap R_{p_i} \smallsetminus H_{p_i}\cap C(\LG)}^{\O}$ est compact, et, en posant 
$$K'=K\cap\overline{C(\LG)}^{\O}\bigcup \cup_{i=1}^k K_{p_i},$$ on obtient
$$\overline{C(\LG)}^{\O} = (\G\cdot K') \bigsqcup \sqcup_{i=1}^k \G\cdot (H_{p_i}\cap \overline{C(\LG)}^{\O}).$$
La partie non cuspidale du c\oe ur convexe est un ferm\'e du compact $\Quotient{\G\cdot K'}{\G}$, elle est donc compacte. L'implication (TF)$\Rightarrow$(PEC) est imm\'ediate puisque la partie \'epaisse du c\oe ur convexe est un ferm\'e de la partie non cuspidale du c\oe ur convexe.\\
Reste \`a voir que (PEC) entra\^ine (PNC). Supposons donc que la partie \'epaisse du c\oe ur convexe soit compacte. Celle-ci \'etant une orbifold à bord, le nombre de ses composantes connexes de bord est fini. Ainsi, $M_{\varepsilon}^{nc}\cap C(M) \smallsetminus M^{\varepsilon}\cap C(M)$ a un nombre fini de composantes connexes. Or,  d'après le lemme \ref{decomposition_partie_fine}, chacune des composantes connexes de $M_{\varepsilon}^{nc}\cap C(M) \smallsetminus M^{\varepsilon}\cap C(M)$ est compacte. Il vient que la partie non cuspidale elle-même est compacte.
\end{proof}

\begin{rema}\label{decomposition_corps_cvx} 
La preuve pr\'ec\'edente montre que sous l'hypoth\`ese (TF), le c\oe ur convexe de $M$ se d\'ecompose en
$$C(M) = (C(M))_{\varepsilon}^{nc} \bigsqcup \sqcup_{i=1}^k \Quotient{\left(H_{p_i}\cap \overline{C(\LG)}^{\O}\right)}{\P_{p_i}},$$
o\`u $(C(M))_{\varepsilon}^{nc}$ est la partie non cuspidale du c\oe ur convexe, qui est compacte, les $\{p_i\}_{1\leqslant i \leqslant k}$ forment un ensemble de repr\'esentants de points paraboliques de $\LG$, les $\{H_{p_i}\}$ sont des horoboules bas\'ees aux points $\{p_i\}$ et $\P_{p_i}=\Stab_{\G}(p_i)$.
\end{rema}

Bouclons une première boucle :

\begin{proof}[Preuve de (PNC) $\Rightarrow$ (GF)]
Tout d'abord, comme la partie non cuspidale du c\oe ur convexe de $M$ est compacte, le nombre de ses composantes connexes de bord est fini. Cela entraîne que $M$ a un nombre fini de cusps.

Soient $p$ un point de l'ensemble limite $\Lambda_{\G}$ et $x$ un point dans l'enveloppe convexe $C(\LG)$ de $\Lambda_{\G}$ dans $\O$. La projection de la demi-droite $[xp)$ sur le quotient $M=\Quo$ est un rayon g\'eod\'esique inclus dans le c\oe ur convexe de $M$. De deux choses l'une: soit ce rayon g\'eod\'esique revient un nombre infini de fois dans la partie non cuspidale du c\oe ur convexe, qui est compacte, et donc le point $p$ est un point limite conique; soit il n'y revient qu'un nombre fini de fois, et il est ainsi ultimement inclus dans une composante connexe de la partie cuspidale de $M$, puisque $M$ a un nombre fini de cusps; le point \ref{partie_fine_cusp} du lemme \ref{decomposition_partie_fine} montre alors que le point $p$ est parabolique, n\'ecessairement uniform\'ement born\'e puisque la partie non cuspidale du c\oe ur convexe est compacte. Le quotient $M=\Quo$ est donc g\'eom\'etriquement fini.
\end{proof}

\subsection{Volume}
  
Nous allons voir ici que l'hypothèse (VF) est \'equivalente à la finitude g\'eom\'etrique sur $\O$. Remarquons que cette hypothèse en regroupe en fait deux: 
\begin{enumerate}
\item[(a)] le $1$-voisinage du c\oe ur convexe de $\Quo$ est de volume fini et
\item[(b)] l'ordre des sous-groupes finis de $\G$ est born\'e.
\end{enumerate}
Dans le point $(a)$, on est oblig\'e de consid\'erer le $1$-voisinage pour prendre en compte les groupes dont l'action serait r\'eductible: dans ce cas, le c\oe ur convexe est d'int\'erieur vide et son volume est donc toujours nul. Si on suppose que les groupes sont irr\'eductibles, on peut alors consid\'erer le c\oe ur convexe et non son $1$-voisinage.\\
En g\'eom\'etrie hyperbolique, le point $(b)$ est inutile lorsque le quotient $\Quo$ est de volume fini ou la dimension est inf\'erieure ou \'egale à $3$. On notera qu'Emily Hamilton \cite{MR1604903} a construit un sous-groupe $\G_0$ de $\mathrm{SO}_{4,1}(\R)$ tel que le $1$-voisinage du c\oe ur convexe est de volume fini mais tel que le groupe $\G_0$ n'est pas de type fini et par suite l'action de $\G_0$ n'est pas g\'eom\'etriquement finie sur $\mathbb{H}^4$.\\

Pour prouver l'\'equivalence, nous utiliserons le fait que l'on peut minorer de façon uniforme le volume des boules de rayon $r>0$ d'une g\'eom\'etrie de Hilbert:

\begin{lemm}[Colbois - Vernicos Th\'eorème 12 de \cite{MR2245997}]\label{Min_vol_boul}
Pour tout $n \geqslant 1$ et tout $r >0$, il existe une constante $v_n(r) >0$ tel que pour tout ouvert proprement convexe $\O$ de $\PP^n$, pour tout point $x$ de $\O$, on a $$\textrm{Vol}_{\O}(B_{\O}(x,r)) \geqslant v_n(r) > 0.$$
\end{lemm}

Bruno Colbois et Constantin Vernicos ont obtenu une in\'egalit\'e quantitative d\'ependant du rayon $r$ des boules. Si l'on veut simplement une in\'egalit\'e qualitative alors il s'agit d'une simple cons\'equence du th\'eorème de Benz\'ecri:

\begin{proof}
Soit $r >0$ une constante. On rappelle la d\'efinition de l'espace des convexes marqu\'es $X^{\bullet}$:
$$X^{\bullet} = \{  (\O,x) \,\mid\, \O \textrm{ est un ouvert proprement convexe de } \PP^n \textrm{ et } x \in \O \}$$

La fonction $f$ qui a un point $(\O,x)$ de $X^{\bullet}$ associe le volume de la boule de $(\O,d_{\O})$ de centre $x$ et de rayon $r$ est continue, strictement positive, et $\ss$-invariante. Or, le th\'eorème de Benz\'ecri \ref{theo_ben} montre que l'action de $\ss$ sur l'espace $X^{\bullet}$ est propre et cocompacte. La fonction $f$ est donc minor\'ee par une constante strictement positive.
\end{proof}

Nous pouvons maintenant donner une:

\begin{proof}[Preuve de (GF)$\Leftrightarrow$(VF)] $\Rightarrow$ La remarque \ref{decomposition_corps_cvx} et l'implication (GF)$\Rightarrow$(TF) montrent que le c\oe ur convexe de $\Quo$ se d\'ecompose en 
$$C(M) = (C(M))_{\varepsilon}^{nc} \bigsqcup \sqcup_{i=1}^k \Quotient{\left(H_{p_i}\cap \overline{C(\LG)}^{\O}\right)}{\P_{p_i}},$$
avec la partie non cuspidale $(C(M))_{\varepsilon}^{nc}$ compacte.\\
D'apr\`es le corollaire \ref{para_cas_gene}, il existe pour chaque point $p_i$, une coupe $\O_{p_i}$ (i.e l'intersection de $\O$ avec un sous-espace projectif) de $\O$ de dimension $d+1\geqslant 2$, contenant $p_i$ dans son bord, et deux ellipso\"ides tangents \`a $\dO_{p_i}$ en $p_i$ qui encadrent $\O$. En particulier, le bord $\dO_{p_i}$ est de classe $\C^{1,1}$ en $p_i$: le bord est de classe $\C^1$ et sa diff\'erentielle est Lipschitz. On peut donc appliquer la proposition \ref{volpic} de l'annexe \`a $\O$, qui montre que chaque partie $\Quotient{\left(H_{p_i}\cap \overline{C(\LG)}^{\O}\right)}{\P_{p_i}}$ est de volume fini.\\
Pour finir, la d\'ecomposition pr\'ec\'edente montre que le c\oe ur convexe se r\'etracte sur sa partie non cuspidale. Le quotient $\Quo$ est donc une orbifold sage; par cons\'equent, le groupe $\G$ est de type fini et même de pr\'esentation finie.\\

$\Leftarrow$ Comme le groupe $\G$ est de type fini, le lemme de Selberg affirme que, quitte à prendre un sous-groupe d'indice fini, on peut supposer que le groupe $\G$ est sans torsion.  Le lemme \ref{volfini_impli_compact} qui suit, appliqu\'e à la partie \'epaisse du c\oe ur convexe, implique que celle-ci est compacte, soit l'hypothèse (PEC) dont on a vu pr\'ec\'edemment qu'elle impliquait (GF).
\end{proof}

\begin{lemm}\label{volfini_impli_compact}
Soit $\G$ un sous-groupe discret et \underline{sans torsion} de $\Aut(\O)$. Si un ferm\'e $\F$ de la partie \'epaisse de $\Quo$ est de volume fini alors il est compact.
\end{lemm}

\begin{proof}
Par d\'efinition de la partie \'epaisse $\O^{\varepsilon}$ (et car le groupe $\G$ est sans torsion), si un point $x$ de $\O$ est dans $\O^{\varepsilon}$ alors la boule $B(x,\varepsilon)$ s'injecte par projection dans $\Quo$. Le lemme \ref{Min_vol_boul} montre que la boule $B(x,\varepsilon)$ a un volume minor\'e par une constante strictement positive ind\'ependante de $x$. Par cons\'equent, on ne peut pas trouver plus de $\textrm{Vol}(\F)/v_n(\varepsilon)$ boules disjointes incluses dans $\F$. Soient $B(x_1,\varepsilon) ,..., B(x_k,\varepsilon)$ un ensemble maximal de boules disjointes incluses dans $\F$. Par maximalit\'e, la r\'eunion finie des boules $B(x_1, 2\varepsilon) ,..., B(x_k, 2\varepsilon)$ recouvre $\F$. L'ensemble $\F$ est donc compact.
\end{proof}

\subsection{Cas particuliers}

La notion de finitude g\'eom\'etrique regroupe, comme on va le voir, des situations un peu diff\'erentes, selon que le quotient est de volume fini ou infini, selon que le c\oe ur convexe est compact ou pas.

\subsubsection*{Cas convexe-cocompact}

Lorsque le c\oe ur convexe du quotient $M=\Quo$ de $\O$ par le sous-groupe discret $\G$ de $\Aut(\O)$ est compact, on dit que l'action de $\G$ sur $\O$ est \emph{convexe-cocompacte} ou que le quotient $M$ lui-même est \emph{convexe-cocompact}. Le corollaire suivant affirme que ces groupes sont exactement ceux dont l'action est g\'eom\'etriquement finie sur $\O$ et qui ne contiennent pas de paraboliques.

\begin{coro}\label{theo_convcocomp}
Soit $\G$ un sous-groupe discret de $\Aut(\O)$. L'action de $\G$ sur $\O$ est convexe-cocompacte si et seulement si tout point de l'ensemble limite $\Lambda_{\G}$ est un point limite conique.
\end{coro}

\begin{proof}
Si l'action de $\G$ sur $\O$ est convexe-cocompacte alors tout point de l'ensemble limite est un point limite conique (remarque \ref{rem_lim_coni}).\\
Inversement, si tout point de l'ensemble limite est un point limite conique, alors $\G$ agit par d\'efinition de façon g\'eom\'etriquement finie sur $\O$. Mais dans ce cas, la partie non cuspidale du c\oe ur convexe de $M$ est le c\oe ur convexe de $M$ tout entier. Le th\'eorème \ref{theo_geo_finie} montre qu'alors le c\oe ur convexe de $M$ est compact.
\end{proof}

\subsubsection*{Action de covolume fini}

Nous obtenons ici la caract\'erisation suivante des actions de covolume fini.

\begin{coro}\label{theo_vol_fini}
Soit $\G$ un sous-groupe discret de type fini de $\Aut(\O)$. L'action de $\G$ sur $\O$ est de covolume fini si et seulement si l'action de $\G$ sur $\dO$ est g\'eom\'etriquement finie et $\Lambda_{\G} = \partial \O$.
\end{coro}

\begin{proof}
Si $\Lambda_{\G} = \partial \O$, alors $C(\LG)=\O$ et le c\oe ur convexe de $\Quo$ est $\Quo$ tout entier. Si l'action de $\G$ sur $\dO$ est g\'eom\'etriquement finie, comme $\Lambda_{\G} = \partial \O$, elle est en fait g\'eom\'etriquement finie sur $\O$. Le th\'eorème \ref{theo_geo_finie} montre alors que $\Quo$ est de volume fini.\\
Comme le groupe $\G$ est de type fini, le lemme de Selberg montre qu'on peut supposer que le groupe $\G$ est sans torsion. Par cons\'equent, le lemme \ref{volfini_impli_compact} montre que la partie \'epaisse de $\Quo$ est compacte. Par cons\'equent, tout point de $\partial \O$ est un point limite conique ou un point parabolique et tout point parabolique est born\'e et de rang maximal. C'est ce qu'il fallait d\'emontrer.
\end{proof}

\begin{coro}\label{dualite_vol_fini}
Soit $\G$ un sous-groupe discret de type fini de $\Aut(\O)$. L'action de $\G$ sur $\O$ est de covolume fini si et seulement si l'action de $\G^*$ sur $\O^*$ est de covolume fini.
\end{coro}

\begin{proof}
Le corollaire \ref{theo_vol_fini} montre que si l'action de $\G$ sur $\O$ est de covolume fini alors l'action de $\G$ sur $\dO$ est g\'eom\'etriquement finie et $\LG=\dO$. La proposition \ref{dual_geo_fini} montre qu'alors l'action de $\G$ sur $\dO^*$ est g\'eom\'etriquement finie et $\Lambda_{\G^*}=\dO^*$. Le corollaire \ref{theo_vol_fini} montre enfin que l'action de $\G^*$ sur $\O^*$ est de covolume fini.
\end{proof}


\section{Hyperbolicit\'e au sens de Gromov}

\subsection{Gromov-hyperbolicit\'e de $(C(\LG),\d)$}

Le but de cette partie est de montrer le r\'esultat suivant.

\begin{theo}\label{eqgromovhyp}
Soient $\G$ un sous-groupe discret de $\Aut(\O)$. L'action de $\G$ sur $\O$ est g\'eom\'etriquement finie sur $\O$ si et seulement si elle est g\'eom\'etriquement finie sur $\dO$ et l'espace $(C(\LG),\d)$ est Gromov-hyperbolique.
\end{theo}

Ce th\'eorème sera cons\'equence des deux lemmes qui suivent:

\begin{lemm}\label{gromov1}
Soit $\G$ un sous-groupe discret de $\Aut(\O)$. Si l'espace m\'etrique $(C(\LG),d_{\O})$ est Gromov-hyperbolique, alors tout point parabolique born\'e est uniform\'ement born\'e.
\end{lemm}
\begin{proof}
Supposons l'espace m\'etrique $(C(\LG),d_{\O})$ Gromov-hyperbolique et choisissons un point parabolique born\'e $p\in \LG$.\\
Fixons une horosphère $\mathcal{H}$ bas\'ee en $p$ et notons $(p\LG)=\{y \in (xp) \ | \ x\in\LG\smallsetminus\{p\}\}$ (voir figure \ref{coquillage}). Comme le point $p$ est un point parabolique born\'e, le groupe $\Stab_{\G}(p)$ agit de façon cocompacte sur $\H\cap (p\LG)$.

\begin{center}
\begin{figure}[h!]
  \centering
\includegraphics[width=8cm]{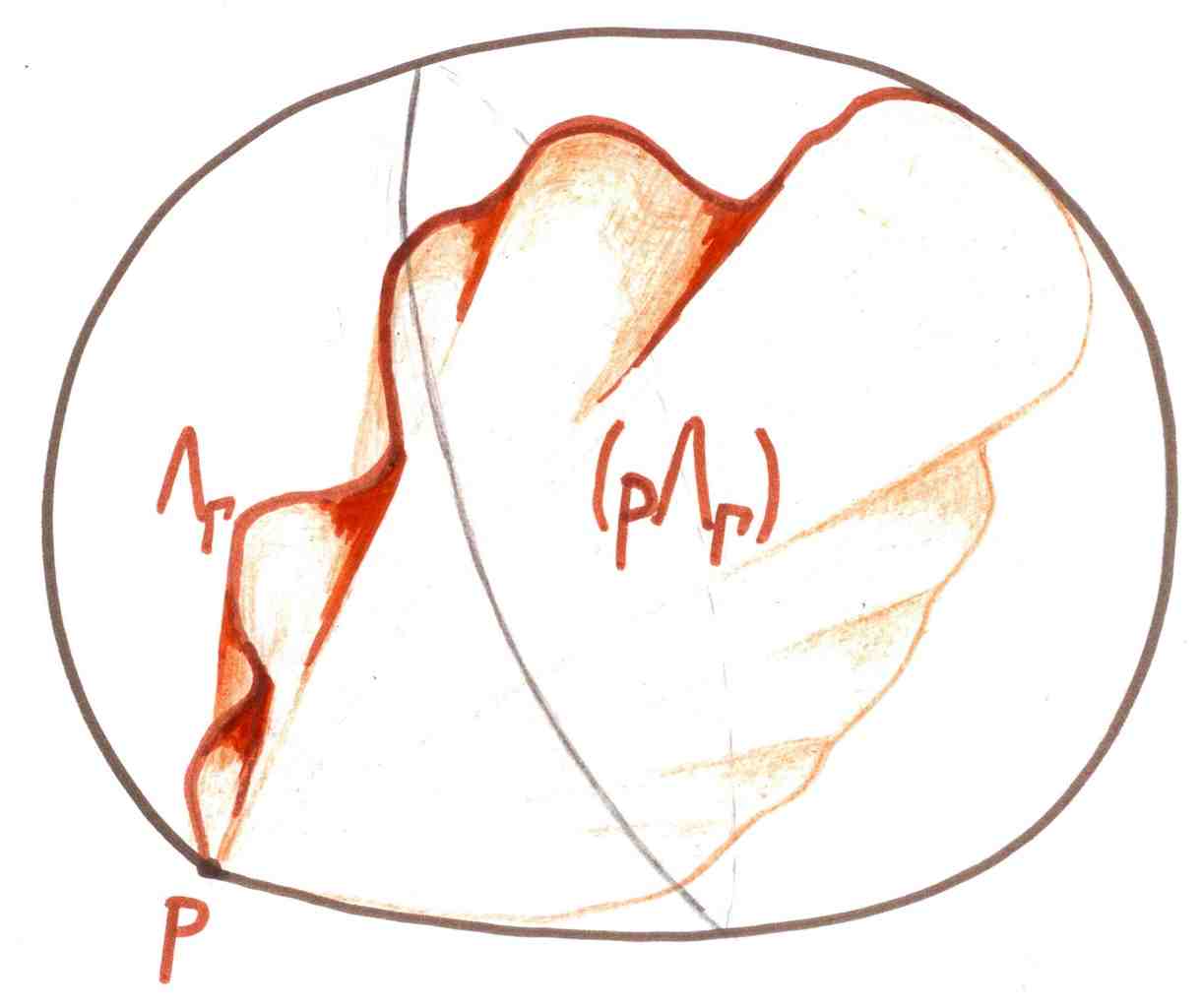}
\caption{L'ensemble $(p\LG)$ \label{coquillage}}
\end{figure}
\end{center}	

On peut identifier l'espace des droites $\mathcal{D}_p(C(\LG))$ à sa trace sur l'horosphère $\mathcal{H}$. On va voir que $\mathcal{H}\cap C(\LG)$ est dans un voisinage born\'e de $\H\cap (p\LG)$, ce qui permettra de conclure que le groupe $\Stab_{\G}(p)$ agit de façon cocompacte sur $\overline{\mathcal{H}\cap C(\LG)}$ et donc aussi sur $\overline{\mathcal{D}_p(C(\LG))}$ (l'adh\'erence est prise respectivement dans $\O$ et dans $\A_p^{n-1}$).\\

L'ensemble $\LG$ est l'ensemble des points extr\'emaux de $C(\LG)$. Ainsi, tout point $x$ de $C(\LG)$ est barycentre d'au plus $n+1$ points de $\LG$. Consid\'erons d'abord l'ensemble $C_2(\LG)$ des points $x\in C(\LG)$ qui sont sur une droite $(ab)$ avec $a,b\in \LG$ (on s'aidera de la figure \ref{grogro1}). Comme l'espace $(C(\LG),d_{\O})$ est Gromov-hyperbolique, le point $x$ est dans un voisinage de taille au plus $\delta$ (pour $d_{\O}$) de $(pa)\cup(pb)$, pour un certain $\delta>0$, ind\'ependant de $x$. Autrement dit, pour tout $x\in\mathcal{H}\cap C_2(\LG)$, il existe un point $y \in (p\LG)$ tel que $\d(x,y)<\delta$. Maintenant, le point $z=(py)\cap\mathcal{H} \in (p\LG)\cap \mathcal{H}$ est le point de $\mathcal{H}$ le plus proche de $y$; en particulier, $d_{\O}(y,z) \leqslant d_{\O}(y,x) < \delta$. L'in\'egalit\'e triangulaire donne que $d_{\O}(x,z)< 2\delta$. On obtient donc que $C_2(\LG)\cap \mathcal{H}$ est dans un voisinage de taille $2\delta$ de  $(p\LG)\cap \mathcal{H}$. On proc\`ede par r\'ecurrence 
pour avoir le r\'esultat pour $C(\LG)$.

\begin{center}
\begin{figure}[h!]
  \centering
\includegraphics[width=7cm]{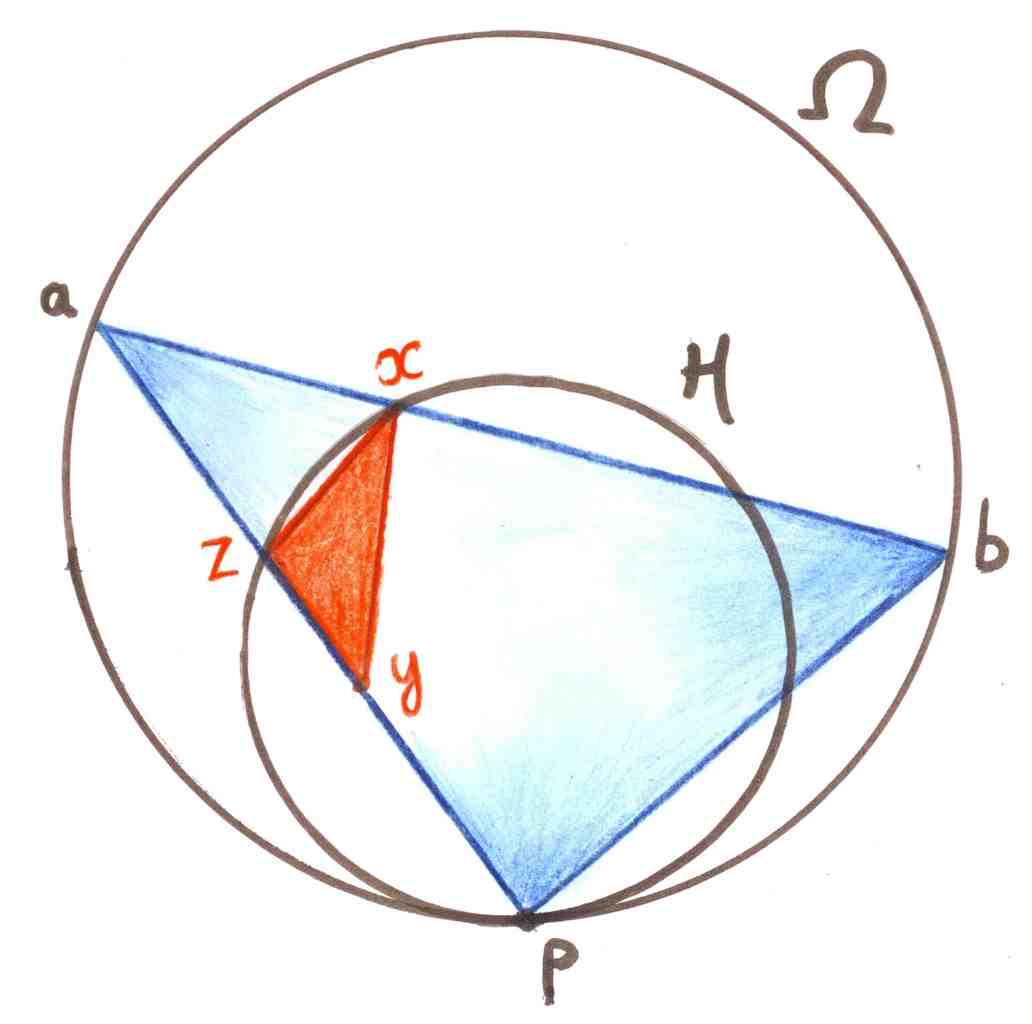}
\caption{Preuve du lemme \ref{gromov1} \label{grogro1}}
\end{figure}
\end{center}

\end{proof}

\begin{lemm}\label{gromov2}
Soit $\G$ un sous-groupe discret de $\Aut(\O)$. Si l'action de $\G$ sur $\O$ est g\'eom\'etriquement finie sur $\O$ alors l'espace m\'etrique $(C(\LG),d_{\O})$ est Gromov-hyperbolique.
\end{lemm}

\begin{rema}
La d\'emonstration qui suit est une am\'elioration de la d\'emonstration du lemme 7.10 de l'article \cite{Marquis:2010fk} qui est elle-même une am\'elioration de la d\'emonstration de la proposition 2.5 de l'article \cite{MR2094116}. Elle est ind\'ependante des deux pr\'ec\'edentes mais leur lecture pr\'ealable peut aider.
\end{rema}

\begin{proof}
On va proc\'eder par l'absurde en supposant qu'il existe une suite de triangles $(x_ny_nz_n)$ de $C(\LG)$ dont la taille $\delta_n=\sup\{\d(u_n,[xz]),\d(u_n,[y_n z_n])\}$ tend vers l'infini, $u_n$ \'etant un point du segment $[x_n y_n]$.\\
Quitte à extraire, on peut supposer que toutes les suites convergent dans $\overline{C(\LG)}$ (l'adh\'erence est prise dans $\PP^n$), et on note $x,y,z,u$ les limites correspondantes.\\
On va distinguer deux cas.

\begin{itemize}
 \item Supposons que $u$ est un point de $\O$. Dans ce cas, il faut au moins, pour que $\delta_n$ puisse tendre vers l'infini, que les points $x,y,z$ soient à l'infini, autrement dit dans $\LG$, et qu'ils soient deux à deux distincts. Or, l'ouvert $\O$ \'etant strictement convexe, la distance de $u$ à la droite $(xz)$ est finie, d'où une contradiction.
 \item Supposons maintenant que $u$ est un point de $\dO$. En utilisant l'action de $\G$, on aurait pu, avant extraction, faire en sorte que la suite $(u_n)$ reste dans un domaine fondamental convexe localement fini $D\subset \overline{C(\LG)}$. Le point limite $u\in\dO$ est alors dans l'adh\'erence du domaine $D$ dans $\PP^n$ et dans $\LG$; c'est donc un point parabolique uniform\'ement born\'e de $\LG$, d'apr\`es le lemme \ref{lem_pbord}.\\
Prenons alors deux ellipsoïdes $\E^{int}$ et $\E^{ext}$ comme dans le corollaire \ref{para_cas_gene}, et notons $C=\Cone(u,C(\LG))$. On a
$$C\cap\E^{int} \subset \O\cap C \subset C\cap\E^{ext}.$$
On considère maintenant une isom\'etrie hyperbolique $\g$ de l'espace hyperbolique $(\E^{ext},d_{\E^{ext}})$, dont le point r\'epulsif est $u$ et le point attractif un point $v\in C\cap\partial\E^{ext}$ quelconque. On fixe un hyperplan $H$ s\'eparant les points $u$ et $v$, et on note $H'=\g(H)$, en s'arrangeant pour que $H$ et $H'$ intersectent ent $C(\LG)$. On note $A= \E^{ext}\cap C(H,H')$, où $C(H,H')$ repr\'esente l'ensemble d\'elimit\'e par $H$ et $H'$ et ne contenant ni $u$ ni $v$.\\
Pour chaque \'el\'ement $u_n$, il existe $k_n\in\Z$ tel que $\g^{k_n}(u_n)\in A$. On pose $u_n'= \g^{k_n}(u_n)$, et on fait de même pour $x_n',y_n',z_n'$. Il revient alors au même, par isom\'etrie, de regarder la suite de triangles $(x'_ny_n'z_n')$ et de points $(u_n')$ dans la g\'eom\'etrie de Hilbert d\'efinie par $\O_n=\g^{k_n}(\O)$. On va même remplacer le convexe $\O_n$ par $\O_n'=\O_n \cap \E^{ext}$, la taille du triangle $(x'_ny_n'z_n')$ \'etant plus grande dans $\O_n'$ que dans $\O_n$ (voir figure \ref{grogro2}).

\begin{center}
\begin{figure}[h!]
  \centering
\includegraphics[width=12cm]{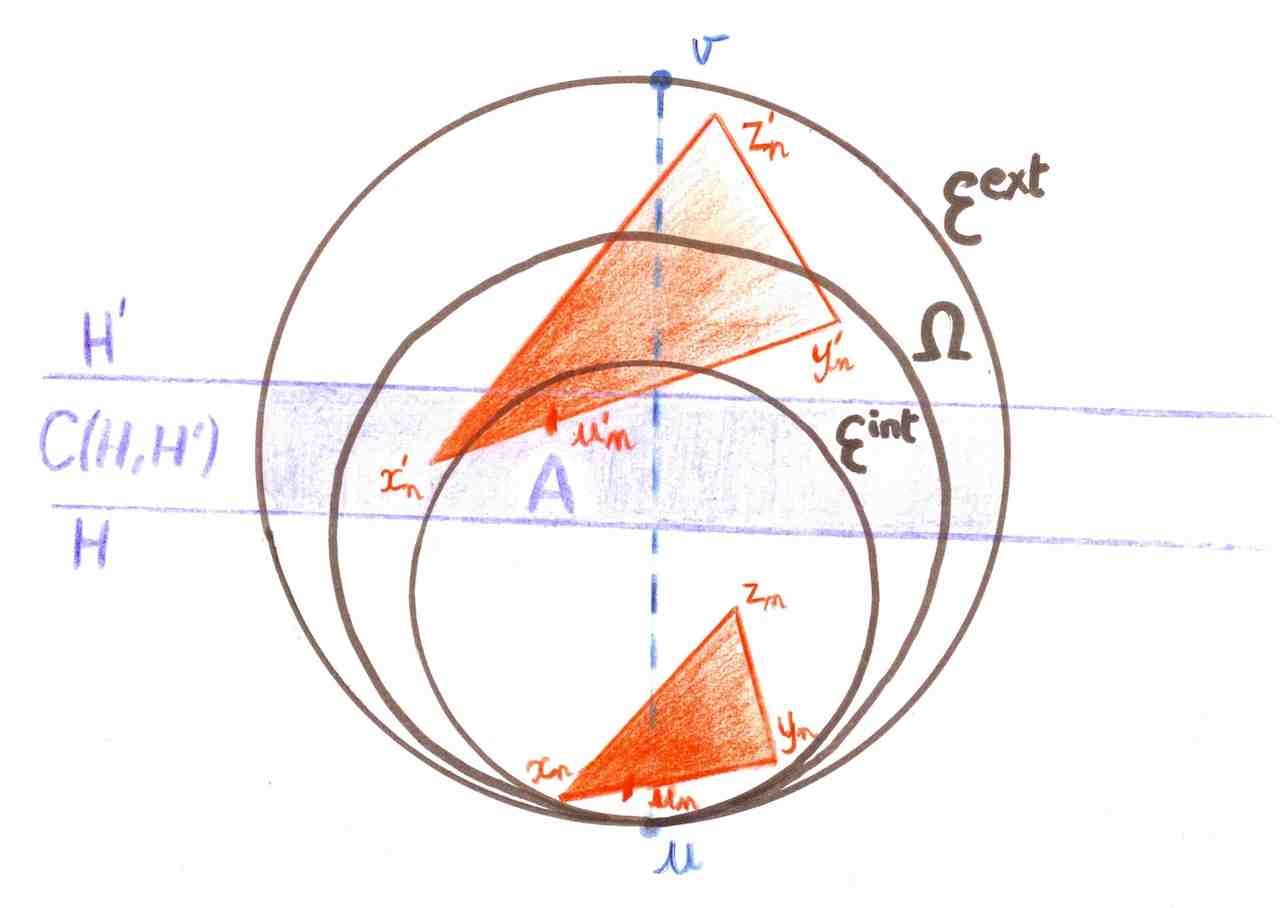}
\caption{Preuve du lemme \ref{gromov2} \label{grogro2}}
\end{figure}
\end{center}

Quitte à extraire à nouveau, on peut supposer que toutes ces suites convergent, et on note $x',y',z',u'$ leurs limites. Il n'est pas dur de voir que $u' \in [uv]\cap A$: en effet, le point $u_n'$ est dans $D_n=\g^{k_n}(D)$ et $D_n$ tend vers $[uv]$ car $u$ est un point parabolique uniform\'ement born\'e. Comme $\O_n'$ est coinc\'e entre $\g^{k_n}(\E^{int})$ et $\E^{ext}$, la suite de convexes $(\O_n')$ tend, tout comme $(\g^{k_n}(\E^{int}))$, vers $\E^{ext}$. Les points $x',y',z'$ sont quant à eux des point de $\partial\E^{ext}$. Autrement dit, on obtient à la limite un triangle $x'y'z'$ d'un espace hyperbolique, dont la taille est n\'ecessairement born\'ee. D'où une contradiction avec l'hypoth\`ese $\delta_n\to +\infty$.
\end{itemize}

\end{proof}

Comme corollaire de la proposition \ref{eqgromovhyp}, on peut \'enoncer le r\'esultat suivant dans le cas d'une g\'eom\'etrie de Hilbert Gromov-hyperbolique.

\begin{coro}
Pour une g\'eom\'etrie de Hilbert \emph{Gromov-hyperbolique}, les notions de finitude g\'eom\'etrique sur $\O$ et sur $\dO$ sont \'equivalentes.
\end{coro}

Notons enfin un autre corollaire dans le cas d'une action de covolume fini.

\begin{coro}\label{ghyp_vol_fini}
Si l'ouvert convexe $\O$ admet un quotient de volume fini, alors l'espace m\'etrique $(\O,d_{\O})$ est Gromov-hyperbolique.
\end{coro}

\subsection{Gromov-hyperbolicit\'e du groupe $\G$}

Rappelons qu'un groupe de type fini est Gromov-hyperbolique si son graphe de Cayley, muni de la m\'etrique des mots, l'est. De façon plus g\'en\'erale, nous prendrons la d\'efinition suivante de groupe relativement hyperbolique:

\begin{defi}
Soient $\G$ un groupe et $(\P_i)_i$ une famille de sous-groupes de type fini de $\G$. On dit que le groupe $\G$ est \emph{relativement hyperbolique relativement aux groupes $(\P_i)_i$} lorsqu'il existe un espace Gromov-hyperbolique propre $X$ et une action g\'eom\'etriquement finie de $\G$ sur $X$ (au sens de la d\'efinition \ref{geofini_ghyp}) telle que le stabilisateur de tout point parabolique de $\LG$ est conjugu\'e à l'un des groupes $(\P_i)_i$.
\end{defi}

Les r\'esultats de la partie pr\'ec\'edente permettent donc d'affirmer le fait suivant.

\begin{prop}
Si $\G$ est un sous-groupe discret de $\Aut(\O)$ agissant de façon g\'eom\'etriquement finie sur $\O$, alors le groupe $\G$ est relativement hyperbolique relativement à ses sous-groupes paraboliques maximaux.
\end{prop}

En fait, on peut changer l'hypothèse d'action g\'eom\'etriquement finie sur $\O$ en action g\'eom\'etriquement finie sur $\dO$ via le travail d'Asli Yaman. Elle a montr\'e le th\'eorème suivant qui donne une caract\'erisation topologique des groupes relativement hyperboliques (dans \cite{MR2039323}).

\begin{theo}[Yaman \cite{MR2039323}]\label{yaman}
Soient $M$ un compact parfait non vide et m\'etrisable , et $\G$ un groupe. Supposons que le groupe $\G$ agit par une action de convergence sur $M$ tel que tout point de $M$ est un point limite conique ou un point parabolique born\'e, que l'ensemble des points paraboliques modulo l'action de $\G$ est fini et que les stabilisateurs des points paraboliques sont de type fini. Alors le groupe $\G$ est relativement hyperbolique relativement aux stabilisateurs de ses points paraboliques.
\end{theo}

On obtient alors le r\'esultat \emph{a priori} plus satisfaisant:

\begin{prop}
Soit $\G$ un sous-groupe discret de $\Aut(\O)$ agissant de façon g\'eom\'etriquement finie sur $\dO$. Alors le groupe $\G$ est relativement hyperbolique relativement à ses sous-groupes paraboliques maximaux.
\end{prop}

\begin{proof}
On prend pour compact $M$ l'ensemble limite $\LG$. Le th\'eorème \ref{action_conv} montre que l'action de $\G$ sur $\LG$ est une action de convergence (d\'efinition \ref{def_action_conv}).\\
On note $P$ un ensemble de repr\'esentants des points paraboliques modulo $\G$. L'ensemble $P$ est fini. En effet, l'action du stabilisateur de tout point parabolique $p$ sur $\LG \smallsetminus \{ p \}$ est cocompacte, on peut donc choisir l'ensemble $P$ de telle fa\c con que $p$ soit isol\'e dans $P$. On peut donc faire en sorte que tous les points de $P$ soient isol\'es, auquel cas $P$ de est un sous-ensemble discret du compact $\LG$: $P$ est donc fini.\\
Il nous reste à v\'erifier que les stabilisateurs des points paraboliques sont de type fini. Or, tout sous-groupe discret d'un groupe de Lie nilpotent connexe est de type fini. Ainsi, les sous-groupes paraboliques de $\G$ sont de type fini: c'est le corollaire 2 de la partie 2.10 du livre \cite{MR0507234} de Raghunathan. 
\end{proof}


\section{Petites dimensions}

\subsection{La dimension 2}\label{main2}

En dimension 2, la situation est beaucoup plus simple qu'en dimension sup\'erieure. La proposition suivante a presque \'et\'e montr\'e par l'un des auteurs dans \cite{Marquis:2009kq}.

\begin{theo}
Soient $\O \subset \PP^2$ et $\G$ un sous-groupe discret de $\Aut(\O)$. Les propositions suivantes sont \'equivalentes:
\begin{enumerate}[(1)]
\item le c\oe ur convexe est de volume fini;
\item l'action de $\G$ sur $\dO$ est g\'eom\'etriquement finie;
\item l'action de $\G$ sur $\O$ est g\'eom\'etriquement finie;
\item le groupe $\G$ est de type fini.
\end{enumerate} 
\end{theo}

\begin{proof}[\'El\'ements de d\'emonstration]
L'implication (2)$\Rightarrow$(3) est \'evidente puisqu'ici le bord de $\O$ est de dimension 1. L'implication (3)$\Rightarrow$(4) est une partie du th\'eorème \ref{theo_geo_finie}.

L'implication (4)$\Rightarrow$(1), a d\'ejà \'et\'e montr\'ee dans \cite{Marquis:2009kq} (proposition 6.16) et on ne reproduirera pas la d\'emonstration de ce r\'esultat ici. Remarquons simplement que la classification des surfaces (compactes ou non) montre que le groupe fondamental d'une surface est de type fini si et seulement si cette dernière est hom\'eomorphe à une surface compacte à laquelle on a enlev\'e un nombre fini de points. Le reste de la d\'emonstration est une \'etude attentive de la g\'eom\'etrie des bouts d'un tel quotient dans le cadre de la g\'eom\'etrie de Hilbert.

Pour montrer l'implication (1)$\Rightarrow$(2), on montre plutôt l'implication (1)$\Rightarrow$(4). En effet, si un groupe $\G$ v\'erifie (1) et (4) alors il v\'erifie l'hypothèse (VF); par cons\'equent, le th\'eorème \ref{theo_geo_finie} montre que $\G$ v\'erifie (3), qui entra\^ine (2).\\
C'est le th\'eorème 5.22 de \cite{Marquis:2009kq} qui montre que si le groupe $\G$ v\'erifie (1), alors il est de type fini. Il serait un peu long de reproduire ici la d\'emonstration de ce r\'esultat. Mais l'id\'ee principale est que pour une g\'eom\'etrie de Hilbert de dimension 2, il existe une borne uniforme strictement positive minorant l'aire des triangles id\'eaux (ce r\'esultat est d\^u à Constantin Vernicos, Patrick Verovic et Bruno Colbois dans \cite{MR2116715}); c'est un analogue du fait que tout triangle id\'eal du plan hyperbolique a une aire \'egale à $\pi$.
\end{proof}

\subsection{La dimension 3}\label{main3}

Le r\'esultat principal en dimension $3$ est le suivant, qui peut être prouv\'e ''à la main``.

\begin{prop}
Soient $\O\subset\PP^3$ et $\P$ un sous-groupe parabolique de $\Aut(\O)$ fixant le point $p\in\dO$. Alors le groupe $\P$ pr\'eserve un ellipsoïde $\E$ tangent à $\O$ en $p$. $\P$ est donc conjugu\'e \`a un sous-groupe de $\so{3}$; en particulier, $\P$ est virtuellement isomorphe à $\Z$ ou $\Z^2$.
\end{prop}

\begin{proof}
Soit $\g$ un \'el\'ement parabolique de $\P$, qu'on voit comme matrice de $\s4$. La d\'ecomposition de Jordan de $\g$ permet d'\'ecrire $\g$ comme le produit d'une matrice unipotente $\g_u$ et d'une matrice elliptique. Le paragraphe \ref{laius} montre que la seule possibilit\'e pour $\g_u$ est la matrice suivante:
$$
\begin{tabular}{ll}
$
\underbrace{
\left(
\begin{array}{ccc}
1 & 1 & 0\\
   & 1  & 0\\
   &     & 1\\
   &     &\\
\end{array}
\right.
}_H
$
&
$
\left.
\begin{array}{l}
 0\\
1\\
0\\
1\\
\end{array}
\right),
$
\\
$
\begin{array}{ccc}
\underbrace{\,}_p &  & \\
\end{array}
$
&
\\
\end{tabular}
$$

\noindent o\`u $H$ l'hyperplan tangent à $\O$ en $p$. Par cons\'equent, l'action de $\g_u$ (resp. $\g$) sur l'espace $\A^{2}_p$ est une action par translation (resp. vissage). La partie lin\'eaire de l'action de $\G$ sur $\A^2_p$ est incluse dans un groupe compact. Il vient que le groupe $\G$ pr\'eserve un produit scalaire sur $\A^2_p$.


On en d\'eduit que $\P$ est inclus dans un conjugu\'e de $\so{3}$, c'est-à-dire que $\P$ pr\'eserve un ellipsoïde, qui est n\'ecessairement tangent à $\O$ en $p$.
\end{proof}

Cela permet d'obtenir le
\begin{coro}\label{dim3fini}
En dimension $3$, les notions de finitude g\'eom\'etrique sur $\O$ et sur $\dO$ sont \'equivalentes.
\end{coro}
\begin{proof}
Il s'agit de montrer que, \'etant donn\'e une action g\'eom\'etriquement finie d'un groupe $\G$ sur $\dO$, tout point parabolique born\'e $p\in\LG$ est en fait uniform\'ement born\'e. Or, on vient de voir que les sous-groupes paraboliques sont, en dimension $3$, conjugu\'es dans $\so{3}$. Le corollaire \ref{soborne} permet de conclure.
\end{proof}

\subsection{Un contre-exemple}\label{cex}

Pour trouver un exemple d'une action g\'eom\'etriquement finie sur $\dO$ mais pas g\'eom\'etriquement finie sur $\O$, il faudra, d'après les deux parties pr\'ec\'edentes, chercher en dimension sup\'erieure ou \'egale à $4$. On a vu aussi que, dès que la g\'eom\'etrie de Hilbert \'etait Gromov-hyperbolique, les deux notions \'etaient \'equivalentes. Enfin, on a vu dans le corollaire \ref{soborne} que, si le stabilisateur d'un point parabolique born\'e \'etait conjugu\'e dans $\SO$, alors ce point \'etait en fait uniform\'ement born\'e.\\
On peut r\'esumer tout cela dans l'\'enonc\'e suivant:

\begin{prop}\label{bilan_geo}
Soit $\G$ un sous-groupe discret de $\Aut(\O)$. Les propositions suivantes sont \'equivalentes:
\begin{enumerate}[(i)]
 \item l'action de $\G$ sur $\O$ est g\'eom\'etriquement finie;
 \item l'action de $\G$ sur $\dO$ est g\'eom\'etriquement finie et les sous-groupes paraboliques de $\G$ sont conjugu\'es à des sous-groupes paraboliques de $\SO$;
 \item l'action de $\G$ sur $\dO$ est g\'eom\'etriquement finie et l'espace m\'etrique $(C(\LG),\d)$ est Gromov-hyperbolique.
\end{enumerate}
\end{prop}

On notera au passage le corollaire suivant:

\begin{coro}\label{fin_dual}
L'action de $\G$ sur $\O$ est g\'eom\'etriquement finie si et seulement si l'action de $\G^*$ sur $\O^*$ est g\'eom\'etriquement finie.
\end{coro}

\begin{proof}
On sait d\'ejà que l'action de $\G$ sur $\dO$ est g\'eom\'etriquement finie si et seulement si l'action de $\G^*$ sur $\dO^*$ est g\'eom\'etriquement finie (proposition \ref{dual_geo_fini}). Or, le dual d'un sous-groupe parabolique de $\SO$ est un sous-groupe parabolique de $\SO$ puisque $\SO$ est autodual.
\end{proof}

On en vient à pr\'esent aux contre-exemples annonc\'es dans l'introduction:

\begin{prop}\label{contrex_geo_finie}
Il existe un ouvert proprement convexe $\O$ de $\PP^4$, strictement convexe et à bord $\C^1$, qui admet une action d'un sous-groupe discret d'automorphismes $\G$ dont l'action est g\'eom\'etriquement finie sur $\dO$ mais pas g\'eom\'etriquement finie sur $\O$.
\end{prop}

\begin{prop}\label{contrex_conv_cocomp}
Il existe un ouvert proprement convexe $\O$ de $\PP^4$, strictement convexe et à bord $\C^1$, et un sous-groupe discret $\G$ de $\Aut(\O)$ dont l'action est convexe-cocompacte et l'adh\'erence de Zariski n'est ni $\s{5}$ ni conjugu\'ee à $\so{4}$. 
\end{prop}

\paragraph*{Construction du contre-exemple via les repr\'esentations sph\'eriques de $\s{2}$}
$\,$
\par{
L'action de $\s{2}$ sur $\R^2$ induit une action $\rho_d$ de $\s{2}$ sur l'espace vectoriel $V_d$ des polynômes homogènes de degr\'e $d$ en deux variables, qui est de dimension $d+1$. De plus, toute repr\'esentation irr\'eductible de dimension finie de $\s{2}$ est \'equivalente à l'une des repr\'esentations $\rho_d:\s{2}\rightarrow \textrm{GL}(V_d)$ pour un $d\geqslant 1$.
}
\\
\par{
Il est facile de voir que $\rho_d$ pr\'eserve un ouvert proprement convexe de $\PP(V_d)$ si et seulement si $d$ est pair. En effet, si $d$ est impair alors $\rho_d(-Id_2)=-Id_{V_d}$: par cons\'equent, $\rho_d$ ne peut pr\'eserver d'ouvert proprement convexe.  Notons $\C_{min}$ l'ensemble des polynômes convexes de $V_d$ et $\C_{max}$ l'ensemble des polynômes positifs de $V_d$. Ce sont deux cônes proprement convexes de $V_d$. Ils sont non vides si et seulement si $d$ est pair et $\C_{min}$ est inclus dans $\C_{max}$. En fait, $\C_{max}$ est le c\^one dual de $\C_{min}$. Enfin, tous deux sont pr\'eserv\'es par $\rho_d$. En fait, on peut même montrer que tout cône convexe proprement convexe de $V_d$ pr\'eserv\'e par $\rho_d$ contient $\C_{min}$ et est contenu dans $\C_{max}$. Vinberg \'etudie le cas d'un groupe semi-simple quelconque dans \cite{MR565090}, on pourra aussi trouver un \'enonc\'e dans l'article \cite{MR1767272}, proposition 4.7. 
}
\\
\par{
On notera $\O_0=\PP(\C_{min})$ et $\O_{\infty}= \PP(\C_{max})$. Il n'est pas difficile de voir que $\O_0 \neq \O_{\infty}$ si et seulement si $d \geqslant 4$. Par cons\'equent, on peut introduire l'ouvert $\O_r= \{ x \in \O_{\infty} \,|\, d_{\O_{\infty}}(x,\O_0) < r\}$, c'est-à-dire le $r$-voisinage de $\O_0$ dans $(\O_{\infty},d_{\O_{\infty}})$. La proposition suivante montre que les $\O_r$ sont convexes. 
 }

\begin{lemm}[Corollaire 1.10 de \cite{Cooper:2011fk}]
Le $r$-voisinage (pour $d_{\O}$) d'une partie convexe d'un ouvert proprement convexe est convexe.
\end{lemm}

Nous allons montrer la proposition suivante:

\begin{prop}\label{contrex}
Si $d=4$ et $r\neq 0,\infty$ alors les ouverts proprement convexes $\O_r$ sont strictement convexes et à bord $\C^1$.
\end{prop}

\begin{proof}[D\'emonstration de la proposition \ref{contrex_geo_finie} et de la proposition \ref{contrex_conv_cocomp}]
Choisissons un r\'eel $r>0$ et posons $\O=\O_r$.\\
Pour la proposition \ref{contrex_geo_finie}, il suffit de prendre un r\'eseau $\G$ non cocompact de $\s{2}$ et de remarquer que tout \'el\'ement parabolique de $\rho_4(\G)$ est conjugu\'e à un bloc de Jordan de taille 5. L'action de $\rho_4(\G)$ est bien sûr g\'eom\'etriquement finie sur $\dO$ mais la proposition \ref{bilan_geo} (ii) montre qu'elle n'est pas g\'eom\'etriquement finie sur $\O$. On pourra même remarquer que l'enveloppe convexe de $\LG\smallsetminus \{p\}$ dans $\A^3_p$ est $\A^3_p$ tout entier (où $p$ est n'importe quel point de $\LG$; on rappelle que $\LG$ est un cercle d'un point de vue topologique).

Pour la proposition \ref{contrex_conv_cocomp}, il suffit de prendre un r\'eseau $\G'$ cocompact de $\s{2}$ ou un sous-groupe discret $\G'$ convexe-cocompact de $\s{2}$. Dans tous les cas, l'action de $\rho_4(\G')$ sera convexe-compacte mais l'adh\'erence de Zariski de $\G'$ dans $\s{5}$ est $\rho_4(\s{2})$ qui est incluse dans $\mathrm{SO}_{2,3}(\R)$. On rappelle que la repr\'esentation irr\'eductible de $\s{2}$ de dimension $2n$ est incluse dans le groupe symplectique alors que celle de dimension $2n+1$ est incluse dans $\textrm{SO}_{n,n+1}(\R)$.
\end{proof}

Nous allons avoir besoin de plusieurs lemmes pour d\'emontrer la proposition \ref{contrex}.

\begin{lemm}\label{elli}
On suppose $d$ pair. Si $\g$ est un \'el\'ement elliptique non trivial de $\s{2}$ (c'est-\`a-dire qui n'est pas dans le centre) alors $\rho_d(\g)$ possède un unique point fixe sur $\PP(V_d)$. En particulier, tout point fixe d'un \'el\'ement elliptique appartient à $\O_0$.
\end{lemm}

\begin{proof}
Si $\g$ est un \'el\'ement elliptique de $\s{2}$ non trivial alors les valeurs propres de $\g$ s'\'ecrivent $e^{\pm i\theta}$ pour un certain $\theta \notin \pi\Z$. Il vient alors que les valeurs propres de $\rho_4(\g)$ sont les nombres : $ e^{di\theta}, \cdots, 1, \cdots, e^{-di\theta}$. Par cons\'equent, $\rho_d(\g)$ fixe un unique point de $\PP(V_d)$. Il nous reste à montrer que ce point est dans $\O_0$.

Le sous-groupe \`a 1-paramètre $K$ engendr\'e par $\g$ est un sous-groupe compact de $\s{2}$ de dimension 1. Le groupe $K$ pr\'eserve l'ouvert proprement convexe $\O_0$, il pr\'eserve donc aussi l'isobarycentre de toute orbite. Ainsi, l'unique point fixe de $\rho_d(\g)$ appartient à $\O_0$.  
\end{proof}

\begin{lemm}\label{libre}
On suppose $d$ pair. L'action de $\s{2}$ sur $\O_{\infty} \smallsetminus \O_0$ est propre et libre.
\end{lemm}

\begin{proof}
L'action de $\s{2}$ sur $\O_{\infty}$ est propre, donc l'ensemble des points fixes des \'el\'ements hyperboliques et paraboliques de $\s{2}$ est dans le compl\'ementaire de $\O_{\infty}$. Par suite, l'action de de $\s{2}$ sur $\O_{\infty} \smallsetminus \O_0$ est propre et libre via le lemme \ref{elli}.
\end{proof}

\begin{lemm}\label{dim4}
On suppose que $d=4$. Tout ouvert proprement convexe pr\'eserv\'e par $\rho_d$ est l'un des $\O_r$ pour $r\in \R_+\cup \{ \infty \}$. En particulier, le dual d'un $\O_r$ est un certain $\O_{r'}$ et il existe un unique $r_0$ tel que $\O_{r_0}$ soit autodual. Enfin, tous les  $\O_r$ sont strictement convexes et à bord $\C^1$, si $r\neq 0,\infty$.
\end{lemm}

\begin{proof}
Le lemme \ref{libre} montre que l'action de $\s{2}$ sur $\dO_r \smallsetminus \Lambda_{\s{2}} \subset \O_{\infty}\smallsetminus \O_0$ est libre. Or, si $d=1$ le groupe $\s{2}$ et la sous-vari\'et\'e $\dO_r$ ont la même dimension ($3$):  cela montre que les orbites de cette action sont ouvertes dans $\dO_r \smallsetminus \Lambda_{\s{2}}$. De plus, comme on a retir\'e l'ensemble limite de $\dO_r$, les orbites sont ferm\'ees. Enfin, comme la vari\'et\'e $\dO_r$ est une 3-sphère et l'ensemble limite est un cercle dont le plongement est donn\'e par la courbe Veronese, l'espace $\dO_r \smallsetminus \Lambda_{\s{2}}$ est donc connexe. Par suite, l'action de $\s{2}$ sur $\dO_r \smallsetminus \Lambda_{\s{2}} \subset \O_{\infty}$ est transitive.\\
Ceci montre que tout ouvert proprement convexe pr\'eserv\'e par $\rho_d$ est l'un des $\O_r$. Le dual d'un $\O_r$ est donc un $\O_{r'}$. L'existence d'un unique $\O_r$ autodual est simplement due au fait que la dualit\'e renverse les inclusions.\\
On se donne un $r\neq \infty$. Si $\O_r$ n'est pas strictement convexe, il existe un point de $\dO_r \smallsetminus \Lambda_{\s{2}}$ qui n'est pas un point extr\'emal. Or, l'action de $\s{2}$ sur $\dO_r \smallsetminus \Lambda_{\s{2}}$ est transitive: aucun point de $\dO_r \smallsetminus \Lambda_{\s{2}}$ n'est extr\'emal et donc $\O_r = \O_0$.\\
Enfin, si $r\not = 0,\infty$, le dual de $\O_r$ est un $\O_{r'}$ avec $r'\not = 0,\infty$. Comme $\O_{r'}$ est strictement convexe, le bord de $\O_r$ est de classe $\C^1$.
\end{proof}

\begin{rema}
On a vu que $\O_0$ n'\'etait pas strictement convexe. Une \'etude attentive de $\rho_d$ permet de voir que $\O_{\infty}$ n'est pas strictement convexe. En effet, tout \'el\'ement hyperbolique $\g$ de $\s{2}$ possède 5 valeurs propres r\'eelles distinctes pour son action sur $V_d$. On peut v\'erifier que la droite propre $p^0_{\g}$ associ\'ee à la troisi\`eme (si elles sont rang\'ees par ordre croissant) appartient au bord de $\O_{\infty}$ mais pas à l'ensemble limite. Enfin, on peut v\'erifier que le segment $[p^0_{\g} p^+_{\g}]$ est inclus dans le bord de $\O_{\infty}$, où $p^+_{\g}$ d\'esigne le point fixe attractif de $\g$. Par cons\'equent, $\O_0$ et $\O_{\infty}$ ne sont ni strictement convexes ni à bord $\C^1$.
\end{rema}

\begin{proof}[D\'emonstration de la proposition \ref{contrex}]
Elle est incluse dans le lemme \ref{dim4}.
\end{proof}

\appendix

\section{Sur le volume des pics, par les auteurs et Constantin Vernicos}\label{appendix_with_costia}

Le but de cette annexe est de prouver le r\'esultat suivant.

\begin{prop}\label{volpic}
Soient $\O$ un ouvert convexe de $\R^n$, $p$ un point du bord $\dO$ en lequel $\dO$ est de classe $\C^1$.
Supposons qu'il existe une coupe de $\O$ de dimension $\geqslant 2$, contenant $p$ en son bord, et dont le bord est $\C^{\alpha}$ en $p$, pour un certain $\alpha>1$.\\
Alors tout cône $C$ de sommet $p$ et de base $B\subset \O$ compacte est de volume fini.
\end{prop}

Soit $\O$ un ouvert convexe de $\R^n$ euclidien. Rappelons que le volume de Busemann $\Vol_{\O}$ de la g\'eom\'etrie de Hilbert $(\O,\d)$ est donn\'e par
$$d\Vol_{\O}(x) = \frac{v_n}{\Vol (B(T_x\O))} d\Vol,$$
où $\Vol$ est le volume de Lebesgue de $\R^n$ et $v_n$ le volume de Lebesgue de la boule unit\'e de $\R^n$.\\

Par exemple, un simple calcul montre le
\begin{lemm}\label{dim1}
\begin{enumerate}[(i)]
 \item Soit $\O=(-a,a)$, avec $a>1$. Le volume de Busemann est donn\'e au point $x\in\O$ par
$$d\Vol_{\O}(x) = \frac{a}{a^2-x^2} dx.$$
 \item Soit $\O=(0,+\infty)$. Le volume de Busemann est donn\'e au point $x\in\O$ par
$$d\Vol_{\O}(x) = \frac{2dx}{x}.$$
\end{enumerate}
\end{lemm}

Le lemme suivant nous sera \'egalement bien utile:

\begin{lemm}\label{prod}
Soient $n\geqslant 1$ et $\O$ un ouvert convexe de $\R^n$, de base $(e_1,\cdots,e_n)$. Notons, pour $x\in \O$, $\O_i(x) = \O \cap (x+\R\cdot e_i)$. Il existe $K_n>0$ tel que, pour tout $x\in\O$,
$$\Vol(B(T_x\O)) \geqslant K_n \prod_{i=1}^n \Vol_i (B(T_x\O_i(x))),$$
en notant $\Vol_i$ le volume de Lebesgue de $\R.e_i$.
Autrement dit, il existe $\kappa_n>0$ tel que, pour tout $x\in\O$,
$$d\Vol_{\O}\leqslant \kappa_n d\Vol_{\O_1(x)}\cdots d\Vol_{\O_n(x)}.$$
\end{lemm}
\begin{proof}
Il suffit de voir que la boule unit\'e tangente $B(T_x\O)$ contient toujours l'enveloppe convexe des points $z_i^{\pm} = \partial B(T_x\O) \cap \R_{\pm}\cdot e_i$, pour $1\leqslant i\leqslant n$.
\end{proof}

\begin{proof}[D\'emonstration de la proposition \ref{volpic}]
Si une telle coupe existe, il en existe en particulier une de dimension $2$. On peut donc supposer qu'il existe une telle coupe de dimension $2$.\\

Maintenant, voyons qu'il suffit de prouver le r\'esultat pour un ouvert convexe bien choisi et un cône assez g\'en\'eral, qui sont les suivants. Prenons le point $p$ pour origine, pour convexe l'ensemble
$$\O=\{x_n > |x_1|^{\alpha} + \sum_{i=2}^{n-1} |x_i|\},$$
et pour cône
$$C=\{x_n < \frac{1}{2},\ x_n > 2\sum_{i=1}^{n-1} |x_i|\}.$$
C'est une situation assez g\'en\'erale au sens où, \'etant donn\'e un convexe dont le bord est $\C^1$ au point $p$ et un cône de sommet $p$ comme dans l'\'enonc\'e, on peut choisir une carte affine, une norme euclidienne et des coordonn\'ees, avec origine $p$, de telle façon que, au moins au voisinage de $p$, le convexe contienne un convexe du type pr\'ec\'edent et le cône soit contenu dans un cône du type pr\'ec\'edent.\\

On peut maintenant faire le calcul. Pour $x = (x_1,\cdots,x_n)\in\O$ et $1\leqslant k\leqslant n$, notons $\O_k(x) = \O \cap (x+\R.e_k)$ la coupe du convexe $\O$ selon $e_k$, qui est donc un convexe de dimension 1.\\
Pour  $2\leqslant k\leqslant n-1$, $\O_k(x)$ est un segment de demi-longueur
$$a_k(x) = x_n - |x_1|^{\alpha} - \sum_{i=2, i\not =k}^{n-1} |x_i|.$$
Le lemme \ref{dim1} nous donne que 
$$d\Vol_{\O_k(x)} = \frac{a_k(x)}{a_k(x)^2-x_k^2}\ dx_k = \frac{a_k(x)}{(a_k(x)-|x_k|)(a_k(x)+|x_k|)}\ dx_k \leqslant \frac{dx_k}{a_k(x)-|x_k|}.$$
Si $x$ est dans $C$, on a $\sum_{i=1}^{n-1} |x_i| < \frac{x_n}{2},$ 
et donc
$$a_k(x)-|x_k| = x_n - |x_1|^{\alpha} - \sum_{i=2}^{n-1} |x_i| \geqslant x_n - \sum_{i=1}^{n-1} |x_i| \geqslant \frac{x_n}{2}.$$
Au final,
$$d\Vol_{\O_k(x)} \leqslant \frac{2dx_k}{x_n}.$$
De même, le convexe $\O_1(x)$ est un segment de demi-longueur
$$a_1(x) =  (x_n - \sum_{i=2}^{n-1} |x_i|)^{\frac{1}{\alpha}},$$
et on a donc, pour $x\in C$,
$$d\Vol_{\O_1(x)} = \frac{a_1(x)}{(a_1(x)-|x_1|)(a_1(x)+|x_1|)}\ dx_1 \leqslant
\frac{dx_1}{(x_n - \sum_{i=2}^{n-1} |x_i|)^{\frac{1}{\alpha}}-|x_1|} 
\leqslant \frac{dx_1}{\left(\displaystyle\frac{x_n}{2}\right)^{\frac{1}{\alpha}}-|x_1|}.$$
Enfin $\O_n(x)=(0,+\infty)$ et donc, pour $x\in C$,
$$d\Vol_{\O_n(x)} = \frac{2dx_n}{x_n}.$$

\ \\
Du lemme \ref{prod}, on tire ainsi
$$
d\Vol_{\O}(x) \leqslant  \kappa_n d\Vol_{\O_1(x)}\cdots d\Vol_{\O_n(x)} \leqslant
\displaystyle \kappa_n \frac{2^{n-2}}{(x_n)^{n-1}} \frac{1}{\left(\displaystyle\frac{x_n}{2}\right)^{\frac{1}{\alpha}}-|x_1|}\ dx
$$
L'int\'egrale sur $C$ se majore alors ainsi, en utilisant les sym\'etries, $K$ et $K'$ \'etant des constantes qui grandissent:
$$\int_C d\Vol_{\O} \leqslant K \int_{x_n=0}^{\frac{1}{2}}  \prod_{k=1}^{n-1} \int_{x_k=0}^{\frac{x_n}{2}}  \frac{1}{(x_n)^{n-1}} \frac{1}{\left(\displaystyle\frac{x_n}{2}\right)^{\frac{1}{\alpha}}-|x_1|} \ dx_1\cdots dx_n.$$
Ainsi,
$$\int_C d\Vol_{\O} \leqslant K' \int_{0}^{\frac{1}{2}} \int_{0}^{\frac{x_n}{2}} \frac{1}{x_n} \frac{1}{\left(\displaystyle\frac{x_n}{2}\right)^{\frac{1}{\alpha}}-|x_1|} \ dx_1\ dx_n = K' \int_{0}^{\frac{1}{2}} \frac{1}{x_n} \ln \frac{1}{1 - \left(\displaystyle\frac{x_n}{2}\right)^{1-\frac{1}{\alpha}}}\ dx_n.$$
Cette dernière int\'egrale est finie puisqu'en $0$, l'int\'egrande est \'equivalente à $\frac{1}{(x_n)^\frac{1}{\alpha}}$, et $\frac{1}{\alpha}<1$.

\end{proof}


\backmatter
\bibliographystyle{alpha}

\newpage

\section{Erratum/Addendum to: Finitude g\'eom\'etrique en g\'eom\'etrie de Hilbert by Pierre-Louis Blayac and Ludovic Marquis}\label{errata_part}

\begin{altabstract}
	We amend  Theorems~1.3 and 1.11 of \cite{CM2014finitude}: Finitude g\'eom\'etrique en g\'eom\'etrie de Hilbert. We seize the opportunity to show that in round Hilbert geometry, geometrical finiteness \eqref{item:GF_on_boundary}  is equivalent to cusp-uniform action and to fill some small gaps that appear in two other proofs of \cite{CM2014finitude}.
\end{altabstract}

\subsection{The published statement}

Let $\O$ be an open subset of $\R\PP^d$ which is properly convex, i.e.~contained in an affine chart of $\R\PP^d$ where it is bounded.
Suppose further $\O$ is \emph{round}, in the sense that it has $\Cc^1$-boundary and is strictly convex (any segment contained in the boundary must be reduced to a point).
Finally, let $\Gamma$ be a discrete subgroup of $\PGL_{d+1}(\R)$ that preserves $\Omega$.
Recall that, using a famous $\G$-invariant metric on $\O$, denoted $d_\O$ and called the \emph{Hilbert metric}, one can check that  $\G$ acts properly discontinuously on $\O$, and $\O/\G$ is called a convex projective orbifold.
For more detailed reminders on convex projective geometry and the Hilbert metric, see \cite[\textsection 2]{CM2014finitude}.

The main goal of \cite{CM2014finitude} was to introduce a notion of geometrical finiteness for the action of $\Gamma$ on $\O$ and for the underlying orbifold $\O/\G$, and then to study this notion, in particular by giving various characterisations of it, in the spirit of \cite{BowditchGF,Bowditch_GF}.
Before we describe these characterisations, let us recall some notations from \cite{CM2014finitude}. 

\medskip

The \emph{limit set} of $\G$ is $\LG=\overline{\G x} \smallsetminus \G x\subset\partial\O$, which is independent of the choice of an $x\in\O$.
The \emph{convex core} $\Cc(\LG)$ is the convex hull in $\O$ of the limit set.
More generally, we will use the notation $\Cc (\cdot)$ to denote convex hulls.

Note that for any $p \in \PP(\R^{d+1}) $, the projective space  $\PP (\R^{d+1}/p)$ identifies with the space of lines of $\PP(\R^{d+1}) $ containing $p$.
If $p \in \dO$ then we denote by $\Dc_p (\O) $ the space of lines containing $p$ and intersecting $\O$.
Since $\dO$ is $\Cc^1$ at $p$, the space of lines $\Dc_p (\O) $ identifies with the affine space $\mathbb{A}_p = \PP (\R^{d+1}/p) \smallsetminus \PP (T_p \dO/p)$, where $T_p \dO$ is the tangent space to $\dO$ at $p$. 
The map $s_p : \overline{\O} \smallsetminus \{ p\} \to \mathbb{A}_p$ given (through the former identification) by $q \mapsto (pq)$, will be called the \emph{stereographic projection from $p$}.


Recall that $p\in \dO$ is a \emph{parabolic point} if its stabilizer $\G_p$ is infinite and parabolic, which is equivalent to saying that $\g_nx\to p$ for any injective sequence $(\g_n)_n\subset\G_p$ and any $x\in\O$ (see \cite[\textsection 3.5]{CM2014finitude} for more characterisations of parabolicity); this implies $p\in\LG$.
A parabolic point is \emph{bounded} if the action of $\G_p$ on $\LG \smallsetminus \{ p\}$ is cocompact.
A parabolic point is \emph{uniformly bounded} if the action of $\G_p$ on $s_p(\Cc(\LG))$ is cocompact.
Note that   $ s_p (\Cc(\LG))$ is the convex hull in $\A_p$ of $s_p(\LG\smallsetminus\{p\})$, so uniformly bounded implies bounded.

Finally, a point $p\in\partial\O$ is called \emph{conical} if there are $x\in\O$ and $(\g_n)_n\subset\G$ such that $\g_nx\to p$ and $\g_nx$ stays at bounded Hilbert distance from the ray $[x,p)$; this implies $p\in\LG$.

\medskip

We defines the following properties for the action of $\G$ on $\O$:
\begin{itemize}
	\item[(\namedlabel{item:GF_on_boundary}{\texttt{gf}})] : Every point of $\Lambda_\Gamma$ is either conical or bounded parabolic.
	
	\item[(\namedlabel{item:GF_on_Omega}{\texttt{GF}})] : Every point of $\Lambda_\Gamma$ is either conical or uniformly bounded parabolic.
	
	\item[(\namedlabel{item:gf_cusps_lobat}{\texttt{HC}})] :  \eqref{item:GF_on_boundary} holds and for each parabolic point $p$, the group $\G_p$ is conjugate into $\mathrm{O}_{d,1} (\R)$.
	
	\item[(\namedlabel{item:region_parabolique}{\texttt{TF}})] :   There exists a $\Gamma$-invariant family $\mathcal{P}$ of points $p \in \LG$, a family of standard\footnote{We do not recall the technical definition of standard parabolic regions since it will not be used here, see \cite[\textsection 7.3]{CM2014finitude}.} regions $R_p$ centered at $p$, such that $({R_p})_p$ is $(\G,\G_p)$-precisely equivariant (see Definition~\ref{def:precisely equivariant}), the action of $\G$ on $\overline{\O} \smallsetminus (\LG \cup \bigcup_p R_p)$ is cocompact.
	
	\medskip
	
	\item[(\namedlabel{item:thick_part}{\texttt{PEC}})] :  The thick part\footnote{The thick part consists of the projections of $x\in\O$ such that $\{ \g \in \G \, |\, d_\O (x, \g x) < \varepsilon \}$ generates a finite group, given a sufficiently small $\varepsilon$, see Section~\ref{sec:thickthin}.} of the convex core is compact.
	
	\item[(\namedlabel{item:non_cusp_part}{\texttt{PNC}})] :  The non-cuspidal part\footnote{For us the non-cuspidal part is the union of the thick part with the components of the thin parts that consists of tubular neighborhoods of short geodesics, see \cite[\textsection 6.2]{CM2014finitude}.} of the convex core is compact.
	
	\item[(\namedlabel{item:cusp_uniform}{\texttt{CU}})] :  The action $\G \curvearrowright \O$ is \emph{cusp-uniform} i.e.\ there exists a $\Gamma$-invariant family $\mathcal{P}$ of points $p \in \LG$, a family of horoballs $H_p$ center at $p$, such that $({H_p})_p$ is $(\G,\G_p)$-precisely equivariant, and the action of $\G$ on $\Cc(\LG)\smallsetminus \bigcup_p H_p$ is cocompact.\footnote{See Section~\ref{sec:horoballs} for reminders on horoballs.}
	
	\medskip
	
	\item[(\namedlabel{item:convex_core}{\texttt{VF}})${}_R$] :  $\G$ is finitely generated and the uniform $R$-neighborhood (for the Hilbert metric) of the convex core $\Cc(\LG)/\G$  is of finite volume.\footnote{Here our volume form is the Hausdorff measure of the Hilbert metric, see \cite[\textsection 2.1]{CM2014finitude}.}
	
	\item[(\namedlabel{item:convex_core0}{\texttt{VF}})${}_0$] :  $\G$ is finitely generated and the convex core $\Cc(\LG)/\G$ is of finite volume for the Hilbert volume form from $\O\cap {\rm Span}(\Cc(\LG))$.
	
	\item[(\namedlabel{item:convex_core_ghyp}{\texttt{Hyp}})] :  The convex core $\Cc(\LG)$ is Gromov-hyperbolic\footnote{Recall that a geodesic metric space is Gromov-hyperbolic if for some $\delta$ all geodesic triangles are $\delta$-thin: any side is in the $\delta$-neighborhood of the union of the two other sides, see e.g.\ \cite[\textsection  III.H.1]{bridsonhaeflige}.} for the Hilbert metric of $\O$.
	
	\item[(\namedlabel{item:generic}{\texttt{Gen}})] :  The limit set $\LG$ spans $\R\PP^d$, or its dual spans the dual projective space.\footnote{Recall that $\G$ also preserves a properly convex open set $\O^*$ in the projective space of linear forms on $\R^{d+1}$, and hence has a limit set there too, see \cite[\textsection 2.3]{CM2014finitude}.}
\end{itemize}

\vspace{1em}

Note that for all $R,R'>0$, it is easy to check that \eqref{item:convex_core}${}_R$ and \eqref{item:convex_core}${}_{R'}$ are equivalent and imply \eqref{item:convex_core0}${}_0$.
However, \eqref{item:convex_core0}${}_0$ does not implies \eqref{item:convex_core}${}_1$, for example suppose we have two discrete subgroups $G, H < \mathrm{Isom} (\Hb^d)$ such that $G$ stabilizes a proper subspace $V<\Hb^d$ and $H$ is generated by a loxodromic element whose axis $A$ is disjoint from $V$.
If $G'<G$ and $H'<H$ are sufficiently small finite-index subgroups then one can play ping-pong (see e.g.\ \cite{maskit}) to prove that the group $\G$ generated by $G'$ and $H'$ is discrete, is isomorphic to the free product $G\ast H$, the convex cores $C_1\subset V/G$ and $C_2\subset A/H$ of $G$ and $H$ embed (disjointly) in the convex core  $C\subset \Hb^d/\G$ of $\G$, and, fixing a point $x_0\in C$, there is a constant $K>0$ such that any point $x\in C$ at distance $r$ from $x_0$ is at distance at most $Ke^{-r}$ from $C_1\sqcup C_2$.
In particular, if $V$ has codimension $1$ and $V/G$ is an abelian cover of a closed hyperbolic manifold, then one can check that $C$ has finite volume, as was mentioned to us by D.\ Cooper.

Note also, that the assumption $\G$ finitely generated cannot be removed from  \eqref{item:convex_core}${}_R$ has shown by \cite{hamilton_counter_example}. One can see \eqref{item:generic} as a genericity assumption.
It holds when $\G$ is irreducible, and a fortiori when it is Zariski-dense.

\vspace{1em}

Those properties were linked by the following theorem, \emph{which is wrong}.

\begin{theorem}\cite[Thm.\,1.3\,\&\,1.11,\ Prop.\,1.4]{CM2014finitude}\label{thm original}
	For any $\G$ acting on a round convex $\O$, the assertions \eqref{item:GF_on_Omega}, \eqref{item:region_parabolique}, \eqref{item:gf_cusps_lobat}, \eqref{item:thick_part}, \eqref{item:non_cusp_part} , \eqref{item:convex_core}${}_1$  and (\eqref{item:GF_on_boundary}$\&$\eqref{item:convex_core_ghyp}) are equivalent.
	
	Morever they all imply \eqref{item:GF_on_boundary} but are not equivalent to it.
\end{theorem}

The former proof used the following pattern, which is also recapitulated in Figure~\ref{fig:old pattern}.
We indicated in brackets where to find the proof in the original paper.

\begin{itemize}
	\item \eqref{item:GF_on_Omega} $\Rightarrow$ \eqref{item:region_parabolique} [Prop.~7.21 \& 7.23], whose proof is correct.
	
	\item \eqref{item:region_parabolique} $\Rightarrow$ \eqref{item:thick_part} $\Rightarrow$ \eqref{item:non_cusp_part}\footnote{Typo in the sentence ``Preuve de \eqref{item:region_parabolique}  $\Rightarrow$ \eqref{item:non_cusp_part} $\Rightarrow$ \eqref{item:thick_part} '', first proof of section~8.2. It should have been written: ''Preuve de \eqref{item:region_parabolique}  $\Rightarrow$ \eqref{item:thick_part} $\Rightarrow$ \eqref{item:non_cusp_part} ''.
		Note that the implication \eqref{item:non_cusp_part} $\Rightarrow$ \eqref{item:thick_part} is trivial since the thick part is a closed subset of the non-cuspidal part.} [\textsection 8.2], whose proofs are corrects.
	
	\item \eqref{item:non_cusp_part} $\Rightarrow$ \eqref{item:GF_on_Omega} [\textsection 8.2], this proof is wrong, and in fact the statement is wrong.
	However, the implication \eqref{item:non_cusp_part} $\Rightarrow$ \eqref{item:GF_on_boundary} is true.
	The proof of the implication \eqref{item:non_cusp_part} $\Rightarrow$ \eqref{item:GF_on_boundary} appears as a step in  \eqref{item:non_cusp_part} $\Rightarrow$ \eqref{item:GF_on_Omega}. 
	The proof of this implication is incomplete but can be corrected using the same strategy as the original paper.
	We will correct it with a slightly different strategy.
	
	\item \eqref{item:GF_on_Omega} $\Rightarrow$ \eqref{item:convex_core}${}_1$ [\textsection 8.3], whose proof is incomplete as only \eqref{item:convex_core}${}_0$ is proved; we will fix this.
	
	\item \eqref{item:convex_core}${}_1$ $\Rightarrow$ \eqref{item:thick_part} [\textsection 8.3], whose proof is correct (as stated the Lemma~8.5 used in the proof is incorrect but the proof is easy to fix, see Remark~\ref{cor:lemma8.5}).
	
	\item  \eqref{item:GF_on_Omega} $\Rightarrow$ \eqref{item:gf_cusps_lobat} [Cor.~7.18], which is wrong in full generality but true if we ask \eqref{item:generic}; we will fix this.
	
	\item   \eqref{item:gf_cusps_lobat} $\Rightarrow$ \eqref{item:GF_on_Omega} [Prop.~7.21\footnote{Typo in the hypotheses: one needs to assume $p$ is \emph{bounded} parabolic.}], whose proof is correct.
	
	\item  \eqref{item:GF_on_Omega} $\Rightarrow$ (\eqref{item:GF_on_boundary}$\&$\eqref{item:convex_core_ghyp}) [Th.~9.1] is true but the proof has a small gap; we will fix this.
	
	\item  (\eqref{item:GF_on_boundary}$\&$\eqref{item:convex_core_ghyp})   $\Rightarrow$ \eqref{item:GF_on_Omega} [Th.~9.1] is false (Lemma~9.2 being wrong).
	
	\item \eqref{item:GF_on_boundary} $\not\Rightarrow$ \eqref{item:GF_on_Omega} [Prop.~10.6] whose proof by counter-example is correct; in fact this counter example satisfies \eqref{item:convex_core_ghyp} and \eqref{item:convex_core}${}_1$ as we explain in Section~\ref{sec:overview} and in \cite{Next}.
\end{itemize}


\begin{remark}(The error  in  (\eqref{item:GF_on_boundary}$\&$\eqref{item:convex_core_ghyp})   $\Rightarrow$ \eqref{item:GF_on_Omega}). The error is hidden in the sentence: ``On peut identifier l'espace des droites $\Dc_p (\Cc(\LG ))$ à sa trace sur l'horosphère $\Hc$.'' meaning that we can identified $\Dc_p (\Cc(\LG ))$ and $\Hc \cap \Cc(\LG)$, which is wrong. In fact, if $\Hc_t$ is family of horosphere such that the corresponding family of horoball decreases, then stereographic projection on $\mathbb{A}_p$ of $\Hc_t \cap \Cc(\LG)$ is a closed subset $F_t$ of $\mathbb{A}_p$. For each $t$, the group $\G_p$ acts cocompactly on $F_t$ but this family of closed subset is increasing and $\G_p$ may not act cocompactly on their union.
\end{remark}

\begin{remark}(The error in  \eqref{item:non_cusp_part} $\Rightarrow$ \eqref{item:GF_on_Omega})
	The error is hidden in the sentence ``n\'ecessairement uniform\'ement born\'e puisque la partie non cuspidale du c\oe ur convexe est 	compacte.'', implicitly the authors had in mind that the intersection  $\partial \O_{\varepsilon} (\G_p) \cap \Cc(\LG)$  can be identified with  $\Dc_p (\Cc(\LG ))$. Similarly to the above mistake, the stereographic projection on $\mathbb{A}_p$ of $\partial \O_{\varepsilon} (\G_p) \cap \Cc(\LG)$ is a closed subset $E_\varepsilon$ of $\mathbb{A}_p$. For each $\varepsilon$, the group $\G_p$ acts cocompactly on $F_\varepsilon$ but this family of closed subset is increasing (as $\varepsilon \to 0$) and $\G_p$ may not act cocompactly on their union.
\end{remark}

\begin{remark}(\eqref{item:convex_core}${}_1$ $\Rightarrow$ \eqref{item:thick_part})\label{cor:lemma8.5}
First, for any $R>0$ there is a constant $c_R >0$ independent of the convex $\Omega$ and of the point $x \in \Omega$ such that $\Vol_\Omega (B(x,R)) \geqslant c_R$ (\cite[Thm.12]{colbois_vernicos_spectre} or see e.g. \cite[Lem.8.4]{CM2014finitude}). Second, if $x$ is in the $\varepsilon$-thick part then the ball $B(x,\varepsilon)$ of $\Omega$ embeds in $\O/\G$. Hence, if $x$ is in the convex core then $B(x,\varepsilon)$ embeds in the $1$-neighborhood of the convex core $\Cc(\LG)/\G$ (assuming that $\varepsilon < 1$). Crampon and the second author conclude erroneously that such a ball embeds in $\Cc(\LG)/\G$.
	
Leading to the erroneous conclusion that if the action of $\Gamma$ satisfies \eqref{item:convex_core}$_0$ then the $\varepsilon$-thick part of the convex core can contain only finitely many disjoint balls of radius $\varepsilon$, hence is compact. When, in fact, one needs to assume \eqref{item:convex_core}$_1$ to conclude that the $\varepsilon$-thick part of the convex core can contain in its $1$-neighborhood only finitely many disjoint balls of radius $\varepsilon$, and hence must be compact.
\end{remark}

\begin{figure}
	\centering
	\begin{tikzpicture}[>=latex]
		\tikzstyle{implique}=[-{Implies},double equal sign distance]
		\tikzstyle{equivalent}=[implies-implies,double equal sign distance]
		\node (PNC) at (0,0){\eqref{item:non_cusp_part}};
		\node (GF) at (0,4) {\eqref{item:GF_on_Omega}};
		\node (gf+hyp) at (-4,4) {\eqref{item:GF_on_boundary}$\&$\eqref{item:convex_core_ghyp}};
		\node (TF) at (4,4) {\eqref{item:region_parabolique}};
		\node (PEC) at (4,0) {\eqref{item:thick_part}};
		\node (VF) at (2,2) {\eqref{item:convex_core}${}_1$};
		\node (gf) at (-4,0) {\eqref{item:GF_on_boundary}};
		\node (HC) at (2,6) {\eqref{item:gf_cusps_lobat}};
		
		\draw[implique] (GF) -- (TF);
		\draw[implique] (HC) -- (GF);
		\draw [implique] (GF) -- (gf);
		\draw [implique,red!85!black] (PNC) -- (GF) node[midway, sloped,scale=.1cm]{$\times$};
		\draw[implique] (TF) -- (PEC);
		\draw [implique,orange] (PNC) -- (gf);
		\draw [equivalent] (PNC) -- (PEC);
		\draw [implique] (VF) -- (PEC);
		\draw [implique,orange] (GF) -- (VF);
		\draw [-{Implies},double,red,double equal sign distance] (gf+hyp) to [out=30,in=150] node[midway, sloped,scale=.1cm]{$\times$} (GF);
		\draw [-{Implies},double,red,double equal sign distance] (GF) to [out=90,in=180]  node[midway, sloped,scale=.1cm]{$\times$} (HC);
		\draw [implique,orange] (GF) -- (gf+hyp);
		
	\end{tikzpicture}
	\caption{Old pattern of implications: black arrows were correctly proved in the former paper, the red arrows are mistakes of the former paper, the orange arrows need to be fixed (they are badly written, are incomplete, or have mistakes in the former paper).}
	\label{fig:old pattern}
\end{figure}

%

\begin{figure}
	\centering
	\begin{tikzpicture}[>=latex,scale=1.1]
		\tikzstyle{arrondi}=[rounded corners=4pt]
		\tikzstyle{implique}=[-{Implies},double equal sign distance]
		\tikzstyle{equivalent}=[implies-implies,double equal sign distance]
		\node (TF) at (2,1) {\eqref{item:region_parabolique}};
		\node (GF) at (2,-1) {\eqref{item:GF_on_Omega}};
		\node (HC) at (0,0) {\eqref{item:gf_cusps_lobat}};
		\node (GF+Gen) at (-2,0) {\eqref{item:GF_on_Omega}\&\eqref{item:generic}};
		
		\node (VF) at (4,-1) {\eqref{item:convex_core}${}_1$};
		\node (gf+hyp) at (4,1) {\eqref{item:GF_on_boundary}$\&$\eqref{item:convex_core_ghyp}};
		
		\node (PEC) at (6,-1) {\eqref{item:thick_part}};
		\node (PNC) at (8,-1){\eqref{item:non_cusp_part}};
		\node (gf) at (6,1) {\eqref{item:GF_on_boundary}};
		\node (CU) at (8,1) {\eqref{item:cusp_uniform}};

		\draw[implique] (HC) -- (GF);
		\draw[equivalent] (TF) -- (GF);
		\draw [implique,green!60!black] (GF) -- (VF);
		\draw[implique,green!60!black] (GF) -- (gf+hyp);
		\draw[implique] (gf+hyp) -- (gf);
		\draw [implique] (VF) -- (PEC);
		\draw [implique,green!60!black] (GF+Gen) -- (HC);
		\draw [equivalent,green!60!black] (gf) -- (CU);
		\draw [equivalent,green!60!black]  (PNC) -- (gf);
		\draw [equivalent] (PNC) -- (PEC);

		\draw[arrondi] (1.5,-1.3) rectangle (2.5,1.3);
		\draw[arrondi] (5.5,-1.3) rectangle (8.52,1.3);
		
	\end{tikzpicture}
	\caption{New pattern of implications: black arrows where correctly proved in the former paper, the green ones repair the mistakes of the former paper.}
	\label{fig:new pattern}
\end{figure}

\subsection{A correct statement}

In the present paper we prove the following result, which corrects Theorem~\ref{thm original}.
Figure~\ref{fig:new pattern} shows the new pattern of the proof.

\begin{theorem}\label{thm:erratum}
	Let $\O$ be a round convex of $\R\PP^d$ and $\G \leqslant \Aut(\Omega)$. Then:
	\begin{enumerate}
		\item (\eqref{item:GF_on_Omega}\&\eqref{item:generic})
		$\Longrightarrow$
		\eqref{item:gf_cusps_lobat} $\Longrightarrow$ \eqref{item:GF_on_Omega} $\Longleftrightarrow$ \eqref{item:region_parabolique}.
		
		\item \eqref{item:thick_part} $\Longleftrightarrow$ \eqref{item:non_cusp_part} $\Longleftrightarrow$ \eqref{item:GF_on_boundary} $\Longleftrightarrow$ \eqref{item:cusp_uniform}.
		
		\item \eqref{item:GF_on_Omega} $\Longrightarrow$ \eqref{item:convex_core}${}_1$ $\Longrightarrow$  \eqref{item:GF_on_boundary}.	
		\item \eqref{item:GF_on_Omega} $\Longrightarrow$ \eqref{item:GF_on_boundary}$\&$\eqref{item:convex_core_ghyp} $\Longrightarrow$  \eqref{item:GF_on_boundary}.
	\end{enumerate}
\end{theorem}

A counter-example to the reciprocal of the implication \eqref{item:gf_cusps_lobat} $\Longrightarrow$ \eqref{item:GF_on_Omega} is given in Section~\ref{sec:contrex}. It is trivial to find an example satisfying \eqref{item:gf_cusps_lobat} but not \eqref{item:generic}.
We will provide in a separate article counter-examples to all implications which are not equivalence in Theorem~\ref{thm:erratum}.(3-4), see Section~\ref{sec:overview} for an overview. 

\begin{theorem}[{\cite{Next}}]\label{thm:non_implication}
	Let $\O$ be a round convex set of $\R\PP^d$ and $\G \leqslant \Aut(\Omega)$. Then:
	\begin{enumerate}
		\item The condition  \eqref{item:convex_core}${}_1$ does not imply the condition  \eqref{item:GF_on_Omega}.
		\item The condition  \eqref{item:GF_on_boundary}$\&$\eqref{item:convex_core_ghyp} does not imply the condition  \eqref{item:GF_on_Omega}.
		\item The condition \eqref{item:GF_on_boundary} does not imply the condition \eqref{item:convex_core}${}_1$.
		\item The condition  \eqref{item:GF_on_boundary} does not imply the condition \eqref{item:GF_on_boundary}$\&$\eqref{item:convex_core_ghyp}.
	\end{enumerate}
	Indeed, for any non-uniform lattice $\G$ of $\SL_2 (\R)$, if $\rho: \SL_2 (\R) \to \SL_5 (\R)$ is the $5$-dimensional irreducible representation of $\SL_2 (\R)$ then, there exists  $\rho(\G)$-invariant round convex domains $\O_0$, $\O_1$ of $\R\PP^4$ such that :
	\begin{enumerate}[1.]
		\item $\rho (\G) \curvearrowright \O_0,\O_1$ are \eqref{item:GF_on_boundary}, but not \eqref{item:GF_on_Omega}.
		\item The convex core of $\O_0/\rho(\G)$ is of finite (nonzero) volume and
		\item $\Cc(\LG)$ is Gromov-hyperbolic for the Hilbert metric of $\O_0$.
		\item  While the convex core of $\O_1/\rho(\G)$ is of infinite volume and
		\item  $\Cc (\LG)$ is not Gromov-hyperbolic for the Hilbert metric of $\O_1$.
	\end{enumerate}
\end{theorem}


As we mentioned, in \cite[Prop.\,1.4]{CM2014finitude}, the authors exhibit examples of pairs $(\Omega,\Gamma)$ which satisfies \eqref{item:GF_on_boundary} but not \eqref{item:GF_on_Omega}. We use those examples to show the existence of $\O_0$ in Theorem~\ref{thm:non_implication}.
The construction of $\Omega_1$ is more involved.

\begin{remark}
	Fix a discrete subgroup $\G\leqslant \Aut(\O)$ preserving at least one round convex set of the projective space, such that $\G$ is non-elementary (not virtually nilpotent).
	As one can see from Theorem~\ref{thm:non_implication}, the properties \eqref{item:convex_core}${}_0$, \eqref{item:convex_core}${}_1$ and \eqref{item:convex_core_ghyp} depend on the choice of the $\G$-invariant round convex set $\O$: they might hold for one domain but not for another.
	
	However, all the other properties studied in this paper are independent of the choice of $\O$.
	This comes from the classical fact that the limit set $\Lambda_\G$ is independent of $\O$.
	Indeed recall that an element $g\in\SL_{d+1}(\R)$ is \emph{proximal} if it has an attracting fixed point in $\R\PP^d$.
	Then one can check that the limit set $\Lambda_\G$ is the closure in $\R\PP^d$ of the set of attracting fixed points of proximal elements of $\G$, also called proximal limit set \footnote{An element $\g$ is proximal if and only if it is a hyperbolic automorphism of $\O$ in the sense of the classification theorem \cite[Th.\,3.3]{CM2014finitude}.
		This theorem also easily implies that any attracting fixed point of a proximal element is in the limit set, so the proximal limit set is contained in the limit set.
		To prove the other inclusion, consider $p$ in the limit set and $p_0$ any point in the proximal limit set.
		Then $\g_nx\to p\in\partial\O$ for some sequence $\g_n$.
		Up to extracting $\g_n y\to p$ for any $y\in\partial\O\smallsetminus\{q\}$, for some $q$ (see\cite[Prop.\,4.8]{CM2014finitude}).
		Since $\G$ is non-elementary there is $\alpha\in\G$ such that $\alpha p_0\neq q$, so $\g_n\alpha p_0\to p$, so $p$ is in the proximal limit set.}.
	From the definitions, one immediately sees that the properties \eqref{item:GF_on_boundary}, \eqref{item:GF_on_Omega}, \eqref{item:gf_cusps_lobat} and \eqref{item:generic} are independent of $\O$, and Theorem~\ref{thm:erratum} implies the properties \eqref{item:region_parabolique},  \eqref{item:thick_part}, \eqref{item:non_cusp_part} and \eqref{item:cusp_uniform} are independent of $\O$ too.
\end{remark}

The second author warmly thanks the first author for pointing out to him the mistake in the former paper and his help to find and write the proper statement. 
The second author also thanks A. Zimmer for pointing out to him the second point of Theorem~\ref{thm:non_implication} using the same $\O_0$ that we will use. 
The authors thank B. Fl\'echelles and D. Cooper for interesting discussions and useful comments.

\subsection{Plan of proof}

The main results are \eqref{item:GF_on_boundary} $\Longleftrightarrow$ \eqref{item:cusp_uniform} and \eqref{item:GF_on_boundary} $\Longleftrightarrow$ \eqref{item:non_cusp_part}, whose proofs are extremely similar, so our goal is to prove them both at the same time, by proving a more general result.

In Section~\ref{sec:conical case} we establish a short independent lemma, useful in the proofs of \eqref{item:GF_on_boundary} $\Longleftrightarrow$ \eqref{item:cusp_uniform} and \eqref{item:GF_on_boundary} $\Longleftrightarrow$ \eqref{item:non_cusp_part}.

\medskip

In Section~\ref{sec:star}, we prove that \eqref{item:GF_on_boundary} is equivalent to a whole family of properties.
More precisely, we show that, given any precisely equivariant family of star domains that satisfy a certain convexity condition, asking $\G$ to act geometrically finitely on $\partial\O$ (i.e.~ asking \eqref{item:GF_on_boundary}) is equivalent to asking that $\G$ acts cocompactly on the complement in the convex core of the family of star domains.

\medskip

In Section~\ref{sec:cusp_uniform}, we check that horoballs satisfy the above convexity condition (because horoballs are convex), and obtain the equivalence \eqref{item:GF_on_boundary} $\Longleftrightarrow$ \eqref{item:cusp_uniform} as a consequence.
The condition \eqref{item:cusp_uniform} is not present in the original paper \cite{CM2014finitude} but it should have been, so we seize the opportunity to give a proof. 
A proof of the implication \eqref{item:GF_on_boundary} $\Longrightarrow$ \eqref{item:cusp_uniform} is also given in \cite[Prop.~3.3]{bray_tiozzo}

\medskip

In Section~\ref{sec:non_cuspidal_part}, we check that the star domains obtained in the thick-thin decomposition of $\O$ also satisfy the above-mentioned convexity condition, and obtain the equivalence \eqref{item:GF_on_boundary} $\Longleftrightarrow$ \eqref{item:non_cusp_part} as a consequence.

\medskip

In Section~\ref{sec:contrex} we give a counterexample to \eqref{item:GF_on_Omega} $\Rightarrow$ \eqref{item:gf_cusps_lobat}, and then in Section~\ref{sec:GF imp HC} we prove that  \eqref{item:GF_on_Omega} $\Rightarrow$ \eqref{item:gf_cusps_lobat} holds under the additional genericity assumption \eqref{item:generic}.

\medskip

In Sections~\ref{sec:GF imp VF} and \ref{sec:GF imp Hyp} we fill in the gaps in the proofs of respectively \eqref{item:GF_on_Omega} $\Rightarrow$ \eqref{item:convex_core}${}_1$ and \eqref{item:GF_on_Omega} $\Rightarrow$ \eqref{item:convex_core_ghyp}.

\medskip

The only missing implication of Theorem~\ref{thm:erratum} is the implication \eqref{item:region_parabolique} $\Rightarrow$ \eqref{item:GF_on_Omega}.
This implication is not present in the original paper  \cite{CM2014finitude}.
Because it was done through the erroneous implication \eqref{item:non_cusp_part} $\Rightarrow$ \eqref{item:GF_on_Omega}.
Nevertheless, it is easy to check that \eqref{item:region_parabolique} $\Rightarrow$ \eqref{item:GF_on_boundary} (for instance using Proposition~\ref{prop:cocpct implies gf} below) and the proof of  \cite[Prop.~7.21]{CM2014finitude} shows that  if \eqref{item:region_parabolique} holds then all bounded parabolic points are in fact uniformly bounded.

%
%
%
%

\subsection{Dirichlet domain and conical limit points}\label{sec:conical case}

This section only contains a short independent lemma saying that Dirichlet domains do not accumulate on conical limit point.
The argument is standard, see for instance \cite[Prop.1.10]{roblin_smf}.
This lemma, as well as Dirichlet domains and the ideas in \cite[Prop.1.10]{roblin_smf}, will be used to prove that \eqref{item:GF_on_boundary} implies cocompactness properties (see Proposition~\ref{prop:gf implies cocpct}).

Let $\Omega$ be a round convex subset of $\R\PP^d$, $\Gamma\subset\Aut(\Omega)$ discrete. If $o$ is a point of $\O$, the \emph{Dirichlet domain based at $o$} is 
$$
\Dc = \{ x \in \O \, |\, \forall \g \in \G, \,  d_\O (x,o)  \leqslant d_\O(x,\gamma o) \}
$$
Note that $\Dc$ is a closed subset of $\O$, and that the translates of $\Dc$ by $\G$ cover $\Omega$.

Using the fact that $\O$ is strictly convex, one can check that if $\G$ is torsion-free then the translates of $\Dc$ by $\G$ have disjoint interiors and that those interiors have the form $\{ x \in \O \, |\, d_\O (x,o)  < d_\O(x,\gamma o) ,\, \forall \g \in \G\smallsetminus\{1\}\}$, but we will not need this fact.
Note that the translates of $\Dc$ by $\G$ may intersect on their interiors if $\O$ is not strictly convex: this happens for instance in the case of a $\Z$-action on a triangle generated by a diagonal $3\times 3$ matrix with diagonal entries $2,2,1/4$.

In the following lemma we check that Dirichlet domains cannot contain conical limit points at their boundary at infinity.
Compare with \cite[Lem.~8.2]{CM2014finitude}, which has a similar result for a different kind of fundamental domains.

\begin{lemma}\label{lem:boundary_pt_of_fund_dom_are_not_conical}
	Let $\Omega$ be a round convex subset of $\R\PP^d$, $\Gamma\subset\Aut(\Omega)$ discrete. If $p \in \partial \Dc \cap \partial \O$ then $p$ is not a conical limit point.
\end{lemma}

\begin{proof}[Proof]
	There exists $(x_n)_n \in \Dc^\N$ such that $x_n \to p$. Assume $p$ is a conical limit point. Then, there exists also $(\g_m)_m \in \G^\N$ such that $\g_m (o)$ converges conically to $p$, i.e.\ there exists $(y_m)_m\subset [o,p)$ tending to $p$ such that $(d_\O(\g_mo,y_m))_m$ is bounded.
	Since for any $m$ we have 
	$$b_p(o,y_m)=\lim_{y\in[y_m,p)\to p}d_\O(o,y)-d_\O(y_m,y)=d_\O(o,y_m),$$
	thus
	$$
	b_p (o, \g_m (o) )= b_p(o,y_m)+b_p(y_m,\g_m (o))\geq d_\O(o,y_m) - d_\O(y_m,\g_m (o)) \underset{m\to\infty}{\longrightarrow} + \infty 
	$$
	However, for any $m$, since $x_n \to p$, we also have:
	$$
	b_p (o, \g_m (o) ) = \lim_{n \to + \infty} d_\O (o,x_n) - d_\O (\g_m (o), x_n).
	$$
	Since $x_n \in \Dc$ one has:
	$$
	d_\O (o,x_n) - d_\O (\g_m (o), x_n) \leqslant 0.
	$$
	So $b_p (o, \g_m (o) ) \leqslant 0$ for any $m$, absurd.
\end{proof}

\subsection{Cocompactness at parabolic points}\label{sec:star}

\subsubsection{Strongly star-shaped domains}

We will need a class of well-behaved domains of $\O$ centered at parabolic points that encompasses both horoballs (see Section~\ref{sec:cusp_uniform}) and components of the thin part (see Section~\ref{sec:non_cuspidal_part}).
Since the components of the thin part are not necessarily convex, we will use the larger class of star domains.
Unfortunately star-shapedness alone will be too weak for our purposes: we will need an important extra convexity assumption which will be stated directly inside Lemma~\ref{lemm:key_ameliorationbisbis}.
Roughly, a star domain $B$ satisfies this condition if it contains the convex hull of a smaller star domain.


\begin{definition_english}\label{def:strongly starshaped}
	Let $\Omega$ be a round convex subset of $\R\PP^d$ and $p\in\partial\O$.
	An open subset $B\subset \O$ is called \emph{strongly star-shaped} at $p$ if for every $x\in\partial\O\smallsetminus\{p\}$, the interval $(x,p)$ intersects $\partial B$ at exactly one point $y\in\O$, the interval $(y,p)$ is contained in $B$, and the interval $(x,y)$ is outside of $\overline B$.
	
	Observe that this implies that
	\begin{enumerate}[i.]
		\item $B$ is star-shaped at $p$,
		\item \label{item:dB in Omega} $\partial B\smallsetminus\{p\}\subset\Omega$,
		\item \label{item:stereo} $\partial \O \smallsetminus \{ p \}$ maps homeomorphically onto $\partial B$ via the stereographic projection (in a $G$-equivariant way if $B$ is invariant under some $G\subset\Stab(p)\subset\Aut(\O)$), and
		\item \label{item:stereobis} the stereographic projection from $\overline\O\smallsetminus (B\cup\{p\})$ to $\partial B$ is surjective, continuous, $G$-equivariant and proper.
	\end{enumerate}
\end{definition_english}

In other words, a strongly star-shaped open subset of $\O$ at $p$ is the ``interior'' of a hypersurface of $\O$ that maps homeomorphically onto $\partial \O \smallsetminus \{ p \}$ via the stereographic projection.

\subsubsection{A key cocompactness lemma about parabolic subgroups}

In this section we prove a cocompactness result for the parabolic subgroups,
inspired by the argument in \cite[Prop.\,8.12]{blayac_zhu_ergo}.

First we recall the following more classical properness result about parabolic subgroup.
Note that in the reference we are using there is a typo: they define $\mathcal{O}_\G:=\O\smallsetminus \LG$ whereas it should be $\mathcal{O}_\G:=\overline\O\smallsetminus\LG$.

\begin{fact}[{\cite[Lem.4.5]{CM2014finitude}}]\label{lem:parabolic_transf_push_to_p}
	Let $\Omega$ be a round convex subset of $\R\PP^d$ and $\Gamma\subset\Aut(\Omega)$ discrete.
	Then $\G$ acts properly discontinuously on $\overline\O\smallsetminus\LG$.
	
	In particular, applying this to a parabolic subgroup $\G_p$ fixing $p\in\LG$ we get that $\G_p$ acts properly discontinuously on $\overline\O\smallsetminus\{p\}$.
\end{fact}

	%

Now comes the key cocompactness lemma.
The formulation involving finite-index subgroups of the stabiliser of the parabolic point is an unfortunate necessary technicality.
It will be used in Section~\ref{sec:non_cuspidal_part}.

\begin{lemma}\label{lemm:key_ameliorationbisbis}
	Let $\Omega$ be a round convex subset of $\R\PP^d$,
	$\Gamma\subset\Aut(\Omega)$ discrete non-elementary,
	and $p\in \Lambda_\G$ be a bounded parabolic fixed point with stabilizer $\G_p$.
	Consider $G\subset \G_p$ a finite index subgroup and $B^-\subset B^+\subset \O$ two $G$-invariant strongly star-shaped open subsets at $p$ such that the convex hull of $B^-$ is contained in $B^+$.
	
	Then the action of $G$ on $\partial B^+ \cap \Cc(\LG)$ is cocompact.
\end{lemma}

	\begin{figure}
		\centering
		\definecolor{uququq}{rgb}{0.25,0.25,0.25}
		\begin{tikzpicture}[line cap=round,line join=round,>=triangle 45,x=1.0cm,y=1.0cm,scale=2]
			\clip(-2.6,-2.01) rectangle (2.49,1.92);
			\draw [rotate around={0:(0,0)}] (0,0) ellipse (2cm and 1.73cm);
			\draw [rotate around={0:(-1.3,0)}] (-1.3,0) ellipse (0.7cm and 0.63cm);
			\draw (-1.73,0.87)-- (-1.2,-1.39);
			\draw [dotted] (-1.99,-0.2)-- (1.89,0.57);
			\begin{scriptsize}
				\fill [color=black] (-2,0) circle (.8pt);
				\draw[color=black] (-2.1,0) node {$p$};
				\fill [color=uququq] (1.89,0.57) circle (.8pt);
				\draw[color=uququq] (2.2,0.66) node {$\g(\eta_{1})$};
				\fill [color=black] (-1.65,0.54) circle (.8pt);
				\draw[color=black] (-1.5,0.5) node {$x$};
				\fill [color=uququq] (-1.73,0.87) circle (.8pt);
				\draw[color=uququq] (-1.75,1) node {$\eta_{1}$};
				\fill [color=uququq] (-1.2,-1.39) circle (.8pt);
				\draw[color=uququq] (-1.14,-1.6) node {$\eta_{2}$};
				\fill [color=black] (-1.99,-0.2) circle (.8pt);
				\draw[color=black] (-2.3,-0.2) node {$\g(\eta_{2})$};
				\fill [color=uququq] (-0.605,0.07) circle (.8pt);
				\draw[color=uququq] (-0.9,0.15) node {$\g (x)$};
			\end{scriptsize}
		\end{tikzpicture}
		\caption{Illustration of the proof of Lemma~\ref{lemm:key_ameliorationbisbis}\\
			(For simplicity, $x$ is here in the convex hull of only two points of the limit set)}
		\label{fig:gf_cocompactness_on_horo}
	\end{figure}

	\begin{proof}[Proof]
		Since $p$ is bounded parabolic and $G$ has finite index in $\G_p$, there exists $K\subset \Lambda_\Gamma \smallsetminus \{ p \}$ compact such that $G\cdot K=\Lambda_\Gamma \smallsetminus \{ p \}$.
		Let $L$ be the set of points $x\in \overline\O$ such that $[x,q]$ does not intersect $B^-$ for some $q\in K$.
		To finish this proof, it suffices to check that $L'=L\cap \partial B^+ \cap \Cc(\LG)$ is compact and that $G\cdot L'=\partial B^+ \cap \Cc(\LG)$.
		
		First we check $L'$ is compact.
		It is clear that $L$ is closed in $\overline\O$, and hence compact: Let $(x_n)_n \in L^\N$ such that $x_n \to x$, hence $[x_n,q_n]\cap B^-=\emptyset$ for some $q_n\in K$. Up to extracting, we can assume $q_n\to q\in K$, and to conclude that $x \in L$, we note that $[x,q]$ cannot intersect the open set $B^-$, otherwise the segments $[x_n,q_n]$ would also intersect it for $n$ large enough.
		
		Note that $\partial B^+\cap \overline{\Cc(\Lambda_\G)}=\partial B^+\cap \Cc (\Lambda_\G)\cup\{p\}$ since $B^+$ is strongly star-shaped (see \ref{item:dB in Omega} in Definition~\ref{def:strongly starshaped}), and $p$ is not in $L$ since $B^-$ is strongly star-shaped.
		Thus 
		\[ L'=L\cap \partial B^+ \cap \Cc(\LG)=L\cap \partial B^+\cap \overline{\Cc (\Lambda_\G)}\]
		is compact.
		
		It remains to check that for any  $x\in \partial B^+ \cap \Cc(\LG)$ there exists $g\in G$ such that $gx\in L'$.
		Since $x \in \Cc (\LG)$, there exists $(\eta_i)_{i=1,...,d+1}$ in $\Lambda_\Gamma \smallsetminus \{ p \}$ such that $x$ is in the convex hull of $p$ and the $(\eta_{i})_{i=1,...,d+1}$.
		We claim that there exists $i$ such that $[x,\eta_{i}] \cap B^- = \varnothing$.
		By contradiction, if this is not the case, then for each $i$, there exists $x_{i}\in [x,\eta_{i}] \cap B^-$.
		Since $B^-$ is strongly star-shaped at $p$, there also exists $y\in[x,p] \cap B^-$.
		One can then check that $x$ is in the convex hull of $\{x_{i}\}_i\cup\{y\}$, and hence that $x$ lies in $B^+$ by our assumption that the convex hull of $B^-$ is contained in $B^+$.
		This contradicts $x\in\partial B^+$.
		
		Since $G\cdot K=\Lambda_\Gamma \smallsetminus \{ p \}$, there is $g\in G$ such that $g\eta_i\in K$.
		Then $[gx,g\eta_i]$ does not intersect $B^-$, so $gx\in L'$, which concludes the proof.
	\end{proof}
	
	
	\begin{lemma}\label{lemm:key_ameliorationbisbisbis}
		Let $\Omega$ be a round convex subset of $\R\PP^d$,
		$\Gamma\subset\Aut(\Omega)$ discrete non-elementary,
		and $p\in \Lambda_\G$ be a bounded parabolic fixed point with stabilizer $\G_p$.
		Consider $G\subset \G_p$ a finite index subgroup and $B\subset\O$ a $G$-invariant strongly star-shaped open subset at $p$ such that the action of $G$ on $\partial B \cap \Cc(\LG)$ is cocompact.
		
		Then the action of $G$ on $\overline{\Cc(\LG)}\smallsetminus (B\cup\{p\})$ is cocompact.
	\end{lemma}
	\begin{proof}[Proof]
		This is an immediate consequence of \ref{item:stereobis} and the fact that the image of $\overline{\Cc(\LG)}\smallsetminus (B\cup\{p\})$ under the stereographic map from $\overline\O\smallsetminus (B\cup\{p\})$ to $\partial B$ is exactly $\partial B \cap \Cc(\LG)$.
	\end{proof}

\subsection{The general result}\label{sec:general}

In this section we prove a general result that \eqref{item:GF_on_boundary} is equivalent to a whole family of properties which encompasses \eqref{item:cusp_uniform} and \eqref{item:non_cusp_part}.
As a consequence, the equivalences \eqref{item:GF_on_boundary}$\Leftrightarrow$\eqref{item:cusp_uniform} and \eqref{item:GF_on_boundary}$\Leftrightarrow$\eqref{item:non_cusp_part} will be particular cases of the results of this section.

Let us recall the definition of $(\G,(\G_p)_p)$-precisely equivariant.

\begin{definition_english}\label{def:precisely equivariant}
	Let $\Omega\subset\R\PP^d$ be a round convex subset, $\Gamma\subset\Aut(\Omega)$ a discrete subgroup and $\Pc \subset\partial\O$ a $\G$-invariant subset.
	A \emph{$(\G,(\G_p)_p)$-equivariant} family $(B_p)_{p\in\Pc}$ of domains of $\O$ is a family of domains such that  $\g B_p = B_{\g p}$ for all $p\in\Pc$ and $\g\in\G$.
	
	It is called \emph{$(\G,(\G_p)_p)$-precisely equivariant} if moreover $\overline B_p \cap \overline B_q = \varnothing$ for all distinct $p\neq q\in\Pc$.
\end{definition_english}

\subsubsection{Cocompactness consequences of \eqref{item:GF_on_boundary}}

We can now state and prove one of the two main results of this section.
The proof is standard, see for instance \cite[Prop.\,1.10]{roblin_smf}.

\begin{proposition_english}\label{prop:gf implies cocpct}
	Let $\Omega$ be a round convex subset of $\R\PP^d$ and
	$\Gamma\subset\Aut(\Omega)$ discrete, non-elementary, and geometrically finite on $\partial\O$ (Assumption~\eqref{item:GF_on_boundary}).
	Let $\Pc \subset\Lambda_\G$ be the set of parabolic points, and denote by $\G_p\subset\G$ the stabilizer of each $p\in \Pc$.
	Consider a $(\G,(\G_p)_p)$-equivariant family $(B_p)_{p\in \Pc}$ of domains.
	Suppose that $\G_p$ acts cocompactly on $\overline{\Cc(\LG)}\smallsetminus(B_p\cup\{p\})$ for every $p\in\Pc$.
	Then the action of $\G$ on $\Cc(\LG) \smallsetminus \bigcup_p B_p$ is cocompact.
\end{proposition_english}

\begin{proof}[Proof]
	We fix a point $o \in \Omega$ and consider the (Dirichlet) domain:
	$$
	\Dc = \{ x \in \O \, |\, \forall \g \in \G, \,  d_\O (x,o)  \leqslant d_\O(x,\gamma o) \}.
	$$
	Recall that it is a closed subset of $\O$, and that the translates of $\Dc$ by $\G$ cover $\Omega$. Consider the closed subset $X = \Dc \cap \Cc(\LG) \smallsetminus  \bigcup_p B_p$ of $\O$.
	Let us show that $X$ is bounded.
	
	Assume it is not, then there exists a sequence $(x_n)_{n\in \N} \in X ^\N$ such that $x_n \to p \in \Lambda_\G$. 
	Lemma~\ref{lem:boundary_pt_of_fund_dom_are_not_conical} shows that $p$ is not a conical limit point.
	Hence,  $p$ is a bounded parabolic point since $\G \curvearrowright \partial \O$ is geometrically finite.

	\medskip
	
	By our assumption, up to extracting a subsequence, there exists $\g_n \in \G_p$ such that $\g_n (x_n) \to z \in \overline{\O}\smallsetminus\{p\}$, see Figure~\ref{fig:gf_implies_stuff}.
	In particular $\g_n \to + \infty$ and so $\g_n(o) \to p$, by Fact~\ref{lem:parabolic_transf_push_to_p}.
	\medskip
	
	Pick any $y$ in the interval $(z,p)$, which is contained in $\O$ since it is strictly convex.
	Since $\g_n[x_n,o]\to[z,p]$, we can find $y_n\in \g_n(x_n,o)\subset\O$ that converge to $y$.
	
	\begin{align*}
		d_\O (x_n , o ) & = d_\O (\g_n(x_n) , \g_n(o) ))\\
		& = d_\O (\g_n(x_n) , y_n ) + d_\O (y_n , \g_n(o) )\\
		& \geqslant d_\O (\g_n(x_n) , o ) - d_\O (o , y_n ) +
		d_\O (\g_n(o) , o )- d_\O (o , y_n ) 
	\end{align*}
	So:
	\begin{equation*}
		\underbrace{d_\O (x_n , o ) -  d_\O (\g_n(x_n) , o )}_{\leqslant 0 \,\,\, \textrm{since } x_n \in \Dc}   \geqslant
		- \underbrace{2d_\O (o , y_n )}_{\to d_\O(o,y)} + \underbrace{d_\O (\g_n(o) , o )}_{\to + \infty}.
	\end{equation*}
	Absurd.
\end{proof}

\begin{figure}
	\centering
	\definecolor{dcrutc}{rgb}{0.86,0.08,0.24}
	\begin{tikzpicture}[line cap=round,line join=round,>=triangle 45,x=1.0cm,y=1.0cm,scale=2.2]
		\clip(-2.85,-2.12) rectangle (2.86,1.95);
		\draw [rotate around={0:(0,0)}] (0,0) ellipse (2cm and 1.73cm);
		\draw [dotted] (-1.8,-0.43)-- (1.34,0.53);
		\draw [dotted] (-2,0)-- (1.26,0.3);
		\draw [dotted] (-1.73,0.36)-- (0.04,0.87);
		\begin{scriptsize}
			\fill [color=black] (-2,0) circle (0.5pt);
			\draw[color=black] (-2.1,0) node {$p$};
			\fill [color=black] (-1.73,0.36) circle (0.5pt);
			\draw[color=black] (-1.65,0.45) node {$x_n$};
			\fill [color=black] (0.04,0.87) circle (0.5pt);
			\draw[color=black] (0.09,0.95) node {$o$};
			\fill [color=black] (1.34,0.53) circle (0.5pt);
			\draw[color=black] (1.6,0.55) node {$\g_n(x_n)$};
			\fill [color=black] (1.26,0.3) circle (0.5pt);
			\draw[color=black] (1.31,0.2) node {$z$};
			\fill [color=black] (-1.8,-0.43) circle (0.5pt);
			\draw[color=black] (-1.55,-0.45) node {$\g_n(o)$};
			\fill [color=black] (-1.11,0.08) circle (0.5pt);
			\draw[color=black] (-1.07,0.17) node {$y$};
			\fill [color=black] (-0.97,-0.18) circle (0.5pt);
			\draw[color=black] (-0.92,-0.09) node {$y_n$};
		\end{scriptsize}
	\end{tikzpicture}
	\caption{Illustration of the proof of Proposition~\ref{prop:gf implies cocpct}}
	\label{fig:gf_implies_stuff}
\end{figure}

\subsubsection{\eqref{item:GF_on_boundary} as a consequence of cocompactness}

We now state the second main result of this section, which can be described as a converse to the first main result Proposition~\ref{prop:gf implies cocpct}.

\begin{proposition_english}\label{prop:cocpct implies gf}
	Let $\Omega$ be a round convex subset of $\R\PP^d$ and
	$\Gamma\subset\Aut(\Omega)$ discrete non-elementary.
	Consider a $\G$-invariant subset $\Pc\subset\LG$ and denote by $\G_p\subset\G$ the stabilizer of any $p\in\Pc$.
	Consider a $(\G,\G_p)$-precisely equivariant family $(B_p)_{p\in\Pc}$ of domains with $B_p$ strongly star-shaped at $p$.
	
	Suppose that the action of $\G$ on $\Cc(\LG)\smallsetminus \bigcup_p B_p$ is cocompact.
	
	Then $\G$ acts geometrically finitely on $\partial\O$ (Assumption~\eqref{item:GF_on_boundary}), $\Pc$ is the set of bounded parabolic points in $\LG$, and $\G_p$ acts cocompactly on $\partial B_p\cap\Cc(\LG)$ for every $p\in\Pc$.
\end{proposition_english}

\begin{proof}[Proof]
	Let $q \in \mathcal{P}$.
	Since $(B_p)_{p\in\Pc}$ is $(\G,(\G_p)_p)$-precisely equivariant,
	$$\left(\partial B_q \cap \Cc (\LG)\right)/\G_q\hookrightarrow\left(\Cc(\LG)\smallsetminus \bigcup_p B_p\right)/\G$$
	is an embedding with closed image.
	In particular, $(\partial B_q \cap \Cc (\LG))/\G_q$ is compact, in other words
	the action of $\G_q$ on $\partial B_q \cap \Cc (\LG)$ is cocompact.
	
	Moreover, $\LG \smallsetminus \{ q \}$ embeds $\G_q$-equivariantly in $\partial B_q\cap \Cc (\LG)$ since $B_q$ is strongly star-shaped at $q$ (see \ref{item:stereo}), so the action of $\G_q$ on $\LG \smallsetminus \{ q \}$ is proper and cocompact.
	This implies that $q$ is bounded parabolic.

	Let $q \in\LG\smallsetminus \mathcal{P}$.
	Consider $o\in \Cc (\LG)$.
	Note that the geodesic ray $[o,q[$ is not eventually contained in any $B_p$, for any $p \in \mathcal{P}$, since $\overline B_p\cap \partial\O=\{p\}$ as $B_p$ is star-shaped at $p$ (see \ref{item:dB in Omega}).
	Hence there exists $x_n \in [o,q[$ such that $x_n \to q$ and $x_n \in \Cc(\LG)\smallsetminus \bigcup_p B_p$. 
	So, by cocompactness of the action, up to extracting a subsequence, there exists $\g_n \in \G$ such that $\g_n(x_n) \to z \in \O$, which implies that $q$ is conical ($(\g_n^{-1}z)_n$ converge to $q$ while remaining at bounded distance from $[o,q)$).
\end{proof}

\subsection{\eqref{item:GF_on_boundary}$\Longleftrightarrow$ \eqref{item:cusp_uniform}}\label{sec:cusp_uniform}

Let $\Omega$ be a round convex subset of $\R\PP^d$.
In this section we recall the definition of horoballs and some basic facts, e.g.\ that horoballs are the images of $\Omega$ under a projective transformation, and hence are round convex subsets of $\R\PP^d$.
We then use Section~\ref{sec:star} to prove \eqref{item:GF_on_boundary}$\Longleftrightarrow$ \eqref{item:cusp_uniform}.

\subsubsection{Horoballs}\label{sec:horoballs}

We first give a very algebraic definition of horoballs, and then describe them more geometrically via Busemann functions, using
a result of Benoist.
See also \cite[\textsection 2.2]{CM2014finitude} and \cite[p.16]{CLT2015cvxisom}

\begin{definition_english}\label{def:horoball}
	Let $p\in\partial\Omega$ and $x\in\Omega$.
	Let $q\in\partial\Omega$ be such that $x\in[p,q]$.
	Consider a basis $v_1,\dots,v_{d+1}\in \R^{d+1}$ such that $p=[v_1]$, $q=[v_2]$, $x=[v_1+v_2]$, and $[v_i]\in T_p\partial\Omega$ for each $i\geq 3$.
	Then the horosphere $W\subset\Omega$ centered at $p$ passing through $x$ is the image $g\partial\Omega\smallsetminus\{p\}$ of $\partial\Omega\smallsetminus\{p\}$ under the projective transformation 
	$$g=\begin{pmatrix} 1 & 1 & 0 \\ 0 & 1 & 0 \\ 0 & 0 & \mathrm I_{d-1} \end{pmatrix}.$$
	Note that $g$ fixes any affine chart not containing $T_p\partial\O$, in which it acts as a translation in the direction of the line through $p$ and $q$, sending $q$ to $x$.
	
	The open horoball $H$ with boundary $W\cup\{p\}$ is $g\Omega$, which is a round convex subset of $\Omega$; in particular it is strongly star-shaped at $p$.
	Note that $T_p\partial H=T_p\partial\Omega$.
	
	This does not depend on the choice of $v_1,\dots,v_{d+1}$.
	Indeed, one can check that $\partial H\smallsetminus\{p\}$ is the set of $x'\in\Omega$ such that, if $x'\in[p,q']$ for  $q'\in\partial\Omega$, then the two lines $\overline{qq'}$ and $\overline{xx'}$ intersect in the hyperplane $T_p\partial\Omega$. 
\end{definition_english}

\begin{fact}[{\cite[\textsection 3.2.3-4 \& Fig.7]{CD1}}]
	For all $p\in\partial\O$ and $x,y \in \Omega$, 
	$$b_p(x,y):=\underset{z\to p}{\lim} \,\, d_\O(x,z) - d_\O (y,z)$$
	is well defined.
	
	Moreover, the horosphere centered at $p$ through any given $x\in\Omega$ is $\{y\in\O:b_p(x,y)=0\}$.
	The associated open horoball is $\{y\in\O:b_p(x,y)>0\}$.
\end{fact}

It is clear that projective transformations map horoballs to horoballs.
The following states that parabolic groups preserve each horoball centered at the point they fix.

\begin{fact}[{\cite[Th.\,3.3]{CM2014finitude},\ \cite[Prop.\,3.3]{CLT2015cvxisom}}]
	For all $p\in\partial\O$, $x\in\O$ and $\g\in\Aut(\O)$ preserving $p$, the translation length of $\g$ is exactly $|b_p(x,\g x)|$.
	
	In particular, if $\g$ is parabolic (or elliptic) then it preserves each horoball centered at $p$.
\end{fact}
\begin{proof}[Proof]
	Note that for every $y\in\O$ we have
	$$b_p(y,\g y) = b_p(y,x) + b_p(x,\g x) + b_p(\g x,\g y) = b_p(x,\g x) + b_p(y,x) + b_p(x,y) = b_p(x,\g x).$$
	Morever, by the triangle inequality $|b_p(y,\g y)|\leq d_\O(y,\g y)$ for every $y\in\O$.
	
	As a consequence, $|b_p(x,\g x)|$ is bounded from above by the translation length.
	
	If the translation length is zero then we are done.
	Otherwise, $\g$ is hyperbolic and fixes exactly two points of $p,q\in\partial\O$ (see \cite[\textsection 3.1]{CM2014flot}). 
	Then one can check that if $x\in[p,q]$ then $|b_p(x,\g x)|$ equals $d_\O(x,\g x)$, and hence equals the translation length of $\g$.
\end{proof}


%


Finally, the following result says that geometrically finite groups always admit precisely equivariant families of horoballs.
The result is not stated the same way in the reference, but the link is not hard to make.

\begin{fact}[{\cite[Lem.\,8.11]{blayac_zhu_ergo}}]\label{fact:precisely inv exist}
	Let $\Omega$ be a round convex subset of $\R\PP^d$.
	Let $\Gamma\subset\Aut(\Omega)$ be discrete non-elementary and act geometrically finitely on $\partial\O$.
	Then there exists a $(\G,(\G_p)_p)$-precisely equivariant family of horoballs centered at the parabolic points of $\Lambda_\G$.
\end{fact}

\subsubsection{Applications of Sections~\ref{sec:star} and \ref{sec:general}}

\begin{lemma}\label{lemm:key_amelioration}
	Let $\Omega$ be a round convex subset of $\R\PP^d$.
	Let $\Gamma\subset\Aut(\Omega)$ be discrete non-elementary.
	Let $p\in \Lambda_\G$ be a bounded parabolic fixed point. For any open horoball $H_p$ centered at $p$, the action of $\G_p$ on $\overline{\Cc(\LG)}\smallsetminus(H_p\cup\{p\})$ is cocompact.
\end{lemma}

\begin{proof}[Proof]
	This is an immediate corollary of Lemmas~\ref{lemm:key_ameliorationbisbis} and \ref{lemm:key_ameliorationbisbisbis}, using the fact that horoballs are round convex with their center in their boundary (Definition~\ref{def:horoball}) and invariant under the associated parabolic subgroups.
\end{proof}


\begin{proof}[Proof of \eqref{item:GF_on_boundary}$\Leftrightarrow$ \eqref{item:cusp_uniform}]
	This is an immediate corollary of  Lemma~\ref{lemm:key_amelioration}, Propositions~\ref{prop:gf implies cocpct} and \ref{prop:cocpct implies gf}, and Fact~\ref{fact:precisely inv exist}.
\end{proof}

\subsection{\eqref{item:GF_on_boundary} $\Longleftrightarrow$ \eqref{item:non_cusp_part}} \label{sec:non_cuspidal_part}

In this section we recall the definition of the thin part of convex projective manifolds and some basic facts, e.g.\ that the components of the thin part in $\O$ are star-shaped.
We then prove that they also satisfy the extra convexity condition of Lemma~\ref{lemm:key_ameliorationbisbis}.
We then use Sections~\ref{sec:star} and \ref{sec:general} to prove \eqref{item:GF_on_boundary} $\Longleftrightarrow$ \eqref{item:non_cusp_part}.

\subsubsection{Thick-thin decomposition}\label{sec:thickthin}

Let $\Omega$ be a  convex domain of $\R\PP^d$, $\G\subset \Aut(\O)$ discrete and $\epsilon>0$.
We use the following notation.
\begin{enumerate}
	\item $S_\varepsilon (x) := \{ \g \in \G \, |\, d_\O (x, \g x) < \varepsilon \}$ for $x\in\Omega$;
	\item $\G_\varepsilon (x) := \langle S_\varepsilon (x) \rangle$ (the subgroup generated by $S_\varepsilon (x)$) for $x\in\Omega$;
	\item $\O_\varepsilon (\G) := \{ x \in \O \, | \, \G_\varepsilon(x) \textrm{ is infinite} \}$ is the $\varepsilon$-thin part of $\O$, its complement is the $\varepsilon$-thick part;
ilon	\item $\mathcal{O}_\epsilon(A) := \{x\in \O \, | \, d_\O(x,\gamma x)<\epsilon, \ \forall \gamma\in A\}$ for $A\subset \G$;
	%
	%
	%
	%

	%
	
\end{enumerate}

Let us recall the Margulis lemma for convex projective geometry.

\begin{fact}[\cite{Margu_ludo_mikl} \& \cite{CLT2015cvxisom}]\label{lem:margulis}
	There exists $\varepsilon_0 >0$ which only depends on the dimension $d$, such that for every convex domain $\O$ of $\R\PP^d$, any $\G\subset \Aut(\O)$ discrete, any $0 < \epsilon \leqslant \epsilon_0$, any $x \in \O$, the group $\G_\varepsilon (x)$ is virtually nilpotent.
\end{fact}

Assuming $\Omega$ is a round convex domain, one can use the previous Margulis lemma to obtain a thick-thin decomposition, and more precisely a nice decomposition of the thin part (see \cite[Lem.~6.2]{CM2014finitude}).

If $\varepsilon< \varepsilon_0$, then the thin part $\O_\varepsilon (\G)$ is the disjoint union of $(\O_\varepsilon (G))_G$ (in fact the \emph{closures} are pairwise disjoint), where $G$ runs over the maximal parabolic subgroups of $\G$ and the centralizers of hyperbolic elements of translation length less than $\varepsilon$.

The \emph{$\varepsilon$-noncuspidal part} is the complement of the union of $(\O_\varepsilon (G))_G$, where $G$ runs over the parabolic subgroups of $\G$.

\begin{fact}[{\cite[Lem.\,6.2.1 \& Cor.\,3.16]{CM2014finitude}}]\label{lem:precisely_equivariant}
	If $\varepsilon< \varepsilon_0$ and $\Pc\subset \LG$ is the set of parabolic points, then $({\O_\epsilon(\G_p)})_{p\in\Pc}$ is $(\G,(\G_p)_p)$-precisely equivariant.
\end{fact}
%

\subsubsection{Star-shapedness and the weak convexity condition}

Here we check that the components of the thin part, as well as the domains of the form $\mathcal{O}_\varepsilon(A)$ defined in the previous section, are strongly star-shaped and satisfy the extra weak convexity condition in Lemma~\ref{lemm:key_ameliorationbisbis}.

\begin{fact}\label{fact:thinpart starshaped}
	Let $\Omega$ be a round convex subset of $\R\PP^d$ and $\G_p\subset\Aut(\Omega)$ a discrete infinite parabolic subgroup fixing $p\in\partial\O$.
	
	Then  $\O_\varepsilon(\G_p)$ is strongly star-shaped at $p$ for any $\varepsilon$ (see Definition~\ref{def:strongly starshaped}).
	
	Moreover, $\mathcal{O}_\varepsilon(A)$ is also strongly star-shaped at $p$ for any finite $A\subset\G_p$ that generates an infinite group. 
\end{fact}

Note also that $(\partial\O_\varepsilon(\G_p))_\varepsilon$ and  $(\partial\mathcal{O}_\varepsilon(A))_\varepsilon$ foliate $\O$.

The above fact is a consequence of the following elementary result (which uses the fact that $\O$ is round).

\begin{fact}[{\cite[Lem.\,3.4]{CD1}}]\label{fact:moving_goes_to_zero}
	Let $\Omega$ be a round convex subset of $\R\PP^d$ and $\g\in\Aut(\O)$ a parabolic or elliptic transformation
	fixing $p\in\partial\O$.
	Consider a straight geodesic $(p_t)_{t\in\R}\subset\O$ going to $p$ as $t\to\infty$.
	
	Then either $\g$ fixes the geodesic or $t\mapsto d_\O(p_t,\g p_t)$ is decreasing from $\infty$ to $0$.
\end{fact}

%

%

Before we discuss the weak convexity condition needed in Lemma~\ref{lemm:key_ameliorationbisbis}, let us discuss briefly the link between $\O_\varepsilon(\G_p)$ and $\mathcal{O}_\varepsilon(A)$, where $\Omega\subset\R\PP^d$ is round convex, $\G_p\subset\Aut(\Omega)$ is discrete infinite parabolic and fix $p\in\partial\O$, the subset $A\subset\G_p$ is finite and generates an infinite group, and $\varepsilon>0$.

\begin{enumerate}
	\item $\mathcal{O}_\varepsilon(A)\subset\O_\varepsilon(\G_p)$.
	\item $\mathcal{O}_\varepsilon(A)$ is not necessarily $\G_p$-invariant; it is if $A$ is invariant under conjugacy.
	\item $\G_p$ being virtually nilpotent, it admits a torsion-free nilpotent finite-index subgroup $G\subset\G_p$, whose center has a nontrivial element $g\in G$; then $\mathcal{O}_\varepsilon(g)$ is $G$-invariant.
\end{enumerate}

Let us now turn to the weak convexity condition needed in Lemma~\ref{lemm:key_ameliorationbisbis}.
We will need the following estimate on the Hilbert metric, which gives control on the distance between two segments via the distance between the endpoints.

\begin{fact}[{\cite[Lem.\,8.3]{crampon} \& \cite[Lem.\,5.2]{topmixing}}]\label{fact:crampon}
	Let $\O$ be a convex domain.
	Consider two segments $[x,y],[x',y']\subset\O$ and two points $z\in[x,y]$ and $z'\in[x',y']$ such that $\tfrac{d_\O(x,z)}{d_\O(x,y)}=\tfrac{d_\O(x',z')}{d_\O(x',y')}$.
	Then
	$$ d_\O(z,z') \leq d_\O(x,x') + d_\O(y,y'). $$
\end{fact}

\begin{corollary}\label{coro:crampon}
	Let $\O$ be a convex domain and $y \in \partial \Omega$. Consider two segments $[x,y),[x',y)\subset\O$ and two points $z\in[x,y)$ and $z'\in[x',y)$ such that $d_\Omega (x,z) = d_\Omega (x',z')$. Then
	$$ d_\O(z,z') \leq d_\O(x,x'). $$
\end{corollary}

\begin{proof}[Proof]
	Let $(y_n)_n$ a sequence of points of the segment $(z,y)$ converging to $y$. Let $(z'_n)_n$ be the sequence of points of $(y_n,x')$ such that: 
	$\tfrac{d_\O(x,z)}{d_\O(x,y_n)}=\tfrac{d_\O(x',z'_n)}{d_\O(x',y_n)}$. By Fact~\ref{fact:crampon}, $ d_\O(z,z'_n) \leq d_\O(x,x')$. Hence, it is enough to show that $(z'_n)_n$ converges to $z'$.
	
	The ratio $\tfrac{d_\O(x,y_n)}{d_\O(x',y_n)}$ converges to $1$ since $| d_\O(x,y_n) - d_\O(x',y_n) | \leqslant d_\O(x,x')$. So, $d_\O(x',z'_n)$ converges to $d_\O(x,z) = d_\O(x',z')$, giving that $z'_n \to z'$ since $z'_n$ is on the segment $(y_n,x')$ which converges to the segment $(y,x')$.
\end{proof}

\begin{corollary}\label{cor:Oeps almost convex}
	Let $\Omega$ be a convex domain of $\R\PP^d$.
	For all $\epsilon>0$ and $A\subset \Aut(\O)$,
	the convex hull of $\mathcal{O}_\epsilon (A)$ is contained in $\mathcal{O}_{(d+1)\epsilon}(A)$.
\end{corollary}
\begin{proof}[Proof]
	It suffices to prove by induction on $k\geq 1$ that any convex combination of $k$ points of $\mathcal{O}_\epsilon (A)$ is in $\mathcal{O}_{k\epsilon}(A)$.
	
	If $k=1$ then this is obvious.
	
	Suppose $k\geq 2$ and the property we want to prove for convex combinations of fewer than $k$ points.
	Let $z$ be a convex combination of $k$ points of $\mathcal{O}_\epsilon(A)$.
	Then $z\in[x,y]$ where $x\in \mathcal{O}_\epsilon(A)$ and $y$ is a convex combination of $k-1$ points of $\mathcal{O}_\epsilon(A)$.
	By the inductive hypothesis we have $y\in \mathcal{O}_{(k-1)\epsilon}(A)$.
	
	Consider $\gamma \in A$, and let us check that $d_\O(z,\gamma z)<k\epsilon$.
	We have $\gamma z\in[\gamma x,\gamma y]$ and $\tfrac{d_\O(x,z)}{d_\O(x,y)}=\tfrac{d_\O(\gamma x,\gamma z)}{d_\O(\gamma x,\gamma y)}$, so by Fact~\ref{fact:crampon}
	\begin{equation*} d_\O(z,\gamma z) \leq d_\O(x,\gamma x) + d_\O(y,\gamma y) < \epsilon + (k-1)\epsilon = k\epsilon.\qedhere \end{equation*}
\end{proof}

\subsubsection{Applications of Sections~\ref{sec:star} and \ref{sec:general} bis}

We now apply the results from previous sections to establish  \eqref{item:GF_on_boundary} $\Leftrightarrow$ \eqref{item:non_cusp_part}.

\begin{lemma}\label{lemm:key_ameliorationbis}
	Let $\Omega$ be a round convex subset of $\R\PP^d$,
	$\Gamma\subset\Aut(\Omega)$ discrete non-elementary,
	and $p\in \Lambda_\G$ be a bounded parabolic fixed point with stabilizer $\G_p$.
	Consider $G\subset \G_p$ a finite index subgroup and $A\subset G$ finite, that generates an infinite subgroup, and { invariant under conjugation by elements of $G$}.
	Then the action of $G$ on $\overline{\Cc(\LG)}\smallsetminus(\mathcal{O}_\epsilon(A)\cup\{p\})$ is cocompact for any $\epsilon$.
\end{lemma}
Note that $G$ preserves $\mathcal{O}_\epsilon(A)$ because $A$ is invariant under conjugation.
\begin{proof}[Proof]
	This is an immediate corollary of Fact~\ref{fact:thinpart starshaped}, Corollary~\ref{cor:Oeps almost convex} and Lemmas~\ref{lemm:key_ameliorationbisbis} and \ref{lemm:key_ameliorationbisbisbis}.
\end{proof}

\begin{corollary}\label{cor:key amelioration}
	Let $\Omega$ be a round convex subset of $\R\PP^d$,
	$\Gamma\subset\Aut(\Omega)$ discrete non-elementary,
	and $p\in \Lambda_\G$ be a bounded parabolic fixed point with stabilizer $\G_p$.
	Then the action of $\G_p$ on $\overline{\Cc(\LG)}\smallsetminus(\O_\epsilon(\G_p)\cup\{p\})$ is cocompact for any $\epsilon$.
\end{corollary}

\begin{proof}[Proof]
	By Fact~\ref{lem:margulis}, we can find  a finite-index torsion-free nilpotent subgroup $G\subset\G_p$.
	Since $G$ is nilpotent it has a nontrivial element $g$ in the center.
	By definition, $\mathcal{O}_\epsilon (g)\subset \O_\epsilon (\G_p)$ and hence 
	$$\overline{\Cc(\LG)}\smallsetminus(\O_\epsilon(\G_p)\cup\{p\})\subset \overline{\Cc(\LG)}\smallsetminus(\mathcal{O}_\epsilon(g)\cup\{p\}).$$
	We conclude using Lemma~\ref{lemm:key_ameliorationbis}.
\end{proof}

\begin{proof}[Proof of \eqref{item:GF_on_boundary} $\Leftrightarrow$ \eqref{item:non_cusp_part}]
	This is a corollary of Corollary~\ref{cor:key amelioration}, Fact~\ref{lem:precisely_equivariant} and Propositions~\ref{prop:gf implies cocpct} and \ref{prop:cocpct implies gf}.
\end{proof}

\subsection{Counterexample to \eqref{item:GF_on_Omega} $\Rightarrow$ \eqref{item:gf_cusps_lobat}}\label{sec:contrex}


In this section we use a reducible representation of $\SL_2(\R)$ to construct an example of group $\G$ satisfying \eqref{item:GF_on_Omega} but not  \eqref{item:gf_cusps_lobat}.
Let $\tau:\SL_2 (\R)\to\SL_3(\R)$ be an irreducible representation.
Consider the reducible semisimple representation $\rho:\SL_2(\R)\to\SL_5(\R)$ such that for any $g\in \SL_2(\R)$ we have
$$\rho(g)=\begin{pmatrix}\tau(g)&0\\0&g\end{pmatrix}.$$

Note that by definition $\rho(\SL_2(\R))$ preserves the supplementary subspaces $\R^3\times\{0\}$ and $\{0\}\times\R^2$ of $\R^5$.
Moreover in $\R^3\times\{0\}$ it preserves a properly convex (relatively) open cone $C=C\times\{0\}$, and of course it also preserves the open convex cone $C\times\R^2$ which is not properly convex.

The projectivisation $D=\PP(C)$ is a 2-dimensional properly convex disc and $\O_{\max}=\PP(C\times\R^2)$ is an open  convex subset of $\R\PP^4$ which is contained in some affine chart, where it is $D\times \R^2$.
Their relative boundaries are denoted by $\partial D$ and $\partial\O_{\max}$.

The following result describes all $\rho(\SL_2(\R))$-invariant convex domains.

\begin{fact}\label{fait:cvx sl2inv}
	We have the following.
	\begin{enumerate}
		\item  The proximal limit set of $\rho(\SL_2(\R))$ is $\partial D$.
		\item $\rho(\SL_2(\R))$ acts properly discontinuously on $\O_{\max}$; more precisely the orbit of any compact set accumulates on all of $\partial D$ (and only there).
		
		\item  For any $x\in \O_{\max}\smallsetminus D$, the stabilizer is trivial.
		
		\item For any $x\in \O_{\max}\smallsetminus D$, the disjoint union $\partial D\sqcup \rho(\SL_2(\R))\cdot x$ is the boundary of an invariant round convex domain $\O\subset\O_{\max}$. Moreover, every invariant convex domain $\O'$ is obtained in this way. 
	\end{enumerate}
\end{fact}

\begin{proof}[Proof]
	\begin{enumerate}
		\item Let  $g \in \SL_2 (\R)$ be proximal whose (real) eigenvalues have norm $\lambda>\nicefrac1\lambda$.
		Then the norms of the eigenvalues of $\tau(g)$ are $\lambda^2>1>\nicefrac1{\lambda^2}$.
		Thus the biggest norm of eigenvalues of $\rho(g)$ is $\lambda^2$, and the corresponding eigenline is exactly the eigenline of $\tau(g)$, embedded in $\R^5$ via $\R^3\to\R^3\times\{0\}$.
		This concludes the proof since it is well known that the proximal limit set of $\tau(\SL_2(\R))$ in $\R\PP^2$ is $\partial D$.
		
		\item Let $x \in \O_\mathrm{max}$ and $(g_n)_n$ a sequence of element of $\SL_2 (\R)$ such that $g_n \to \infty$. Let $\|\cdot \|$ be the induced norm on the space of $m \times m$ real matrix, by the canonical scalar product on $\R^m$, there exists $C > 1$ such that:
		$$
		\forall g \in \SL_2 (\R), \qquad
		C^{-1} \| g \|^2 \leqslant \|\tau (g) \| \leqslant C \| g \|^2
		$$
		For another constant $C_2 > 1$, one has:
		$$
		\forall g \in \SL_2 (\R), \qquad
		C_2^{-1} \| g \|^2 \leqslant \|\rho (g) \| \leqslant C_2 \| g \|^2
		$$
		Hence, 
		up to extraction, we may assume that $\tau(g_n)/\| \tau (g_n) \|$ converge to a rank-one matrix $T \in \mathcal{M}_3 (\R)$ such that $\mathrm{Im} (T) \subset \partial D$ and $\ker (T) \cap \partial D$ is a singleton. Thanks to the estimate, the matrix $\rho(g_n)/\|\rho (g_n) \|$ converges to the matrix:
		$$
		\begin{psmallmatrix}
			T & & \\
			& 0 & 0 \\
			& 0 & 0\\
		\end{psmallmatrix}
		$$
		Hence, $\rho (g_n) \cdot x$ converges to $T (x) \in \partial D$.
		
		\item Let $[(x,y)] \in \O_{\max} \smallsetminus D$ with $x\in \R^3$ and $y \in \R^2$ and consider $g \in \SL_2 (\R)$ such that $\rho (g) \cdot [(x,y)] = [(x,y)]$. 
		We may assume that $\Stab_{\tau(\SL_2(\R)}([x]) = \tau(\mathrm{SO}_2 (\R))$, hence we get that $g \in \mathrm{SO}_2 (\R)$. 
		The only two rotations of $\R^2$ that preserves a line are $\mathrm{Id}$ and $-\mathrm{Id}$, so we get that $g = \pm \mathrm{Id}$. 
		But, the fixed point set of $\rho (- \mathrm{Id})$ is $\PP(\R^3 \times \{ 0\})  \cup \PP(\{ 0\} \times \R^2)$, hence $g = \mathrm{Id}$.
		
		\item Take $x \in \O_{\max} \smallsetminus D$.
		The interior of the convex hull of $\rho(\SL_2(\R))\cdot x$ in $\O_{\max}$ is an invariant convex domain $\O$.
		The orbit $\rho(\SL_2(R))\cdot x$ of $x$  accumulates on $\partial D$ and only there (by (2.)), which implies $\rho(\SL_2(R))\cdot x \sqcup \partial D$ is compact and $\overline\O$ is the convex hull of $\rho(\SL_2(R))\cdot x \sqcup \partial D$ and is properly convex.
		
		Thus the extremal points of $\overline\O$ are in $\rho(\SL_2(R))\cdot x \sqcup \partial D$.
		They cannot be all in $\partial D$, otherwise $\overline\O\subset\overline D$ which contradicts $x\not\in \overline D$.
		Thus at least one point of $\rho(\SL_2(R))\cdot x$ must be extremal, and then all points of $\rho(\SL_2(R))\cdot x$ are extremal since $\rho(\SL_2(R))$ maps extremal points to extremal points.
		Moreover one can also check that any point $p\in\partial D$ is extremal.
		(Otherwise there would be $a,b\in\partial\O$ such that $p\in (a,b)$: if $a,b\in\partial D$ then $p\in D$, absurd, and if one of $a,b$ is in $\rho(\SL_2(R))\cdot x$ then $p\in\O_{\max}$, absurd too.)
		
		We proved that the set of extremal points is exactly $\rho(\SL_2(R))\cdot x \sqcup \partial D$, which is in particular contained in $\partial \O$.
		The orbit $\rho(\SL_2(R))\cdot x$ of $x$ is open in $\partial \O$ by Brouwer's invariance of the domain theorem, thanks to (3.).
		The orbit accumulates on $\partial D$ and only there (by (2.)), hence is closed in $\partial \O \smallsetminus \partial D$.
		A classical result of topology shows that $\partial \O \smallsetminus \partial D$ is connected (see e.g.\ \cite[Prop.\,2.B.1.b]{Hatcher}).
		Hence, $\partial \O = \rho(\SL_2(R))\cdot x \sqcup \partial D$, and $\O$ is strictly convex since all points of the boundary are extremal.
		
		Let $\O'\subset\R\PP^4$ be an invariant properly convex open set.
		Then $\partial\O'$ contains the proximal limit set of $\rho(\SL_2(\R))$, i.e.\ $\partial D$.
		By convexity $\overline\O'$ must then contain $D$, and $\O'$ intersects $\O_{\max}$.
		However, $\overline\O'$ cannot intersect $\partial\O_{\max} \smallsetminus \partial D$.
		If it did then by applying powers of a suitable element of $\rho(\SL_2(\R))$ there would be a point of $\PP(\{0\}\times\R^2)$ in $\overline\O'$, and hence there would in fact be the whole $\PP(\{0\}\times\R^2)$ inside $\overline \O'$, which would compromise $\O'$'s proper convexity.
		As a consequence, $\partial\O'$ intersects $\O_{\max}$ at some point, say at the point $x$.
		Then $\O\subset\O'$, and by our results above we get $\partial \O= \rho(\SL_2(R))\cdot x \sqcup \partial D\subset\partial \O'$, which implies at once that $\O'=\O$.
		
		
		
		Finally, the dual representation $\rho^\ast$ is conjugated to $\rho$, hence the dual convex of $\O$ is strictly convex too, hence $\O$ has $\Cc^1$-boundary.\qedhere
	\end{enumerate}
\end{proof}

Next we prove that the image under $\rho$ of a geometrically finite subgroup of $\SL_2(\R)$ satisfies \eqref{item:GF_on_Omega} but not  \eqref{item:gf_cusps_lobat}.

\begin{proposition_english}
	For any $\rho(\SL_2(\R))$-invariant round convex domain $\O\subset\O_{\max}$, for any discrete subgroup  $\G\subset \SL_2(\R)$, if $\G$ is finitely generated (which means geometrically finite in the classical sense), then $\rho(\G)$ acts geometrically finitely on $\O$, but no parabolic subgroup is conjugate into $\mathrm{O}_{4,1}(\R)$ (even though all parabolic points are uniformly bounded).
\end{proposition_english}

\begin{proof}[Proof]
	By Fact~\ref{fait:cvx sl2inv}, the proximal limit set $\LG$ of $\rho(\G)$ is contained in $\partial D$, and the convex hull $\Cc (\LG)$ is contained in $D$, which is, we recall, isometric to the Poincar\'e disc.
	This implies $\G$ acts geometrically finitely on $\O$.
	
	Indeed every point $p$ of $\LG$ corresponds to a point $q$ of the limit set of $\G$ acting on $\Hb^2$.
	If $q$ is conical (there is $(\g_n)_n\subset \G$ such that $(\g_n o)_n$ converges to $q$ while remaining at bounded distance from $[o,q)$, for $o\in \Hb^2$) then $p$ is conical too.
	If $q$ is bounded parabolic for the action of $\G$ on $\Hb^2$ then the stabiliser $\G_q$ acts cocompactly on $\partial \Hb^2 \smallsetminus \{q\}$, hence $\rho(\G_q)$, which is the stabiliser of $p$, acts cocompactly on $\partial D \smallsetminus \{p\}$, which contains the stereographic projection of $\Cc(\LG)$, hence $p$ is uniformly bounded parabolic.
	
	Parabolic subgroups of $\rho(\G)$ are virtually conjugate to the group generated by the following matrix,
	$$ 
	\begin{psmallmatrix}
		1 & 1 & \nicefrac12 &&\\
		& 1 & 1 &&\\
		&& 1 &&\\
		&&& 1 & 1 \\
		&&&& 1
	\end{psmallmatrix}
	$$
	 which is not conjugate into $\mathrm{O}_{4,1}(\R)$.\qedhere
\end{proof}

\subsection{Under \eqref{item:generic}, uniformly bounded cusp groups are conjugate into $\mathrm{O}_{d,1}(\R)$}\label{sec:GF imp HC}

In this section, we prove that under the genericity assumption \eqref{item:generic}, the stabilisers of uniformly parabolic points are conjugate into  $\mathrm{O}_{d,1}(\R)$.
In particular, this establish the implication  (\eqref{item:GF_on_Omega}\&\eqref{item:generic})$\Rightarrow$\eqref{item:gf_cusps_lobat}.

\begin{proposition_english}\label{prop:oscultation}
	Let $\Omega$ be a round convex subset of $\R\PP^d$,
	$\Gamma\subset\Aut(\Omega)$ discrete non-elementary,
	and $p\in \Lambda_\G$ be a uniformly bounded parabolic fixed point with stabilizer $\G_p$.
	
	Suppose that the limit set $\LG$ spans the whole $\R\PP^d$, or that its dual spans the  dual of $\R\PP^d$.
	
	Then $\G_p$ is conjugate to a parabolic subgroup of $\mathrm{O}_{d,1}(\R)$.
	
	Moreover it preserves a projective subspace $\R\PP^{r+1}\subset \R\PP^d$ where $r$ is the rank of $\G_p$, that contains $p$ and intersects $\O$, and preserves ellipsoids $\Ec^{\mathrm{int}}\subset \Ec^{\mathrm{ext}}$ such that
	$$\Ec^{\mathrm{int}}\cap \mathrm{Cone}(p,\Cc(\LG)) \subset \O\cap \mathrm{Cone}(p,\Cc(\LG))\subset \Ec^{\mathrm{ext}}\cap \mathrm{Cone}(p,\Cc(\LG)),$$
	where $\mathrm{Cone}(p,\Cc(\LG))$ is the union of the lines through $p$ and a point of $\Cc(\LG)$.
\end{proposition_english}
\begin{proof}[Proof]
	We can assume that $\LG$ spans $\R\PP^d$ since the other case is dual.
	
	Let $\Ab^{d-1}$ be the affine chart of $\PP(\R^{d+1}/p) \smallsetminus \PP(T_p\partial\O/p)$, on which $\G_p$ acts properly discontinuously by affine transformation (see Fact~\ref{lem:parabolic_transf_push_to_p}), and preserves and acts cocompactly on the closed convex projection $K$ of $\Cc(\LG)$ (since $p$ is uniformly bounded).
	
	Let $\Ab^{r}\subset \Ab^{d-1}$ be a maximal affine subspace contained in $K$ and $K'\subset \Ab^{d-1}/\Ab^r$ the projection of $K$, which can be thought as the set of maximal affine subspaces contained in $K$.
	Note that $K'\subset \Ab^{d-1}/\Ab^r$ does not contain any line, and
	that $K$ is isomorphic to $\Ab^r\times K'$.
	
	Observe that $\G_p$ acting cocompactly on $K$ implies that $K'$ must be compact: indeed if it were not then $K'$ would be homeomorphic to a halfspace, and so would be $K$, but no halfspace can be acted on properly discontinuously and cocompactly by a discrete group.
	(If a group $G$ acts properly discontinuously and cocompactly on $\R^n\times \R_{\geq 0}$ then it acts properly discontinuously and cocompactly on both the boundary $\R^n\times\{0\}$ and the double $\R^{n+1}$, which is impossible.)
	
	Moreover, $K'$ has nonempty interior by our assumption that $\LG$ spans $\R\PP^d$.
	
	The group $\G_p$ acts on $\Ab^{d-1}/\Ab^r$ by affine transformations and preserves the compact convex subset $K'$ with nonempty interior, so $\G_p$ must fix the barycenter of $K'$ which is in its interior, and $\G_p$ must preserve some Euclidean structure on $\Ab^{d-1}/\Ab^r$.
	
	This barycenter lifts to a $\G_p$-invariant maximal affine subspace of $K$, which we assume to be $\Ab^r$ without loss of generality, and on which the action of $\G_p$ is cocompact.
	Then $\Ab^r$ lifts to a $(r+1)$-dimensional $\G_p$-invariant subspace of $\R\PP^d$ which contains $p$ and intersects $\O$.
	Up to changing basis we assume that this subspace is $\R\PP^{r+1}=\PP(\R^{r+2}\times\{0\})\subset \R\PP^d=\PP(\R^{d+1})$.
	
	The intersection $\O'=\O\cap \R\PP^{r+1}$ is $\G_p$-invariant, and $\G_p$ acts properly discontinuously and cocompactly on $\partial \O'\smallsetminus\{p\}$ since we have an equivariant identification with $\Ab^r$ via the stereographic projection.
	
	By \cite[Th.\,7.14]{CM2014finitude} this implies that the restriction of $\G_p$ to $\R^{r+2}=\R^{r+2}\times\{0\}\subset\R^{d+1}$ is conjugate to a parabolic subgroup of $\mathrm{O}_{r+1,1}(\R)$ of rank $r$; up to conjugating everything we assume that this restriction is contained in $\mathrm{O}_{r+1,1}(\R)$.
	By Bieberbach's Theorem (see e.g.\ \cite[Th.\,5.4.4]{Ratcliffe}), up to changing the basis $\mathcal{B}$ of $\R^{r+2}$, the group $\G_p$ has a finite-index normal subgroup isomorphic to $\Z^r$ such that the restriction of any $k\in\Z^r$ to $\R^{r+2}$ acts by
	$$
	\begin{pmatrix}
		1 & k^t & \frac {||k||^2}2 \\
		0 & I_{r-1} & k\\
		0 & 0 & 1
	\end{pmatrix}.
	$$
	
	The fact that $\G_p$ preserves an Euclidean structure on $\Ab^{d-1}/\Ab^r$ means that its action on $\R^{d-r-1}=\R^{d+1}/\R^{r+2}$ preserves an inner product, say the standard one.
	
	To prove that $\G_p$ is conjugate to a parabolic subgroup of $\mathrm{O}_{d,1}(\R)$ it suffices to find a $\G_p$-invariant subspace of $\R^{d+1}$ which is supplementary to $\R^{r+2}$.
	The set $E$ of subspaces supplementary to $\R^{r+2}$ is an affine space on which $\G_p$ acts by affine transformations.
	Thus it suffices to check that $\Z^r\subset \G_p$ preserves such a subspace, i.e.\ fixes a point of $E$.
	Indeed, since $\Z^r$ is a normal subgroup of $\G_p$, the subspace $E'\subset E$ of $\Z^r$-fixed points is $\G_p$-invariant.
	As $\Z^r$ acts trivially on $E'$, the $\G_p$-action descends to an affine action of $\G_p/\Z^r$ which is a finite group, and hence has a fixed point.
	
	To write the matrices, we first choose $\mathcal{B}$ for the $(r+2)$ first elements of our basis. 
	Then we choose the remaining $(d-r-1)$ elements of the basis of $\R^{d+1}$ in a lift of $T_p \partial \O$ in such a way that: an element $k\in\Z^r$ acts on $\R^{d+2}$ by
	$$
	\begin{pmatrix}
		1 & k^t & \frac {||k||^2}2  & D_k\\
		0 & I_{r-1} & k & C_k\\
		0 & 0 & 1 & B_k\\
		0 & 0 & 0 & A_k
	\end{pmatrix},
	$$
	where $A_k$ is an orthogonal matrix.
	Since those last $(d-r-1)$ elements were chosen in a lift of $T_p \partial \O$, we get that $B_k=0$.


	Now if we put the $(r+2)$-th element of the basis of $\R^{d+1}$ into last position, then the first $d$ vectors form a basis of the lift $\R^d$ of $T_p\partial\O$, and the new action of $k\in\Z^r$ will be given by
	$$
	\begin{pmatrix}
		1 & k^t   & D_k & \frac {||k||^2}2\\
		0 & I_{r-1} & C_k & k\\
		0 & 0 & A_k&0\\
		0 & 0 & 0 & 1
	\end{pmatrix},
	$$
	and our goal is to arrange the elements of the basis from $(r+2)$-th to $d$-th so that $C_k=0$ and $D_k=0$ (and so that those vectors are still in $\R^{d}$, which is the lift of $T_p\partial\O$).
	
	We can diagonalise simultaneously all $A_k$ for $k\in\Z^r$, in the complex field.
	This gives us 
	\begin{itemize}
		\item  $v_1,\dots,v_\alpha\in \R^d/\R^{r+1}$ which are real-eigenvectors for all $A_k$'s with eigenvalue $1$,
		\item $w_1,\dots,w_\beta\in\R^d/\R^{r+1}$ eigenvectors such that the eigenvalue of $A_k$ for $w_j$ is $(-1)^{k\cdot \epsilon_j}$ for some $\epsilon_j\in \Z^r\smallsetminus 2\Z^r$,
		\item and  $P_1,\dots ,P_\gamma\subset \R^d/\R^{r+1}$ invariant planes such that each $A_k$ acts as a rotation on $P_j$ with angle $k\cdot \theta_j$ for some $\theta_j\in\R^r\smallsetminus \pi \Z^r$.
	\end{itemize}
	Using this basis, the new action of $k\in\Z^r$ is given by
	$$\left(
	\begin{array}{*{12}c}
		1 & k^t     & d^1_k & \cdots & d^\alpha_k & \bar d^1_k               & \cdots     & \bar d^\beta_k               & \tilde d^1_k        & \cdots & \tilde d^\gamma_k        & \frac {||k||^2}2 \\
		& I_{r-1} & c^1_k & \cdots & c^\alpha_k & \bar c^1_k               & \cdots     & \bar c^\beta_k               & \tilde c^1_k        & \cdots & \tilde c^\gamma_k        & k                \\
		&         & 1     &        &            &                          &            &                              &                     &        &                          &                  \\
		&         &       & \ddots &            &                          &            &                              &                     &        &                          &                  \\
		&         &       &        & 1          &                          &            &                              &                     &        &                          &                  \\
		&         &       &        &            & (-1)^{k\cdot \epsilon_1} &            &                              &                     &        &                          &                  \\
		&         &       &        &            &                          & \ddots     &                              &                     &        &                          &                  \\
		&         &       &        &            &                          &            & (-1)^{k\cdot \epsilon_\beta} &                     &        &                          &                  \\
		&         &       &        &            &                          &            &                              & R_{k\cdot \theta_1} &        &                          &                  \\
		&         &       &        &            &                          &            &                              &                     & \ddots &                          &                  \\
		&         &       &        &            &                          &            &                              &                     &        & R_{k\cdot \theta_\gamma} &                  \\
		&         &       &        &            &                          &            &                              &                     &        &                          & 1
	\end{array}
	\right),$$
	%
	
	By dealing with each column (of width 1 or 2) independently, we can assume that we are in one of the three following elementary cases:
	\begin{enumerate}
		\item $A_k=1$ for every $k$;
		\item $A_k=(-1)^{k\cdot \epsilon}$ for every $k$, where $\epsilon\in\Z^r\smallsetminus 2\Z^r$;
		\item $A_k=R_{k\cdot \theta}$ for every $k$, where $\theta\in\R^r\smallsetminus 2\pi \Z^r$.
	\end{enumerate}
	
	{\bf Case 2.}
	Since $\epsilon\not\in 2\Z^r$ we can find $k$ such that $(-1)^{k\cdot \epsilon}=-1$, and the action of $k$ has a unique $-1$-eigenvector, which is invariant under the whole group $\Z^r$, and which makes $C_\ell$ and $D_\ell$ zero for every $\ell\in\Z^r$.
	
	{\bf Case 3.}
	Since $\theta\not\in \pi\Z^r$ we can find $k$ such that $R_{k\cdot \theta}$ is a nontrivial rotation, and the action of $k$ has a unique invariant plane in $\R^{d+1}$ where it acts as this rotation, which is invariant under the whole group $\Z^r$, and which hence makes $C_\ell$ and $D_\ell$ zero for every $\ell\in\Z^r$.
	
	{\bf Case 1.}
	Let us show that $C_k$ has to be zero for any $k$, no matter what choices for the basis have been made before.
	
	The action of $k$ is given by
	$$
	\begin{pmatrix}
		1 & k^t   & D_k & \frac {||k||^2}2\\
		0 & I_{r-1} & C_k & k\\
		0 & 0 & 1&0\\
		0 & 0 & 0 & 1
	\end{pmatrix}.
	$$
	The fact that we have a group action implies that for all $k,\ell\in\Z^r$ we have $C_{k+\ell}=C_k+C_{\ell}$ and $D_{k+\ell}=D_k+D_\ell+k^t \cdot C_\ell$.
	
	Hence there is a matrix $C$ such that $C_k=C\cdot k$ for every $k$.
	
	Since $D_{k+\ell}=D_{\ell+k}$ we have $k^t C\ell=\ell^t Ck$ for all $k,\ell$ and hence $C$ is symmetric.
	Set $D'_k:=D_k-\tfrac12 k^t Ck$ and note that $D'_{k+\ell}=D'_k+D'_\ell$ for all $k,\ell$, so there is a row vector $D$ such that $D_k=Dk + \tfrac12 k^t Ck$ for any $k$.
	
	The affine action of $k$ on the affine horizontal hyperplane with height $1$ is given by 
	$$
	v\mapsto k\cdot v=
	\begin{pmatrix}
		1 & k^t   & Dk+\tfrac12 k^t C k\\
		0 & I_{r-1} & Ck\\
		0 & 0 & 1
	\end{pmatrix}
	v+
	\begin{pmatrix}
		\frac {||k||^2}2\\
		k\\
		0
	\end{pmatrix}
	$$
	We know that for every $v$ the first entry of $k\cdot v$ must go to $+\infty$ as $k$ leaves every compact set.
	This implies that $C$ has to be zero.
	Indeed, if $x^t Cx\neq 0$ for some unit vector $x$ then we can find a diverging sequence $(k_n)_n$ with direction converging to $x$ such that $\tfrac12 k_n^t C k_n\sim \tfrac12 x^t C x||k_n||^2$ and taking $v$ with last entry $-2/x^t Cx$ and all other entries zero, the first entry of $k_n\cdot v$ is 
	$$O(||k_n||)+\frac12 x^t C x||k_n||^2\cdot (-2/x^t Cx)+\frac12||k_n||^2=O(||k_n||)-\frac12||k_n||^2\to-\infty,$$
	which is absurd.
	
	The action of $k$ on $\R^{d+1}$ is now given by 
	$$
	\begin{pmatrix}
		1 & k^t   & Dk & \frac {||k||^2}2\\
		0 & I_{r-1} & 0 & k\\
		0 & 0 & 1&0\\
		0 & 0 & 0 & 1
	\end{pmatrix}.
	$$
	Replacing the $(r+2)$-th vector of the basis by 
	$$
	\begin{pmatrix}
		0\\
		-D^t\\
		1\\
		0
	\end{pmatrix}
	$$
	yields a new action of the form
	$$
	\begin{pmatrix}
		1 & k^t   & 0 & \frac {||k||^2}2\\
		0 & I_{r-1} & 0 & k\\
		0 & 0 & 1&0\\
		0 & 0 & 0 & 1
	\end{pmatrix},
	$$
	which is what we needed to conclude the proof.
\end{proof}

\subsection{A correction on the proof of
	\eqref{item:GF_on_Omega} $\Rightarrow$ \eqref{item:convex_core}${}_1$}\label{sec:GF imp VF}

The implication \eqref{item:GF_on_Omega} $\Rightarrow$ \eqref{item:convex_core}${}_1$ is a consequence of  \eqref{item:GF_on_Omega} $\Rightarrow$\eqref{item:cusp_uniform} and the following lemma.

We adopt here a slightly different strategy than in \cite{CM2014finitude}.
There the crucial ingredient were estimates on the Hilbert volumes of cones in convex domains, with apex in the boundary, obtained in collaboration with Vernicos, using the Busemann volumes.
We think it is possible to adapt this strategy to prove  \eqref{item:convex_core}${}_1$ (instead of just  \eqref{item:convex_core}${}_0$), but it would be much more complicated.
Here upper estimates on volumes are obtained by covering our set with a well chosen collection of balls with the same radius.

\begin{lemma}[{\cite[p.\,48]{CM2014finitude}}]\label{lem:finite volume}
	Let $\G$ be a discrete group preserving a round convex open subset $\O$, and $p\in\LG$ a uniformly bounded parabolic point with stabiliser $\G_p$.
	Fix a closed horoball $H\subset \O$ at $p$ and let $N$ be the $1$-neighborhood of $H\cap \Cc(\LG)$ in $\O$.
	Then $N/\G_p$ has finite Hilbert volume.
\end{lemma}

\begin{proof}[Proof]
	Since $p$ is uniformly bounded parabolic, $\G_p$ acts cocompactly on $\Cone(p,\Cc(\LG))\cap \partial H \smallsetminus \{p\}$.
	Fixing a point $x_0$ in this set, there exists $R>0$ such that all the other points are at Hilbert distance at most $R$ from the $\G_p$-orbit of $x_0$.
	
	For each $t>0$ let $x_t$ be the point of $[x_0,p)$ at distance $t$ from $x_0$ and $H_t$ the horoball with $x_t$ in its boundary.
	Note that  $\Cone(p,\Cc(\LG))\cap \partial H_t \smallsetminus \{p\}$ is contained in $\G_p\cdot B_\O(x_t,R)$.
	Indeed if $z_t\in \Cone(p,\Cc(\LG))\cap \partial H_t \smallsetminus \{p\}$ then the line through $p$ and $z_t$ crosses $\partial H$ at some point $z_0$.
	Then there is $\g\in\G_p$ such that $\g z_0 \in  B_\O(x_0,R)$, and $d_\O(\g z_t,x_t)\leq d_\O(\g z_0,x_0)<R$ by Corollary~\ref{coro:crampon}.
	
	This implies that $N$ is contained in
	$$N\subset \bigcup_{n\in\N}\bigcup_{ \g\in\G_p}\g B_\O(x_n,R+2)=\bigcup_{n\in\N} B_\O(\G_p\cdot x_n,R+2).$$
	Let $\pi:\O\to\O/\G_p$ be the projection map and $\Vol$ the quotient measure on $\O/\G_p$.
	Then
	$$\Vol(\pi(N))\leq \sum_{n\in\N}\Vol(\pi(B_\O(x_n,R+2)))$$
	By definition of the quotient measure on $\O/\G_p$, to compute the volume of the quotient of $B_\O(x_n,R+2)$ one can either find a fundamental region for the action of $\G_p$ or one can consider all the points of $B_\O(x_n,R+2)$ and then divide by the number of orbit points in $B_\O(x_n,R+2)$:
	$$\Vol(\pi(B_\O(x_n,R+2)))=\int_{x\in B_\O(x_n,R+2))}\frac1{\#\{\g\in\G_p:\g x\in B_\O(x_n,R+2)\}}d\Vol_\O(x).$$
	Here we use this second idea, except that we apply it to $V=B_\O(\G_p\cdot x_n,R+2)\cap B_\O(x_n,R+3)$ instead of $B_\O(x_n,R+2)$, both have the same projection under $\pi$.
	Note that if $x\in V$ then $\gamma_0x \in B_\Omega(x_n,R+2)$ for some $\gamma_0$, and then for each $\gamma'\in \Gamma$, if $\g' x_n\in B_\O(x_n,1)$ then $\g' \g_0x\in B_\O(x_n,R+3)$.
	Using this we observe the following.
	\begin{align*}
		\Vol(\pi(B_\O(x_n,R+2))) = \Vol(\pi(V))
		&= \int_{x\in V}\frac1{\#\{\g\in\G_p:\g x\in V\}}d\Vol_\O(x)\\
		& \leq \int_{x\in V}\frac1{\#\{\g\in\G_p:\g x\in B_\O(x_n,R+3)\}}d\Vol_\O(x)\\
		&\leq \frac{\Vol_\O(B_\O(x_n,R+3))}{\#\{\g'\in\G_p:\g' x_n\in B_\O(x_n,1)\}} \\
		& \leq \frac{C}{\#\{\g\in\G_p:\g x_n\in B_\O(x_n,1)\}},
	\end{align*}
	where $C$ is a constant that depends on $R$ (see for instance \cite[Th.\,12]{colbois_vernicos_spectre}).
	
	From Proposition~\ref{prop:oscultation} we know there is a $\G_p$-invariant ellipsoid $\Ec\subset \Cc(\LG)$ of dimension 1 plus the rank of $\G_p$ tangent to $p$.
	
	Up to shrinking $H$ and making a different choice of $x_0$, we can assume that $x_0\in \Ec$.
	Then  the Hilbert distance in $\Ec$ from $x_t$ to $x_0$ is $t$ plus some constant.
	With this it is classical to deduce that that $\#\{\g\in\G_p:\g x_n\in B_{\Ec}(x_n,1)\}$ increases exponentially fast with $n$, and hence so does the bigger number $\#\{\g\in\G_p:\g x_n\in B_\O(x_n,1)\}$ (recall that distances in $\O$ are smaller than in $\Ec$ since $\Ec\subset\O$, see e.g.\,\cite[\textsection 2.1]{CM2014finitude}).
	This makes $\sum_{n\in\N} \mathrm{Vol} (\pi(N\cap B_\O(\G_p\cdot x_n,R+2)))$ summable and concludes the proof.
\end{proof}

\subsection{A correction on the proof of \eqref{item:GF_on_Omega} $\Rightarrow$ (\eqref{item:GF_on_boundary}$\&$\eqref{item:convex_core_ghyp})}\label{sec:GF imp Hyp}

\begin{lemma}[{\cite[Lem.\,9.3]{CM2014finitude}}]
	Let $\G$ be a discrete group preserving a round convex open subset $\O$. If $\G$ acts \eqref{item:GF_on_Omega} on $\Omega$ then the metric space $(\Cc(\LG),d_\O)$ is Gromov-hyperbolic.
\end{lemma}

The lemma is correct but there is a mistake at the end of the proof in \cite[Lem.\,9.3]{CM2014finitude}. 
Let us reproduce the proof, with some minor modifications, to exhibit the mistake and explain how to fix it.
One can assume that $\LG$ spans $\R\PP^d$ (up to restricting to the span).
The proof works by contradiction: one assumes there is a sequence of fatter and fatter triangles in the convex core $C=\Cc(\LG)$, with vertices $x_n,y_n,z_n$ and a point $u_n$ on the side $[x_n,y_n]$ whose Hilbert distance to the other sides goes to infinity.

If the projection of $u_n$ in $C/\G$ stayed in a compact set, then up to translating the sequence of triangles and extracting a subsequence we could assume $u_n$ converges to a point $u\in\O$ while $x_n,y_n,z_n$ converge to points $x,y,z\in\partial\O$, but then by strict convexity of $\O$, the point $u$ would be at finite distance from one of the sides of the (possibly degenerate) triangle $(x,y,z)$.
Thus the projection of $u_n$ in $C/\G$ does not stay in a compact set, and up to extraction we may assume it is contained in a single cusp and leaves every compact set (using that the action is geometrically finite).

Up to translating the triangles we can then assume that $u_n$ lies in a fixed horoball $H$ of $\O$ about a uniformly bounded parabolic point $u\in\partial\O$.
Then $u_n\to u$.
Up to translating again with elements of $\Gamma_u$, we may also assume that the intersection point $h_n\in\partial H\cap (u\,u_n)$ converges to a point $h\in\O$ (here $(u\,u_n)$ denotes the line spanned by $u$ and $u_n$).

Letting ${\rm Co}$ be the cone at $u$ spanned by $C$, by Proposition~\ref{prop:oscultation} there are $\G_u$-invariant osculating ellipsoids $\Ec^{\rm int}\subset\Ec^{\rm ext}$ such that ${\rm Co}\cap \Ec^{\rm{int}} \subset {\rm Co}\cap \O \subset {\rm Co}\cap \Ec^{\rm{ext}}$, and there is a $\G_u$-invariant subspace $S\subset\R\PP^d$ of dimension one plus the rank of $\G_u$, such that it intersects $\O$ and $S\cap \O\subset {\rm Co}$ (see Proposition~\ref{prop:oscultation} and its proof).
One can check that the Hilbert distance from $u_n$ to $S\cap \O$ tends to zero: fix $o\in S\cap C$ and recall that, because $u$ is a $\Cc^1$ point of $\partial\O$, the Hilbert distance between the rays $[h_n,u)$ and $[o,u)$ tends to zero as we get closer to $u$, so there is $o_n\in[o,u)\subset S\cap C$ whose distance to $u_n$ tends to zero (see for instance \cite[Lem.\,3.4]{CD1}).

Now we fix a one-parameter subgroup $(\g^t)_t$ of hyperbolic automorphisms of $\Ec^{\rm ext}$ that preserve $S$, $u$ and the line $(o\,u)$ and use it to recenter the whole picture: we select times $k_n$ such that $\g^{k_n}o_n=o$.




Now comes the small error: The sentence ``Comme $\g^{k_n}\O\cap\Ec^{\rm ext}$ est coinc\'e entre $\gamma^{k_n} (\Ec^{\rm{int}} )$ et $\Ec^{\rm{ext}}$'' is incorrect (only when restricting to ${\rm Co}$ does the inclusion $\g^{k_n}({\rm Co}\cap \Ec^{\rm{int}}) \subset \g^{k_n}(\O)$ become true), hence the conclusion ``la suite de convexes $(\g^{k_n}(\O)\cap\Ec^{\rm ext})$ tend, tout comme $(\gamma^{k_n} (\Ec^{\rm{int}} ))$, vers $\Ec^{\rm{ext}}$'' is incorrect too, and there are examples where $\g^{k_n}(\O)$ can converge to a convex with empty interior.
To make the proof work we must recenter the picture via a sequence slightly different than $(\gamma^{k_n})_n$.


By \cite[Lem.~2.8]{benoist_conv_hyp} of Benoist (following Benz\'ecri), and up to extracting a subsequence, there exists a sequence $(g_n)_n$ of projective transformations such that $\O_n := g_n (\O) \to \O_{\infty}$ and $g_n$  equal to $\gamma^{k_n}$ in restriction to $S$, and also such that $\O_\infty$ intersects $S$.
Note that $g_n(\O\cap S)$ converges to $\Ec_{\rm ext}\cap S$ and also to $\O_\infty\cap S$, which is therefore an ellipsoid.
Note also that since $d_\O(u_n,o_n)\to 0$ and $g_n(o_n)=o$, we have $g_n(u_n)\to o$.

Suppose, up to extracting again, that the closures of the images of the convex core $C_n=g_n(C)$ converge to a closed convex set $\overline C_\infty$.
To conclude, it suffices to show that $\overline C_\infty\subset S$.
Indeed we can then conclude as earlier: we know $u_n':=g_n(u_n)\to o$ and up to extracting we also have $x_n':=g_n(x_n)\to x'$ and $y_n':=g_n(y_n)\to y'$ and $z_n':=g_n(z_n)\to z'$, all three contained in $\overline C_\infty\cap \partial\O_\infty$, hence $S\cap \partial\O_\infty$ which is an ellipsoid.
By strict convexity of ellipsoids this means $u'$ is at finite Hilbert distance from $[x',z']$ or $[y',z']$, which contradicts that the Hilbert distance from $u'_n$ to $[x'_n,z'_n]\cup[y'_n,z'_n]$ goes to infinity.


Assume that there exists $(\simple'_n)_n$ in $C_n$ such that $\simple'_n \to \simple'_\infty \in \overline C_\infty$ but $\simple'_\infty \notin S$.
In particular, $d_{\O_n} (\simple'_n,u'_n)$ is bounded from below by some constant $\epsilon>0$ (since $u_n'\to o\in S$).
Let $\simple''_n$ be the intersection point of $[\simple'_n,u'_n]$ with the Hilbert sphere of radius $\epsilon$ around $u_n'$.
Then $(\simple''_n)_n$ is a sequence in $C_n$ converging to a point of $[o,\simple'_\infty]\subset \overline C_\infty$, which is not $o$, so this limit point is not in $S$ (otherwise $\simple'_\infty$ would be too).
To simplify the notations we can assume that $\simple'_n=\simple''_n$, so that $d_{\O_n} (\simple'_n,u'_n)=\epsilon$ for all $n$.

Let $\simple_n = g_n^{-1} (\simple'_n)$, which remains at bounded Hilbert distance from $u_n$, and hence like $u_n$ converges to $u$.
Denote by $a_n$ the intersection point of the line $(u \, \simple_n)$ with $\partial \O$, which is not $u$. 
Using that $u$ is uniformly bounded parabolic we can find $\pi_n \in \G_u$ such that $b_n := \pi_n (a_n) \to b_\infty\neq u$ after possibly extracting.

Let $\double_n := \pi_n (\simple_n)$, and observe it also converges to $u$.
Recall $o \in S \cap \O$.
Since $u$ is a $\Cc^1$-point of the boundary of $\O$, there exists a $\triple_n \in [o,u)$ such that $d_\O (\double_n, \triple_n)  \to 0$ (see again \cite[Lem.\,3.4]{CD1}). 

Hence, the Hilbert distance $d_{\O_n} (g_n \circ \pi_n^{-1} (\triple_n) , \simple'_n)$ tends to zero, whereas $g_n \circ \pi_n^{-1} (\triple_n)$ is in $S$. Contradiction. 

\subsection{Counterexample to \eqref{item:convex_core}${}_1$ $\Rightarrow$  \eqref{item:GF_on_Omega} and (\eqref{item:GF_on_boundary}$\&$\eqref{item:convex_core_ghyp}) $\Rightarrow$ \eqref{item:GF_on_Omega}: an overview}\label{sec:overview}

The counterexamples to the implications \eqref{item:convex_core}${}_1$ $\Rightarrow$  \eqref{item:GF_on_Omega} and (\eqref{item:GF_on_boundary}$\&$\eqref{item:convex_core_ghyp}) $\Rightarrow$ \eqref{item:GF_on_Omega} are similar to the example in Section~\ref{sec:contrex}, in the sense that they also come from a representation $\rho$ of $\SL_2(\R)$ into $\SL_5(\R)$ that preserves round convex domains of $\R\PP^4$, except that this time we use an irreducible representation of $\SL_2(\R)$, following \cite[\textsection 10.3]{CM2014finitude}.

As in Section~\ref{sec:contrex}, these $\rho(\SL_2(\R))$-invariant domains $\O$ can be described explicitly, as well as the convex hull $\Cc$ of the proximal limit set, and $\SL_2(\R)$ acts cocompactly (but not transitively) on $\Cc$.

If $\G\subset\SL_2(\R)$ is a noncocompact lattice, then $\rho(\G)$ does not act geometrically finitely on $\O$, because the maximal parabolic subgroups are not conjugate into ${\rm O}_{4,1}(\R)$ and the limit set spans the whole $\R\PP^4$ (see Proposition~\ref{prop:oscultation}).
However, that $\SL_2(\R)$ acts cocompactly on $\Cc$ implies that $\Cc$ is quasi-isometric to $\SL_2(\R)$, and hence is Gromov-hyperbolic.
Moreover, using the ideas from the proof of Lemma~\ref{lem:finite volume}, one can further prove that the quotient under $\rho(\G)$ of the $1$-neighborhood of $\Cc$ has finite volume.

All the above will be written in a forthcoming paper \cite{Next}.
We will also include other kinds of counterexamples, to  \eqref{item:GF_on_boundary} $\Rightarrow$ \eqref{item:convex_core}${}_1$ and  \eqref{item:GF_on_boundary} $\Rightarrow$ \eqref{item:convex_core_ghyp}.
In fact these examples will involve the same groups $\rho(\G)$ (where $\rho$ is the irreducible representation of $\SL_2(\R)$ in $\SL_5(\R)$ and $\G$ a noncocompact lattice of $\SL_2(\R)$), but with different more subtle $\rho(\G)$-invariant round domains, which are not invariant under the whole $\rho(\SL_2(\R))$.

\bibliographystyle{alpha}

\renewcommand\refname{Bibliography}

\backmatter
\end{document}